\documentclass[a4paper, 11pt, onecolumn]{article} 

\usepackage[margin=2.5cm]{geometry}

\usepackage{xcolor}
\usepackage{amsmath}
\usepackage{amsthm}
\usepackage{indentfirst}
\usepackage{bm}
\usepackage{amsfonts,amssymb}
\usepackage{graphicx}
\usepackage{mathrsfs}
\usepackage{subfigure}
\usepackage{graphicx}
\usepackage{multirow}
\usepackage{float}
\usepackage{algorithmic}
\usepackage{algorithm}
\usepackage{paralist}
\usepackage{url}

\newtheorem{definition}{Definition}
\newtheorem{theorem}{Theorem}
\newtheorem{lemma}{Lemma}
\newtheorem{proposition}{Proposition}
\newtheorem{corollary}{Corollary}
\newtheorem{remark}{Remark}
\newcommand{\Poincare}{Poincar\'e}

\DeclareMathOperator{\Shrink}{Shrink}
\DeclareMathOperator{\Div}{div}
\renewcommand{\div}{\Div}
\newcommand{\RR}{\mathbb{R}}
\newcommand{\CC}{\mathbb{C}}
\newcommand{\ZZ}{\mathbb{Z}}
\DeclareMathOperator{\loc}{loc}
\DeclareMathOperator{\dom}{dom}

\DeclareMathOperator{\Ran}{ran}
\DeclareMathOperator{\TV}{TV}
\DeclareMathOperator{\BV}{BV}
\DeclareMathOperator{\BD}{BD}
\DeclareMathOperator{\BGV}{BGV}
\DeclareMathOperator{\TGV}{TGV}
\DeclareMathOperator{\ICTGV}{ICTGV}
\DeclareMathOperator{\osci}{osci}
\DeclareMathOperator{\dist}{dist}
\DeclareMathOperator{\linspan}{span}
\DeclareMathOperator{\range}{ran}
\DeclareMathOperator*{\argmin}{arg\,min}

\newcommand{\infconv}{\hspace{0.08333em}\Box\hspace{0.08333em}}
\newcommand{\inprod}{\cdot}
\newcommand{\tensor}{\otimes}
\newcommand{\symgrad}{\mathcal{E}}
\newcommand{\grad}{\nabla}
\newcommand{\set}[2]{\{#1:#2\}}
\newcommand{\sett}[1]{\{#1\}}
\newcommand{\radon}{\mathcal{M}}
\newcommand{\expE}{\mathrm{e}}
\newcommand{\im}{\mathrm{i}}
\newcommand{\conj}[1]{\overline{#1}}
\newcommand{\seq}[1]{\{#1\}}

\newcommand{\open}[2]{{]{#1,#2}[}}
\newcommand{\leftopen}[2]{{]{#1,#2}]}}

\newcommand{\abs}[2][]{|{#2}|_{#1}}
\newcommand{\norm}[2][]{\|{#2}\|_{#1}}
\newcommand{\scp}[3][]{\langle{#2},{#3}\rangle_{#1}}
\newcommand{\dd}[1]{\,\mathrm{d}{#1}}
\newcommand{\bdry}{\partial}

\newcommand{\embeds}{\hookrightarrow}

\newcommand{\tabincell}[2]{\begin{tabular}{@{}#1@{}}#2\end{tabular}}

\begin{document}

{ \renewcommand{\addcontentsline}[3]{}%
	
\title{Infimal convolution of oscillation total generalized variation
  for the recovery of images with structured texture} \author{Yiming
  Gao%
  \thanks{Department of Mathematics, Nanjing University of Science and
    Technology, 210094 Nanjing, Jiangsu, China. Email:
    \texttt{gao\_yiming@hotmail.com}.} \and Kristian Bredies%
  \thanks{Institute of Mathematics and Scientific Computing,
    University of Graz, Heinrichstra\ss{}e 36, A-8010 Graz,
    Austria. Email: \texttt{kristian.bredies@uni-graz.at}. The
    Institute for Mathematics and Scientific Computing is a member of
    NAWI Graz (\texttt{http://www.nawigraz.at/}).}  } \date{\today}
\maketitle

\paragraph{Abstract:}
We propose a new type of regularization functional for images called
oscillation total generalized variation (TGV) which can represent
structured textures with oscillatory character in a specified
direction and scale. The infimal convolution of oscillation TGV with
respect to several directions and scales is then used to model images
with structured oscillatory texture. Such functionals constitute a
regularizer with good texture preservation properties and can flexibly
be incorporated into many imaging problems. We give a detailed
theoretical analysis of the infimal-convolution-type model with
oscillation TGV in function spaces. Furthermore, we consider
appropriate discretizations of these functionals and introduce a
first-order primal-dual algorithm for solving general variational
imaging problems associated with this regularizer. Finally, numerical
experiments are presented which show that our proposed models can
recover textures well and are competitive in comparison to existing
state-of-the-art methods.



\section{Introduction}
\label{sec:introduction}

After the introduction of total variation (TV) to image processing
\cite{ROF}, variational models were developed for a wide range of
applications in imaging during the last decades.
In that context, variational models that base on image decomposition
have also become increasingly popular. The concept of these models is
to decompose an image into two or more components which inherit
distinguishing characteristics.  The classical TV (ROF) model for
image denoising can, for instance, be regarded as a decomposition
model:
\begin{equation}
\min_{\substack{u\in \BV(\Omega), \ v \in L^2(\Omega)\\ u + v = f}}
\ \frac{\lambda}2 \norm{v}^2 + J(u),
\end{equation}
where $f$ is the observed noisy image, $J(u)=\int_\Omega |Du|$ is the
total variation, i.e., the Radon norm of the distributional
derivative, and $\|\cdot\|$ represents the $L^2$ norm, both associated
with a bounded domain $\Omega \subset
\RR^d$.  %
In such an approach, we decompose the image to a noise-free image
component $u$ belonging to $\BV(\Omega)$ and a noise or small-scale
texture component in $L^2(\Omega)$. In \cite{meyer}, Y.~Meyer pointed
out some limitations of the above model and introduced a new space
$G(\Omega)=W^{-1,\infty}(\Omega)$, which is larger than $L^2(\Omega)$,
to model %
oscillating patterns that are typical for structured texture:
\begin{equation}
\label{eq:tv_g_norm}
\min_{\substack{u \in \BV(\Omega), \ v \in G(\Omega)\\ u+v = f}} 
\lambda \norm[G]{v} + J(u),
\end{equation}
where $G(\Omega)$ denotes the space
$G(\Omega)=\{v=\div g:g=(g_1,\ldots,g_d)\in L^\infty(\Omega)^d\}$ with
the norm
$\|v\|_G=\inf\ \{\|g\|_\infty:v=\div g,g=(g_1,\ldots,g_d)\in
L^\infty(\Omega)^d\}$
and again, $J(u) = \int_\Omega \abs{Du}$.  In this model, $u$
represents a piecewise constant function, consisting of homogeneous
regions with sharp boundaries and is usually called cartoon
component. The other part $v$ in $G(\Omega)$ contains oscillating
patterns, like textures and noise. This model can be considered as the
original cartoon-texture decomposition model. However, the $G$-space
may be difficult to handle in implementations. Vese and Osher
\cite{vesedecom} approximated the $G$-space by a Sobolev space of
negative differentiability order $W^{-1,p}(\Omega)$, which can, in
practice, more easily be implemented by partial differential
equations. Later, Aujol et al.~\cite{aujoldecom} made a
modification to the $G$-norm term in~\eqref{eq:tv_g_norm} which is
replaced by constraining $v$ to the set
$G_\mu(\Omega)=\{v\in G(\Omega):\|v\|_G\leq\mu\}$. As a consequence,
the problem can be solved alternatingly by Chambolle's projection
method \cite{Chambolle} with respect to the two variables $u$ and $v$.

Certainly, the aim of above models is to find an appropriate norm to
describe textures. However, in general, the above norms can represent
all kinds of oscillatory parts. As a consequence, since noise can be
regarded as small-scale oscillating texture, these models do not deal
well with noise. In order to tackle this disadvantage, in
\cite{lowrank}, Schaeffer and Osher incorporated robust principal
component analysis (PCA) \cite{RPCA} into cartoon-texture
decomposition models based on the patch method. For the texture part,
the authors suggest to decompose the image into small, non-overlapping
patches, written as vectors. The collection of the patch vectors is
assumed to be (highly) linearly dependent and thus to have low
rank. This suggests to minimize the nuclear norm of the patch-vector
matrix. However, the choice of the block-size of the patches for this
model is challenging and also, the numerical optimization procedures
for the associated functionals have to rely on computationally
expensive matrix-decomposition methods such as singular-value
thresholding, for instance.  Alternatively, rather than considering
the cartoon and texture components separately, another efficient
texture preservation approach is using non-local methods which is also
based on patches. For more details, we refer readers to
\cite{nlm1,nlm2,nltv1,nltv2}. Furthermore, in order to enhance and
reconstruct line structures well, Holler and Kunisch \cite{ICTVMartin}
proposed an infimal convolution of TGV-type~(ICTGV) functionals which
can capture some directional line textures. Disadvantages are, of
course, the bias of this regularization approach towards certain
directions and the fact that we cannot expect to recover texture that
is not composed of different line structures. In \cite{DTGV}, Kongskov
and Dong proposed a new directional total generalized variation (DTGV)
functional to capture the directional information of an image, which
can be considered as a special case of ICTGV.

Another approach to decompose an image is to compute sparse
representations based on dictionaries. %
The basic idea is to find two suitable dictionaries, one adapted to
represent textures, and the other to represent the cartoon part. In
the papers \cite{starck1,starck2}, Elad et al.~combined the
basis pursuit denoising (BPDN) algorithm and the TV-regularization
scheme to the following decomposition model:
\begin{equation}
\min_{\alpha_{\mathrm{t}},\alpha_{\mathrm{n}}} \|\alpha_{\mathrm{t}}\|_1+\|\alpha_{\mathrm{n}}\|_1+\frac{\lambda}{2}\|f-T_{\mathrm{t}}\alpha_{\mathrm{t}}-T_{\mathrm{n}}\alpha_{\mathrm{n}}\|^2+\gamma J(T_{\mathrm{n}}\alpha_{\mathrm{n}}),
\end{equation}
where $T_{\mathrm{t}},T_{\mathrm{n}}$ are dictionary synthesis
operators for the texture and cartoon parts
$\alpha_{\mathrm{t}},\alpha_{\mathrm{n}}$, respectively, and $J$ is again the
total variation. Penalizing with TV forces the cartoon content
$T_{\mathrm{n}}\alpha_{\mathrm{n}}$ to have a sparser gradient and to
be closer to a piecewise constant image. Here, the choice of the
dictionary synthesis operators $T_{\mathrm{t}},T_{\mathrm{n}}$ is the
main challenge for this approach. %
The best choice for one image depends on experience and may be not
suitable for another image with different structures. The authors
recommended that for the texture content, one may use local discrete
cosine transform (DCT) or the Gabor transform, and for the cartoon
content one may use curvelet, ridgelets, contourlets and so on.
This approach is reported to perform well for the restoration of
texture images \cite{starck1,starck2}.  Inspired by this, Cai, Osher
and Shen \cite{framelet} constructed a cartoon-texture model with
tight framelets \cite{frameletinp} and local DCT which can easily be
realized algorithmically by split Bregman iterations. %

Recently, filtering approaches have been widely used in image
decomposition. Gilboa et al.~proposed a generalization of the
structure-texture decomposition method based on TV scale space which
is also known as TV flow in
\cite{SpectralDecom1,SpectralDecom2,Horesh}. The aim is to construct a
TV spectral framework to analyze dominant features in the transform
domain, filter in that domain, and then perform an inverse transform
back to the spatial domain to obtain the filtered response.  In
\cite{SpectralDecom3}, Burger et al.~extended this notion to general
one-homogeneous functionals, not only TV regularization. Furthermore,
Buades et al.~proposed a non-linear filter that decomposes the image
into a geometric and an oscillatory part in \cite{Buades1,Buades2}, in
which a non-linear low-high frequency decomposition is computed by
using isotropic filters. However, the cartoon parts produced by these
filters are not well-structured and textures also contain other
oscillatory components. Furthermore, these methods are not variational
models and can not easily be transferred to other image processing
problems such as solving inverse problems, for
instance. %

In this paper, we propose a new regularization which can represent
structured texture that consists of different directions and frequencies.
The regularization is called oscillation total generalized variation (oscillation TGV) and defined as follows:
\begin{equation}\label{TTGV}
  \TGV^{\osci}_{\alpha,\beta,\mathbf{c}}(u) =\min_{w\in \text{BD}(\Omega)} \alpha\|\grad u-w\|_\mathcal{M}+\beta\|\mathcal{E}w+\mathbf{c}u\|_\mathcal{M}, \quad\text{for $u\in L^1(\Omega)$},
\end{equation}
where $\alpha,\beta>0$, $\BD(\Omega)$ denotes the space of vector
fields of bounded deformation, the operator $\grad$ denotes the weak
derivative, the operator $\mathcal{E}$ is the weak symmetrized
derivative $\mathcal{E}w=\frac{1}{2}(\nabla w+\nabla w^T)$ and
$\|\cdot\|_\mathcal{M}$ is the Radon norm. 
Moreover,
$\mathbf{c} \in \RR^{d \times d}$ and $\mathbf{c}u$ corresponds to the
matrix field given by
$(\mathbf{c} u)_{ij}=c_{ij}u$ %
for $i,j=1,\ldots,d$. In~\eqref{TTGV}, we set
$\mathbf{c} = \omega \tensor \omega$ for some $\omega \in \RR^d$,
i.e., $c_{ij} = \omega_i\omega_j$. 
For a precise description
of these notions, we ask for the reader's patience until
Section~\ref{sec:osci_tgv}.
This regularization functional reduces to the
original second-order TGV functional when $\mathbf{c}=0$, i.e.,
\begin{equation}
  \label{eq:TGV}
  \TGV_{\alpha,\beta}^2(u) = \min_{w \in \BD(\Omega)} \alpha \norm[\radon]{\grad u - w} + \beta \norm[\radon]{\symgrad w},
\end{equation}
which has been widely-used in imaging applications, such as, for
instance, denoising, (dynamic) MRI and QSM reconstruction, diffusion
tensor image restoration, JPEG and MPEG decompression, zooming,
optical flow estimation
\cite{TGVdecom,TGVdecom2,TGVMPEG,TGVinv2,TGV,TGVinv1,ICTVMartin,TGVMR,MRPET,TGVqsm,TGVoptical,DynamicMR,TGVTensor},
and many more.
Furthermore, we consider $m$-fold infimal convolution of oscillation TGV:
\begin{align}
  \notag
  \ICTGV^{\osci}_{\vec{\alpha},\vec{\beta},\vec{\mathbf{c}}}(u) 
  &= (\TGV^{\osci}_{\alpha_1,\beta_1,\mathbf{c}_1}\infconv\ldots\infconv \TGV^{\osci}_{\alpha_m,\beta_m,\mathbf{c}_m})(u) \\
  \label{ICTGV}
  &= \inf_{u_1 + \ldots + u_m = u} \ \sum_{i=1}^m 
    \TGV^{\osci}_{\alpha_i,\beta_i, \mathbf{c}_i}(u_i)
\end{align}
where $\vec{\alpha}=(\alpha_1\ldots,\alpha_m)$,
$\vec{\beta}=(\beta_1,\ldots,\beta_m)$ and
$\vec{\mathbf{c}}=(\mathbf{c}_1,\ldots,\mathbf{c}_m)$. In
applications, $\mathbf{c}_1=0$ such that $u_1$ represents the cartoon
part interpreted as piecewise smooth function, and $u_2,\ldots,u_m$
are the texture components with the directions depending on the values
of $\mathbf{c}_2,\ldots,\mathbf{c}_m$. We give a detailed theoretical
analysis in an infinite-dimensional function space setting and apply
regularization with
$\ICTGV^{\osci}_{\vec{\alpha},\vec{\beta},\vec{\mathbf{c}}}$ to
several image processing problems.

The paper is organized as follows. In Section~\ref{sec:osci_tgv}, we
start with the introduction of the oscillation TGV functional which,
on the one hand, measures smoothness but possesses, on the other hand,
a kernel consisting of structured oscillations. Then, we give a
detailed analysis of it and show some main properties. In
Section~\ref{seq:inf_conv_osci_tgv}, we focus on
$\ICTGV^{\osci}_{\vec{\alpha},\vec{\beta},\vec{\mathbf{c}}}$. First,
we establish a theoretic analysis for the general functionals of
infimal-convolution type and apply it to infimal convolutions of
oscillation TGV.  As a result, we obtain well-posedness for Tikhonov
regularization of linear inverse problems with
$\ICTGV^{\osci}$-penalty.  The discretization and a first-order
primal-dual optimization algorithm for solving Tikhonov-regularized
linear inverse problems with infimal-convolution of oscillation TGV
regularization are presented in Section~\ref{sec:discretization}.  In
particular, we show how the differential operators in $\TGV^{\osci}$
have to be discretized such that the discrete kernel coincides with
discrete structured oscillations. In Section~\ref{sec:applications},
we illustrate our approach for several imaging problems such as
denoising, inpainting and MRI reconstruction,
and demonstrate the efficiency of our new regularization compared to a
variety of methods, including state-of-the-art methods. Finally,
conclusions are drawn in Section~\ref{sec:conclusions}.

\section{Oscillation TGV and basic properties}
\label{sec:osci_tgv}

This section is devoted to the introduction of oscillation total
generalized variation ($\TGV^{\osci}$) and the derivation of some
immediate properties.  Throughout this paper, we will focus on
second-order TGV and its generalizations although further
generalization to higher orders might also be possible. 

Let us first introduce the required spaces and notions.
For simplicity, we assume for the rest of the paper that
$\Omega\subset \RR^d$ is a bounded Lipschitz domain where $d \geq 1$
is the image dimension.  If $\mu$ is an $X$-valued distribution on
$\Omega$ where $X \in \sett{\RR, \RR^d, S^{d \times d}}$ and
$S^{d \times d}$ denotes the set of symmetric $d \times d$ matrices
equipped with the Frobenius matrix inner product and norm, then the
Radon norm of $\mu$ is given as
\[
\norm[\radon]{\mu} = \sup \ \set{\scp{\mu}{\phi}}{\phi \in C_c^{\infty}(\Omega,X),
	\ \norm[\infty]{\phi} \leq 1}
\]
where $C_c^\infty(\Omega,X)$ denotes the space of $X$-valued test
functions on $\Omega$, i.e., the set of arbitrarily smooth
$\phi: \Omega \to X$ with compact support, and $\scp{\cdot}{\cdot}$
denotes the duality pairing of distributions and test functions. Here,
the supremum is $\infty$ whenever the set is unbounded from above. In
case $\norm[\radon]{\mu} < \infty$, the distribution $\mu$ can
continuously be extended, by density, to a Radon measure in
$\radon(\Omega,X) = C_0(\Omega,X)^*$ in which case the Radon norm is
the dual norm and we can identify $\mu$ with its corresponding
measure. Note that in particular, the density of
$\set{\phi \in C_c^\infty(\Omega,X)}{\norm[\infty]{\phi} \leq 1}$ in
the closed unit ball of $C_0(\Omega,X)$ can easily be established
using smooth cut-off functions and the fact that for non-negative
mollifiers $\rho$ we have
$\norm[\infty]{\rho \ast \phi} \leq \norm[\infty]{\phi}$.

A $\RR$-valued distribution $u$ always admits a $\RR^d$-valued
distribution as weak derivative which we denote by $\nabla u$ and
which is given by $\scp{\nabla u}{\phi} = \scp{u}{-\div \phi}$ for
each $\phi \in C_c^\infty(\Omega, \RR^d)$. Likewise, for an
$\RR^d$-valued distribution $w$, the weak symmetrized derivative is a
$S^{d \times d}$-valued distribution which is denoted by $\symgrad w$
and defined by $\scp{\symgrad w}{\phi} = \scp{w}{-\div \phi}$ for each
$\phi \in C_c^\infty(\Omega,S^{d \times d})$. With the usual (tacit)
identification of functions in $u \in L^1_{\loc}(\Omega, X)$ with
$X$-valued distributions via
$T_u(\phi) = \int_\Omega u \inprod \phi \dd{x}$ for
$\phi \in C_c^\infty(\Omega, X)$, the spaces of bounded variation
$\BV(\Omega)$ and bounded deformation $\BD(\Omega)$ are given by
\[
\begin{aligned}
\BV(\Omega) &= \set{u \in L^1(\Omega)}{\norm[\radon]{\grad u} <
	\infty}, & \norm[\BV]{u} &= \norm[1]{u} + \norm[\radon]{\grad u}, \\
\BD(\Omega) &= \set{w \in L^1(\Omega,\RR^d)}{\norm[\radon]{\symgrad
		w} < \infty}, & \norm[\BD]{w} &= \norm[1]{u} +
\norm[\radon]{\symgrad w}.
\end{aligned}
\]
These spaces are well-studied Banach spaces and we refer to
\cite{AmbrosioBD,BV1,BV2,TemamBD} for an overview and more details.

\subsection{Definition of oscillation TGV}

We start with a motivation for the definition as in~\eqref{TTGV}.
Observe that functions $u$ representing sinusoidal oscillations along
a fixed direction are given as follows:
\begin{equation}\label{oscifunc}
u(x)=C_1\cos(\omega \inprod x)+C_2\sin(\omega \inprod x)
\quad \text{for} \quad
x \in \Omega.
\end{equation}
Here, $C_1, C_2 \in \RR$ are constants, $\omega\in\RR^d$,
$\omega \neq 0$ is a direction vector and $\inprod$ represents the
standard inner product of vectors in
$\RR^d$. Figure~\ref{texture_conti} shows a visualization of these
oscillatory functions~\eqref{oscifunc} in $\Omega\subset\RR^2$ with
different values of $\omega_1$ and $\omega_2$. Our goal is to find a
differential equation that is, in some sense, compatible with the
second-order TGV functional according to \cite{TGV} and such that the
functions $u$ according~\eqref{oscifunc} are exactly the solutions for
all $C_1,C_2 \in \RR$. For this purpose, consider the equation
\begin{equation}\label{kerfunc}
\mathcal{E}\nabla u+\mathbf{c}u=0,
\end{equation}
where $\mathbf{c}=\omega\otimes\omega$ is the tensor (or outer)
product, i.e., $c_{ij}=\omega_i\omega_j$, $i,j=1,\ldots,d$.  Indeed,
this equation already possesses the desired properties.

\begin{figure}
	\center{} 
	\subfigure[]{%
		\begin{minipage}{0.15\linewidth}
			\includegraphics[width=\textwidth]{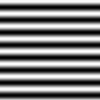}
			\\[-1em]
	\end{minipage}}
	\subfigure[]{%
		\begin{minipage}{0.15\linewidth}
			\includegraphics[width=\textwidth]{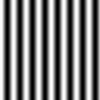}
			\\[-1em]
	\end{minipage}}
	\subfigure[]{%
		\begin{minipage}{0.15\linewidth}
			\includegraphics[width=\textwidth]{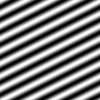}
			\\[-1em]
	\end{minipage}}
	\subfigure[]{%
		\begin{minipage}{0.15\linewidth}
			\includegraphics[width=\textwidth]{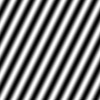}
			\\[-1em]
	\end{minipage}}
	\subfigure[]{%
		\begin{minipage}{0.15\linewidth}
			\includegraphics[width=\textwidth]{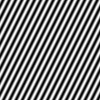}
			\\[-1em]
	\end{minipage}}
	\subfigure[]{%
		\begin{minipage}{0.15\linewidth}
			\includegraphics[width=\textwidth]{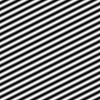}
			\\[-1em]
	\end{minipage}}
	\caption{Visualization of sinusoidal oscillations on $\Omega\subset \RR^2$ according to~\eqref{oscifunc}. (a) $\omega_1=0,\omega_2=0.5$; (b) $\omega_1=0.5,\omega_2=0$; (c) $\omega_1=0.25,\omega_2=0.5$; (d) $\omega_1=0.5,\omega_2=0.25$; (e) $\omega_1=1,\omega_2=0.5$; (f) $\omega_1=0.5,\omega_2=1$.}
	\label{texture_conti}
\end{figure}

\begin{lemma}\label{contiker}
  Let $\mathbf{c} = \omega \tensor \omega$ with $\omega \in \RR^d$,
  $\omega \neq 0$. Then, a function $u \in C^2(\Omega)$ solves~\eqref{kerfunc}
  if and only if $u$ has a representation~\eqref{oscifunc} for some
  $C_1,C_2 \in \RR$.
\end{lemma}

\begin{proof}
  Note that it suffices to prove the statement locally, i.e., without
  loss of generality, $\Omega$ is an open hypercube in $\RR^d$ of the
  form
  $\Omega = \open{a_1}{b_1} \times \cdots \times \open{a_d}{b_d}$.

  Clearly,
  $(\symgrad \grad )_{ij} = \frac{\partial^2}{\partial x_i \partial
    x_j}$
  and
  $\frac{\partial^2}{\partial x_i \partial x_j} \sin(\omega \inprod x)
  = -\omega_i \omega_j \sin(\omega \inprod x)$
  as well as
  $\frac{\partial^2}{\partial x_i \partial x_j} \cos(\omega \inprod x)
  = -\omega_i \omega_j \cos(\omega \inprod x)$. Thus, if $u \in C^2(\Omega)$
  has the representation~\eqref{oscifunc}, then~\eqref{kerfunc} follows. 
  
  For the converse direction, we proceed inductively with respect to
  the dimension $d$. Starting with $d = 1$, equation~\eqref{kerfunc}
  reads as $\frac{\partial^2 u}{\partial x_1^2} u + \omega_1^2 u = 0$
  where $\omega_1 \neq 0$ which is an ordinary differential equation
  with solutions of the form~\eqref{oscifunc}. Next, fix $d \geq 1$
  and suppose that~\eqref{kerfunc} implies~\eqref{oscifunc} for some
  $C_1,C_2 \in \RR$ whenever $u \in C^2(\Omega)$ for
  $\Omega \subset \RR^d$ an open hypercube.  Let
  $\omega \in \RR^{d+1}$ such that $\omega \neq 0$. After a possible
  interchange of axes, we can assume, without loss of generality, that
  $(\omega_1,\ldots,\omega_d) \neq 0$. Let $\Omega \subset \RR^{d+1}$
  be an open hypercube and $u \in C^2(\Omega)$ satisfy~\eqref{kerfunc}
  (in dimension $d+1$). Then, in particular,
  $\frac{\partial^2 u}{\partial x_i \partial x_j} + \omega_i\omega_j u
  = 0$
  for $i,j= 1,\ldots,d$. Denoting by $x'$ the projection of a vector
  $x \in \RR^{d+1}$ onto the first $d$ components, i.e.,
  $x' = (x_1,\ldots,x_d)$, the induction hypothesis gives
  \[
  u(x) = C_1(x_{d+1}) \sin(\omega' \inprod x') + C_2(x_{d+1})
  \cos(\omega' \inprod x') \quad \text{for} \quad x \in \Omega.
  \]
  As $x \mapsto \omega' \inprod x'$
  is non-constant, we are able to choose $x, \bar{x} \in \Omega$ such that
  the matrix 
  \[
  \begin{bmatrix}
    \sin(\omega' \inprod x') & \cos(\omega' \inprod x') \\
    \sin(\omega' \inprod \bar{x}') & \cos(\omega' \inprod \bar{x}')
  \end{bmatrix}
  \]
  has full rank.  Using that
  $\frac{\partial^2 u}{\partial x_{d+1}^2} + \omega_{d+1}^2 u = 0$ in
  neighborhoods of $x$ and $\bar{x}$ consequently yields that there
  are constants $C_{1,1}, C_{1,2}, C_{2,1}, C_{2,2}$ such that
  \[
  C_k(x_{d+1}) =
  \begin{cases}
    C_{k,1} x_{d+1} + C_{k,2} & \text{if} \ \omega_{d+1} = 0, \\
    C_{k,1} \sin(\omega_{d+1}x_{d+1}) + C_{k,2}
    \cos(\omega_{d+1} x_{d+1}) & \text{otherwise},
  \end{cases}
  \quad 
  \]
  for $k=1,2$.  If $\omega_{d+1} = 0$, then
  $\frac{\partial^2 u}{\partial x_i \partial x_{d+1}} + \omega_i
  \omega_{d+1} u = 0$
  for $i \in \sett{1,\ldots,d}$ such that $\omega_i \neq 0$ gives that
  $C_{1,1} \cos(\omega' \inprod x') - C_{2,1} \sin(\omega' \inprod x') =
  0$
  for $x \in \Omega$. Consequently, $C_{1,1} = C_{2,1} = 0$ and the
  induction step is complete.

  In the case $\omega_{d+1} \neq 0$,
  $\frac{\partial^2 u}{\partial x_i \partial x_{d+1}} + \omega_i
  \omega_{d+1} u = 0$
  for $i \in \sett{1,\ldots,d}$ such that $\omega_i \neq 0$ implies
  \[
  \begin{aligned}
    u(x) &= \bigl( C_{1,2} \sin(\omega_{d+1} x_{d+1}) - C_{1,1}
    \cos(\omega_{d+1} x_{d+1}) \bigr) \cos(\omega' \inprod x') \\
    & \quad + \bigl( C_{2,1} \cos(\omega_{d+1} x_{d+1}) - C_{2,2}
    \sin(\omega_{d+1} x_{d+1}) \bigr) \sin(\omega' \inprod x')
  \end{aligned}
  \]
  for $x \in \Omega$. Comparing coefficients yields $C_{2,2} = -C_{1,1}$
  as well as $C_{2,1} = C_{1,2}$. Plugging this into the representation
  of $u$ gives, by virtue of the angle sum identities for sine and cosine,
  \[
  u(x) = C_{2,1} \sin(\omega' \inprod x' + \omega_{d+1} x_{d+1}) 
  + C_{2,2} \cos(\omega' \inprod x' + \omega_{d+1} x_{d+1}) =
  C_{2,1} \sin(\omega \inprod x) + C_{2,2} \cos(\omega \inprod x)
  \]
  for $x \in \Omega$, proving the induction step for
  $\omega_{d+1} \neq 0$.
\end{proof}

Based on this observation, we aim at designing a functional that
is suitable to measure piecewise oscillatory functions $u$ with direction
$\omega \neq 0$, i.e.,
\begin{equation*}
  u=\sum_{i=1}^{n}\chi_{\Omega_i}(x)u_i(x), \quad \text{with}\quad \chi_{\Omega_i}(x)=\begin{cases}
    1,&\text{if $x\in \Omega_i$},\\
    0,&\text{otherwise},
  \end{cases}
\end{equation*}
where $\Omega_1\cup\ldots\cup \Omega_n=\Omega$,
$\Omega_i\cap \Omega_j=\emptyset$, $\forall i\neq j$ and
$u_i(x)=C_{1,i}\cos(\omega \inprod x) + C_{2,i}\sin(\omega \inprod
x)$, for some $C_{1,i}, C_{2,i} \in\RR$, $i=1,\ldots,n$.

Let us now motivate the definition of oscillation TGV. It is easily
seen that the kernel of second-order TGV according to~\eqref{eq:TGV}
can be obtained by solving $\grad u - w = 0$ and $\symgrad w = 0$ for
$u \in \BV(\Omega), w \in \BD(\Omega)$. This leads to all $u$ for
which $\symgrad \grad u = 0$ holds in the distributional sense which
can be shown to be exactly the affine functions.  Based on this and
the fact that oscillatory functions are solving~\eqref{kerfunc}, we
see in turn that this class can be represented by $\grad u - w = 0$
and $\symgrad w + \mathbf{c}u = 0$ for
$u \in \BV(\Omega), w \in \BD(\Omega)$.  This leads to the following
generalization of second-order TGV we call oscillation TGV.

\begin{definition}
Let $\alpha, \beta > 0$ and $\mathbf{c} = \omega \tensor \omega$ where
$\omega \in \RR^d$. For $u \in L^1(\Omega)$, the oscillation TGV
functional is defined as follows:
\begin{equation}\label{TTGV2}
\TGV^{\osci}_{\alpha,\beta,\mathbf{c}}(u) =\min_{w\in \BD(\Omega)} \alpha\|\grad u-w\|_\mathcal{M}+\beta\|\mathcal{E}w+\mathbf{c}u\|_\mathcal{M}
\end{equation}
if $u \in \BV(\Omega)$ and
$\TGV^{\osci}_{\alpha,\beta,\mathbf{c}}(u) = \infty$ if
$u \in L^1(\Omega) \setminus \BV(\Omega)$.
\end{definition}

The usage of the minimum will be justified later.
Note that for $\omega = 0$, the definition reduces to the well-known
second-order TGV with weights $\alpha,\beta$, i.e., we have
$\TGV^{\osci}_{\alpha,\beta,\mathbf{c}}(u) = \TGV_{\alpha,\beta}^2(u)$
if $\omega = 0$.  One can easily see that oscillation TGV is proper
and convex; we focus on showing its lower semi-continuity in
$L^p$-spaces along with some basic properties in the following.

\subsection{Basic properties}
For the purpose of establishing lower semi-continuity, we need some
tools and notions from convex analysis, see, for instance,
\cite{Rockafellar,Combettes}, for an overview.  For a functional
$F: X\rightarrow \leftopen{-\infty}{\infty}$, the effective domain is
defined as $\dom F=\{x\in X:F(x)<\infty\}$, and $F$ is called proper if $\dom F\neq \emptyset$. The indicator
functional is defined as
\begin{equation*}
\mathcal{I}_{K}(x)=
\begin{cases}
0  &\text{if  $x\in K$} ,\\
+\infty &\text{otherwise}.\\
\end{cases}
\end{equation*}
As usual, we also denote by $\langle\cdot,\cdot\rangle$ the duality
product of $X^\ast$ and $X$ between the Banach spaces $X^\ast$ and
$X$.  Moreover, recall that the convex conjugate functional for a
proper $F: X \to \leftopen{-\infty}{\infty}$ is given by
$F^\ast:X^\ast\to \leftopen{-\infty}{\infty}$ with
$F^\ast(x^\ast)=\sup\limits_{x \in X} \ \langle x^\ast,x\rangle-F(x)$.
With these prerequisites, we give a predual formulation of oscillation
TGV which will then imply lower semi-continuity and in particular
justify the minimum in the definition~\eqref{TTGV2}.

\begin{proposition}\label{dualTTGV}
  Fix $\alpha, \beta > 0$ and $\mathbf{c} = \omega \tensor \omega$,
  $\omega \in \RR^d$. %
  For each $u\in L^1(\Omega)$, we have
  \begin{equation}\label{dualTTGV2}
    \TGV^{\osci}_{\alpha,\beta,\mathbf{c}}(u)=\sup\left\{\int_\Omega u(\div^2\phi+\mathbf{c}\inprod\phi) \dd{x}: \phi\in C_c^\infty(\Omega,S^{d\times d}),\|\phi\|_\infty\leq\beta,\|\div\phi\|_\infty\leq\alpha \right\},
  \end{equation}
  where 
  $\mathbf{c} \inprod \phi$ denotes the pointwise Frobenius inner
  product of $\mathbf{c}$ and the symmetric matrix field
  $\phi$. Moreover, the minimum in the definition of
  $\TGV^{\osci}_{\alpha,\beta,\mathbf{c}}(u)$ according
  to~\eqref{TTGV2} is obtained.
\end{proposition}
\begin{proof}
  Firstly observing that $C_c^\infty(\Omega,S^{d\times d})$ is dense
  in the Banach space $C_0^2(\Omega,S^{d\times d})$ with respect to
  the $C^2$ norm given by
  $\norm[C^2]{\phi} = \max\ (\norm[\infty]{\phi}, \norm[\infty]{\grad
    \phi}, \norm[\infty]{\grad^2 \phi})$,
  we can replace $C_c^\infty(\Omega,S^{d\times d})$
  in~\eqref{dualTTGV2} by $C_0^2(\Omega,S^{d\times d})$: Indeed, for
  $\phi \in C_0^2(\Omega,S^{d \times d})$ such that
  $\norm[\infty]{\phi} \leq \beta$ and
  $\norm[\infty]{\div \phi} \leq \alpha$ one may choose a sequence
  $\seq{\phi^n}$ in $C_c^2(\Omega,S^{d \times d})$ such that
  $\lim_{n \to \infty} \phi^n = \phi$ in
  $C_0^2(\Omega, S^{d \times d})$.  Convolving with smooth mollifiers,
  we can moreover ensure that, without loss of generality, each
  $\phi^n \in C_c^\infty(\Omega,S^{d \times d})$.  Now, letting
  $\lambda_n = \max \ (\norm[\infty]{\phi^n}/\beta, \norm[\infty]{\div
    \phi^n}/\alpha, 1)$
  we see by continuous embedding
  $C_0^2(\Omega, S^{d \times d}) \embeds C_0(\Omega,S^{d \times d})$
  as well as continuity of
  $\div: C_0^2(\Omega, S^{d \times d}) \to C_0(\Omega,\RR^d)$ that
  $\lim_{n \to \infty} \lambda_n = 1$. Thus,
  $\lim_{n \to \infty} \lambda_n^{-1} \phi^n = \phi$ and for each $n$,
  we have $\norm[\infty]{\lambda_n^{-1} \phi^n} \leq \beta$ as well as
  $\norm[\infty]{\div \lambda_n^{-1} \phi^n} \leq \alpha$. This proves
  the density.

  Choosing $X=C_0^2(\Omega,S^{d\times d})$,
  $Y=C_0^1(\Omega,\mathbb{R}^d)$ with norm
  $\norm[{C^1}]{\eta} = \max\ ( \norm[\infty]{\eta},
  \norm[\infty]{\grad \eta})$,
  $\Lambda=\div\in\mathcal{L}(X,Y)$, and defining
  \begin{equation*}
    \begin{aligned}
      F_1&: X\rightarrow \leftopen{-\infty}{\infty}, 
      & F_1(\phi) &=\mathcal{I}_{\{\|\cdot\|_\infty\leq\beta\}}(\phi)-\langle\mathbf{c}u,\phi \rangle,\\
      F_2&:  Y\rightarrow \leftopen{-\infty}{\infty},
      & F_2(\eta) & = \mathcal{I}_{\{\|\cdot\|_\infty\leq\alpha\}}(\eta)-\langle u,\div\eta\rangle,
    \end{aligned}
  \end{equation*}
  where $\scp{\inprod}{\inprod}$ is the $L^1$--$L^\infty$ duality
  pairing, we can equivalently write the supremum in~\eqref{dualTTGV2}
  as
  \begin{equation*}
    -\inf_{\phi\in X} F_1(\phi)+F_2(\Lambda\phi).
  \end{equation*}
  According to \cite{Attouch}, the Fenchel--Rockafellar duality
  formula is valid if
  \begin{equation*}
    Y=\bigcup_{\lambda\geq0}\lambda(\dom F_2-\Lambda \dom F_1).
  \end{equation*}
  This is indeed true since 
  for each $\eta\in Y$, we can write
  $\eta=\lambda(\lambda^{-1}\eta- \div 0)$ with
  $\|\lambda^{-1}\eta\|_\infty\leq\alpha$ and $0\in \dom F_1$. Thus,
  there is no duality gap between the primal and the dual problems and
  the dual problem admits a solution, i.e.,
  \begin{equation*}
    \Bigl( \inf_{\phi\in X}  F_1(\phi)+F_2(\Lambda\phi) \Bigr)
    + \Bigl( \min_{w\in Y^\ast} F_1^\ast(-\Lambda^\ast w)+F_2^\ast(w) \Bigr) =0.
  \end{equation*}
  (In the case the primal functional is unbounded from below, this
  equation has to be understood in the sense that the dual functional
  is constant $\infty$.)  Now, we calculate the convex conjugates on
  the right-hand side.
  For, a distribution $w \in Y^\ast = C_0^1(\Omega,\RR^d)^\ast$ we
  have that $-\Lambda^* w = -\div^* w = \symgrad w$ in the
  distributional sense: Indeed, testing $-\div^*w$ with
  $\phi \in C_c^\infty(\Omega,S^{d \times d})$ gives
  $\scp{-\div^* w}{\phi} = \scp{w}{-\div \phi} = \scp{\symgrad
    w}{\phi}$
  for the duality pairing $\scp{\cdot}{\cdot}$ of distributions and
  test functions. Thus,
  \begin{equation*}
    F_1^\ast(-\Lambda^\ast w)=\sup_{\scriptstyle \phi\in X, \atop \scriptstyle \|\phi\|_\infty\leq\beta}
    \langle w, - \div \phi\rangle+ \langle\mathbf{c}u,\phi \rangle
    = \sup_{\substack{\phi \in C_c^\infty(\Omega,S^{d\times d}), \\ \norm[\infty]{\phi} \leq \beta}}
    \langle\mathcal{E} w,\phi\rangle+ \langle\mathbf{c}u,\phi \rangle
    =
    \beta\|\mathcal{E}w+\mathbf{c}u\|_\mathcal{M}
  \end{equation*}
  where, by definition of the Radon norm for distributions,
  $\norm[\radon]{\symgrad w + \mathbf{c}u} = \infty$ if 
  $\symgrad w + \mathbf{c}u$ is not in
  $\radon(\Omega,S^{d \times d})$.  
  Note that for the above identity, we used that
  $\set{\phi \in C_c^\infty(\Omega,S^{d \times
      d})}{\norm[\infty]{\phi} \leq \beta}$
  is dense in
  $\set{\phi \in C_0^2(\Omega,S^{d \times d})}{\norm[\infty]{\phi}
    \leq \beta}$
  which can be seen analogously to the density argument at the
  beginning of the proof.  Now, since $u \in L^1(\Omega)$, we have
  $\mathbf{c}u \in L^1(\Omega, S^{d \times d})$ and
  $\norm[\radon]{\symgrad w + \mathbf{c}u}$ is finite if and only if
  $\symgrad w \in \radon(\Omega,S^{d \times d})$. But this is the case
  if and only if $w \in \BD(\Omega)$.

  Furthermore, we see that $\div^\ast = -\grad$ in the distributional
  sense, so by arguments analogous to the above, it follows that
  \begin{equation*}
    F_2^\ast(w)=\sup_{\scriptstyle \eta\in Y, \atop \scriptstyle \|\eta\|_\infty\leq\alpha}\langle w,\eta\rangle+\langle u, \div\eta\rangle=\alpha\|\grad u-w\|_\mathcal{M}
  \end{equation*}
  where, again, $\norm[\radon]{\grad u - w} = \infty$ if
  $\grad u - w \notin \radon(\Omega, \RR^d)$. Consequently, for
  $w \in \BD(\Omega)$, the norm is finite if and only if
  $u \in \BV(\Omega)$. 

  These considerations yield that indeed, minimizing
  $F_1^\ast(-\Lambda^\ast w) + F_2^\ast(w)$ is equivalent to
  minimizing the functional in~\eqref{TTGV2}, that it suffices to
  choose $w$ from $\BD(\Omega)$ instead of $Y^*$, and that the minimum
  is $\infty$ if and only if $u \notin \BV(\Omega)$.  This gives the
  desired identity~\eqref{dualTTGV2}.

  In particular, as the optimal value for the dual problem is
  attained, there always exists a $w \in \BD(\Omega)$ which admits the
  minimum in~\eqref{TTGV2}.
\end{proof}

Since the dual formulation of oscillation TGV is a pointwise supremum
of continuous functions, we can easily derive that it is lower
semi-continuous with respect to the $L^p$-topology for each
$p \in [1,\infty]$. In summary, it is a semi-norm that shares the
following properties:

\begin{proposition}
  \label{prop:tgv_osci_lsc}
  The functional $\TGV^{\osci}_{\alpha,\beta,\mathbf{c}}$ is a proper,
  and lower semi-continuous one-homo\-geneous functional on
  $L^1(\Omega)$ that satisfies the triangle inequality.  Further,
  $\TGV^{\osci}_{\alpha,\beta,\mathbf{c}}(u) = 0$ if and only if $u$
  satisfies~\eqref{oscifunc} for some $C_1,C_2 \in \RR$ in case of
  $\omega \neq 0$ and $u$ affine in case of $\omega = 0$.
\end{proposition}

\begin{proof}
  We only discuss the case $\omega \neq 0$ as, by the above remark,
  $\TGV^{\osci}_{\alpha,\beta,\mathbf{c}}$ reduces to second-order TGV in case
  $\omega = 0$ for which the claimed properties were established in
  \cite{TGV}.

  By Proposition~\ref{dualTTGV},
  $\TGV^{\osci}_{\alpha,\beta,\mathbf{c}}$ is the convex conjugate of
  a convex indicator functional associated with a nonempty, convex and
  balanced set. Thus, it is proper, convex and positive
  one-homogeneous; we refer to the proof of statement 1 of Proposition
  3.3 in \cite{TGV}. In particular,
  $\TGV^{\osci}_{\alpha,\beta,\mathbf{c}}$ satisfies the triangle
  inequality. Further, for each
  $\phi \in C_c^\infty(\Omega,S^{d \times d})$, the mapping
  $u \mapsto \int_\Omega u (\div^2 \phi + \mathbf{c} \inprod \phi)
  \dd{x}$
  is linear and continuous on $L^1(\Omega)$, so
  $\TGV^{\osci}_{\alpha,\beta,\mathbf{c}}$ is lower semi-continuous on
  $L^1(\Omega)$ as the pointwise supremum of continuous functionals,
  see~\eqref{dualTTGV2}.

  Finally, for $u \in L^1(\Omega)$, we have
  $\TGV^{\osci}_{\alpha,\beta,\mathbf{c}}(u) = 0$ if and only if
  $\int_\Omega u (\div^2 \phi + \mathbf{c} \inprod \phi) \dd{x} = 0$
  for all $\phi \in C_c^\infty(\Omega, S^{d \times d})$ with
  $\norm[\infty]{\phi} \leq \beta$,
  $\norm[\infty]{\div \phi} \leq \alpha$. The latter, however, is
  equivalent to
  $\int_\Omega u (\div^2 \phi + \mathbf{c} \inprod \phi) \dd{x} = 0$
  for all $\phi \in C_c^2(\Omega, S^{d \times d})$ which, in turn,
  means $\symgrad \grad u + \mathbf{c} u = 0$ in $\Omega$ in the weak
  sense.
  
  We show that this equation is also satisfied in the strong sense.
  For this purpose, observe that it is in particular true for
  $u \ast \rho_\epsilon$ with a mollifier
  $\rho_\epsilon \in C_c^\infty(B_\epsilon(0))$ and each subdomain
  $\Omega' \subset \Omega$ with
  $\dist(\Omega', \bdry\Omega) > \epsilon$.  Indeed, this follows from
  \[
  \int_{\Omega'} (u \ast \rho_\epsilon) (\div^2 \phi + \mathbf{c}
  \inprod \phi) \dd{x} = \int_\Omega u \bigl(\div^2 (\phi \ast
  \bar\rho_\epsilon) + \mathbf{c} \inprod (\phi \ast
  \bar\rho_\epsilon) \bigr) \dd{x} = 0
  \]
  where $\bar \rho_\epsilon(x) = \rho_\epsilon(-x)$, since
  $\phi \ast \bar \rho_\epsilon \in C_c^\infty(\Omega, S^{d \times d})$
  whenever $\phi \in C_c^\infty(\Omega', S^{d \times d})$. But
  $u \ast \rho_\epsilon \in C^\infty_c(\Omega')$, so we have
  $\symgrad \grad (u \ast \rho_\epsilon) + \mathbf{c}(u \ast
  \rho_\epsilon) = 0$
  in $\Omega'$ also in the strong sense. Thus, $u \ast \rho_\epsilon$
  satisfies~\eqref{oscifunc} in $\Omega'$. Fixing $\Omega'$ and
  letting $\epsilon \to 0$ now gives $u \ast \rho_\epsilon \to u$ in
  $L^1(\Omega')$. In particular, $u \ast \rho_\epsilon$ converges to
  $u$ in the finite-dimensional space of functions in $L^1(\Omega')$
  satisfying~\eqref{oscifunc}, which implies by closedness of
  finite-dimensional spaces that $u$ also satisfies~\eqref{oscifunc}
  for some $C_1, C_2 \in \RR$ and, consequently, $u \in C^2(\Omega')$.
  As $\Omega'$ can be chosen arbitrarily such that
  $\overline{\Omega'}$ is compact in $\Omega$, the latter also holds
  true in $\Omega$.
\end{proof}

Furthermore, $\TGV^{\osci}_{\alpha,\beta,\mathbf{c}}$ is translation invariant
as well as rotationally invariant with respect to a rotated $\mathbf{c}$:
\begin{proposition}\label{rotation}
  For each $y \in \RR^d$ and $u \in L^1(\Omega)$, if
  $\tilde u(x) = u(x-y)$, then
  $\TGV^{\osci}_{\alpha,\beta,\mathbf{c}}(\tilde{u}) =
  \TGV^{\osci}_{\alpha,\beta,\mathbf{c}}(u)$,
  i.e., $\TGV^{\osci}_{\alpha,\beta,\mathbf{c}}$ is
  translation-invariant.

  Furthermore, for each orthonormal matrix
  $O\in\mathbb{R}^{d\times d}$ and $u\in L^1(\Omega)$, if
  $\tilde{u}(x)=u(Ox)$, then
  $\TGV^{\osci}_{\alpha,\beta,\mathbf{c}}(\tilde{u}) =
  \TGV^{\osci}_{\alpha,\beta,\tilde{\mathbf{c}}}(u)$
  with $\tilde{\mathbf{c}} = O\omega \tensor O\omega$.  In other
  words, ${\rm TGV}^{\osci}_{\alpha,\beta,\mathbf{c}}$ is rotationally
  invariant.
\end{proposition}

\begin{proof}
  First, let us define the following set
  \begin{equation*}
    K_{\alpha,\beta}(\Omega)=\left\{\phi\in 
      C_c^\infty(\Omega,S^{d\times d}): \|\phi\|_\infty\leq\beta, \ 
      \|\div\phi\|_\infty\leq\alpha \right\}.
  \end{equation*}
  For $\phi \in K_{\alpha,\beta}(\Omega + y)$ we set
  $\tilde \phi(x) = \phi(x + y)$. Then,
  $\norm[\infty]{\tilde \phi} \leq \beta$,
  $\div \tilde \phi(x) = \div \phi(x+y)$ and
  $\div^2 \tilde \phi(x) = \div^2 \phi(x+y)$ such that
  $\norm[\infty]{\div \tilde \phi} \leq \alpha$ and, consequently,
  $\tilde \phi \in K_{\alpha,\beta}(\Omega)$. Interchanging the roles
  of $\phi$ and $\tilde \phi$, it follows that
  $\phi \in K_{\alpha,\beta}(\Omega + y)$ if and only if
  $\tilde \phi \in K_{\alpha,\beta}(\Omega)$. For
  $\phi \in K_{\alpha,\beta}(\Omega + y)$ it holds that
  \[
  \int_{\Omega + y} \tilde u(x) \bigl ( \div^2 \phi(x) + \mathbf{c}
  \inprod \phi(x) \bigr) \dd{x} = \int_\Omega u(x) \bigl( \div^2
  \tilde \phi(x) + \mathbf{c} \inprod \tilde \phi(x) \bigr) \dd{x},
  \]
  hence taking the supremum yields with the above that
  $\TGV_{\alpha,\beta,\mathbf{c}}^{\osci}(\tilde u) =
  \TGV_{\alpha,\beta,\mathbf{c}}^{\osci}(u)$.

  For $\phi \in K_{\alpha,\beta}(O^T\Omega)$ we set
  $\tilde \phi (x) = O \phi(O^T x) O^T$ which satisfies
  $\tilde \phi \in C_c^2(\Omega, S^{d \times d})$ as well as
  $\norm[\infty]{\tilde \phi} \leq \beta$, the latter due to the
  invariance of the Frobenius norm under orthonormal similarity
  transforms. Moreover, we have
  \[
  \div \tilde \phi(x) = \div \phi(O^Tx) O^T, \qquad
  \div^2 \tilde \phi(x) = \div^2 \phi(O^T x),
  \]
  see, for instance, (A.4) in \cite{TGV}. Therefore,
  $\norm[\infty]{\div \tilde \phi} \leq \alpha$, so
  $\tilde \phi \in K_{\alpha,\beta}(\Omega)$.  Interchanging the roles
  of $\phi$ and $\tilde \phi$ it follows that
  $\phi \in K_{\alpha,\beta}(O^T\Omega)$ if and only if
  $\tilde \phi \in K_{\alpha,\beta}(\Omega)$.

  Eventually, for each $\phi\in K_{\alpha,\beta}(O^T\Omega)$, we have 
  \begin{equation*}
    \begin{aligned}
      \int_{O^T\Omega} \tilde u(x)
      \bigl(\div^2\phi(x)+\mathbf{c}\inprod\phi(x) \bigr) \dd{x} &=
      \int_{O^T\Omega} u(Ox)
      \bigl(\div^2\phi(x)+\mathbf{c}\inprod\phi(x) \bigr) \dd{x}
      \\
      &=\int_\Omega u(x)
      \bigl(\div^2\phi(O^Tx)+\mathbf{c}\inprod\phi(O^Tx) \bigr)
      \dd{x} \\
      &=\int_\Omega u(x) \bigl(\div^2\tilde{\phi}(x)+\mathbf{c}\inprod
      (O^T\tilde{\phi}(x)O) \bigr)\dd{x} \\
      &= \int_\Omega u(x) \bigl(\div^2\tilde{\phi}(x)+
      (O\mathbf{c}O^T) \inprod \tilde{\phi}(x) \bigr) \dd{x}.
    \end{aligned}
  \end{equation*}
  Taking the supremum and noting that
  $O \mathbf{c} O^T = O\omega \tensor O\omega = \tilde{\mathbf{c}}$
  then leads to
  $\TGV^{\osci}_{\alpha,\beta,\mathbf{c}}(\tilde u)
  =\TGV^{\osci}_{\alpha,\beta,\tilde{\mathbf{c}}}(u)$.
\end{proof}

We would like to examine Banach spaces that are associated with the
semi-norm $\TGV^{\osci}_{\alpha,\beta,\mathbf{c}}$. To construct a
norm, we add, as it is also done in case of $\TGV$, the $L^1$-norm. The
resulting Banach space
\[
\BGV^{\osci}(\Omega) = \set{u \in
  L^1(\Omega)}{\TGV^{\osci}_{\alpha,\beta,\mathbf{c}}(u) < \infty},
\qquad
\norm[\BGV^{\osci}]{u} = \norm[1]{u} + \TGV^{\osci}_{\alpha,\beta,\mathbf{c}}(u)
\]
then coincides with $\BV(\Omega)$. Indeed, this is a consequence of
the equivalence of norms which is proven in the following lemma.
\begin{lemma}\label{BVTBGV}
  For each $\omega \in \RR^d$, there exist constants
  $0 < c < C < \infty$ such that for each $u\in L^1(\Omega)$, we have
  \begin{equation}\label{TTGVBV}
    c\|u\|_{\BV}\leq\|u\|_1+\TGV^{\osci}_{\alpha,\beta,\mathbf{c}}(u)\leq C\|u\|_{\BV}.
  \end{equation}
\end{lemma}
\begin{proof}
  Setting $w=0$ in the definition~\eqref{TTGV2} implies immediately
  that
  \begin{equation*}
    \TGV^{\osci}_{\alpha,\beta,\mathbf{c}}(u)\leq 
    \alpha\|\grad u\|_\mathcal{M}+\beta\|\mathbf{c} u\|_1
    \leq \alpha \TV(u) + \beta \abs{\omega}^2 \norm[1]{u}.
  \end{equation*}
  Adding the $L^1$-norm on both sides give the second inequality
  of~\eqref{TTGVBV} for $C=\max\{\alpha, 1 + \beta\abs{\omega}^2 \}$.
  
  For the converse inequality, we may assume $u \in \BV(\Omega)$ since
  otherwise, $\TGV^{\osci}_{\alpha,\beta,\mathbf{c}}(u)=\infty$ by
  definition.  Due to the norm equivalence in $\BV(\Omega)$ and
  $\BGV^2_{\alpha,\beta}(\Omega)$ shown, e.g., in
  \cite{TGVTensor,TGVinv1,TGVinv2}, there exists a constant $C_1>0$
  such that for each $w \in \BD(\Omega)$,
  \begin{equation*}
    \begin{split}
      C_1\|u\|_{\BV}&\leq \|u\|_1+\TGV^2_{\alpha,\beta}(u)\\
      &\leq\|u\|_1+\alpha\|\grad u-w\|_\mathcal{M}+\beta\|\mathcal{E} w\|_\mathcal{M}\\
      &=\|u\|_1+\alpha\|\grad u-w\|_\mathcal{M}+\beta||\mathcal{E} w+\mathbf{c}u-\mathbf{c}u\|_\mathcal{M}\\
      &\leq \|u\|_1+\alpha\|\grad u-w\|_\mathcal{M}+\beta\|\mathcal{E} w+\mathbf{c}u\|_\mathcal{M}+\beta\|\mathbf{c}u\|_1\\
      &\leq C_2(\|u\|_1+\alpha\|\grad
      u-w\|_\mathcal{M}+\beta\|\mathcal{E}
      w+\mathbf{c}u\|_\mathcal{M})
    \end{split}
  \end{equation*}
  where $C_2=\max\{1,1 + \beta\abs{\omega}^2 \}$. Choosing
  $c=\frac{C_1}{C_2}$ and minimizing over $w \in \BD(\Omega)$ finally
  yields the first inequality of~\eqref{TTGVBV}.
\end{proof}

Depending on the results of Lemma \ref{BVTBGV}, we can further deduce
following inequality of \Poincare--Wirtinger type which implies the
coercivity of oscillation TGV in the $\BV$-space.

\begin{proposition}\label{TBGVcoer}
  Let
  $P: L^1(\Omega) \rightarrow
  \ker(\TGV^{\osci}_{\alpha,\beta,\mathbf{c}})$
  a linear and continuous projection onto the kernel
  $\ker(\TGV^{\osci}_{\alpha,\beta,\mathbf{c}})$ (which exists),
  then there is a $C>0$ such that
  \begin{equation}\label{embedding}
    \|u-Pu\|_{\BV}\leq C \TGV^{\osci}_{\alpha,\beta,\mathbf{c}}(u) \quad \text{for all}\quad u\in \BV(\Omega).
  \end{equation}
\end{proposition}
\begin{proof}
  If the statement is not true, then there exists a sequence
  $\{u^n\}$ in $\BV(\Omega)$ such that
  \begin{equation*}
    \|u^n-Pu^n\|_{\BV}=1,  \quad \TGV^{\osci}_{\alpha,\beta,\mathbf{c}}(u^n)\leq \frac{1}{n},
  \end{equation*}
  for all $n\in\mathbb{N}$.
  By the compact embedding $\BV(\Omega)\hookrightarrow
  L^1(\Omega)$, it follows that
  $\lim\limits_{n\rightarrow\infty}u^n-Pu^n=u^\ast$
  in $L^1(\Omega)$
  for some subsequence (not relabeled).  
  We have $P(u^n-Pu^n)=0$
  for all $n$,
  consequently, the limit has to satisfy $Pu^\ast=0$
  by continuity of $P$ in $L^1(\Omega)$.
  
  As $P$
  projects onto the kernel of
  $\TGV^{\osci}_{\alpha,\beta,\mathbf{c}}$,
  we have $\TGV^{\osci}_{\alpha,\beta,\mathbf{c}}(Pu^n)=0$
  for all $n$
  and can deduce that
  $\TGV^{\osci}_{\alpha,\beta,\mathbf{c}}(u^n-Pu^n)=\TGV^{\osci}_{\alpha,\beta,\mathbf{c}}(u^n)$.
  On the other hand, the lower semi-continuity of
  $\TGV^{\osci}_{\alpha,\beta,\mathbf{c}}$ implies
  \begin{equation*}
    \TGV^{\osci}_{\alpha,\beta,\mathbf{c}}(u^\ast)\leq \liminf_{n\rightarrow\infty}\TGV^{\osci}_{\alpha,\beta,\mathbf{c}}(u^n)=0,
  \end{equation*}
  hence
  $u^\ast\in \ker(\TGV^{\osci}_{\alpha,\beta,\mathbf{c}})\cap \ker(P)$.
  As $P$ is a projection with
  $\range(P) = \ker(\TGV^{\osci}_{\alpha,\beta,\mathbf{c}})$, this
  implies $u^* = Pu^* = 0$, so
  $\lim\limits_{n\rightarrow\infty}u^n-Pu^n = 0$ in $L^1(\Omega)$.

  Now, according to Lemma~\ref{BVTBGV}, there is a $C_1>0$ such that
  \begin{equation}\label{upu2}
    \|u^n-Pu^n\|_{\BV}\leq C_1 \bigl (\|u^n-Pu^n\|_1+\TGV^{\osci}_{\alpha,\beta,\mathbf{c}}(u^n) \bigr),
  \end{equation}
  so 
  $\lim\limits_{n\rightarrow\infty}u^n-Pu^n=0$ in $\BV(\Omega)$
  which is a contradiction
  to $\|u^n-Pu^n\|_{\BV}=1$ for all $n$.
\end{proof}

\begin{remark}
  \label{rem:osci_tgv_coer}
  By continuous embedding $\BV(\Omega) \embeds L^p(\Omega)$ for $1 \leq 
  p \leq d/(d-1)$ we have in particular coercivity of
  $\TGV^{\osci}_{\alpha,\beta,\mathbf{c}}$
  in $L^p(\Omega)$ according to
  \[
  \norm[p]{u - Pu} \leq C \TGV^{\osci}_{\alpha,\beta,\mathbf{c}}(u)
  \quad
  \text{for each} \quad u \in L^p(\Omega).
  \]
\end{remark}

 \section{Infimal convolution of oscillation TGV functionals}
\label{seq:inf_conv_osci_tgv}

So far, we examined oscillation TGV functionals associated with a
single direction $\omega \in \RR^d$. Next, we aim at incorporating
multiple directions $\omega_1,\ldots,\omega_m \in \RR^d$ into one
functional in order to represent structured oscillations in an image
with respect to multiple orientations and frequencies. This will be
done in terms of the infimal convolution.
Recall that, given two functionals $J_1$ and $J_2$, their infimal
convolution is defined by
\begin{equation}
  (J_1\infconv J_2)(u)=\inf_v J_1(u-v)+J_2(v)=\inf_{u=u_1+u_2} J_1(u_1)+J_2(u_2).
\end{equation}
For $m$ functionals $J_1,\ldots,J_m$ the definition is iterated leading to
\[
(J_1 \infconv \ldots \infconv J_m)(u) = \inf_{u = u_1 + \ldots + u_m}
\sum_{i=1}^m J_i(u_i).
\]
We aim at studying the $m$-fold
infimal convolution of oscillation TGV given as follows:
\begin{equation}\label{ICTGV2}
  \ICTGV^{\osci}_{\vec{\alpha},\vec{\beta},\vec{\mathbf{c}}}(u)= (\TGV^{\osci}_{\alpha_1,\beta_1,\mathbf{c}_1}\infconv\ldots\infconv \TGV^{\osci}_{\alpha_m,\beta_m,\mathbf{c}_m})(u)=\inf_{u=u_1+\ldots+u_m} \sum_{i=1}^{m}  \TGV^{\osci}_{\alpha_i,\beta_i,\mathbf{c}_i}(u_i).
\end{equation} 
Here, $\vec\alpha = (\alpha_1,\ldots,\alpha_m)$,
$\vec\beta = (\beta_1,\ldots,\beta_m)$ are vectors of positive weights
and $\vec{\mathbf{c}} = (\mathbf{c}_1,\ldots,\mathbf{c}_m)$ is given
by $\mathbf{c}_i = \omega_i \tensor \omega_i$ for
$\omega_1,\ldots,\omega_m \in \RR^d$.

It is well-known that the infimal convolution of the convex
functionals is also convex.  Also, since
$\TGV^{\osci}_{\alpha_i,\beta_i,\mathbf{c}_i}(0) = 0$ for each $i$, one can
easily obtain that
$\ICTGV^{\osci}_{\vec{\alpha},\vec{\beta},\vec{\mathbf{c}}}$ is
proper. Therefore, in the following section, we concentrate on the
coercivity and lower semi-continuity of
$\ICTGV^{\osci}_{\vec{\alpha},\vec{\beta},\vec{\mathbf{c}}}$ in
$L^p$-spaces.

As it is typical for infimal convolutions, the lower semi-continuity
and coercivity of
$\ICTGV^{\osci}_{\vec{\alpha},\vec{\beta},\vec{\mathbf{c}}}$ is not
immediate. We study these properties in a more abstract framework that
also allows for generalization to a larger class of regularizers.

\begin{definition}
  Let $X$ be a real Banach space and $\Phi: X \to [0,\infty]$. Then,
  $\Phi$ is called a \emph{lower semi-continuous semi-norm} on $X$
  if $\Phi$ satisfies the following properties:
  \begin{compactenum}
  \item $\Phi$ is proper,
  \item $\Phi$ is positive one-homogeneous, i.e., $\Phi(\lambda u) 
    = \abs{\lambda} \Phi(u)$ for all $\lambda \in \RR$, $u \in X$,
  \item $\Phi$ satisfies the triangle inequality,
    i.e., $\Phi(u+v) \leq \Phi(u) + \Phi(v)$ for all $u,v \in X$,
  \item $\Phi$ is lower semi-continuous, i.e.,
    $\Phi(u) \leq \liminf_{n \to \infty} \Phi(u^n)$ whenever
    $u^n \to u$ in $X$ as $n \to \infty$,
  \item there is a linear, continuous and onto projection operator
    $P: X \to \ker(\Phi)$.
  \end{compactenum}
  Further, $\Phi$ is called \emph{coercive} in $X$ if there exists a
  $C > 0$ such that
  \[
  \norm[X]{u - Pu} \leq C \Phi(u) \quad \text{for all} \quad u \in X.
  \]
\end{definition}

Clearly, $\TGV^{\osci}_{\alpha,\beta,\mathbf{c}}$ are lower
semi-continuous semi-norms on $\BV(\Omega)$ and $L^p(\Omega)$ for all
$1 \leq p \leq \infty$ (see Proposition~\ref{prop:tgv_osci_lsc} and
recall that one can always find continuous projections onto
finite-dimensional subspaces). They are coercive in $\BV(\Omega)$ and
in $L^p(\Omega)$ for $1 \leq p \leq d/(d-1)$, see
Proposition~\ref{TBGVcoer} and Remark~\ref{rem:osci_tgv_coer}.

In the following, we fix a Banach space $X$ and let $\Phi_1$, $\Phi_2$
be lower semi-continuous semi-norms on $X$ with associated projections
\[
P_1: X \to \ker(\Phi_1), \qquad
P_2: X \to \ker(\Phi_2).
\]
Further, assume that there is a linear, continuous and onto projection
operator
\[
Q: X \to \ker(\Phi_1) \cap \ker(\Phi_2).
\]
In this situation, we have the following basic estimate.

\begin{lemma}\label{exactlemma}
  Let $\Phi_1, \Phi_2$ be lower semi-continuous semi-norms on $X$ and
  $Q: X \to \ker(\Phi_1) \cap \ker(\Phi_2)$ be a linear, continuous
  and onto projection.  Then, there is a $C>0$ such that
  \begin{equation}\label{vbdd}
    \|v\|_{X}\leq C (\|v-P_1v\|_X+\|v-P_2v\|_X)
    \quad \text{for all}\quad v\in \Ran(I-Q).
  \end{equation}
\end{lemma}

\begin{proof}
  Consider the mapping $T: \Ran(I-Q)\rightarrow \Ran(T)$ according to
  \begin{equation*}
    Tv= \begin{pmatrix}
      v-P_1v\\v-P_2v
    \end{pmatrix},
  \end{equation*} 
  which defines a mapping between Banach spaces since the fact that
  $P_1, P_2, Q$ are continuous projection operators implies that
  $\Ran(I-Q)$ and $\Ran(T)$ are closed.  Obviously, $T$ is linear,
  continuous and surjective.  To see that $T$ is also injective, let
  $Tv = 0$ for some $v \in X$, i.e., $v-P_1v=0$ and $v-P_2v=0$.  Then,
  $v \in \ker(\Phi_1)\cap \ker(\Phi_2)$ which implies $Qv=v$ such that
  $(I-Q)v=0$ follows. Hence, $v\in \Ran(I-Q)\cap\ker(I-Q)=\{0\}$,
  i.e., $v=0$, as $Q$ is, once again, a continuous projection
  operator. By the open mapping theorem, we have that $T^{-1}$ is
  bounded, which leads to the desired result.
\end{proof}

First, we show the exactness and lower semi-continuity of the
infimal convolution of two lower semi-continuous semi-norms under
suitable assumptions.

\begin{lemma}\label{exact}
  Let $\Phi_1, \Phi_2$ be lower semi-continuous and coercive
  semi-norms on the reflexive Banach space $X$ 
  with finite-dimensional kernels.
  Then $\Phi_1 \infconv \Phi_2$ is exact, i.e.,
  for each $u \in X$ there is a $v^* \in X$ such that
  \begin{equation}\label{exact2}
    v^* \in \argmin_{v \in X} \ \Phi_1(u-v) + \Phi_2(v).
  \end{equation}
  Moreover, $\Phi_1 \infconv \Phi_2$ constitutes a lower
  semi-continuous semi-norm on $X$ with finite-dimensional kernel.

  For general $X$, the statements remain true if, for $i=1$ or
  $i=2$, the sublevel sets
  $\set{u \in X}{\norm[X]{u} + \Phi_i(u) \leq C}$ are weakly relatively
  compact for each $C > 0$.
\end{lemma}

\begin{proof}
  Clearly, $\Phi_1 \infconv \Phi_2$ is a non-negative functional.
  Choose a convergent sequence $\{u^n\}$ in $X$ with limit $u$ such
  that $\liminf_{n\to \infty} (\Phi_1 \infconv \Phi_2)(u^n) < \infty$.
  Going to subsequences (without relabeling), we may assume that
  $\lim_{n \to \infty} (\Phi_1 \infconv \Phi_2)(u^n)$ exists and is
  finite.  Thus, we can find a sequence $\{v^n\}$ in $X$ such that
  $\Phi_1(u^n - v^n) + \Phi_2(v^n) \leq (\Phi_1 \infconv \Phi_2)(u^n)
  + \tfrac1n$
  for all $n$, meaning in particular that
  $\liminf_{n \to \infty} \Phi_1(u^n - v^n) + \Phi_2(v^n) \leq \lim_{n \to
    \infty} (\Phi_1 \infconv \Phi_2)(u^n)$.

  Further, as
  $\ker(\Phi_1), \ker(\Phi_2), \ker(\Phi_1) \cap \ker(\Phi_2)$ are
  finite-dimensional subspaces of $X$, we can find linear, continuous
  and onto projections $P_1,P_2,Q$ onto the respective spaces as
  previously stated.

  Now, choosing $\tilde{v}^n=v^n-Qv^n$ instead of $v^n$ does not
  change the functional values since
  \begin{equation*}
    \begin{split}
      \Phi_1(u^n-\tilde{v}^n)+\Phi_2(\tilde{v}^n)
      &=\Phi_1(u^n-v^n+Qv^n)+\Phi_2(v^n-Qv^n)\\
      &=\Phi_1(u^n-v^n)+\Phi_2(v^n),
    \end{split}
  \end{equation*}
  as a consequence of each $Qv^n \in \ker(\Phi_1) \cap \ker(\Phi_2)$.
  In particular, we have that $\Phi_1(u^n-\tilde{v}^n)$ and
  $\Phi_2(\tilde{v}^n)$ are bounded. Since $\Phi_1$ and $\Phi_2$ are
  coercive in $X$, there is a constant $C >0$ such that
  $\|u^n-\tilde{v}^n-P_1(u^n-\tilde{v}^n)\|_X \leq C
  \Phi_1(u^n-\tilde{v}^n)$
  and $\|\tilde{v}^n-P_2\tilde{v}^n\|_X \leq C\Phi_2(\tilde{v}^n)$ for
  each $n$.  Since $\norm[X]{u^n - P_1 u^n}$ is bounded due to the
  convergence $u^n - P_1 u^n \to u - P_1u$, we have that
  $\|\tilde{v}^n-P_1\tilde{v}^n\|_X$ and
  $\|\tilde{v}^n-P_2\tilde{v}^n\|_X$ are bounded, so
  Lemma~\ref{exactlemma} implies that $\{\tilde{v}^n\}$ is bounded in
  $X$. By reflexivity, there exists a weakly convergent subsequence
  (not relabeled) such that
  $\lim\limits_{n\rightarrow\infty}\tilde{v}^n=v^\ast$ weakly in
  $X$. Lower semi-continuity and convexity of $\Phi_1$ and $\Phi_2$
  implies
  $\Phi_1(u-v^\ast)\leq\liminf\limits_{n\rightarrow\infty}
  \Phi_1(u^n-\tilde{v}^n)$
  and
  $\Phi_2(v^\ast)\leq\liminf\limits_{n\rightarrow\infty}
  \Phi_2(\tilde{v}^n)$. Then,
\begin{equation*}
  (\Phi_1 \infconv \Phi_2)(u) \leq \Phi_1(u - v^\ast) + \Phi_2(v^\ast)
  \leq \liminf_{n \to \infty} \Phi_1(u^n - v^n) + \Phi_2(v^n) \leq
  \lim_{n \to \infty} (\Phi_1 \infconv \Phi_2)(u^n).
\end{equation*}
This shows that $\Phi_1 \infconv \Phi_2$ is lower
semi-continuous. Plugging in constant sequences, i.e., $u^n = u$ for
all $n$, yields in particular that
$(\Phi_1 \infconv \Phi_2)(u) = \Phi_1(u - v^\ast) + \Phi_2(v^\ast)$, meaning
that $\Phi_1 \infconv \Phi_2$ is exact.

For $\Phi_1 \infconv \Phi_2$ being a lower semi-continuous semi-norm
in $X$, it remains to show that the functional is proper, positive
one-homogeneous and satisfies the triangle inequality. But these
properties are immediate from the definition and $\Phi_1, \Phi_2$
being lower semi-continuous semi-norms. Finally, note that
$(\Phi_1 \infconv \Phi_2)(u) = 0$ if and only if
$u \in \overline{\ker(\Phi_1) + \ker(\Phi_2)}$. But since, again,
$\ker(\Phi_1), \ker(\Phi_2)$ are finite-dimensional we have that
$\overline{\ker(\Phi_1) + \ker(\Phi_2)} = \ker(\Phi_1) + \ker(\Phi_2)$
is finite-dimensional and therefore, there exists a linear, continuous
and onto projection $P: X \to \ker(\Phi_1) + \ker(\Phi_2)$. This
establishes the remaining property.

To deduce the result without the reflexivity assumption on $X$, we
assume, without loss of generality, that the sets
$\set{u \in X}{\norm[X]{u} + \Phi_2(u) \leq C}$ are weakly relatively
compact for each $C > 0$. Proceeding along the lines of the above
proof and observing that $\{\tilde v^n\}$ as well as
$\{\Phi_2(\tilde v^n)\}$ are still bounded then allows to extract the
weakly convergent subsequence of $\{\tilde v^n\}$. The rest of the
proof is completely analogous.
\end{proof}

In order to use $\Phi_1 \infconv \Phi_2$ as a regularizer, we still
need criteria for coercivity of this infimal convolution
functional. We will obtain this property from a norm equivalence
result that is a generalization of the norm equivalence in
Lemma~\ref{BVTBGV}.

\begin{proposition}
  \label{prop:inf_conv_semi_norms}
  Let $X,Y$ be Banach spaces with $X$ continuously embedded in $Y$.
  Assume that $\Phi_1,\Phi_2$ are lower semi-continuous semi-norms on
  $Y$, are coercive on $X$, possess a finite-dimensional kernel and
  satisfy
  \[
  c_i \norm[X]{u} \leq \norm[Y]{u} + \Phi_i(u) \leq C_i \norm[X]{u}
  \quad \text{for all} \quad u \in X
  \]
  for constants $0 < c_i < C_i < \infty$, $i=1,2$ independent of $u$.
  Then, $\Phi_1 \infconv \Phi_2$ satisfies
  \begin{equation}
     \label{BVTBGVinf}
     c \norm[X]{u} \leq \norm[Y]{u} + (\Phi_1 \infconv \Phi_2)(u) \leq C
     \norm[X]{u} \quad \text{for all} \quad u \in X
  \end{equation}
  for $0 < c < C < \infty$ independent of $u$.
  
  Moreover, if $X$ is compactly embedded in $Y$, then
  $\Phi_1 \infconv \Phi_2$ is a lower semi-continuous semi-norm on $Y$ and
  coercive on $X$, i.e., there exists a $C' > 0$ such that
  \begin{equation}
    \label{infembedding}
    \norm[X]{u - Pu} \leq C'(\Phi_1 \infconv \Phi_2)(u) \quad \text{for
      all} \quad u \in X,
  \end{equation}
  where $P: Y \to \ker(\Phi_1) + \ker(\Phi_2)$ is a linear, continuous
  and onto projector on the kernel of $\Phi_1 \infconv \Phi_2$.
\end{proposition}

\begin{proof}
  We will first establish the norm equivalence of
  $u \mapsto \norm[Y]{u} + (\Phi_1 \infconv \Phi_2)(u)$ with the norm
  in $X$. For that purpose, let $u \in X$ and estimate, employing the
  definition of the infimal convolution and the fact that $\Phi_2$ is
  semi-norm,
  \[
  \norm[Y]{u} + (\Phi_1 \infconv \Phi_2)(u) \leq
  \norm[Y]{u} + \Phi_1(u) \leq C_1 \norm[X]{u}.
  \]
  Next, suppose that the converse estimate is not true. Then, there
  exists a sequence $\{u^n\}$ in $X$ with
  \[
  \norm[X]{u^n} = 1, \qquad
  \norm[Y]{u^n} + (\Phi_1 \infconv \Phi_2)(u^n) \leq \frac1n
  \]
  for all $n$. Hence, $\lim\limits_{n \to \infty} u^n = 0$ in $Y$ and
  one can find a sequence $\{v^n\}$ in $X$ such that
  $\Phi_1(u^n - v^n) \leq \tfrac2n$ and $\Phi_2(v^n) \leq \tfrac2n$
  for all $n$. By the coercivity of $\Phi_1$ and $\Phi_2$ in $X$, we
  have
  \[
  \norm[X]{(u^n - v^n) - P_1(u^n - v^n)} \leq C_1' \Phi_1(u^n - v^n),
  \qquad
  \norm[X]{v^n  - P_2v^n} \leq C_2' \Phi_2(v^n)
  \]
  where $P_i: Y \to \ker(\Phi_i)$ and $C_i'$ are the projections and
  coercivity constants associated with the $\Phi_i$, respectively, for
  $i=1,2$. Consequently,
  \[
  \norm[X]{u^n - P_1u^n - (P_2 - P_1)v^n} \leq \norm[X]{u^n - v^n -
    P_1(u^n - v^n)} + \norm[X]{v^n - P_2v^n}
 \to 0
  \]
  as $n \to \infty$. Since $\norm[Y]{u^n} \to 0$ as $n \to \infty$, it
  follows by continuous embedding $X \embeds Y$ that
  $P_1u^n + (P_2 - P_1)v^n \to 0$ as $n \to \infty$ in $Y$. However,
  this sequence is contained in the finite-dimensional space
  $\ker(\Phi_1) + \ker(\Phi_2)$, so we also have
  $\lim\limits_{n \to \infty} P_1u^n + (P_2 - P_1)v^n = 0$ in
  $X$. Hence, $\norm[X]{u^n} \to 0$ as $n \to \infty$ which is a
  contradiction. So, there has to be a constant $c > 0$ such that
  \[
  c \norm[X]{u} \leq \norm[Y]{u} + (\Phi_1 \infconv \Phi_2)(u) \quad
  \text{for all} \quad u \in X.
  \]

  Now, assume that $X$ is compactly embedded in $Y$. Then,
  $\Phi_1 \infconv \Phi_2$ is a lower semi-continuous semi-norm on
  $Y$. This follows from applying Lemma~\ref{exact} setting $X = Y$
  and observing that the sublevel sets
  $\set{u \in Y}{\norm[Y]{u} + \Phi_i(u) \leq C}$ are relatively compact
  in $Y$ for each $C > 0$ and $i=1,2$. Indeed, the latter follows from
  the norm estimate
  $\norm[X]{u} \leq c_i^{-1}\bigl(\norm[Y]{u} + \Phi_i(u)\bigr)$ as
  well as the compact embedding of $X$ in $Y$.

  Finally, the coercivity of $\Phi_1 \infconv \Phi_2$ in $X$ follows
  in the lines of proof of Proposition~\ref{TBGVcoer}. Choose a
  projector $P: Y \to \ker(\Phi_1) + \ker(\Phi_2)$
  according to the statement which has to exist as the kernels are
  finite-dimensional. 
  Assuming that~\eqref{infembedding} does not hold leads to the
  existence of a sequence $\{u^n\}$ for which
  \[
  \norm[X]{u^n - Pu^n} = 1, \qquad
  (\Phi_1 \infconv \Phi_2)(u^n) \leq \frac1n
  \]
  for each $n$. By compact embedding, for a subsequence of $\{u^n\}$
  (not relabeled), we have $u^n - Pu^n \to u^*$ as $n \to \infty$ in
  $Y$ for some $u^* \in Y$. Since $P$ is a projection, one obtains
  $P(u^n - Pu^n) = 0$ for each $n$ such that by continuity of $P$ in
  $Y$, it follows that $Pu^* = 0$, i.e., $u^* \in \ker(P)$.  Moreover,
  $(\Phi_1 \infconv \Phi_2)(u^n - Pu^n) = (\Phi_1 \infconv
  \Phi_2)(u^n)$
  for each $n$, such that $(\Phi_1 \infconv \Phi_2)(u^n - Pu^n) \to 0$
  as $n \to \infty$.  By lower semi-continuity of
  $\Phi_1 \infconv \Phi_2$ in $Y$,
  $(\Phi_1 \infconv \Phi_2)(u^*) = 0$,
  meaning that $u^* \in \range(P)$. Together with $u^* \in \ker(P)$
  and $P$ being a projection, this implies $u^* = 0$.
  
  Thus, the norm equivalence~\eqref{BVTBGVinf} implies
  \[
  \norm[X]{u^n - Pu^n} \leq c^{-1} \bigl( \norm[Y]{u^n - Pu^n}
  + (\Phi_1 \infconv \Phi_2)(u^n) \bigr) \to 0
  \quad \text{as} \quad n \to \infty,
  \]
  which is a contradiction to $\norm[X]{u^n - Pu^n} = 1$ for all $n$.
  Hence, $\Phi_1 \infconv \Phi_2$ has to be coercive as stated.
\end{proof}

Having a lower semi-continuous and coercive semi-norm $\Phi$ allows to
solve linear inverse problems $Ku = f$ with $\Phi$ as regularization
functional. The following result is easily obtained using techniques from
calculus of variations (see, for instance, \cite{TGVinv1,TGVinv2,TGVmulti}). However,
we provide the argumentation for the sake of completeness.

\begin{proposition}
  \label{prop:general_min}
  Let $X,Y$ be Banach spaces, $X$ be reflexive, $K: X \to Y$ linear
  and continuous, $F: Y \to {]{-\infty, \infty}]}$ convex, lower
  semi-continuous, coercive and bounded from below, and $\Phi$ a lower
  semi-continuous and coercive semi-norm on $X$ with
  finite-dimensional kernel.  Then, there exists a solution to the
  minimization problem
  \begin{equation}
    \label{eq:general_min}
    \min_{u \in X} \ F(Ku) + \Phi(u).
  \end{equation}
\end{proposition}

\begin{proof}
  There is only something to show if the objective functional is
  proper.  In that case, it is bounded from below, so one can choose
  an infimizing sequence $\{u^n\}$ for which the objective functional
  converges to a finite value. Denoting by
  $Z = \ker(\Phi) \cap \ker(K)$, we see that $Z$ has finitely many
  dimensions and therefore, there exists a linear, continuous and onto
  projection operator $Q: X \to Z$.  Denoting by $P: X \to \ker(\Phi)$
  the projection associated with the semi-norm $\Phi$, we can moreover
  achieve that $PQ = Q$.
  
  As $Z$ is a subspace of both
  $\ker(\Phi)$ and $\ker(K)$, we can replace $u^n$ by $u^n - Qu^n$ and
  get
  \[
  F\bigl( K(u^n - Qu^n) \bigr)
  + \Phi(u^n - Qu^n) =
  F(Ku^n)+ \Phi(u^n),
  \]
  implying that $\{u^n - Qu^n\}$ is an infimizing sequence. Thus, we
  can assume without loss of generality that $\{u^n\}$ is an
  infimizing sequence in the reflexive Banach space $\range(I -Q)$.
  Now, $\{u^n\}$ has to be bounded: From the coercivity of $\Phi$ in
  $X$ we get that $\norm[X]{u^n - Pu^n}$ is bounded. Furthermore, as
  $F$ is coercive, $\norm[Y]{Ku^n}$ is bounded, so we have
  \[
  \norm[Y]{KPu^n} \leq \norm[Y]{Ku^n} + \norm{K} \norm[X]{u^n - Pu^n}
  \leq C
  \]
  for some $C > 0$. However, $K$ is injective on
  $\range\bigl(P(I - Q)\bigr) = \range(P-Q)$: For $Ku = 0$ and
  $u \in \range(P - Q) \subset \range(P)$ it follows that
  $u \in \ker(K) \cap \ker(\Phi) = Z$, so $Qu = u$ as well as $Pu = u$
  since $Z \subset \ker(\Phi)$. Hence, $(P - Q)u = 0$ which implies
  that $u = 0$ as $u \in \range(P-Q)$ and $P - Q$ is a projection.
  Now, $\range(P-Q)$ is finite-dimensional, so there exists a constant
  $c > 0$ such that $c \norm[X]{Pu} \leq \norm[Y]{KPu}$ for each
  $u \in \range(I-Q)$. This implies, since $\norm[Y]{KPu^n}$ is
  bounded, that $\norm[X]{Pu^n}$ is bounded and hence, that $\{u^n\}$
  is bounded in $X$.

  By reflexivity of $X$, we can extract a weakly convergent
  subsequence with limit $u^*$. One easily sees that the objective
  functional is convex and lower semi-continuous, so $u^*$ has to be a
  minimizer.
\end{proof}

It is then immediate that Tikhonov functionals for the regularized
solution of $Ku = f$ admit minimizers.

\begin{corollary}
  \label{cor:general_tikh}  
  In the situation of Proposition~\ref{prop:general_min}, let
  $f \in Y$ and $1 \leq q < \infty$. Then, there exists a
  minimizer of 
  \begin{equation}
    \label{eq:general_tikhonov}
    \min_{u \in X} \ \frac{\norm[Y]{Ku - f}^q}{q} + \Phi(u).
  \end{equation}
  and the minimum is finite.
\end{corollary}

The general results presented here allow us to obtain properties of
the general infimal convolution of oscillation TGV model with
an arbitrary number of functionals.

\begin{theorem}
  \label{thm:ictgv}
  The functional
  $\ICTGV_{\vec{\alpha},\vec{\beta},\vec{\mathbf{c}}}^{\osci}$
  according to~\eqref{ICTGV2} is a lower semi-continuous semi-norm on
  $L^1(\Omega)$ with
  \[
  \begin{aligned}
    \ker(\ICTGV_{\vec{\alpha},\vec{\beta},\vec{\mathbf{c}}}^{\osci}) &=
    \linspan\set{x \mapsto \sin(\omega_i \inprod x), x \mapsto
      \cos(\omega_i \inprod x)}{i=1,\ldots,m \ \text{with} \ \omega_i \neq 0} \\
    & \quad
    +
    \begin{cases}
      \set{x \mapsto a \inprod x + b}{a \in \RR^d, b \in \RR}
      & \text{if} \ \omega_i = 0 \ \text{for some} \ i, \\
      \sett{0} & \text{else}.
    \end{cases}
  \end{aligned}
  \]
  It is coercive on $\BV(\Omega)$ and there exist
  constants $0 < c < C < \infty$ such that for each
  $u \in \BV(\Omega)$, it holds that
  \[
  c\norm[\BV]{u} \leq \norm[1]{u} +
  \ICTGV_{\vec{\alpha},\vec{\beta},\vec{\mathbf{c}}}^{\osci}(u) \leq C
  \norm[\BV]{u}.
  \]
\end{theorem}

\begin{proof}
  We proceed inductively with respect to $m$. For $m=1$, the result is
  clear from Propositions~\ref{prop:tgv_osci_lsc} and~\ref{TBGVcoer}
  as well as Lemma~\ref{BVTBGV}, as already mentioned in the beginning
  of this section.  The induction step is then a consequence of
  Proposition~\ref{prop:inf_conv_semi_norms} with $X=\BV(\Omega)$ and
  $Y = L^1(\Omega)$, where $X$ is compactly embedded in $Y$.
\end{proof}

\begin{remark}
  In imaging applications, it is meaningful to choose $\omega_i = 0$
  for some $i$, i.e.,
  $\TGV_{\alpha_i,\beta_i,\mathbf{c}_i}^{\osci} =
  \TGV_{\alpha_i,\beta_i}^2$,
  such that piecewise smooth parts are captured the usual TGV
  functional of second order.
  
  Moreover, one might require that texture parts are sparse in a
  certain sense. The latter can, for instance, be incorporated in the
  infimal convolution model by adding a $L^1$-term such that, for
  $\omega_i \neq 0$, the infimal convolution is performed with respect
  to
  $\gamma_i\norm[1]{\:\cdot\:} +
  \TGV_{\alpha_i,\beta_i,\mathbf{c}_i}^{\osci}$
  for some $\gamma_i > 0$.
  This result is the following variant of~\eqref{ICTGV2}:
  \begin{equation}
    \label{eq:ictgv_l1}
    \ICTGV_{\vec{\alpha},\vec{\beta},\vec{\mathbf{c}},\vec{\gamma}}^{\osci} =
    (\gamma_1 \norm[1]{\:\cdot\:} + \TGV^{\osci}_{\alpha_1,\beta_1,\mathbf{c}_1})
    \infconv \ldots \infconv 
    (\gamma_m \norm[1]{\:\cdot\:} + \TGV^{\osci}_{\alpha_m,\beta_m,\mathbf{c}_m})
  \end{equation}
  where $\vec{\gamma} = (\gamma_1,\ldots,\gamma_m)$ satisfies
  $\gamma_1,\ldots,\gamma_m \geq 0$.

  One can easily see that each
  $\gamma_i\norm[1]{\:\cdot\:} +
  \TGV_{\alpha_i,\beta_i,\mathbf{c}_i}^{\osci}$
  is a lower semi-continuous semi-norm on $L^1(\Omega)$ that is
  coercive on $\BV(\Omega)$. Its kernel, however, is trivial if
  $\gamma_i > 0$. Thus, Theorem~\ref{thm:ictgv} can be extended to
  this variant of $\ICTGV^{\osci}$ with the modification that only $i$
  for which $\gamma_i = 0$ contribute to the kernel.
\end{remark}

Finally, the infimal convolution of oscillation TGV can be used as
regularization term for Tikhonov functionals associated with linear
inverse problems.

\begin{theorem}
  \label{thm:ictgv_tikh}
  Let $1 \leq p<\infty$, $p\leq d/(d-1)$ and
  $K: L^p(\Omega)\rightarrow Y$ a linear and continuous operator
  mapping to a Banach space $Y$, $f\in Y$ and $1 \leq q <
  \infty$. Then, the problem
  \begin{equation}
    \label{eq:ictgv_tikh}
    \min_{u\in L^p(\Omega)} \frac{1}{q}\|Ku-f\|^q_Y + \ICTGV_{\vec{\alpha},
      \vec{\beta}, \vec{\mathbf{c}}}^{\osci}(u)
  \end{equation}
  admits a solution.
\end{theorem}

\begin{proof}
  Without loss of generality, one can assume that $p > 1$ since
  $d/(d-1) > 1$ for all $d$ and $K|_{L^p(\Omega)}: L^p(\Omega) \to Y$
  is linear and continuous if $K: L^1(\Omega) \to Y$ is linear and
  continuous. Then, with $X = L^p(\Omega)$,
  $\ICTGV_{\vec{\alpha}, \vec{\beta}, \vec{\mathbf{c}}}^{\osci}$ is a
  lower semi-continuous semi-norm on $X$ as it is a lower
  semi-continuous semi-norm on $L^1(\Omega)$ and
  $X \embeds L^1(\Omega)$. Moreover, it is coercive on $X$ as it is
  coercive on $\BV(\Omega)$ and $\BV(\Omega) \embeds X$ (see also
  Remark~\ref{rem:osci_tgv_coer}). Hence, application of
  Corollary~\ref{cor:general_tikh} yields the result.
\end{proof}

\begin{remark}
  It is clear that the theorem is still true if one replaces
  $\ICTGV_{\vec{\alpha}, \vec{\beta}, \vec{\mathbf{c}}}^{\osci}$ by
  the variant
  $\ICTGV_{\vec{\alpha},\vec{\beta},\vec{\mathbf{c}},\vec{\gamma}}^{\osci}$
  according to~\eqref{eq:ictgv_l1}.
\end{remark}

\section{Discretization and a numerical optimization algorithm}
\label{sec:discretization}

\subsection{A finite-difference discretization}
We start with a finite-difference discretization of
$\TGV^{\osci}$. Regarding the applications in this paper we just
consider the two-dimensional case and rectangular domains. The
framework can, however, easily be adapted to higher dimensions and
more general domains.  Following essentially the presentation in
\cite{TGV,TGVmulti,TGVPRS}, we first replace the domain
$\Omega\subset\RR^2$ by a discrete rectangular unit-length grid
\begin{equation*}
  \Omega=
  \set{(i,j)}{i,j\in \mathbb{N}, 1\leq i\leq N_1, 
    1\leq j\leq N_2}
\end{equation*}
where $(N_1, N_2)$, $N_1, N_2 \geq 2$ are the image dimensions. Let us
further introduce the vector spaces of functions
\begin{equation*}
U=\{u:\Omega\rightarrow \mathbb{R}\}
\end{equation*}
and $V=U\times U$, $W=U\times U\times U$. The space $U$ will be
endowed with the scalar product
$\langle u, u' \rangle_U =\sum_{(i,j)\in\Omega}^{}u_{i,j}u'_{i,j}$ for
$u,u' \in U$ and the norm $\|u\|_U=\sqrt{\langle u,u\rangle_U}$.
Likewise, the spaces $V$ and $W$ are equipped with the scalar products
$\scp[V]{v}{v'} = \scp[U]{v_1}{v_1'} + \scp[U]{v_2}{v_2'}$ and
$\scp[W]{w}{w'} = \scp[U]{w_1}{w_1'} + \scp[U]{w_2}{w_2'} +
2\scp[U]{w_3}{w_3'}$, respectively.

The oscillation TGV functional will be discretized by finite
differences. The forward and backward partial differentiation
operators are:
\begin{equation*}
\begin{aligned}
(\partial_x^+u)_{i,j}&=\begin{cases}
u_{i+1,j}-u_{i,j},&\text{if}\ 1\leq i<N_1,\\
0,&\text{if}\ i=N_1,
\end{cases}
&
\quad
(\partial_x^-u)_{i,j}&=\begin{cases}
u_{i,j}-u_{i-1,j},&\text{if}\ 1 < i \leq N_1,\\
0,&\text{if}\ i=1,
\end{cases}
\\
(\partial_y^+u)_{i,j}&=\begin{cases}
u_{i,j+1}-u_{i,j},&\text{if}\ 1 \leq j< N_2,\\
0, &\text{if}\ j=N_2,
\end{cases}
&
(\partial_y^-u)_{i,j}&=\begin{cases}
u_{i,j}-u_{i,j-1},&\text{if}\ 1 < j \leq N_2,\\
0, &\text{if}\ j=1,
\end{cases}
\end{aligned}
\end{equation*}
With these, the gradient $\nabla$ and symmetric gradient $\mathcal{E}$
are given as follows:
\begin{equation*}
\nabla: U\rightarrow V,\quad\nabla u =\begin{pmatrix}
\partial_x^+ u\\
\partial_y^+ u\\
\end{pmatrix},
\qquad
\mathcal{E}: V\rightarrow W,\quad\mathcal{E} v =\begin{pmatrix}
\partial_x^- v_1\\
\partial_y^- v_2\\
\frac{1}{2}(\partial_y^- v_1+\partial_x^- v_2)
\end{pmatrix}.
\end{equation*}
The divergences of vector fields and symmetric matrix fields are
defined as the negative adjoint of the gradient and symmetric
gradient, respectively.  These can be expressed in terms of partial
differentiation operators with suitable boundary conditions:
\begin{equation*}
  \begin{aligned}
    (\partial_x^{*+}u)_{i,j}&= \begin{cases}
      u_{2,j},&\text{if}\ i=1,\\
      u_{i+1,j}-u_{i,j},&\text{if}\ 1<i<N_1,\\
      -u_{N_1,j},&\text{if}\ i=N_1,
    \end{cases}
    & \quad
    (\partial_x^{*-}u)_{i,j}&= \begin{cases}
      u_{1,j},&\text{if}\ i=1,\\
      u_{i,j}-u_{i-1,j},&\text{if}\ 1<i<N_1,\\
      -u_{N_1-1,j},&\text{if}\ i=N_1,
    \end{cases}\\
    (\partial_y^{*+}u)_{i,j}&= \begin{cases}
      u_{i,2},&\text{if}\ j=1,\\
      u_{i,j+1}-u_{i,j},&\text{if}\ 1<j<N_2,\\
      -u_{i,N_2},&\text{if}\ j=N_2.
    \end{cases}
    &
    (\partial_y^{*-}u)_{i,j}&= \begin{cases}
      u_{i,1},&\text{if}\ j=1,\\
      u_{i,j}-u_{i,j-1},&\text{if}\ 1<j<N_2,\\
      -u_{i,N_2-1},&\text{if}\ j=N_2,
    \end{cases}
  \end{aligned}
\end{equation*}
such that
\begin{equation*}
\div_1:V\rightarrow U,\quad \div v=\partial_x^{*-} v_1+\partial_y^{*-} v_2,
\qquad
\div_2:W\rightarrow V,\quad \div_2 w=\begin{pmatrix}
  \partial_x^{*+} w_1+\partial_y^{*+} w_3\\
  \partial_x^{*+} w_3+\partial_y^{*+} w_2
\end{pmatrix}.
\end{equation*}
By direct computation, one sees that indeed, $\nabla^*=-\div_1$ as well
as $\mathcal{E}^*=-\div_2$. 
A discrete definition of oscillation TGV can now either be derived
from the primal representation~\eqref{TTGV} or the dual
representation~\eqref{dualTTGV2}.  In any case, we need the notion of
the $1$ and $\infty$ norm in the discrete spaces:
\begin{equation*}
  \begin{aligned}
    u&\in U: & \norm[1]{u} &= \sum_{(i,j)\in \Omega} \abs{u_{i,j}}, &
  \norm[\infty]{u} &= \max_{(i,j)\in\Omega} \ \abs{u_{i,j}}, \\
  v&\in V: & \norm[1]{v} &= \sum_{(i,j)\in\Omega} \bigl((v_1)_{i,j}^2 +
  (v_2)_{i,j}^2\bigr)^{1/2}, 
  &\|v\|_\infty &= \max_{(i,j)\in\Omega}\ \bigl((v_1)_{i,j}^2+(v_2)_{i,j}^2\bigr)^{1/2},\\
  w&\in W:
  &\norm[1]{w} &= \sum_{(i,j) \in \Omega} \bigl((w_1)_{i,j}^2+(w_2)_{i,j}^2+2(w_3)_{i,j}^2\bigr)^{1/2}, \\
  && \|w\|_\infty &= \max_{(i,j)\in\Omega}\
  \bigl((w_1)_{i,j}^2+(w_2)_{i,j}^2+2(w_3)_{i,j}^2\bigr)^{1/2}.
  \end{aligned}
\end{equation*}
With these prerequisites, the discrete functional associated with
$\alpha > 0$, $\beta > 0$ and
$\mathbf{c}=\begin{bmatrix} c_1&c_3\\c_3&c_2
\end{bmatrix}$ reads as follows
\begin{align}
  \notag
  \TGV^{\osci}_{\alpha,\beta,\mathbf{c}}(u) 
  &=
    \min_{v \in V} \ \alpha \norm[1]{\grad u - v} 
    + \beta \norm[1]{\symgrad v + \mathbf{c}u} \\
  \label{eq:discrete_tgv_osci}
  &=
    \max_{p \in V, \ q \in W} \ \scp[U]{u}{\div_1 p + \mathbf{c} \inprod q} \quad \text{subject to} \
    \left\{\begin{aligned}
        \norm[\infty]{p} & \leq \alpha, 
        \norm[\infty]{q} \leq \beta, \\
        p &= \div_2 q.
      \end{aligned}\right.
\end{align}
Indeed, one can easily verify that Fenchel--Rockafellar duality holds,
for instance, by continuity of
$v \mapsto \alpha \norm[1]{\grad u - v} + \beta \norm[1]{\symgrad v +
  \mathbf{c}u}$
in the finite-dimensional space $V$, and that the minimization and
maximization problems in~\eqref{eq:discrete_tgv_osci} are dual to each other.

\subsection{Discrete kernel representation}
We did not detail yet on the choice of $\mathbf{c}$. As our motivation
for oscillation TGV was to capture oscillations of the
type~\eqref{oscifunc} by adapting the kernel of the regularizer
exactly to these functions, our aim is to transfer this property to
the discrete setting.  However, a straightforward adaptation of the
choice of $\mathbf{c}$ in the continuous setting, i.e.,
$\mathbf{c} = \omega \tensor \omega$ for $\omega \in \RR^2$,
$\omega \neq 0$ will lead to a non-trivial kernel only in exceptional
situations. The reason for this is that in the continuous setting,
elements in the kernel satisfy certain differential equations with
entries of $\mathbf{c}$ as coefficients (see the proof of
Lemma~\ref{contiker}). These become difference equations in the
discrete setting whose solutions generally do not coincide with
discrete versions of the continuous solutions if one does not adapt
the coefficients in $\mathbf{c}$. This induces a mismatch that
typically leads to a trivial kernel; see the derivation below for more
details.

Hence, we have to choose the coefficients of $\mathbf{c}$ in order to
fit the above discretization of the differential operators. For this
purpose, we examine the kernel of
$\TGV^{\osci}_{\alpha,\beta,\mathbf{c}}$ according
to~\eqref{eq:discrete_tgv_osci} for a
$\mathbf{c}=\begin{bmatrix} c_1&c_3\\c_3&c_2
\end{bmatrix}$
with $c_1,c_2,c_3 \in \RR$ such that $\mathbf{c} \neq 0$.
Then, one immediately sees that the kernel is given
by all $u \in U$ that satisfy $\mathcal{E}\nabla u+\mathbf{c}u=0$ in
$\Omega$.  This corresponds to
\begin{equation}\label{kerequationdiscrete}
  \left\{\begin{aligned}
      \partial^-_x\partial^+_x u+c_{1}u&=0,\\
      \partial^-_y\partial^+_y u+c_{2}u&=0,\\
      \frac{\partial^-_y\partial^+_x+\partial^-_x\partial^+_y}{2} u+c_{3}u&=0.\\
  \end{aligned}\right.
\end{equation}
For the sake of simplicity we focus, in the following, on the
``interior points'' of $\Omega$ and neglect boundary effects. As a
consequence, Equation~\eqref{kerequationdiscrete} might not be
satisfied at discrete boundary points.  This leads to
\begin{equation}\label{kernelconditiondis}
  \left\{
    \begin{aligned}
      u_{i+1,j}-2u_{i,j}+u_{i-1,j}+c_1u_{i,j}&=0,\\
      u_{i,j+1}-2u_{i,j}+u_{i,j-1}+c_2u_{i,j}&=0,\\
      u_{i+1,j}-u_{i,j}-u_{i+1,j-1}+u_{i,j-1}+u_{i-1,j}-u_{i,j}-u_{i-1,j+1}+u_{i,j+1}+2c_3u_{i,j}&=0,\\
    \end{aligned}
  \right.
\end{equation}
for all $(i,j) \in \ZZ^2$ where $u$ is defined on all discrete points,
i.e., $u: \ZZ^2 \to \RR$.

We would like to achieve that for $\omega \in \RR^2$, $u$
satisfies~\eqref{kernelconditiondis} if and only if there exist
constants $C_1, C_2 \in \RR$ such that for each $(i,j)$ it holds that
\begin{equation}
  \label{eq:discrete_kernel_cond}
  u_{i,j} = C_1 \sin(\omega_1i + \omega_2j) + C_2\cos(\omega_1i + \omega_2j).
\end{equation}
For that purpose let us, in analogy to the presentation in
Section~\ref{sec:osci_tgv}, derive the solution space
for~\eqref{kernelconditiondis}.
Taking into account the first equation
of~\eqref{kernelconditiondis}, which can be regarded as one-dimensional
if we fix the subscript $j$, and rearranging leads to the recurrence
relation
\begin{equation}\label{iteration}
u_{i+1,j}=(2-c_1)u_{i,j}-u_{i-1,j}
\end{equation}
whose solutions can be described as follows.
Denoting by $\lambda \in \CC$ a root of
$t \mapsto t^2 - (2-c_1)t + 1$, it holds that
\begin{equation}
  \label{eq:recurrence_solution}
  u_{i,j} =
  \begin{cases}
    C_{1,j}\lambda^i + C_{2,j} \conj{\lambda}^i & \text{if} \ 
    \lambda \ \text{is a simple root}, \\
    C_{1,j}\lambda^i + C_{2,j} i\lambda^i & \text{else},
  \end{cases}
\end{equation}
for suitable constants $C_{1,j}, C_{2,j} \in \CC$.  Without loss of
generality, $\lambda$ is given by
\begin{equation*}
\lambda=\frac{2-c_1 + \sqrt{c_1^2-4c_1}}{2}
\end{equation*}
where $\sqrt{\cdot}$ is a complex square root that is defined on the
real axis.  Let us first assume that $\lambda$ is simple such that
$\conj{\lambda}$ is the other root. In view of the Euler formulas
$\sin(t) = \tfrac1{2\im} (\expE^{\im t} - \expE^{-\im t})$,
$\cos(t) = \tfrac12 (\expE^{\im t} + \expE^{-\im t})$, the
representation~\eqref{eq:discrete_kernel_cond} can only be achieved if
$\abs{\lambda} = 1$.
This in
turn, is the case if and only if $c_1^2 - 4c_1 < 0$, which is
equivalent to $0 < c_1 < 4$.  Thus, in order to
satisfy~\eqref{eq:discrete_kernel_cond}, the equation
\begin{equation}\label{eigenvalue}
  \lambda =\frac{2-c_1 + \mathrm{i}\sqrt{4c_1-c_1^2}}{2}=\expE^{ \mathrm{i}\omega_1}=\cos(\omega_1) +\mathrm{i}\sin(\omega_1)
\end{equation}
has to hold for some complex root, which means that
\begin{equation}\label{c1}
  c_1= 2 - 2\cos(\omega_1).
\end{equation}
In particular, this can only be true if $\omega_1 \notin \pi \ZZ$, in
which case we can express $u$ by
\begin{equation}\label{ui}
u_{i,j}=D_{1,j}\cos(\omega_1i)+D_{2,j}\sin(\omega_1i)\quad 
\text{for all} \quad i,j
\end{equation}
where $D_{1,j}, D_{2,j} \in \RR$ for each $j$. Fixing, for instance,
$i=1,2$ and noting that the matrix $
\begin{bmatrix}
  \cos(\omega_1) & \sin(\omega_1) \\
  \cos(2\omega_1) & \sin(2\omega_1)
\end{bmatrix}
$
has full rank, the second equation in~\eqref{kernelconditiondis} leads
to the recurrence relation $D_{i,j+1} = (2 - c_2)D_{i,j} - D_{i,j-1}$
for $i=1,2$ and all $j$. Analogous to the above we get, in case of
$0 < c_2 < 4$ that $c_2 = 2 - 2 \cos(\omega_2)$,
$\omega_2 \notin \pi \ZZ$ and
\begin{equation}
  \label{eq:coefficients_dij}
  D_{i,j} = A_i \cos(\omega_2 j) + B_i \sin(\omega_2 j) \quad \text{for}
  \quad i=1,2
\end{equation}
and all $j$.  Plugging this into~\eqref{ui} yields, by virtue of the
angle sum and difference identities,
\[
\begin{aligned}
u_{i,j} &= A_1 \cos(\omega_1 i)\cos(\omega_2 j) + B_1 \cos(\omega_1 i)
\sin(\omega_2 j) + A_2 \sin(\omega_1 i) \cos(\omega_2 j) + B_2 \sin(\omega_1 i) \sin(\omega_2 j) \\
& =  T_1 \cos(\omega_1 i + \omega_2 j) + 
T_2 \cos(\omega_1 i - \omega_2 j) 
+
T_3 \sin(\omega_1 i + \omega_2 j) + 
T_4  \sin(\omega_1 i - \omega_2 j)
\end{aligned}
\]
with $T_1 = \frac{A_1 - B_2}{2}$, $T_2 = \frac{A_1 + B_2}{2}$,
$T_3 =\frac{A_2 + B_1}{2}$ and $T_4 = \frac{A_2 - B_1}{2}$.  If one
plugs the first two equations of~\eqref{kernelconditiondis} into the
third equation, one gets
\begin{equation}
  \label{eq:kernel_combined}
  u_{i+1,j-1} + u_{i-1,j+1} = (2 - c_1 - c_2 + 2c_3) u_{ij}
\end{equation}
which corresponds to, using the above as well as
$c_1 = 2 - 2\cos(\omega_1)$, $c_2 = 2 - 2\cos(\omega_2)$,
\begin{multline*}
  2 \cos(\omega_1 - \omega_2) \bigl( T_1 \cos(\omega_1 i +\omega_2 j)
  +
  T_3 \sin(\omega_1 i +\omega_2 j) \bigr) \\
  + 2 \cos(\omega_1 + \omega_2) \bigl( T_2 \cos(\omega_1 i - \omega_2
  j) + T_4 \sin(\omega_1 i - \omega_2 j) \bigr) \\
  = 2(\cos(\omega_1)
  + \cos(\omega_2) + c_3 - 1) \bigl( T_1 \cos(\omega_1 i + \omega_2 j) + 
  T_2 \cos(\omega_1 i - \omega_2 j) 
  \\
  +
  T_3 \sin(\omega_1 i + \omega_2 j) + 
  T_4  \sin(\omega_1 i - \omega_2 j) \bigr).
\end{multline*}
As $\omega_1, \omega_2 \notin \pi \ZZ$, we have
$\cos(\omega_1 - \omega_2) \neq \cos(\omega_1 + \omega_2)$. Thus,
setting
$c_3 = 1 + \cos(\omega_1 - \omega_2) - \cos(\omega_1) -
\cos(\omega_2)$
implies that the equation can only be true for all $i,j$ if
$T_2 = T_4 = 0$. But this means that $u$ has the desired
representation~\eqref{eq:discrete_kernel_cond}. Note that here, the
above choice of $c_3$ is crucial as any other choice (expect
$c_3 = 1 + \cos(\omega_1 + \omega_2) - \cos(\omega_1) -
\cos(\omega_2)$)
implies a trivial kernel. As mentioned above, this is what typically
happens without an adaptation of $\mathbf{c}$ to the discrete setting.

We still have to discuss the remaining cases. Starting with
$c_2 \in \{0,4\}$ which corresponds, in view of
$c_2 = 2 - 2 \cos(\omega_2)$, to $\omega_2 \in \pi \ZZ$, such that the
root in~\eqref{eq:recurrence_solution} is not simple and the
representation~\eqref{eq:coefficients_dij} becomes
\begin{equation}
  \label{eq:kernel_coefficient2}
  D_{i,j} = (A_i + B_i j)\cos(\omega_2 j) \quad \text{for} \quad i=1,2.
\end{equation}
Now, since $\cos(\omega_2) \in \{-1,1\}$, it follows that 
$\cos(\omega_1i)\cos(\omega_2j) = \cos(\omega_1 i + \omega_2 j)$ as well
as $\sin(\omega_1i)\cos(\omega_2j) = \sin(\omega_1 i + \omega_2 j)$ such that
\[
u_{i,j} = (A_1 + B_1 j) \cos(\omega_1 i + \omega_2 j) + (A_2 + B_2 j)
\sin(\omega_1 i + \omega_2 j).
\]
Employing~\eqref{eq:kernel_combined} then leads to
\begin{multline*}
  2\cos(\omega_1 - \omega_2) \bigl( (A_1 + B_1j) \cos(\omega_1 i +
  \omega_2 j) + (A_2 + B_2j) \sin(\omega_1i +\omega_2j) \bigr) \\
  + 2 \sin(\omega_1 - \omega_2) \bigl(B_1 \sin(\omega_1 i + \omega_2
  j) - B_2 \cos(\omega_1 i + \omega_2 j) \bigr)
  \\
  = 2(\cos(\omega_1) + \cos(\omega_2) + c_3 - 1) \bigl( (A_1 + B_1 j)
  \cos(\omega_1 i + \omega_2 j) + (A_2 + B_2 j) \sin(\omega_1 i +
  \omega_2 j) \bigr),
\end{multline*}
such that setting again
$c_3 = 1 + \cos(\omega_1 - \omega_2) - \cos(\omega_1) -
\cos(\omega_2)$,
this can only be true for all $i,j$ if $B_1 = B_2 = 0$ since
$\sin(\omega_1 - \omega_2) \neq 0$ as a consequence of
$\omega_1 \notin \pi \ZZ$, $\omega_2 \in \pi \ZZ$. But this means that
$u$ has the desired representation~\eqref{eq:discrete_kernel_cond}.
If we interchange the $i$- and $j$-axis, the same reasoning applies in
case of $c_1 \in \{0,4\}$ and $0 < c_2 < 4$, or, equivalently,
$\omega_1 \in \pi \ZZ$ and $\omega_2 \notin \pi \ZZ$.

Finally, if $c_1,c_2 \in \{0,4\}$ which corresponds to
$\omega_1, \omega_2 \in \pi \ZZ$, the
identity~\eqref{ui} becomes
\[
u_{i,j} = (D_{1,j} + D_{2,j} i) \cos(\omega_1 i) 
\quad \text{for all} \quad i,j.
\]
Proceeding as above leads to~\eqref{eq:kernel_coefficient2} and,
consequently,
\[
u_{i,j} = (A_1 + B_1j + A_2i + B_2ij) \cos(\omega_1 i + \omega_2 j).
\]
Employing~\eqref{eq:kernel_combined} again yields
\begin{multline*}
  2 \cos(\omega_1 - \omega_2) \bigl(
  A_1 + B_1 j + A_2 i + B_2(ij-1) \bigr) \cos(\omega_1 i + \omega_2 j) \\
  = 2 ( \cos(\omega_1) + \cos(\omega_2) + c_3 - 1) 
  \bigl(A_1 + B_1j + A_2i + B_2ij \bigr) \cos(\omega_1 i + \omega_2 j).
\end{multline*}
With
$c_3 = 1 + \cos(\omega_1 - \omega_2) - \cos(\omega_1) -
\cos(\omega_2)$
and considering all $i,j$, we can, however, only deduce that
$B_2 = 0$, such that the
representation~\eqref{eq:discrete_kernel_cond} is not obtainable in
this situation. Thus, we have to assume that
$(\omega_1,\omega_2) \notin \pi \ZZ^2$.

In order to be complete let us eventually note that if
$c_1 \notin [0,4]$ or $c_2 \notin [0,4]$, the root $\lambda$
in~\eqref{eq:recurrence_solution} is simple and satisfies
$\abs{\lambda} \neq 1$ such that~\eqref{eq:discrete_kernel_cond} can
also not be obtained. In summary, we have proved the following proposition:

\begin{proposition}
  Let $\omega = (\omega_1, \omega_2) \in \RR^2 \setminus \pi \ZZ^2$ and
  \begin{equation}
    \label{comega}
    \mathbf{c} =
    \begin{bmatrix}
      2 - 2 \cos(\omega_1) & 1 + \cos(\omega_1 - \omega_2) - \cos(\omega_1) - \cos(\omega_2) \\
      1 + \cos(\omega_1 - \omega_2) - \cos(\omega_1) - \cos(\omega_2) & 2 -
      2\cos(\omega_2)
    \end{bmatrix}.
  \end{equation}
  Then, $u: \ZZ^2 \to \RR$ satisfies
  $\symgrad \grad u + \mathbf{c}u = 0$ in the sense
  that~\eqref{kernelconditiondis} holds for all $(i,j) \in \ZZ^2$ if
  and only if there are constants $C_1, C_2 \in \RR$ such that
  \[
  u_{i,j} = C_1 \cos(\omega_1 i + \omega_2 j) + C_2 \sin(\omega_1 i +
  \omega_2 j) \qquad \text{for all} \qquad (i,j) \in \ZZ^2.
  \]
\end{proposition}

In particular, given a pair
$\omega=(\omega_1,\omega_2) \in \RR^2 \setminus \pi\ZZ^2$, we can
represent texture with oscillations in the corresponding direction and
frequency by the discrete functional~\eqref{eq:discrete_tgv_osci} and
$\mathbf{c}$ according to~\eqref{comega}. For
$(\omega_1,\omega_2) = 0$, Equation~\eqref{comega} yields
$\mathbf{c} = 0$, such that~\eqref{eq:discrete_tgv_osci} corresponds
to a discrete $\TGV_{\alpha,\beta}^2$.  Thus, a discrete infimal
convolution of oscillation TGV reads as follows:
\begin{align}
  \notag
  \ICTGV_{\vec{\alpha},\vec{\beta},\vec{\mathbf{c}}}^{\osci}(u) &=
  \min_{\substack{u_1,\ldots, u_m \in U, \\ w_1,\ldots,w_m \in V, \\u = u_1 + \ldots + u_m}} \sum_{i=1}^m
  \alpha_i \norm[1]{\grad u_i - w_i} + \beta_i \norm[1]{\symgrad w_i + \mathbf{c}_i u_i}
  \\
  \label{eq:discrete_ictgv}
  & = \max_{\substack{r \in U, \ p_1,\ldots,p_m \in V, \\ q_1,\ldots,q_m
      \in W}} \ \scp{u}{r} \quad \text{subject to} \quad
    \left\{\begin{aligned}
        \norm[\infty]{p_i} &\leq \alpha_i, \norm[\infty]{q_i} \leq \beta_i, \\
        p_i &= \div_2 q_i, \\
        r &= \div_1 p_i + \mathbf{c}_i \cdot q_i,
    \end{aligned}\right.
\end{align}
where $\vec{\alpha} = (\alpha_1,\ldots,\alpha_m)$,
$\vec{\beta} = (\beta_1,\ldots,\beta_m)$ are vectors of positive
parameters and
$\vec{\mathbf{c}} = (\mathbf{c}_1, \ldots, \mathbf{c}_m)$ are given
by~\eqref{comega} for pairs
$(\omega_{1,1}, \omega_{1,2}),\ldots, (\omega_{m,1}, \omega_{m,2}) \in
\RR^2 \setminus \pi \ZZ^2 \cup \{0\}$.

\begin{remark}
  \label{rem:ictgv_parameter_choice}
  The direction and frequency parameters in~\eqref{comega} can, for
  instance, be chosen as follows.  Suppose, we would like to detect
  eight different texture directions with unit frequency as well as a
  cartoon part, i.e., $m=9$. The first part in the infimal-convolution
  model is the cartoon component and therefore,
  $\omega_{1,1} = \omega_{1,2} = 0$, leading to $\mathbf{c}_1=0$.
  Further, we choose eight pairs $(\omega_1,\omega_2)$ to uniformly
  cover the (unsigned) unit directions:
  \begin{equation}\label{omega}
    \left\{
      \begin{aligned}
        \omega_{i,1}=\sin(\tfrac{(i-2)\pi}{8}), \\
        \omega_{i,2}=\cos(\tfrac{(i-2)\pi}{8}), \\
      \end{aligned}\right.
    \quad \text{for} \quad i=2,3,\ldots,9.
  \end{equation}
  If one aims at obtaining textures of higher frequency, one can, for
  instance, either replace $(\omega_{i,1}, \omega_{i,2})$ by
  $(2\omega_{i,1}, 2\omega_{i,2})$ or add the latter pairs to the
  model, resulting in $m=17$, i.e., one cartoon component, eight
  directions of low frequency and eight directions of high frequency.
  In our experiments, we utilized either the simple model ($m=9$) or
  the extended model ($m=17$).

  See Figure~\ref{texture} for a visualization of the kernels $u$
  associated with $\TGV^{\osci}_{\alpha_i,\beta_i,\mathbf{c}_i}$ for
  the above choice.
\end{remark}

\begin{figure}
  \center{} 
  \subfigure[$k=0$]{%
    \begin{minipage}{0.115\linewidth}
      \includegraphics[width=\textwidth]{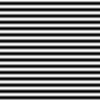}
      \\[-1em]
    \end{minipage}}
  \subfigure[$k=1$]{%
    \begin{minipage}{0.115\linewidth}
      \includegraphics[width=\textwidth]{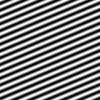}
      \\[-1em]
    \end{minipage}}
  \subfigure[$k=2$]{%
    \begin{minipage}{0.115\linewidth}
      \includegraphics[width=\textwidth]{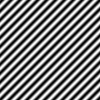}
      \\[-1em]
    \end{minipage}}
  \subfigure[$k=3$]{%
    \begin{minipage}{0.115\linewidth}
      \includegraphics[width=\textwidth]{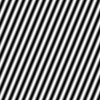}
      \\[-1em]
    \end{minipage}}
  \subfigure[$k=4$]{%
    \begin{minipage}{0.115\linewidth}
      \includegraphics[width=\textwidth]{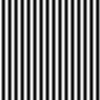}
      \\[-1em]
    \end{minipage}}
  \subfigure[$k=5$]{%
    \begin{minipage}{0.115\linewidth}
      \includegraphics[width=\textwidth]{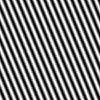}
      \\[-1em]
    \end{minipage}}
  \subfigure[$k=6$]{%
    \begin{minipage}{0.115\linewidth}
      \includegraphics[width=\textwidth]{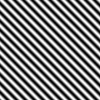}
      \\[-1em]
    \end{minipage}}
  \subfigure[$k=7$]{%
    \begin{minipage}{0.115\linewidth}
      \includegraphics[width=\textwidth]{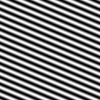}
      \\[-1em]
    \end{minipage}}
  \caption{Elements from the kernel of discrete oscillation TGV
    according to~\eqref{eq:discrete_tgv_osci} for the choice
    $\omega_1 = \sin(\tfrac{k\pi}8)$, $\omega_2 = \cos(\tfrac{k\pi}8)$
    and $\mathbf{c}$ according to~\eqref{omega} and constants $C_1=C_2=1$.}
\label{texture}
\end{figure}

\subsection{A numerical optimization algorithm}
Now, we focus on numerical algorithms for solving the Tikhonov
minimization problem~\eqref{eq:ictgv_tikh} for $q=2$ which is
associated with $\ICTGV^{\osci}$-regularization of the inverse problem
$Ku = f$. We assume that the linear operator $K: U \to Y$, where $Y$
is a Hilbert space is already given in a discretized form. With the
representation~\eqref{eq:discrete_ictgv},
problem~\eqref{eq:ictgv_tikh} reads as
\begin{equation*}
  \min_{\substack{u_1,\ldots,u_m \in U, \\ w_1,\ldots, w_m \in V}} \ 
  \frac{1}{2}\|K\sum_{i=1}^{m}u_i-f\|^2 + \sum_{i=1}^{m} \alpha_i\|\nabla u_i-w_i\|_1+\beta_i\|\mathcal{E} w_i+\mathbf{c}_iu_i\|_1.
\end{equation*}
For the image processing problems we are considering,
we also use a discrete version of the
$\ICTGV^{\osci}_{\vec{\alpha},\vec{\beta},\vec{\mathbf{c}},\vec{\gamma}}$-functionals
according to~\eqref{eq:ictgv_l1} with additional parameters
$\vec{\gamma} = (\gamma_1,\ldots,\gamma_m)$ where $\gamma_i \geq 0$
for each $i=1,\ldots,m$. This is might in particular be beneficial for
texture detection in the presence of noise as it enforces sparsity of
the respective texture component and prevents noise being recognized
as texture.
The modification then leads to the following minimization problem:
\begin{equation}
  \label{eq:discrete_problem}
  \min_{\substack{u_1,\ldots,u_m \in U, \\ w_1,\ldots, w_m \in V}}\
  \frac{1}{2}\|K\sum_{i=1}^{m}u_i-f\|^2
  + \sum_{i=1}^{m} \alpha_i\|\nabla u_i-w_i\|_1+\beta_i\|\mathcal{E} w_i+\mathbf{c}_iu_i\|_1+\gamma_i\|u_i\|_1,
\end{equation}
where we set $\gamma_i = 0$ if $\mathbf{c} = 0$ in order to prevent a
sparsification of the cartoon component (recall that then,
$\TGV^{\osci}_{\alpha,\beta,\mathbf{c}} = \TGV^2_{\alpha,\beta}$).
Our aim is to employ primal-dual algorithms which operate on
equivalent saddle-point formulations of the minimization problem.
For~\eqref{eq:discrete_problem}, one can easily see that
Fenchel--Rockafellar duality is applicable (see, e.g., \cite{Ekeland}), such that primal-dual solutions are equivalent to
solutions of the saddle-point problem
\begin{align}
  \notag
  \min_{\substack{u_1,\ldots,u_m \in U, \\ w_1,\ldots, w_m \in V}}  \ 
  &\max_{\substack{p_1,\ldots, p_m \in V, \\ q_1,\ldots, q_m \in W, \\
  \lambda \in Y}} \ 
  \Bigl( \langle K\sum_{i=1}^{m}u_i,\lambda\rangle + 
    \sum_{i=1}^{m} \langle \nabla u_i-w_i,p_i\rangle+\langle \mathcal{E} w_i+\mathbf{c}_i u_i,q_i\rangle \Bigr) \\
  \label{minmaxTTGVdecom}
  & + \Bigl( \sum_{i=1}^m \gamma_i\|u_i\|_1  \Bigr) -
    \Bigl(\scp{f}{\lambda} + \frac{\|\lambda\|^2}{2} + 
    \sum_{i=1}^{m}\mathcal{I}_{\{\|\,\cdot\,\|_\infty\leq \alpha_i\}}(p_i)
    + \mathcal{I}_{\{\|\,\cdot\,\|_\infty\leq \beta_i\}}(q_i)
    \Bigr).
\end{align}
This reformulation allows to use, e.g., the primal-dual method in
\cite{CP}, which we will shortly describe in the following.
It solves abstract convex-concave saddle-point problems of the form
\begin{equation}\label{pdorig}
\min_{x\in\mathcal{X}} \ \max_{y\in\mathcal{Y}} \ \langle \mathcal{K}x,y\rangle +G(x)-F^\ast(y),
\end{equation}
where $\mathcal{X},\mathcal{Y}$ are Hilbert spaces,
$\mathcal{K}:\mathcal{X}\rightarrow\mathcal{Y}$ is a continuous linear
mapping, and the functionals
$G:\mathcal{X}\rightarrow {]{-\infty,\infty}]}$ and
$F^\ast:\mathcal{Y}\rightarrow {]{-\infty,\infty}]}$ are proper, convex and
lower semi-continuous. The problem~\eqref{pdorig} is associated to the
Fenchel--Rockafellar primal-dual problems
\begin{equation}
  \min_{x\in\mathcal{X}} \ F(\mathcal{K}x)+G(x), \qquad 
\max_{y \in \mathcal{Y}} \ -F^*(y) -G^*(-\mathcal{K}^*y).
\end{equation}
In order to state the algorithm clearly, we have to give the notion of
resolvent operators $(I+\tau\partial G)^{-1}$ and
$(I+\sigma\partial F^\ast)^{-1}$, respectively, which correspond to
the solution operators of certain minimization problems, the so-called
proximal operators:
\begin{equation*}
\begin{split}
& x^\ast=(I+\tau\partial G)^{-1}(\bar{x})=\argmin_{x\in\mathcal{X}} \ \frac{\|x-\bar{x}\|_\mathcal{X}^2}{2}+\tau G(x),\\
&y^\ast=(I+\sigma\partial F^\ast)^{-1}(\bar{y})=\argmin_{y\in\mathcal{Y}} \ \frac{\|y-\bar{y}\|_\mathcal{Y}^2}{2}+\sigma F^\ast(y)
\end{split}
\end{equation*}
where $\tau,\sigma>0$ are step-size parameters we need to choose
suitably. Given the initial point
$(x^0,y^0)\in\mathcal{X}\times\mathcal{Y}$ and setting
$\bar{x}^0=x^0$, the iterative procedure in \cite{CP} can be written
as follows:
\begin{equation}
\label{eq:primal_dual_alg}
\begin{cases}
&y^{n+1}=(I+\sigma\partial F^\ast)^{-1}(y^n+\sigma \mathcal{K}\bar{x}^n),\\
&x^{n+1}=(I+\tau\partial G)^{-1}(x^n-\tau \mathcal{K}^\ast y^{n+1}),\\
&\bar{x}^{n+1}=2x^{n+1}-x^n.
\end{cases}
\end{equation}
It converges if the condition
$\tau\sigma\|\mathcal{K}\|^2 < 1$ is satisfied which corresponds to
choosing appropriate values for $\tau,\sigma>0$.

Next, we will delineate the iteration~\eqref{eq:primal_dual_alg}
adapted to our problem~\eqref{minmaxTTGVdecom} which can be
reformulated into above structure by redefining its variables and
operators as:
\begin{equation*}
\begin{split}
&x=(u_1,w_1,\ldots,u_m,w_m)\in\mathcal{X}=\underbrace{(U\times V) \times\cdots\times (U\times V)}_{\text{$m$-times}},\\
&y=(p_1,q_1,\ldots,p_m,q_m,\lambda)\in\mathcal{Y}=\underbrace{(V\times W) \times\cdots\times (V\times W)}_{\text{$m$-times}} \times Y,\\
& \mathcal{K}=\begin{pmatrix}
\begin{pmatrix}
\nabla & -I \\
\mathbf{c}_1 & \mathcal{E} 
\end{pmatrix}  & 0 & \cdots&0\\
0& \begin{pmatrix}
\nabla & -I \\
\mathbf{c}_2 & \mathcal{E} 
\end{pmatrix} & \cdots &0\\
\vdots& \vdots&  \ddots &\vdots \\
0& 0&\cdots & \begin{pmatrix}
\nabla & -I \\
\mathbf{c}_m & \mathcal{E} 
\end{pmatrix}\\
\begin{pmatrix}
K & 0 
\end{pmatrix}& \begin{pmatrix}
K & 0 
\end{pmatrix} & \cdots &\begin{pmatrix}
K & 0 
\end{pmatrix}
\end{pmatrix},
\end{split}
\end{equation*}
as well as
\begin{equation*}
\begin{split}
  & G(x)=\sum_{i=1}^{m}\gamma_i\|u_i\|_1,\\
  & F^\ast(y) = \langle
  f,\lambda\rangle+\frac{\|\lambda\|^2}{2}+\sum_{i=1}^{m} \mathcal{I}_{\{\|\cdot\|_\infty\leq
    \alpha_i\}}(p_i)+\mathcal{I}_{\{\|\cdot\|_\infty\leq
    \beta_i\}}(q_i).
\end{split}
\end{equation*}
In order to guarantee the convergence condition
$\tau\sigma\| \mathcal{K}\|^2 < 1$, we have to estimate the norm of
the operator $\mathcal{K}$. At first, if $\mathbf{c} =
\begin{bmatrix}
  c_1 & c_3 \\ c_3 & c_2
\end{bmatrix}$, then 
\[
\norm{\mathbf{c}}^2 = \frac{c_1^2}{2} + \frac{c_2^2}{2} + c_3^2 + \frac{c_1 + c_2}{2} \sqrt{(c_1 - c_2)^2 + 4c_3^2}.
\]
If we denote by
$c^2=\max\{\|\mathbf{c}_1\|^2,\ldots,\|\mathbf{c}_m\|^2\}$ and observe that
$\|\nabla\|^2<8$ and $\|\mathcal{E}\|^2<8$, then, after some computations,
one obtains the estimate
\[
\norm{\mathcal{K}}^2 < \frac{c^2 + (c+1) \sqrt{(c-1)^2 + 32} + 17}{2} +
m \norm{K}^2 = C_0\left(c,m,K\right),
\]
thus, in order to ensure convergence, it suffices to choose
$\tau, \sigma >0$ such that $\sigma\tau = \tfrac1{C_0\left(c,m,K\right)}$.

By the above analysis, the primal-dual method for solving the saddle-point
formulation~\eqref{minmaxTTGVdecom} of the imaging
problem~\eqref{eq:discrete_problem} reads as
Algorithm~\ref{TTGVdecomalgo}.
\begin{algorithm}[t]
  \caption{Primal-Dual Method for infimal convolution oscillation TGV}
  \label{TTGVdecomalgo}
  \textbf{Initialization:} Choose $\tau,\sigma$ and $x^0,y^0$, set $\bar{x}^0=x^0$;

  \textbf{Iteration:} Update according to the following steps
  \begin{algorithmic}
    \STATE $\lambda^{n+1}= \bigl(\lambda^n+\sigma(K\sum\limits_{i=1}\limits^{m}\bar{u}_i^n-f) \bigr)/(1+\sigma)$;
    \FOR{$i=1,\ldots,m$} 
    \STATE $p_i^{n+1}=\mathcal{P}_{\alpha_i}\bigl(p_i^n+\sigma(\nabla \bar{u}_i^n-\bar{w}_i^n)\bigr)$;
    \STATE $q_i^{n+1}=\mathcal{P}_{\beta_i}\bigl(q_i^n+\sigma(\mathcal{E}\bar{w}_i^n+\mathbf{c}_i \bar{u}_i^n)\bigr)$;
    \STATE $\tilde{u}_i^{n+1}=u_i^n-\tau(K^\ast\lambda^{n+1}-\div_1 p_i^{n+1}+\mathbf{c}_i q_i^{n+1})$;
    \STATE $u_i^{n+1}=\Shrink_{\tau\gamma_i}(\tilde{u}_i^{n+1})$;
    \STATE $w_i^{n+1}=w_i^n+\tau(p_i^{n+1}+\div_2 q_i^{n+1})$;
    \STATE $\bar{u}_i^{n+1}=2u_i^{n+1}-u_i^n$;
    \STATE $\bar{w}_i^{n+1}=2w_i^{n+1}-w_i^n$;
    \ENDFOR
  \end{algorithmic}
  \textbf{If converged:} Return $u = \sum\limits_{i=1}^m u_i^{n}$;
\end{algorithm}
Note that in the algorithm, the corresponding proximal operators are
already written in an explicit form (see, for instance,
\cite{TGVmulti} for more details).  In particular, the expressions in
the updates of $p_i^{n+1}$, $q_i^{n+1}$ and $u^{n+1}_i$ read as
follows:
\begin{equation*}
  \mathcal{P}_{\eta}(t) = 
  \min \Bigl(1, \frac{\eta}{\abs{t}} \Bigr) t,
  \qquad
  \Shrink_{\eta}(t) = \max \Bigl( 0, 1 - \frac{\eta}{\abs{t}} \Bigr) t.
\end{equation*}

\section{Applications and numerical experiments}
\label{sec:applications}

In this section, we discuss applications of infimal convolution of
oscillation TGV regularization to several image processing problems as
well as associated numerical experiments including comparisons to some
existing state-of-the-art methods for the respective problems. The
applications cover cartoon/texture decomposition, image denoising,
image inpainting and undersampled magnetic resonance imaging (MRI).

In particular, we assess the reconstruction results in terms of the
peak signal to noise ratio (PSNR) and structural similarity index
(SSIM). %
All parameters in the proposed and reference methods were manually
tuned to give the best results in terms of these measures; we report
them for the proposed methods in the figure captions of each case. The
manual tuning of the parameters was performed, for each method, on a
reasonably exhaustive set of parameters and was based on our experience
as well as the recommendations of the authors of the respective
method; we refer to the supplementary material Section D for
additional details on the implementations and parameter choice.  All
experiments were performed in MATLAB R2016a running on a computer
with %
Intel Core i5 CPU at 3.30 GHz and 8 GB of memory.  For some
implementations, we utilized the ``gpuArray'' feature of MATLAB's
parallelization toolbox to achieve an embarrasing parallel
implementation of the code for which the numerical experiments were
also run on a NVIDIA GTX 1070 GPU.  In all of the
examples, the range of image values is $[0,1]$. Note also that in all
experiments, the iteration number was set to a fixed value of 2000, a
number that was experimentally determined to be sufficient for
convergence of the algorithm. %
Moreover, unless stated otherwise, direction and frequency parameters
are chosen according to Remark~\ref{rem:ictgv_parameter_choice}, i.e.,
for $m=9$, there are eight equispaced texture directions and for
$m=17$, there are eight equispaced texture directions and for each
direction, there is a low and high frequency component.  Also note
that we did not perform an estimation of the directions and
frequencies from the data.

\subsection{Cartoon/texture decomposition and image denoising}
The proposed regularization can be applied to both cartoon/texture
decomposition as well as image denoising by setting $K = I$ in the
Tikhonov minimization problem~\eqref{eq:ictgv_tikh} as well as
$Y = L^2(\Omega)$ and $p=q=2$, i.e.,
\begin{equation}
  \label{TTGVdenoise}
  \min_{u \in L^2(\Omega)} \ \frac12 \norm[2]{u - f}^2 +
  \ICTGV^{\osci}_{\vec{\alpha}, \vec{\beta}, \vec{\mathbf{c}}}(u).
\end{equation}
In the case of cartoon/texture decomposition, we can use
$\ICTGV^{\osci}$ according to~\eqref{ICTGV2}, i.e., setting
$\gamma_i = 0$ for all $i$ in the algorithm. In contrast, in the case
of noise, which can be regarded as small-scale textures, the noise
will influence the texture detection and lead to each texture
component having undesired ghost stripes. In order to avoid this
effect, we use the variant with added $L^1$-norm according
to~\eqref{eq:ictgv_l1} in order to sparsify the texture which leads to
replacing
$\ICTGV^{\osci}_{\vec{\alpha}, \vec{\beta}, \vec{\mathbf{c}}}$ by
$\ICTGV^{\osci}_{\vec{\alpha}, \vec{\beta},
  \vec{\mathbf{c}},\vec{\gamma}}$ in~\eqref{TTGVdenoise}.

We performed several numerical experiments in order to study the
effectiveness as well as benefits and potential weaknesses of this
approach.  In our first experiment, we show the performance of the new
regularization
for a synthetic image with one piecewise affine region and four
directional textures corresponding to the angles
$0,\frac{\pi}{4},\frac{\pi}{2}$ and $\frac{3\pi}{4}$. Choosing a
$\ICTGV^{\osci}$ functional as in
Remark~\ref{rem:ictgv_parameter_choice} but with these four directions
instead of eight (leading to $m=5$) gives a decomposition as depicted
in Figure~\ref{decomsyn}.
One can see that the infimal convolution of oscillation TGV models are
almost perfect in decomposing the image into the respective cartoon
and texture parts. Moreover, Figure~\ref{decomsyn} shows denoising
results for the same synthetic image which are also almost perfect. Of
course, the reason for this behavior is that the synthetic ground
truth corresponds, in the respective regions, exactly to the kernels
of the $\TGV^{\osci}$-functionals utilized in the model; leading to a
low value of the $\ICTGV^{\osci}$-functional and a stable minimum of
the objective functional.

\begin{figure}
  \begin{minipage}{0.76\linewidth}
    \center{}
    \subfigure[original]{%
      \begin{minipage}{0.31\linewidth}
        \includegraphics[width=\textwidth]{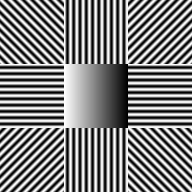}
        \\[-1em]
      \end{minipage}}
    \subfigure[cartoon $u_1$]{%
      \begin{minipage}{0.31\linewidth}
        \includegraphics[width=\textwidth]{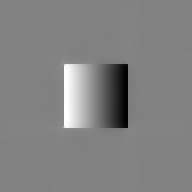}
        \\[-1em]
      \end{minipage}}
    \subfigure[$\text{angle} =0$, $u_2$]{%
      \begin{minipage}{0.31\linewidth}
        \includegraphics[width=\textwidth]{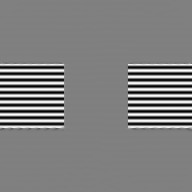}
        \\[-1em]
      \end{minipage}}
    \subfigure[$\text{angle} =\frac{\pi}{4}$, $u_3$]{%
      \begin{minipage}{0.31\linewidth}
        \includegraphics[width=\textwidth]{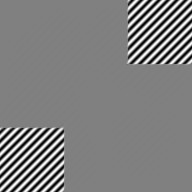}
        \\[-1em]
      \end{minipage}}
    \subfigure[$\text{angle} =\frac{\pi}{2}$, $u_4$]{%
      \begin{minipage}{0.31\linewidth}
        \includegraphics[width=\textwidth]{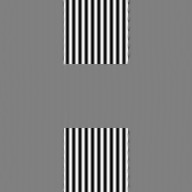}
        \\[-1em]
      \end{minipage}}
    \subfigure[$\text{angle} =\frac{3\pi}{4}$, $u_5$]{%
      \begin{minipage}{0.31\linewidth}
        \includegraphics[width=\textwidth]{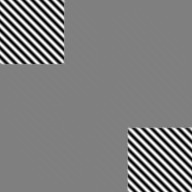}
        \\[-1em]
      \end{minipage}}
  \end{minipage}%
  \begin{minipage}{0.236\linewidth}
    \center{}
    \subfigure[noisy image]{%
      \begin{minipage}{\linewidth}
        \includegraphics[width=\textwidth]{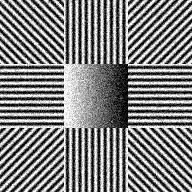}
        \\[-1em]
      \end{minipage}}
    \subfigure[restored image]{%
      \begin{minipage}{\linewidth}
        \includegraphics[width=\textwidth]{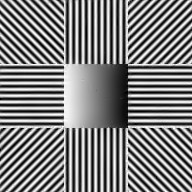}
        \\[-1em]
      \end{minipage}}
  \end{minipage}
  \caption{Decomposition/denoising of a synthetic image. (a)
    Original image of dimensions $192\times192$; (b) Cartoon component
    of the image; (c)--(f) The four directional texture
    components. Parameter choice:
    $\alpha_1=0.06, \beta_1=2\alpha_1, \gamma_1=0,
    \alpha_i=0.5\alpha_1, \beta_i=2\alpha_i, \gamma_i=0,
    i=2,\ldots,5$.
    (g) Noisy image corrupted by Gaussian noise with zero mean
    and standard deviation $\sigma=0.1$; (h) Image restored by
    model~\eqref{TTGVdenoise}. Parameter choice:
    $\alpha_1=0.12, \beta_1=2\alpha_1, \gamma_1=0,
    \alpha_i=0.5\alpha_1, \beta_i=2\alpha_i,
    \gamma_i=0.4\alpha_i, i=2,\ldots,5$. }
  \label{decomsyn}
\end{figure}

\begin{figure}
  \center{}
  \subfigure[noise-free test image]{%
    \begin{minipage}{0.3\linewidth}
      \includegraphics[width=\textwidth]{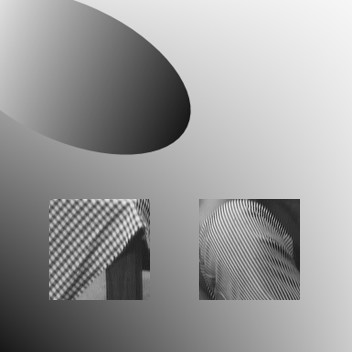}
      \\[-1em]
    \end{minipage}}
  \subfigure[noisy test image ($\sigma=0.05$)]{%
    \begin{minipage}{0.3\linewidth}
      \includegraphics[width=\textwidth]{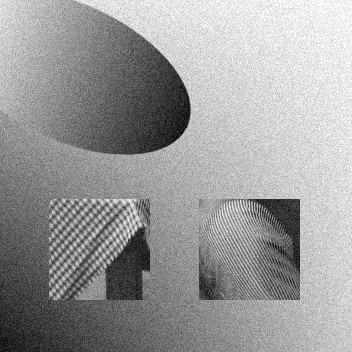}
      \\[-1em]
    \end{minipage}}
  
  \center{} 
  \subfigure[TV-$G$-norm]{%
    \begin{minipage}{0.22\linewidth}
      \includegraphics[width=\textwidth]{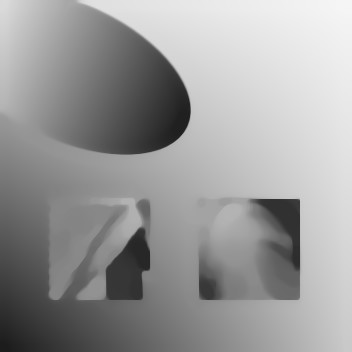}
      \\[\smallskipamount]
      \includegraphics[width=\textwidth]{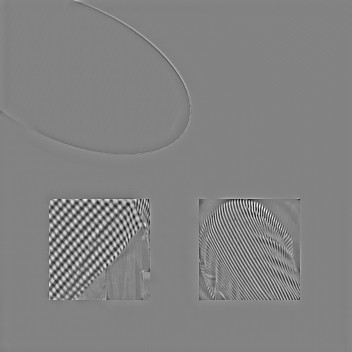}
      \\[-1em]
    \end{minipage}}
  \subfigure[TV-$H^{-1}$ ]{%
    \begin{minipage}{0.22\linewidth}
      \includegraphics[width=\textwidth]{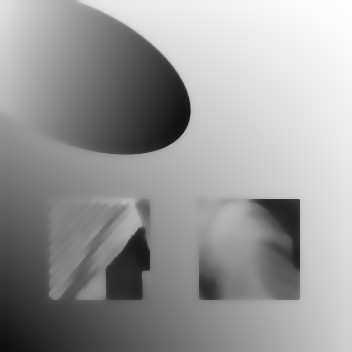}
      \\[\smallskipamount]
      \includegraphics[width=\textwidth]{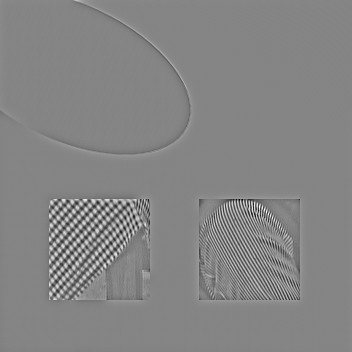}
      \\[-1em]
    \end{minipage}}
  \subfigure[Fra+LDCT]{%
    \begin{minipage}{0.22\linewidth}
      \includegraphics[width=\textwidth]{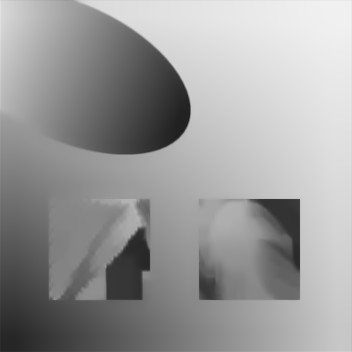}
      \\[\smallskipamount]
      \includegraphics[width=\textwidth]{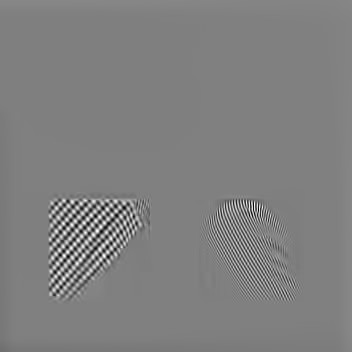}
      \\[-1em]
    \end{minipage}}
  \subfigure[proposed]{%
    \begin{minipage}{0.22\linewidth}
      \includegraphics[width=\textwidth]{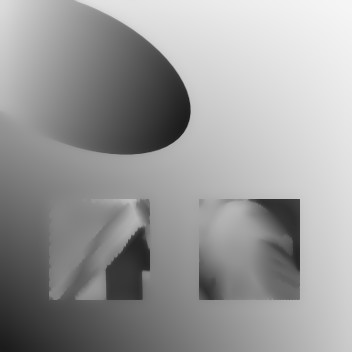}
      \\[\smallskipamount]
      \includegraphics[width=\textwidth]{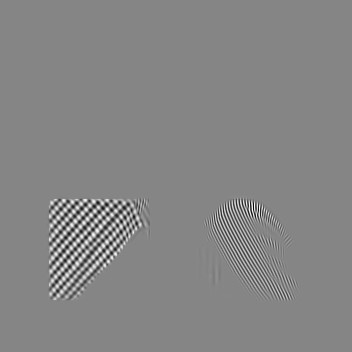}
      \\[-1em]
    \end{minipage}}

  \center{} 
  \subfigure[TV-$G$-norm]{%
    \begin{minipage}{0.22\linewidth}
      \includegraphics[width=\textwidth]{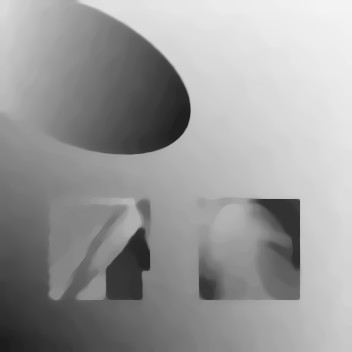}
      \\[\smallskipamount]
      \includegraphics[width=\textwidth]{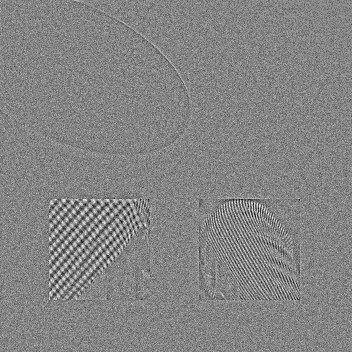}
      \\[-1em]
    \end{minipage}}
  \subfigure[TV-$H^{-1}$ ]{%
    \begin{minipage}{0.22\linewidth}
      \includegraphics[width=\textwidth]{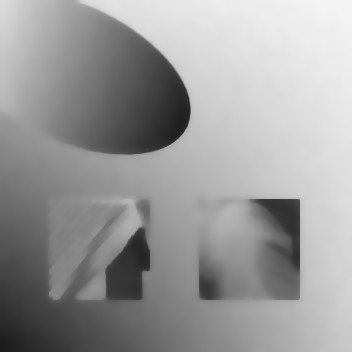}
      \\[\smallskipamount]
      \includegraphics[width=\textwidth]{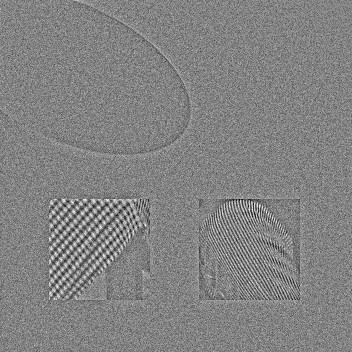}
      \\[-1em]
    \end{minipage}}
  \subfigure[Fra+LDCT]{%
    \begin{minipage}{0.22\linewidth}
      \includegraphics[width=\textwidth]{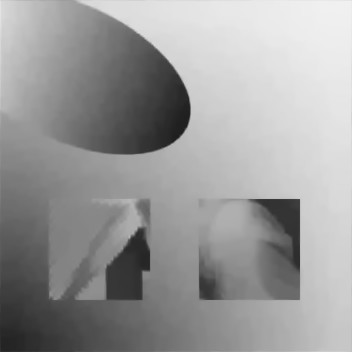}
      \\[\smallskipamount]
      \includegraphics[width=\textwidth]{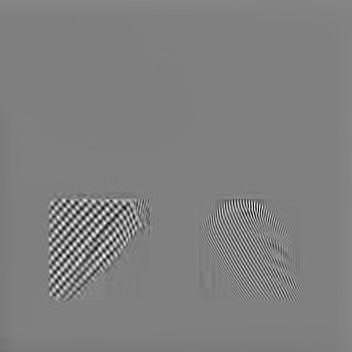}
      \\[-1em]
    \end{minipage}}
  \subfigure[proposed]{%
    \begin{minipage}{0.22\linewidth}
      \includegraphics[width=\textwidth]{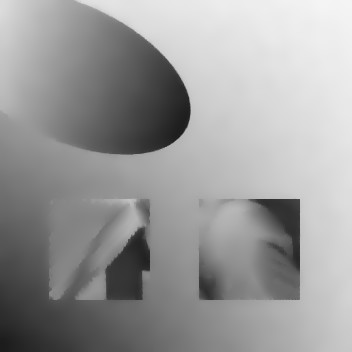}
      \\[\smallskipamount]
      \includegraphics[width=\textwidth]{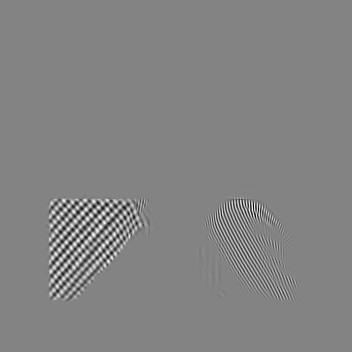}
      \\[-1em]
    \end{minipage}}

  \caption{Comparison of cartoon/texture decomposition methods for a
    synthetic image.  (a) A noise-free test image consisting of a
    cartoon part and textured regions ($352 \times 352$ pixels); (b) A
    noisy version of (a) with Gaussian noise of zero mean and variance
    $\sigma=0.05$; (c)--(f) Decomposition results for the noise-free
    image (top: cartoon part, bottom: texture part); (g)--(j)
    Decomposition results for the noisy image. Parameter choice for
    the proposed model:
    $\alpha_1=0.07, \beta_1=2\alpha_1, \gamma_1=0,
    \alpha_i=0.9\alpha_1, \beta_i=\alpha_i, \gamma_i=0.1\alpha_i,
    i=2,\ldots,17$.}
  \label{decomposition_test}
\end{figure}

In Figure~\ref{decomposition_test} we show the decomposition/denoising
performance for a synthetic image with cartoon part and some textured
regions in comparison to the $\TV$-$G$-norm model \cite{aujoldecom},
$\TV$-$H^{-1}$ model \cite{tvH1} and the framelet/local discrete
cosine transform (Fra+LDCT) model \cite{framelet}. We performed
experiments both for a noise-free and noisy image (with Gaussian noise
of variance $\sigma=0.05$). In the results, one can see that some
cartoon components (such as edges) and undesired information (such as
noise), appear in the texture parts of the $\TV$-$G$-norm and
$\TV$-$H^{-1}$ model which do not appear in the Fra+LDCT and the
proposed model. Comparing the cartoon components of all models, one
can observe staircasing effects in the $\TV$-$G$-norm, $\TV$-$H^{-1}$
and Fra+LDCT model as well as some texture artifacts in the
$\TV$-$H^{-1}$ model. Due to the utilization of $\TGV$ for the cartoon
component, these artifacts do not appear in the $\ICTGV^{\osci}$
model.  Focusing on the texture component, both Fra+LDCT and the
proposed model yield high-quality results. In the Fra+LDCT results,
however, one can, for instance, see some low-frequency components in
the trousers that are attributable to a shading and hence, should
belong to the cartoon part. In contrast, the proposed $\ICTGV^{\osci}$
model does correctly assign these shadings to the cartoon
part. Moreover, the boundaries of the texture seem to be better
delineated in the proposed model compared to Fra+LDCT decomposition
which shows some spill-over effects in the texture part. In summary,
one can observe that the $\ICTGV^{\osci}$ decomposition/denoising
model is effective in separating cartoon, texture and noise.

We also performed numerical experiments for the denoising of three
natural images (``barbara'', ``zebra'' and ``parrots'') and two
different noise levels (Gaussian noise with zero mean and standard
deviation $\sigma= 0.05$ and $0.1$, respectively).  The ``barbara''
and ``zebra'' test images have plentiful textures, some of which are
large scale and some are small. They can be interpreted as different
frequency components, so we used the $\ICTGV^{\osci}$-model with
sixteen texture components according to
Remark~\ref{rem:ictgv_parameter_choice}, i.e., $m=17$ for these
images. For the ``parrots'' image, we performed the experiments with
eight texture components, i.e., $m=9$.

The results were compared to a selection of popular and
state-of-the-art variational image models: $\TGV$ of second order,
non-local $\TV$ (NLTV) \cite{nltv1,nltv2}, ICTGV \cite{ICTVMartin} as
well as framelet/local DCT (Fra+LDCT) \cite{framelet}. Additionally,
we compared to BM3D \cite{bm3d} which is a dedicated denoising method.
For all models, we manually tuned the parameters to give the best peak
signal-to-noise ratio (PSNR). 
  However, as some methods tend to
produce ghost artifacts for PSNR-optimal parameters, we also tuned
parameters in order to yield visually pleasing results; at least from
the subjective viewpoint of the authors (namely, the NLTV, ICTGV,
Fra+LDCT and the proposed model). The PSNR and SSIM values of these
experiments are reported in Tables~\ref{table:denoisingPSNR}
and~\ref{table:denoisingSSIM}, respectively.

\begin{table}\small
	\center{}
	\begin{tabular}{|c|c|c|c|c|c|c|c|}
		\hline
		\hline
		noise level&image&TGV &NLTV&ICTGV&Fra+LDCT& BM3D& proposed model \\
		\hline
		\multirow{3}{*}{$\sigma=0.05$}
		& barbara & 27.41 & 32.05 (31.28) &31.15 (30.33)& 31.70 (31.19) & \textbf{34.43} & 32.21 (31.44)\\
		\cline{2-8}& zebra & 28.74 & 31.16 (30.07) &31.08 (30.43)& 30.81 (30.06) & \textbf{32.25} & 31.33 (30.46) \\
		\cline{2-8}& parrots & 34.77 & 35.10 (34.08) &36.49 (35.63)& 36.10 (35.49) & \textbf{37.40} & 36.61 (36.01) \\
		\hline
		\multirow{3}{*}{$\sigma=0.1$}
		& barbara & 25.34 & 27.58 (27.15) &27.35 (26.71)& 28.05 (27.66) & \textbf{31.06} & 28.45 (27.89)\\
		\cline{2-8}& zebra & 25.24 & 27.35 (26.93) &27.35 (26.83)& 27.06 (26.44) & \textbf{28.80} & 27.65 (27.00) \\
		\cline{2-8}& parrots & 32.51 & 31.39 (30.13) &33.28 (32.35) & 32.91 (32.65) & \textbf{34.08} & 33.32 (32.70) \\
		\hline
	\end{tabular}
	\caption{\label{table:denoisingPSNR} Comparison of denoising performance for noise levels $\sigma=0.05$ and $\sigma=0.1$ in terms of PSNR. (The values in parentheses correspond to the visually-optimized results.)}
\end{table}

\begin{table}\small
	\center{}
        \scalebox{0.9}{%
	\begin{tabular}{|c|c|c|c|c|c|c|c|}
		\hline
		\hline
		noise level&image&TGV &NLTV&ICTGV&Fra+LDCT& BM3D& proposed model \\
		\hline
		\multirow{3}{*}{$\sigma=0.05$}
		& barbara & 0.8043 & 0.9005 (0.8917)&0.8859 (0.8833) & 0.8857 (0.8988) &\textbf{0.9365}& 0.9004 (0.8988)\\
		\cline{2-8}&zebra & 0.8386 & 0.8647 (0.8419) &0.8761 (0.8735) & 0.8653 (0.8671) & \textbf{0.8998} & 0.8859 (0.8777) \\
		\cline{2-8}& parrots & 0.9157 & 0.8956 (0.9062) &0.9324 (0.9272) & 0.9259 (0.9257) & \textbf{0.9416} & 0.9358 (0.9303) \\
		\hline
		\multirow{3}{*}{$\sigma=0.1$}
		& barbara & 0.7297 & 0.7553 (0.7814) &0.7835 (0.7840)& 0.8214 (0.8201) &\textbf{0.8920} & 0.8235 (0.8120)\\
		\cline{2-8}&zebra & 0.7472 & 0.7271 (0.7540) &0.7796 (0.7897)& 0.7868 (0.7795) & \textbf{0.8237} & 0.7965 (0.7948) \\
		\cline{2-8}& parrots & 0.8887 & 0.8312 (0.8530) &0.8975 (0.8889)& 0.8903 (0.8928) & \textbf{0.9025} & 0.8979 (0.8936) \\
		\hline
	\end{tabular}}
	\caption{\label{table:denoisingSSIM} Comparison of denoising performance for noise levels $\sigma=0.05$ and $\sigma=0.1$ in terms of SSIM. (The values in parentheses correspond to the visually-optimized results.)}
\end{table}

\begin{figure}
  \center{} 
    \subfigure[ground truth]{%
      \begin{minipage}{0.22\linewidth}
        \includegraphics[width=\textwidth]{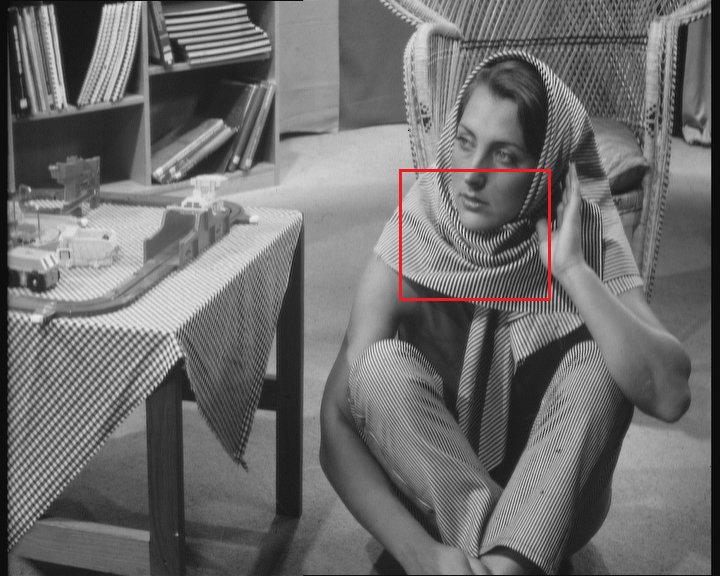}
        \\[\smallskipamount]
        \includegraphics[width=\textwidth]{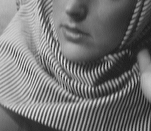}
        \\[-1em]
      \end{minipage}}
  \subfigure[noisy image ($\sigma=0.05$)]{%
    \begin{minipage}{0.22\linewidth}
      \includegraphics[width=\textwidth]{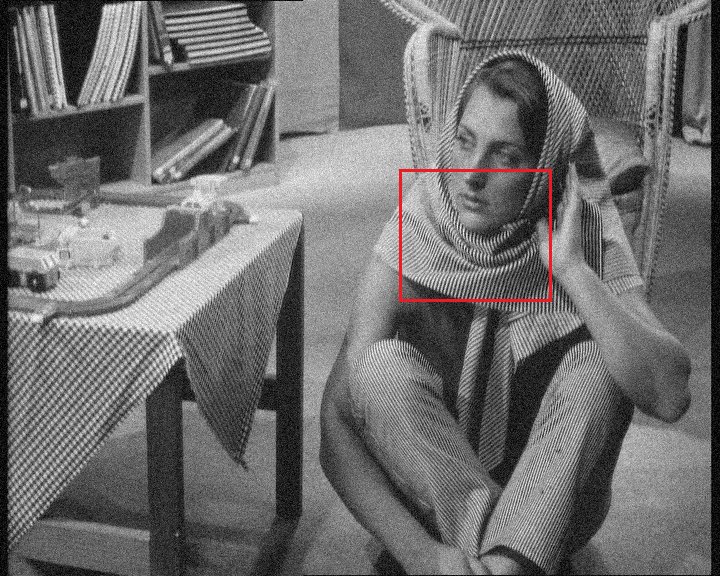} 
      \\[\smallskipamount]
      \includegraphics[width=\textwidth]{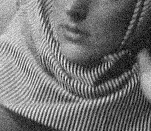}
      \\[-1em]
    \end{minipage}}
  \subfigure[TGV]{%
    \begin{minipage}{0.22\linewidth}
      \includegraphics[width=\textwidth]{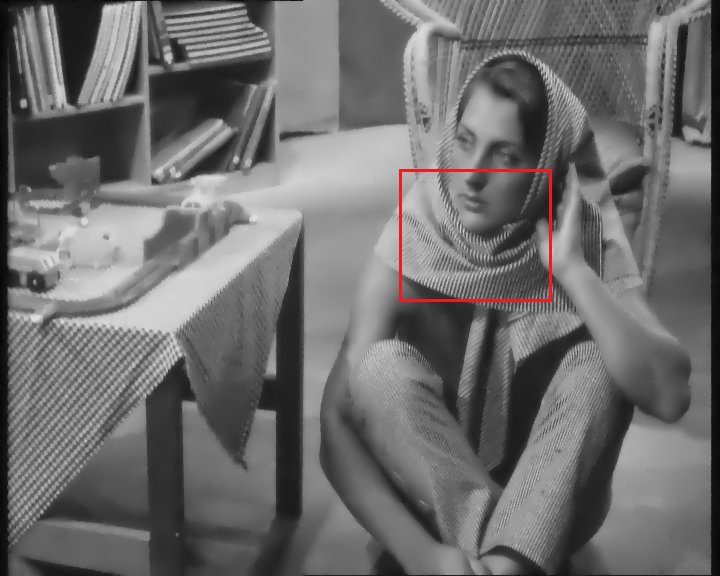}
      \\[\smallskipamount]
      \includegraphics[width=\textwidth]{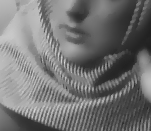}
      \\[-1em]
    \end{minipage}}
  \subfigure[NLTV]{%
    \begin{minipage}{0.22\linewidth}
      \includegraphics[width=\textwidth]{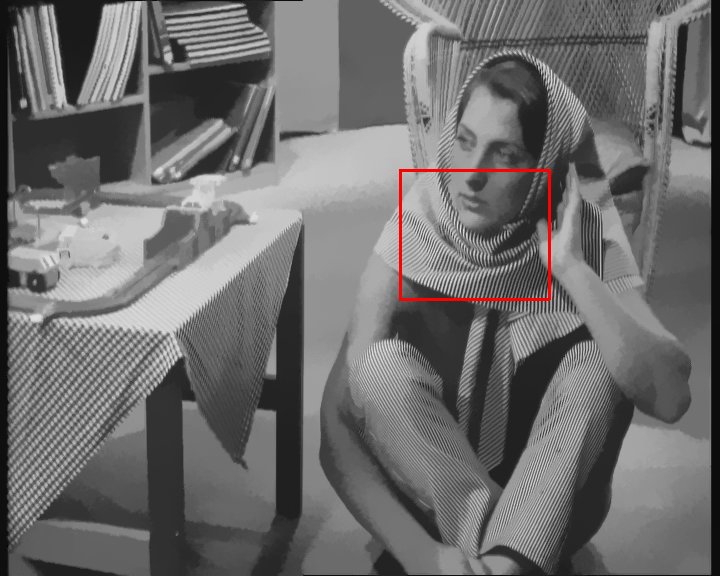}
      \\[\smallskipamount]
      \includegraphics[width=\textwidth]{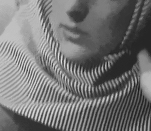}
      \\[-1em]
    \end{minipage}}

  \subfigure[$\ICTGV$]{%
	\begin{minipage}{0.22\linewidth}
		\includegraphics[width=\textwidth]{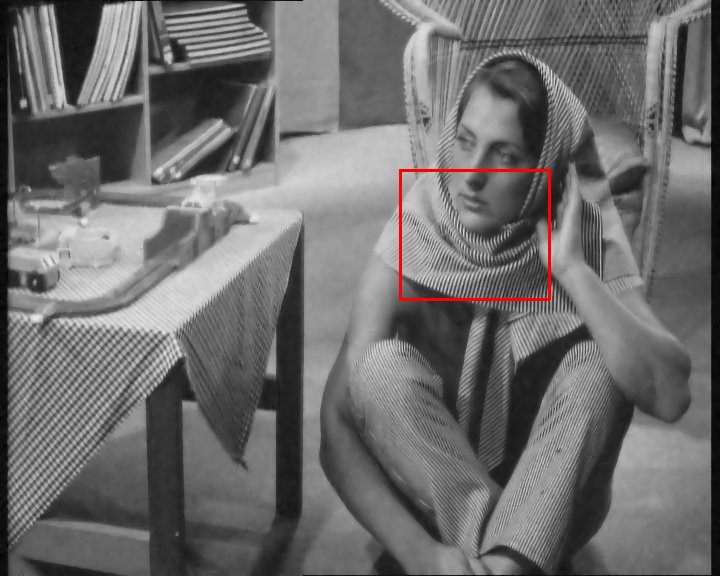}
		\\[\smallskipamount]
		\includegraphics[width=\textwidth]{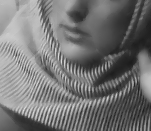}
		\\[-1em]
\end{minipage}}
  \subfigure[Fra+LDCT]{%
    \begin{minipage}{0.22\linewidth}
      \includegraphics[width=\textwidth]{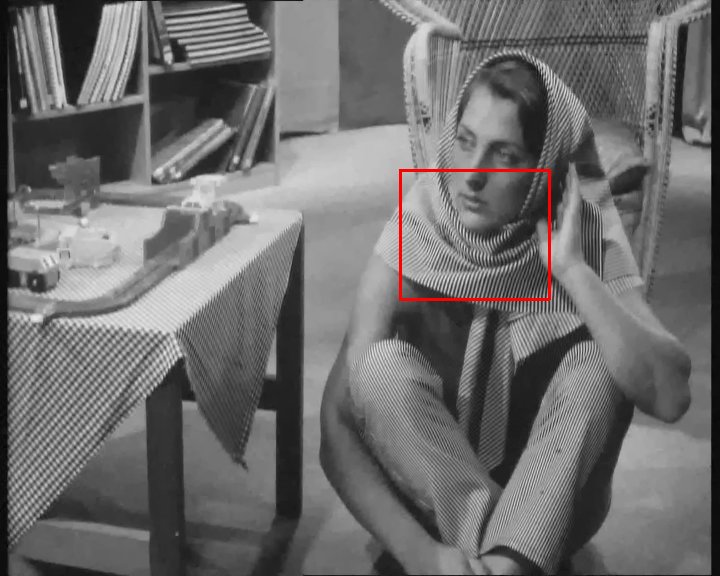}
      \\[\smallskipamount]
      \includegraphics[width=\textwidth]{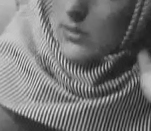}
      \\[-1em]
    \end{minipage}}
  \subfigure[BM3D]{%
    \begin{minipage}{0.22\linewidth}
      \includegraphics[width=\textwidth]{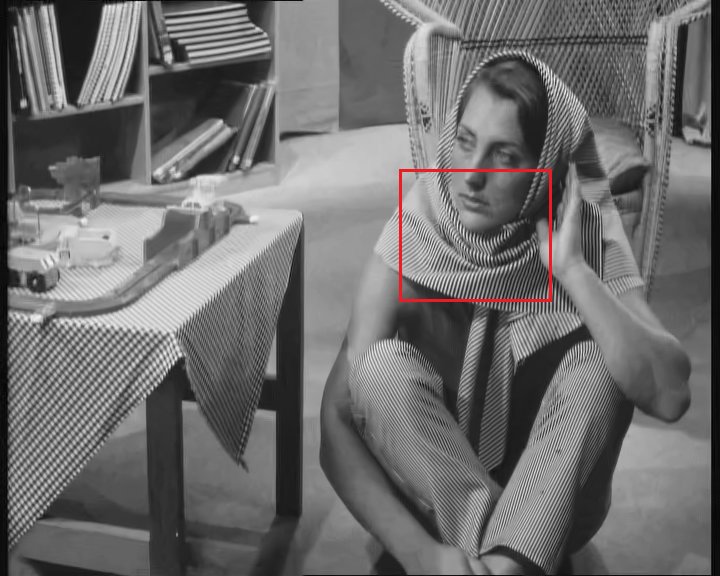}
      \\[\smallskipamount]
      \includegraphics[width=\textwidth]{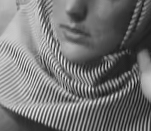}
      \\[-1em]
    \end{minipage}}
  \subfigure[proposed model]{%
    \begin{minipage}{0.22\linewidth}
      \includegraphics[width=\textwidth]{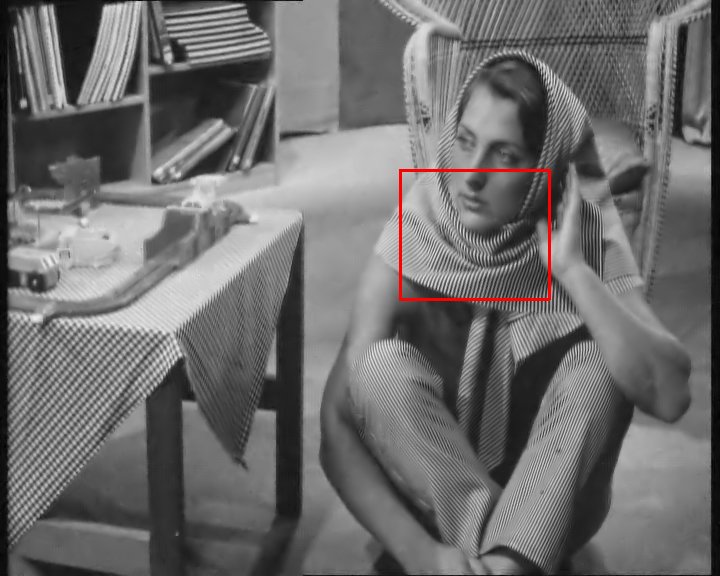}
      \\[\smallskipamount]
      \includegraphics[width=\textwidth]{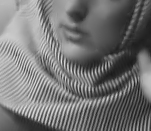}
      \\[-1em]
    \end{minipage}}
  \caption{Denoising results for barbara image with noise level
    $\sigma=0.05$. All parameters are tuned for visually-optimal
    results. Parameter choice for the proposed model:
    $\alpha_1=0.05, \beta_1=\alpha_1, \gamma_1=0,
    \alpha_i=0.9\alpha_1, \beta_i=\alpha_i, \gamma_i=0.1\alpha_i,
    i=2,\ldots,17$.}
  \label{imdenoisebar}
\end{figure}
Figure~\ref{imdenoisebar} shows an example of the denoising
performance of the methods on the ``barbara'' test image with noise
level $\sigma=0.05$ in detail. Additionally, the outcome of the
$\ICTGV^{\osci}$-denoising for all tests is presented in
Figure~\ref{fig:imdenoise_all}.  For the remaining results, we refer
to the supplementary material (Section A).
\begin{figure}
  \center{} 
  \subfigure[barbara ($\sigma = 0.05$)]{%
    \begin{minipage}{0.266667\linewidth}
      \includegraphics[width=\textwidth]{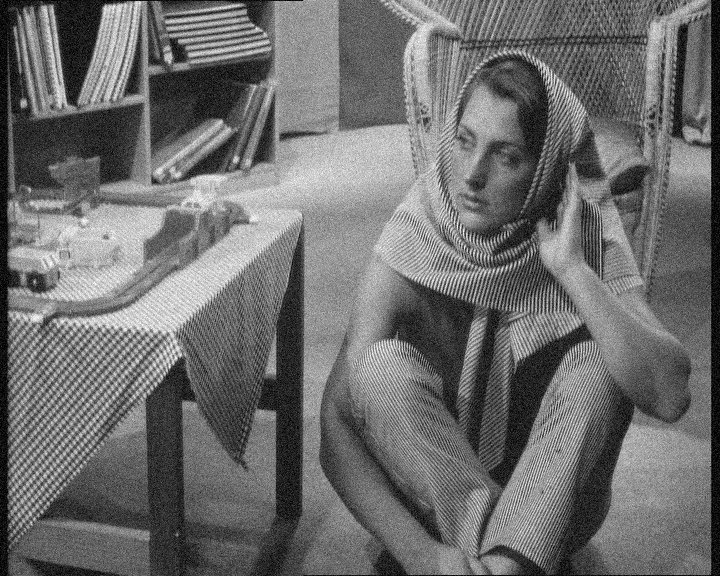}
      \\[\smallskipamount]
      \includegraphics[width=\textwidth]{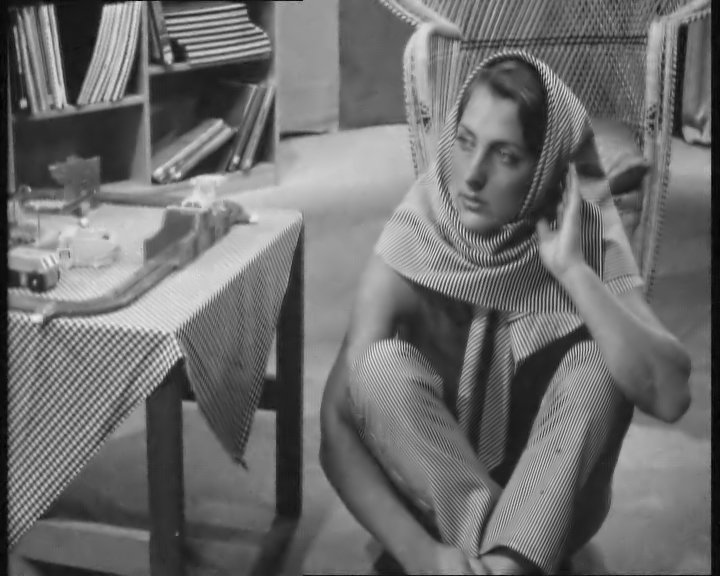}
      \\[-1em]
    \end{minipage}}
  \subfigure[zebra ($\sigma = 0.05$)]{%
    \begin{minipage}{0.32\linewidth}
      \includegraphics[width=\textwidth]{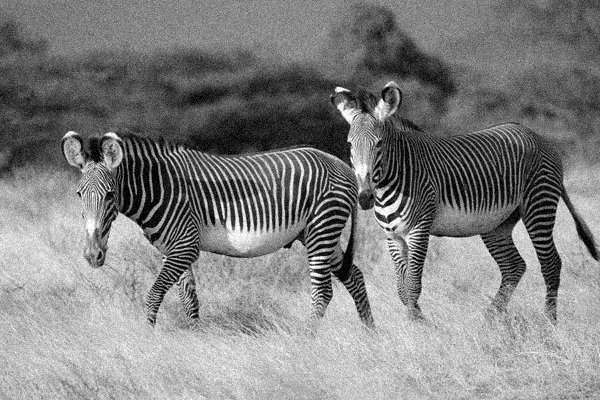}
      \\[\smallskipamount]
      \includegraphics[width=\textwidth]{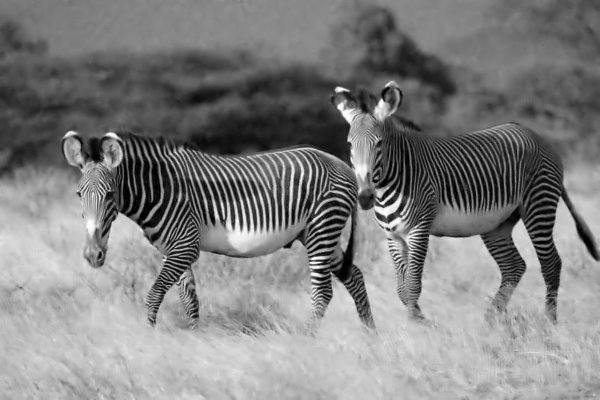}
      \\[-1em]
    \end{minipage}}
  \subfigure[parrots ($\sigma = 0.05$)]{%
    \begin{minipage}{0.32\linewidth}
      \includegraphics[width=\textwidth]{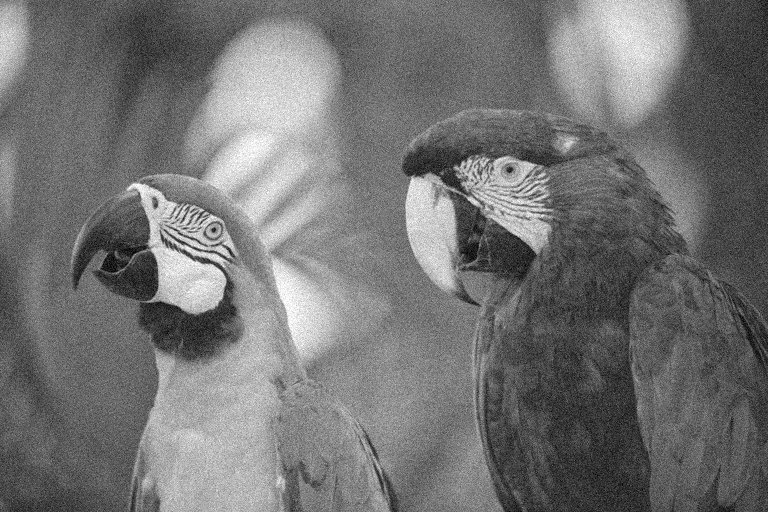}
      \\[\smallskipamount]
      \includegraphics[width=\textwidth]{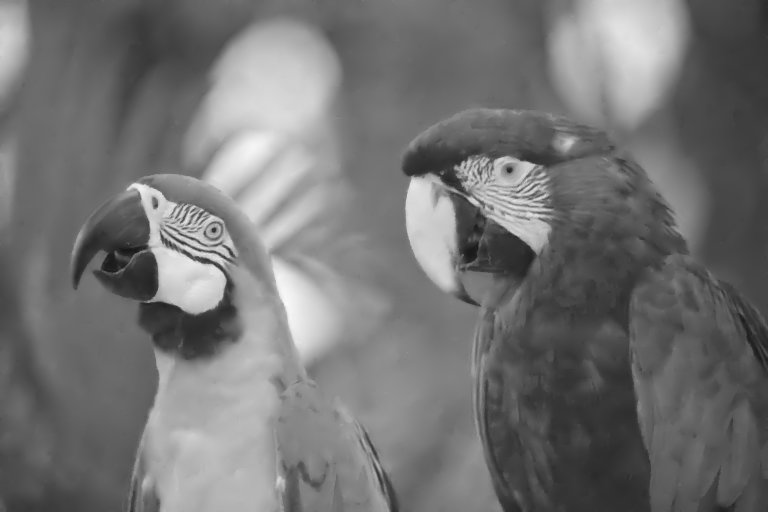}
      \\[-1em]
    \end{minipage}}

  \subfigure[barbara ($\sigma = 0.1$)]{%
    \begin{minipage}{0.266667\linewidth}
      \includegraphics[width=\textwidth]{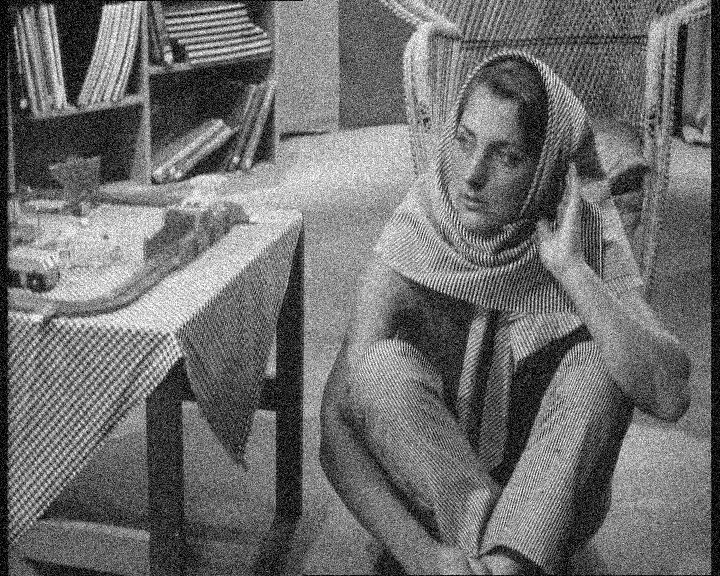}
      \\[\smallskipamount]
      \includegraphics[width=\textwidth]{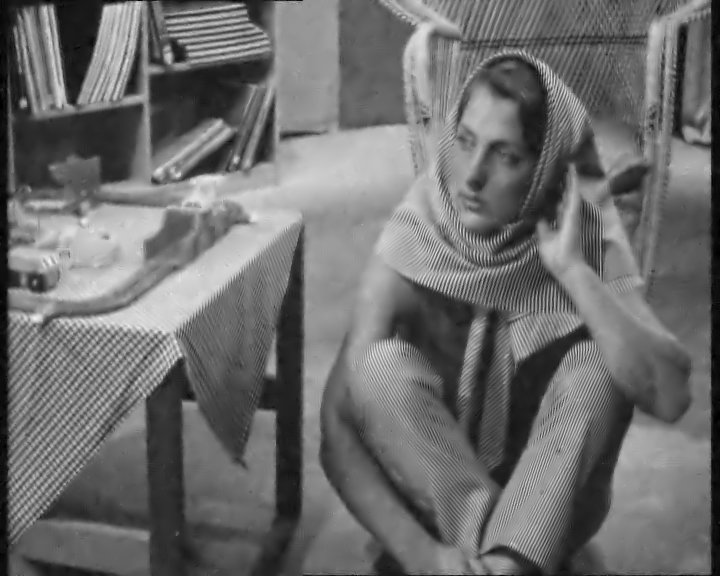}
      \\[-1em]
    \end{minipage}}
  \subfigure[zebra ($\sigma = 0.1$)]{%
    \begin{minipage}{0.32\linewidth}
      \includegraphics[width=\textwidth]{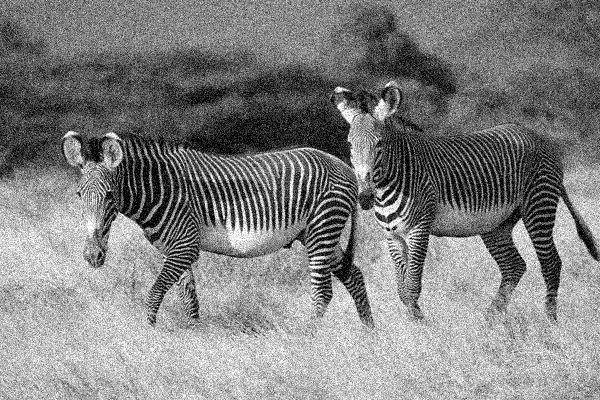}
      \\[\smallskipamount]
      \includegraphics[width=\textwidth]{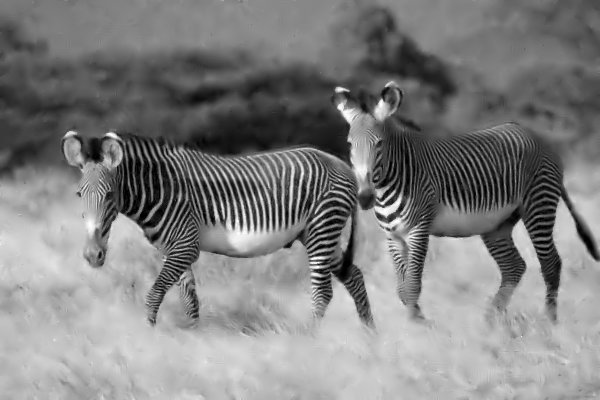}
      \\[-1em]
    \end{minipage}}
  \subfigure[parrots ($\sigma = 0.1$)]{%
    \begin{minipage}{0.32\linewidth}
      \includegraphics[width=\textwidth]{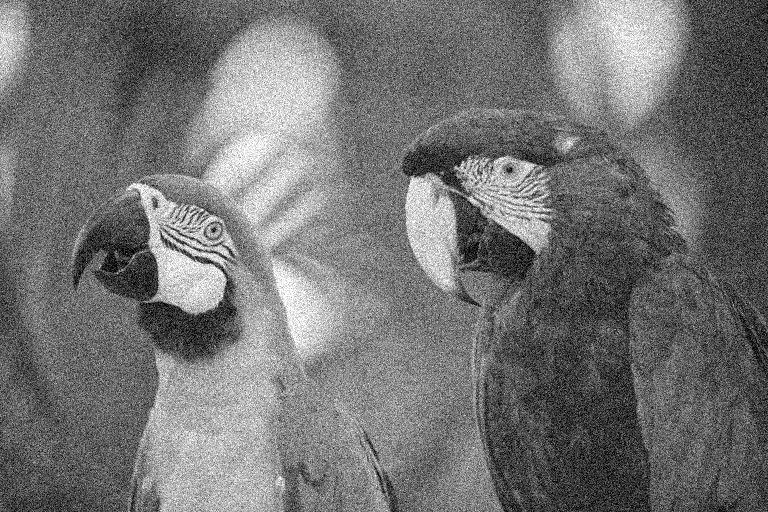}
      \\[\smallskipamount]
      \includegraphics[width=\textwidth]{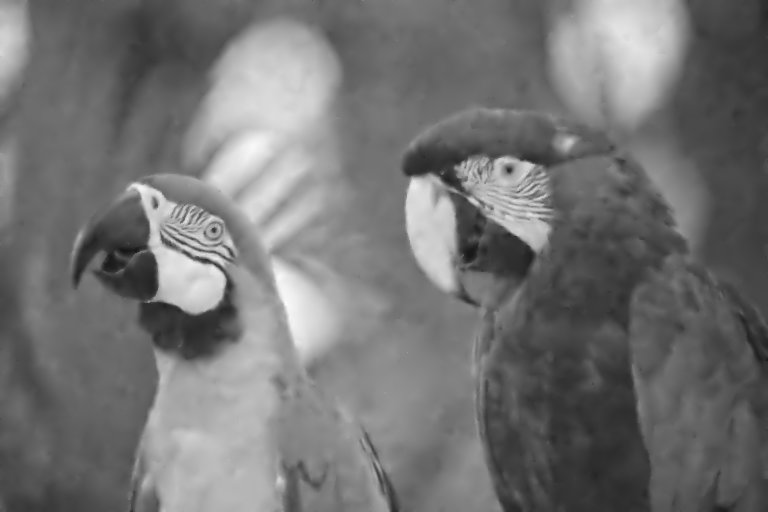}
      \\[-1em]
    \end{minipage}}
  \caption{Results of $\ICTGV^{\osci}$-denoising for natural images at
    different noise levels. For more details and comparisons, see the
    supplemental material.}
  \label{fig:imdenoise_all}
\end{figure}
One can see in Figure~\ref{imdenoisebar} that the proposed model is
able to preserve more textures than the TGV model alone which is
designed for piecewise smooth functions but does not specifically take
textures into account.
In contrast, NLTV is, due to its non-local nature, able to reconstruct
parts of the texture. It suffers, however, from the staircasing
behaviour that is typical for TV-like approaches.
The Fra+LDCT and BM3D method yield smoother reconstructions in pure
cartoon regions while also capturing the texture very well. They
suffer, however, from ghosting artifacts in the smooth regions which
leads, for instance, to a unnatural appearance of the face in the
``barbara'' image. %
The ICTGV model does not show these artifacts, but the texture
denoising is not very well, see the trousers of the ``barbara'' image,
where the line-like textures still look irregular. In comparison, the
proposed $\ICTGV^{\osci}$ model is able to capture both the texture
and the smooth regions without producing ghosting artifacts.
Nevertheless, for very fine textures, such as on the left trouser of
the ``barbara'' image, the Fra+LDCT and BM3D approaches still
reconstruct more of the details. This might be the reason for BM3D
slightly outperforming the proposed methods in terms of quantitative
measures (see Tables~\ref{table:denoisingPSNR}
and~\ref{table:denoisingSSIM}). The BM3D method is, however,
specifically optimized for denoising and not variational, such that a
generalization to inverse problems is not immediate.
Among the variational approaches, the $\ICTGV^{\osci}$ model performs
best in terms of quantitative measures. 

Convergence graphs that show
the evolution of the root mean square error~(RMSE) for TGV, ICTGV and
$\ICTGV^{\osci}$~($m=9,17$) in terms of iteration number and
computation time (with respect to GPU-supported MATLAB
implementations)
can be found in Figure~\ref{converg}.  The experiments were performed
for denoising the ``barbara'' image (dimensions $720\times576$) and
with primal-dual methods for TGV and ICTGV that are similar to
Algorithm~\ref{TTGVdecomalgo}. One can see that the essential
improvement of the RMSE is achieved, for all methods, in a few
hundreds of iterations and within 30 seconds, respectively.  From a
practical point of view this means that the proposed model and
algorithm quickly improves image quality. %
However,
these results stem from an interplay between model and algorithm, and
there might be more efficient algorithms for the same model for which
the behaviour is different in terms of RMSE improvement.
Finally, to give an overview of the computational complexity, we also
ran each algorithm with 2000 iterations for three times and took the
average of the computation time, see, again~Figure~\ref{converg}.

\begin{figure}
  \begin{minipage}{0.31\linewidth}
    \centerline{
      \includegraphics[width=\textwidth]{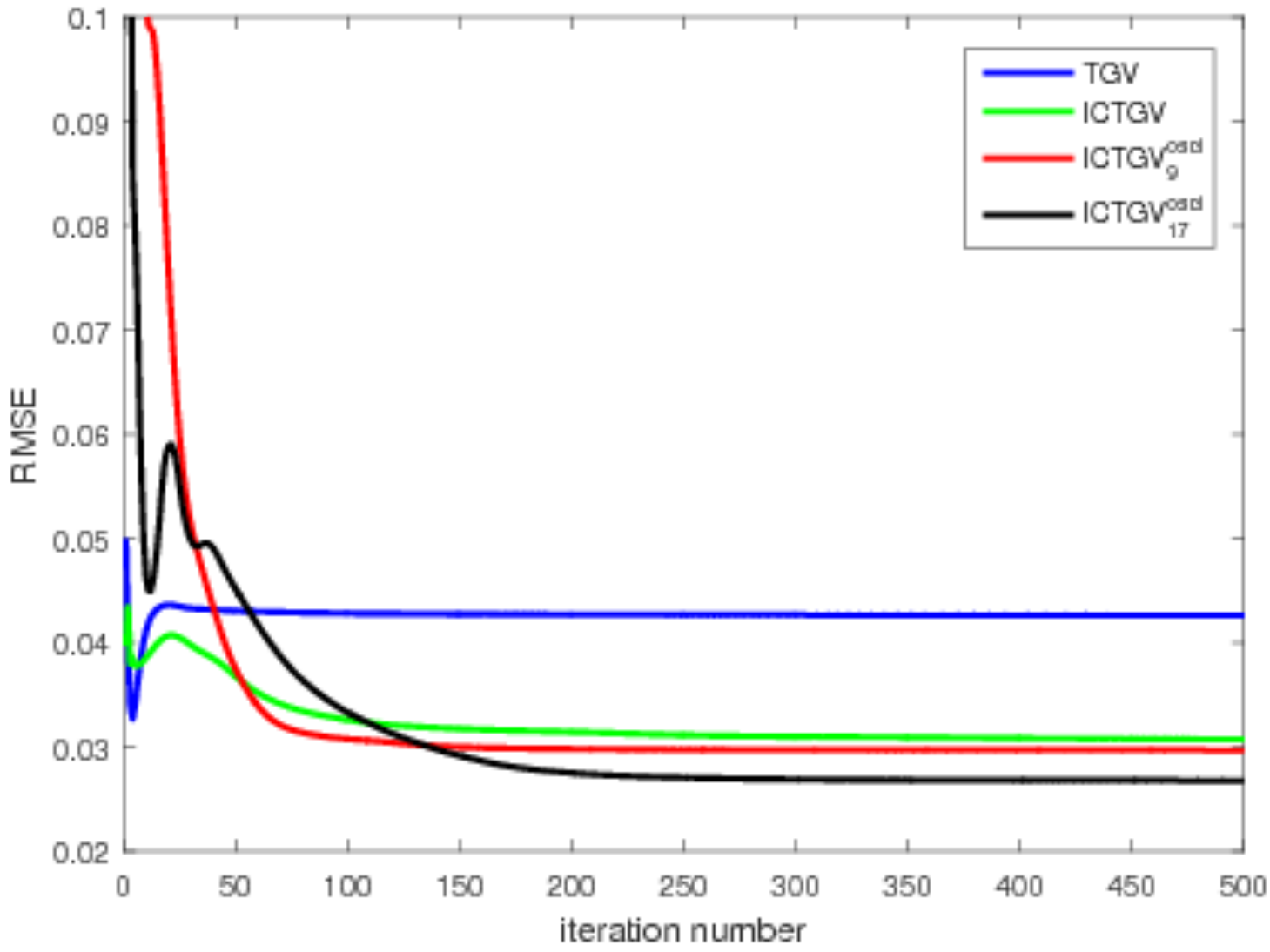}}
    \centerline{(a)}
  \end{minipage}\hfill%
  \begin{minipage}{0.31\linewidth}
    \centerline{
      \includegraphics[width=\textwidth]{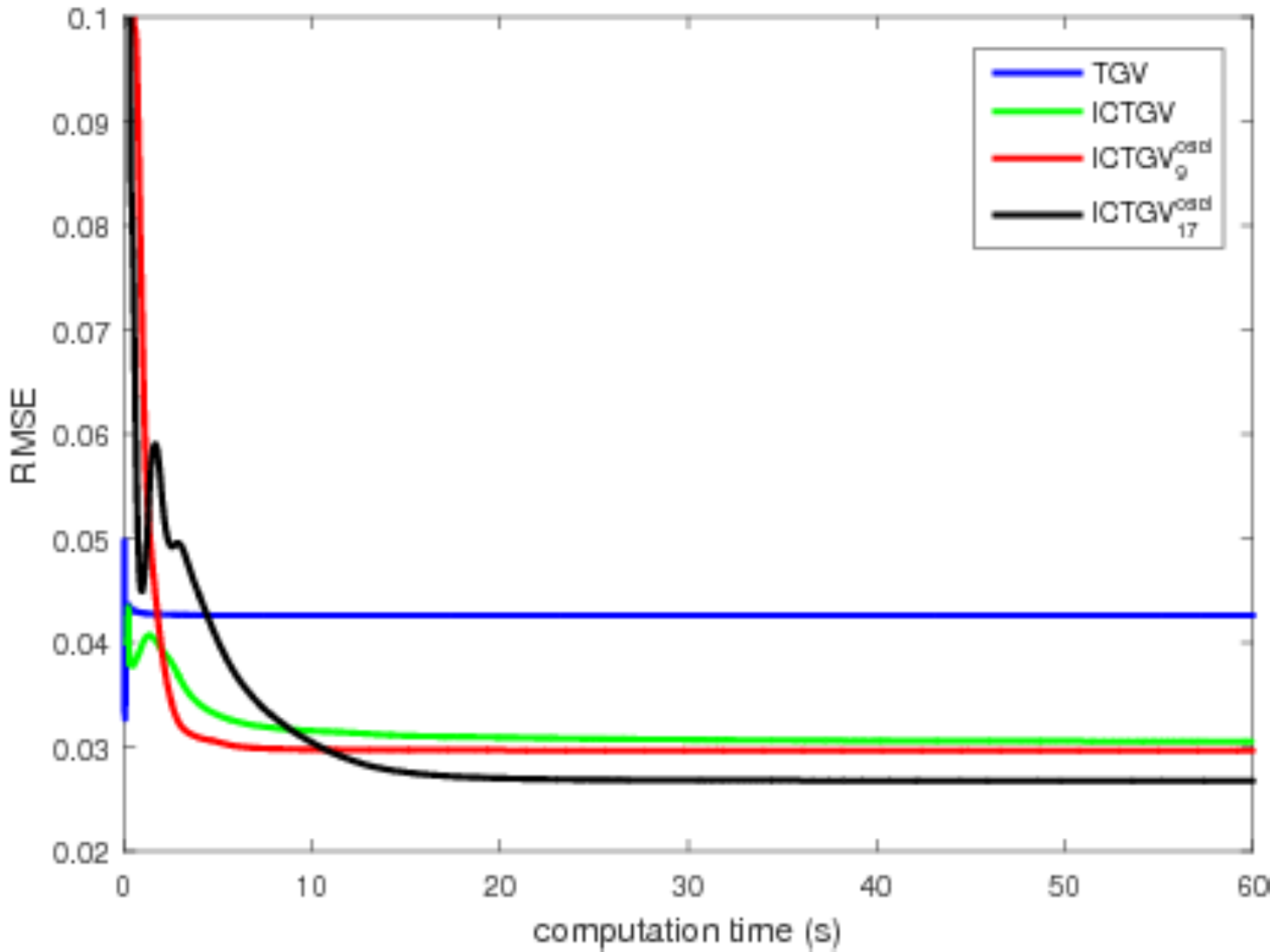}}
    \centerline{(b)}
  \end{minipage}\hfill%
  \begin{minipage}{0.35\linewidth}
    \centerline{
      \begin{tabular}{|c|c|}
        \hline\hline 
        method & time \\
        \hline
        TGV& \phantom{1}14.83s \\ 
        \hline 
        ICTGV&  104.46s\\ 
        \hline 
        $\ICTGV^{\osci}$ ($m=9$)& \phantom{1}87.89s \\ 
        \hline 
        $\ICTGV^{\osci}$ ($m=17$)&  164.97s\\ 
        \hline 
      \end{tabular} }
    \medskip
    \centerline{(c)}
  \end{minipage}
  \caption{Convergence graph and computational complexity of TGV,
    ICTGV and $\ICTGV^{\osci}$ for ``barbara'' image with noise level
    $\sigma=0.05$. (a) the RMSE values versus iterations; (b) the RMSE
    values versus GPU time; (c) the computation time~(GPU) with 2000
    iterations.}
	\label{converg}
\end{figure}

Finally, the supplementary material also contains a
comparison for the ``barbara'' image with a varying number of
oscillation directions in the extended model, i.e., 4, 6, 8, 10 and 12
directions (Section A.3). There, one can see that the restored images
with 4 and 6 directions lose some textures, whereas the majority of
textures is recovered using 8, 10 and 12 directions. Hence, the choice
of 8 directions in~\eqref{omega} indeed offers a good balance between
performance and complexity of the model.

\subsection{Image inpainting}

Inpainting, i.e., the restoration of missing regions of an image, is a
well-studied standard image processing problem.
It can, for example, be used to
repair corruptions in digital photos or in ancient drawings, and to
fill in the missing pixels of images transmitted through a noisy
channel. In this subsection, we will apply the new regularizer to a
restoration problem where pixels are missing randomly with a certain
rate. %
We model this situation as follows. For $\Omega$ the image domain, let
$\Omega' \subset \Omega$ the non-empty subdomain on which data is not
missing. Given $f \in L^2(\Omega')$, we are seeking a
$u^* \in L^2(\Omega)$ that minimizes $\ICTGV^{\osci}$ in the variant
that promotes sparse textures. In terms of the operator
$K: L^2(\Omega) \to L^2(\Omega')$, $Ku = u|_{\Omega'}$ projecting $u$
to the non-corrupted region, this corresponds to solving
\begin{equation}\label{TTGVinp}
  \min_{u \in L^2(\Omega)} \ \ICTGV^{\osci}_{\vec{\alpha},\vec{\beta},\vec{\mathbf{c}},\vec{\gamma}}(u) \quad \text{subject to} \quad Ku = f.
\end{equation}
Proposition~\ref{prop:general_min} applied to
$F = \mathcal{I}_{\{f\}}$ and
$\Phi =
\ICTGV^{\osci}_{\vec{\alpha},\vec{\beta},\vec{\mathbf{c}},\vec{\gamma}}$
yields existence in this situation. 
In order to solve~\eqref{TTGVinp}, one can use a slightly modified
version of Algorithm~\ref{TTGVdecomalgo}: Note that the only difference
between the discrete version of this model and the %
version~\eqref{minmaxTTGVdecom} is the fidelity term, where the term
$\langle
K\sum\limits_{i=1}\limits^{m}u_i-f,\lambda\rangle-\frac{\|\lambda\|^2}{2}$
is replaced by
$\langle K\sum\limits_{i=1}\limits^{m}u_i-f,\lambda\rangle$.
Thus, $\lambda$ in Algorithm~\ref{TTGVdecomalgo} is updated according
to
\begin{equation*}
  \lambda^{n+1}=\lambda^n+\sigma \Bigl(K\sum\limits_{i=1}\limits^{m}\bar{u}_i^n-f \Bigr).
\end{equation*}

\begin{figure}
  \center{} 
    \subfigure[$60\%$ missing]{%
    \begin{minipage}{0.25\linewidth}
      \includegraphics[width=\textwidth]{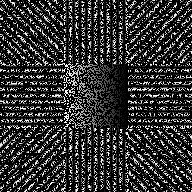}
      \\[-1em]
    \end{minipage}}
    \subfigure[$75\%$ missing]{%
	\begin{minipage}{0.25\linewidth}
		\includegraphics[width=\textwidth]{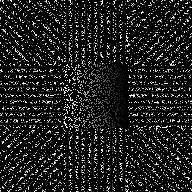}
		\\[-1em]
    \end{minipage}}
    \subfigure[$90\%$ missing]{%
    \begin{minipage}{0.25\linewidth}
      \includegraphics[width=\textwidth]{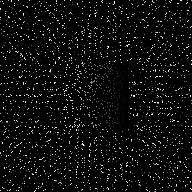}
      \\[-1em]
    \end{minipage}}

  \vspace*{-0.5em}
    \subfigure[inpainted image]{%
    \begin{minipage}{0.25\linewidth}
      \includegraphics[width=\textwidth]{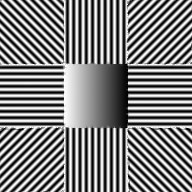}
      \\[-1em]
    \end{minipage}}
    \subfigure[inpainted image]{%
	\begin{minipage}{0.25\linewidth}
		\includegraphics[width=\textwidth]{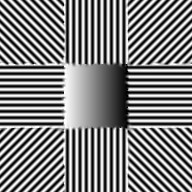}
		\\[-1em]
    \end{minipage}}
    \subfigure[inpainted image]{%
    \begin{minipage}{0.25\linewidth}
      \includegraphics[width=\textwidth]{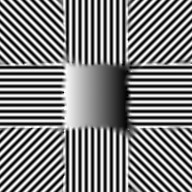}
      \\[-1em]
    \end{minipage}}
  \caption{Inpainting results for a synthetic image ($192 \times 192$
    pixels). (a)--(c) Corrupted image with $60\%$, $75\%$ and $90\%$
    pixels missing, respectively; (d)--(f) Images inpainted by
    model~\eqref{TTGVinp}. Parameter choice:
    $\alpha_1=0.04, \beta_1=\alpha_1, \gamma_1=0, \alpha_i=\alpha_1,
    \beta_i=0.5\alpha_i, \gamma_i=0.02\alpha_i, i=2,\ldots,5$.}
  \label{inpsyn}
\end{figure}
We performed numerical experiments using the inpainting
model~\eqref{TTGVinp} and the described modification of
Algorithm~\ref{TTGVdecomalgo}. Our first experiment aimed at
recovering the synthetic image already utilized in the
decomposition/denoising experiments, see Figure~\ref{decomsyn}. In
this experiment, 60\%, 75\% and 90\% of the pixels were randomly
removed from the image and inpainted by
$\ICTGV^{\osci}$-minimization. The results are shown in
Figure~\ref{inpsyn}. There, one can see that the recovery from 60\%
and 75\% of the missing pixels is almost perfect; whereas for 90\% of
missing pixels, the texture is also perfectly recovered while the
boundaries of the regions deviate from the ground truth. The latter
may be explained by the fact that in this situation, not enough data
is available to unambiguously reconstruct the boundaries. In contrast,
there is still enough information to determine the respective
parameters of the affine and oscillatory parts. Overall, the
experiment shows that $\ICTGV^{\osci}$-regularization is able to
effectively reconstruct oscillatory texture in inpainting problems.

Further, experiments with natural images (``barbara'', ``zebra'' and
``fish'') were carried out. For these images, we also performed a
comparison to TGV of second order and the framelet/local DCT model.
Figure~\ref{iminpaintfish} shows a comparison of the inpainting
performances for selected examples; we refer to the supplementary
material Section B for more detailed results. Again, for the
``barbara'' and ``zebra'' image, sixteen texture components were used,
i.e., $m=17$, and for the ``fish'' image, experiments were performed
with eight texture directions, i.e.,
$m=9$. %
One can clearly see in Figure~\ref{iminpaintfish} that our proposed
model is able to restore the corrupted images in high quality with the
majority of the textures being recovered, such as, for instance, the
table cloth and the scarf in the ``barbara'' image as well as the
stripes on the bodies of the zebra and the fish.  The same holds for
the Fra+LDCT-based model while the TGV-model does not recover well the
texture and leaves many unconnected ``speckles''. The latter is
typical for derivative- and measure-based models and does also apply,
e.g., to $\TV$-inpainting.  %
Furthermore, Tables~\ref{table:inpaintingPSNR}
and~\ref{table:inpaintingSSIM} show the PSNR and SSIM values of the
three models at different rates of missing pixels. As we see, the
proposed model performs better than other two models by these two
evaluation standards. However, despite the favorable quantitative
comparison to Fra+LDCT model, the proposed model has only limited
ability of recovering textures of very high frequency. Looking at the
left trouser in the ``barbara'' image near the arm, there are some
textures that are not inpainted correctly. In contrast, in the result
of Fra+LDCT-model, this region is recovered well.
\begin{figure}
  \center{} 
  \subfigure[$50\%$ missing]{%
    \begin{minipage}{0.23\linewidth}
      \includegraphics[width=\textwidth]{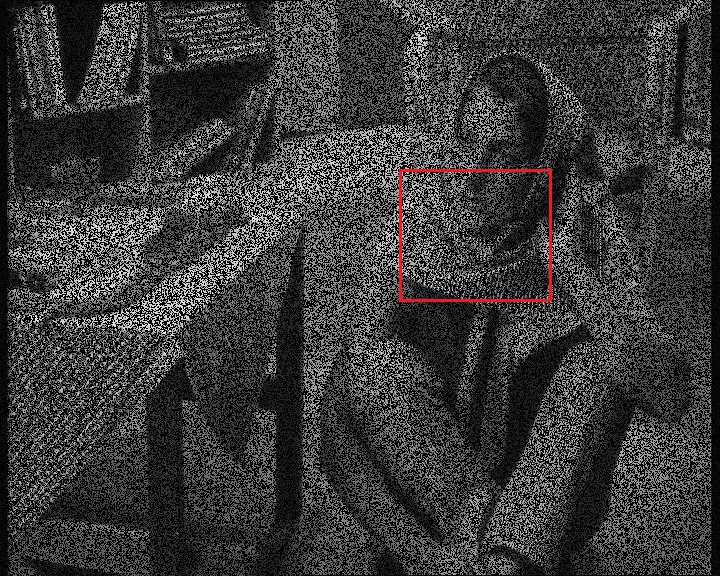}
      \\[\smallskipamount]
      \includegraphics[width=\textwidth]{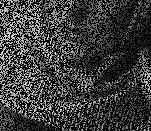}
      \\[-1em]
    \end{minipage}}
  \subfigure[TGV]{%
    \begin{minipage}{0.23\linewidth}
      \includegraphics[width=\textwidth]{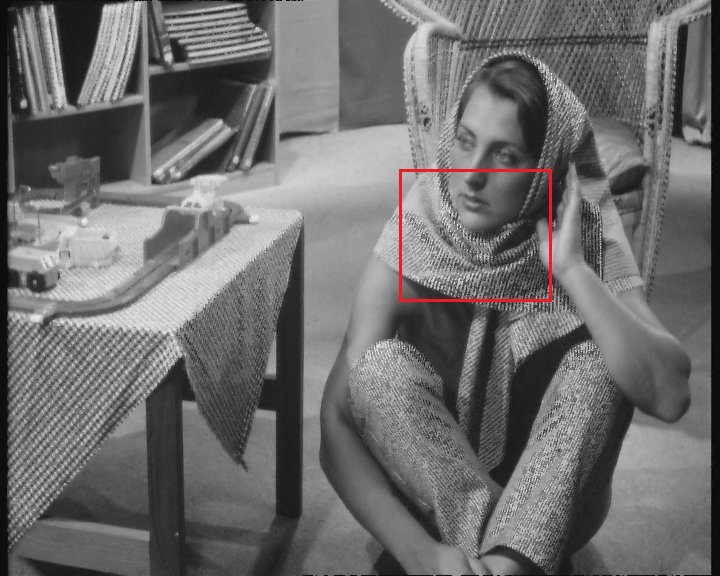}
      \\[\smallskipamount]
      \includegraphics[width=\textwidth]{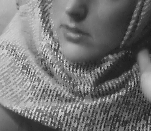}
      \\[-1em]
    \end{minipage}}
  \subfigure[Fra+LDCT]{%
    \begin{minipage}{0.23\linewidth}
      \includegraphics[width=\textwidth]{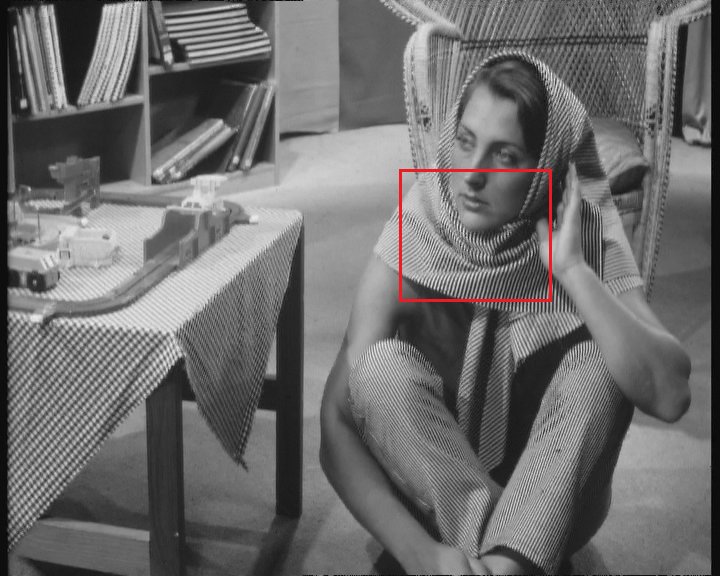}
      \\[\smallskipamount]
      \includegraphics[width=\textwidth]{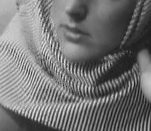}
      \\[-1em]
    \end{minipage}}
  \subfigure[proposed model]{%
    \begin{minipage}{0.23\linewidth}
      \includegraphics[width=\textwidth]{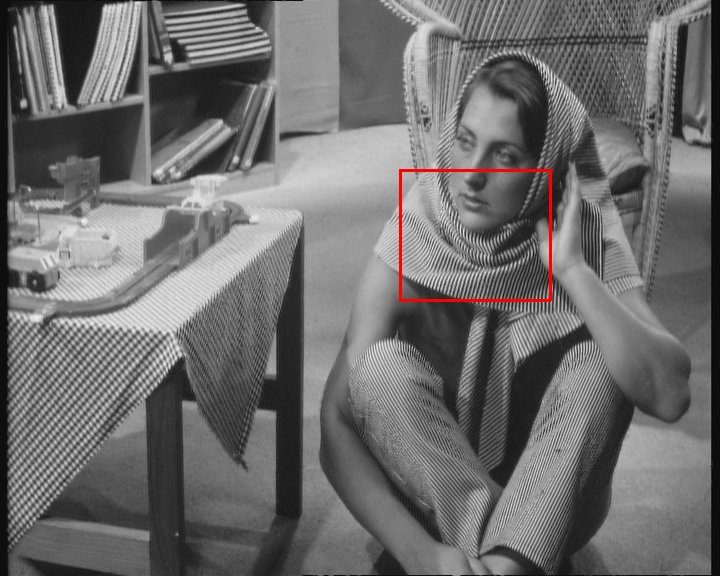}
      \\[\smallskipamount]
      \includegraphics[width=\textwidth]{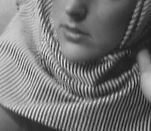}
      \\[-1em]
    \end{minipage}}

  \vspace*{-0.5em}
  \subfigure[$50\%$ missing]{%
    \begin{minipage}{0.23\linewidth}
      \includegraphics[width=\textwidth]{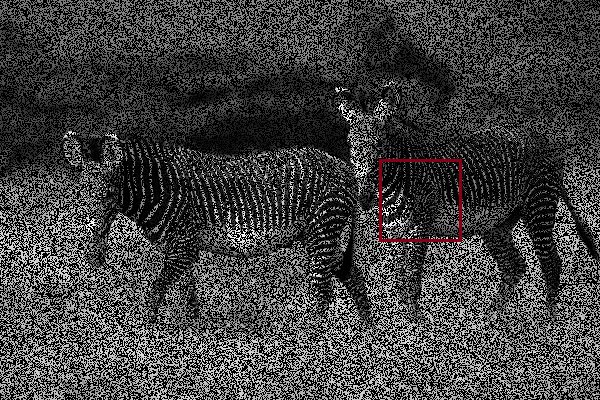}
      \\[\smallskipamount]
      \includegraphics[width=\textwidth]{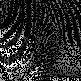}
      \\[-1em]
    \end{minipage}}
  \subfigure[TGV]{%
    \begin{minipage}{0.23\linewidth}
      \includegraphics[width=\textwidth]{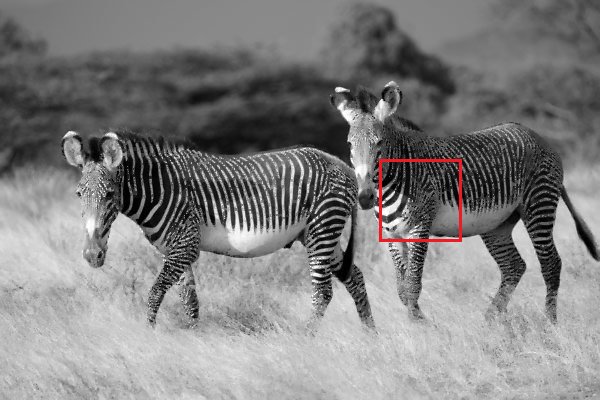}
      \\[\smallskipamount]
      \includegraphics[width=\textwidth]{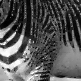}
      \\[-1em]
    \end{minipage}}
  \subfigure[Fra+LDCT]{%
    \begin{minipage}{0.23\linewidth}
      \includegraphics[width=\textwidth]{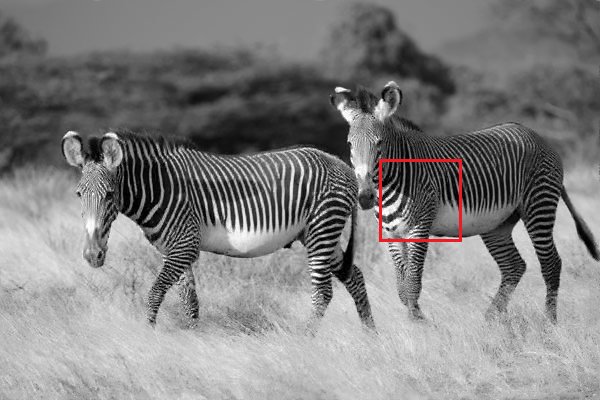}
      \\[\smallskipamount]
      \includegraphics[width=\textwidth]{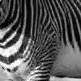}
      \\[-1em]
    \end{minipage}}
  \subfigure[proposed model]{%
    \begin{minipage}{0.23\linewidth}
      \includegraphics[width=\textwidth]{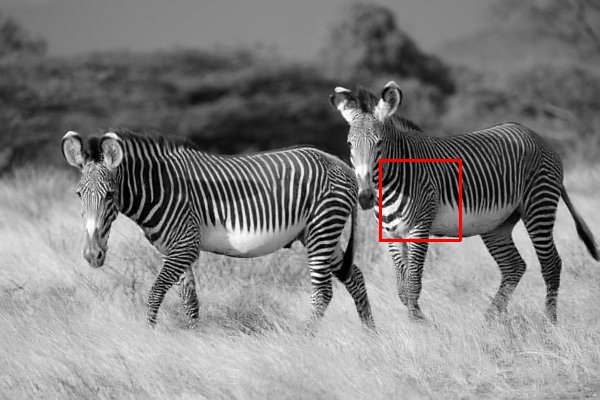}
      \\[\smallskipamount]
      \includegraphics[width=\textwidth]{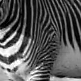}
      \\[-1em]
    \end{minipage}}

  \vspace*{-0.5em}
  \subfigure[$50\%$ missing]{%
    \begin{minipage}{0.23\linewidth}
      \includegraphics[width=\textwidth]{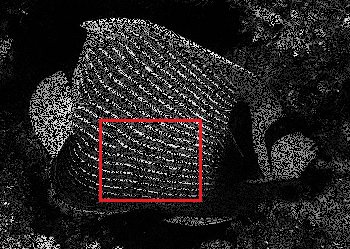}
      \\[\smallskipamount]
      \includegraphics[width=\textwidth]{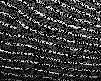}
      \\[-1em]
    \end{minipage}}
  \subfigure[TGV]{%
    \begin{minipage}{0.23\linewidth}
      \includegraphics[width=\textwidth]{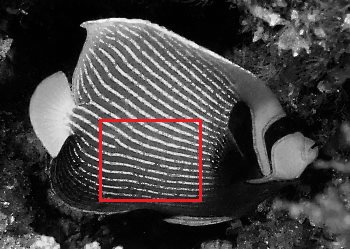}
      \\[\smallskipamount]
      \includegraphics[width=\textwidth]{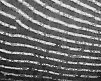}
      \\[-1em]
    \end{minipage}}
  \subfigure[Fra+LDCT]{%
    \begin{minipage}{0.23\linewidth}
      \includegraphics[width=\textwidth]{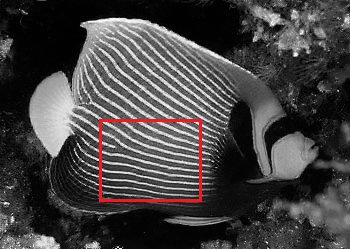}
      \\[\smallskipamount]
      \includegraphics[width=\textwidth]{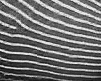}
      \\[-1em]
    \end{minipage}}
  \subfigure[proposed model]{%
    \begin{minipage}{0.23\linewidth}
      \includegraphics[width=\textwidth]{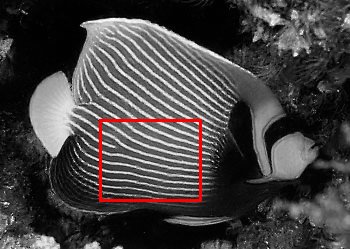}
      \\[\smallskipamount]
      \includegraphics[width=\textwidth]{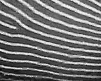}
      \\[-1em]
    \end{minipage}}
  \caption{Inpainting results for ``barbara'', ``zebra'' and ``fish''
    image.  (a)--(d), (e)--(h), (i)--(l): Reconstruction from 50\%
    missing pixels for proposed and reference methods for the
    respective test images. For more details and results, see the supplementary
    material.}
  \label{iminpaintfish}
\end{figure}

\begin{table}\small
  \center{}
  \begin{tabular}{|c|c|c|c|c|c|c|c|c|c|}
    \hline
    \hline
    \multirow{2}{*}{image}
    & \multicolumn{3}{|c|}{TGV} &  \multicolumn{3}{|c|}{Fra+LDCT}&  \multicolumn{3}{|c|}{proposed model} \\
    \cline{2-10}&$50\%$ & $60\%$ &$70\%$ & $50\%$  &$60\%$ & $70\%$&$50\%$ & $60\%$&$70\%$\\
    \hline barbara & 27.49 & 26.26 & 25.09 & 32.75 & 30.76 & 28.53 & \textbf{34.03} & \textbf{31.86} & \textbf{29.49}\\
    \hline zebra &26.00 & 24.23 & 22.65 & 28.50 & 26.67 & 25.08 & \textbf{30.09} & \textbf{28.11} & \textbf{26.42}\\
    \hline fish & 22.70 & 20.81 & 19.01 & 24.54 & 23.01 & 21.22 & \textbf{25.10} & \textbf{23.50} & \textbf{21.54} \\
    \hline
  \end{tabular}
  \caption{\label{table:inpaintingPSNR} Comparison of inpainting performance for different rates of missing pixels in terms of PSNR.}
\end{table}

\begin{table}\small
  \center{}
  \begin{tabular}{|c|c|c|c|c|c|c|c|c|c|}
    \hline
    \hline
    \multirow{2}{*}{image}
    & \multicolumn{3}{|c|}{TGV} &  \multicolumn{3}{|c|}{Fra+LDCT}&  \multicolumn{3}{|c|}{proposed model} \\
    \cline{2-10}&$50\%$ & $60\%$ &$70\%$ & $50\%$  &$60\%$ & $70\%$&$50\%$ & $60\%$&$70\%$\\
    \hline barbara & 0.8942 & 0.8562 & 0.8081 & 0.9549 & 0.9320 & 0.8953 & \textbf{0.9591} & \textbf{0.9390} & \textbf{0.9078}\\
    \hline zebra & 0.9016 & 0.8597 & 0.8051 & 0.9213 & 0.8863 & 0.8438 & \textbf{0.9463} & \textbf{0.9194} & \textbf{0.8839}\\
    \hline fish & 0.8676 & 0.8059 & 0.7255 & 0.8923 & 0.8462 & 0.7902 & \textbf{0.9059} & \textbf{0.8596} & \textbf{0.7977} \\
    \hline
  \end{tabular}
  \caption{\label{table:inpaintingSSIM} Comparison of inpainting performance for different rates of missing pixels in terms of SSIM.}
\end{table}

\subsection{Undersampled magnetic resonance imaging}

Since compressed sensing techniques have been shown to be effective
for image reconstruction in magnetic resonance imaging (MRI) with
undersampling \cite{SparMRI,CSMRI}, i.e., where much less data is available than
usually required, variational methods were becoming increasingly
popular for undersampling MRI \cite{TGVMR,MRITV,ParaMRI} in
particular and medical image processing in general.  As MRI samples
data in the Fourier domain, it is particularly well-suited for these
methods. In this subsection, we discuss
$\ICTGV^{\osci}$-regularization for recovering an image given
incomplete Fourier data. We model the situation as follows. Let
$\Omega \subset \RR^2$ be a bounded image domain and $(\Omega', \mu)$
a finite Borel measure space on a bounded $\Omega' \subset \RR^2$,
where $\Omega'$ represents the set of measured Fourier data.  The
restriction of the Fourier transform $\mathcal{PF}$ then maps
$L^2(\Omega) \to L^2(\Omega',\mu)$ since both $\Omega$ and $\Omega'$
are bounded and the Fourier transform $\mathcal{F}$ on the whole space
maps $L^1(\RR^2)$ functions to continuous functions on $\RR^2$.  Given
$f \in L^2(\Omega',\mu)$, the undersampled MRI reconstruction
corresponds to solving $Ku = f$ for $K = \mathcal{PF}$, such that the
$\ICTGV^{\osci}$-regularized solution (in its sparse-texture variant)
reads as
\begin{equation}\label{TTGVrec}
  \min_{u \in L^2(\Omega)} \ \frac12 \int_{\Omega'} \abs{Ku - f}^2 \dd{\mu} + \ICTGV^{\osci}_{\vec{\alpha},\vec{\beta},\vec{\mathbf{c}},\vec{\gamma}}(u).
\end{equation}
Clearly, the problem~\eqref{TTGVrec} admits a solution, see
Theorem~\ref{thm:ictgv_tikh}. In applications, the Borel measure space
$(\Omega',\mu)$ could, for instance, be $\Omega' \subset \RR^2$ a
bounded domain and $\mu = \mathcal{L}^2$ the Lebesgue measure,
$\Omega'$ a finite collection of line segments and
$\mu = \mathcal{H}^1 \llcorner \Omega'$ the one-dimensional Hausdorff
measure restricted to $\Omega'$ (which models so-called \emph{radial
  sampling} in case of radially centered line segments), or $\Omega'$
is a finite collection of points and
$\mu = \mathcal{H}^0 \llcorner \Omega'$ a collection of delta-measures
(modelling compressed-sensing-inspired sampling).

For the discretization, we chose $\mathcal{F}$ to represent the
discrete Fourier transform on the discrete grid $\Omega$ which is
usually realized by a fast algorithm (FFT). Further, the discrete
domain $\Omega'$ is chosen as $\Omega' \subset \Omega$ and
$\mathcal{P}$ is the projection operator which restricts Fourier
coefficients to $\Omega'$. In total, $K = \mathcal{PF}$ as in the
continuous setting and Algorithm~\ref{TTGVdecomalgo} can be used to
compute a solution.

We performed numerical experiments for undersampled MRI with radial
sampling patterns associated with a varying number of sampling lines;
see Figure \ref{sampleline} for the illustration of a selection mask
with 70 equispaced radial sampling lines which corresponds to a
sampling rate of approximately $14.74\%$ for an image of
$512\times512$ pixels.
\begin{figure}
  \center{} 
  \includegraphics[width=0.3\textwidth]{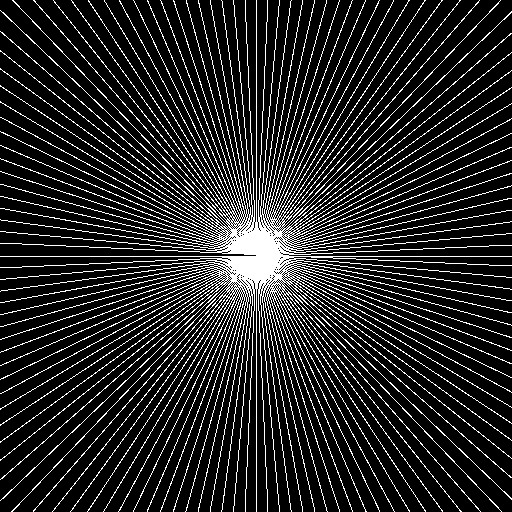}
  \caption{Example of a sampling pattern associated with radial
    sampling (70 equispaced radial lines, image resolution:
    $512\times512$ pixels).}
  \label{sampleline}
\end{figure}
The experimental results for MR image reconstruction are shown in
comparison to TGV-Shearlet MRI reconstruction in \cite{TGVshear} which
is based on simultaneous TGV and shearlet-transform-based
regularization.

\begin{table}
  \center{}
  \begin{tabular}{|c|c|c|c|c|c|c|c|c|}
    \hline
    \hline
    \multicolumn{2}{|c|}{sampling lines}&40 & 50 &60& 70 &80 & 90&100\\
    \hline
    \multirow{3}{*}{TGV-shearlet}
                                        &foot & \textbf{30.81} &\textbf{32.59}& 33.75 &34.99 & 36.51&37.34&38.59\\
    \cline{2-9}&knee & 28.01 &28.77& 29.59 &30.10 & 30.84 &31.50 &32.07\\
    \cline{2-9}&brain & 34.05 & 35.71 & 37.18 &38.44 & 39.66 &40.51 &41.43\\
    \hline
    \multirow{3}{*}{proposed model}
    	&foot & 29.76 & 32.41 & \textbf{33.89} & \textbf{35.27} & \textbf{37.07} & \textbf{38.25} & \textbf{39.66}\\
    \cline{2-9}&knee & \textbf{28.40} & \textbf{29.16} & \textbf{29.98} & \textbf{30.56} & \textbf{31.23} & \textbf{31.82} & \textbf{32.36}\\
    \cline{2-9}&brain & \textbf{34.64} & \textbf{36.97} & \textbf{38.99} & \textbf{40.33} & \textbf{41.65} & \textbf{42.71} & \textbf{43.67}\\
    \hline
  \end{tabular}
  \caption{\label{table:reconstructionPSNR} Comparison of MRI reconstruction performance for different sampling rates in terms of PSNR.}
\end{table}

\begin{table}
  \center{}
  \begin{tabular}{|c|c|c|c|c|c|c|c|c|}
    \hline
    \hline
    \multicolumn{2}{|c|}{sampling lines}&40 & 50 &60& 70 &80 & 90&100\\
    \hline
    \multirow{3}{*}{TGV-shearlet}
                                        &foot & 0.8451 & 0.8669 & 0.8823 & 0.8956 & 0.9098 & 0.9175 & 0.9262\\
    \cline{2-9}&knee & 0.7356 & 0.7656 & 0.7950 & 0.8133 & 0.8348 & 0.8508 & 0.8661\\
    \cline{2-9}&brain & 0.8867 & 0.9059 & 0.9209 & 0.9318 & 0.9408 & 0.9466 & 0.9523\\
    \hline
    \multirow{3}{*}{proposed model}
                                       &foot & \textbf{0.8653} & \textbf{0.9280} & \textbf{0.9465} & \textbf{0.9567} & \textbf{0.9702} & \textbf{0.9774} & \textbf{0.9822}\\
                                       \cline{2-9}&knee & \textbf{0.7581} & \textbf{0.7861} & \textbf{0.8148} & \textbf{0.8344} & \textbf{0.8570} & \textbf{0.8720} & \textbf{0.8878}\\
                                       \cline{2-9}&brain & \textbf{0.9055} & \textbf{0.9421} & \textbf{0.9600} & \textbf{0.9681} & \textbf{0.9743} & \textbf{0.9784} & \textbf{0.9818}\\
    \hline
  \end{tabular}
  \caption{\label{table:reconstructionSSIM} Comparison of MRI reconstruction performance for different sampling rates in terms of SSIM.}
\end{table}

For the MRI reconstruction experiment, we chose some fully sampled MR
test images from an online database (\verb|http://www.mr-tip.com/|): a
$T_1$-weighted foot image ($512 \times 512$ pixels) and a
$T_1$-weighted knee image ($350 \times 350$ pixels); and a
$T_2$-weighted brain image ($490 \times 490$ pixels) (image courtesy
of A.~Farrall, University of Edinburgh).  Undersampled measurements
were then synthetically created by restricting the Fourier
coefficients to radial sampling lines where the number of lines covers
the range $40,50,\ldots, 100$. Reconstructions were obtained from
these measurements using the TGV+Shearlet model and by
minimizing~\eqref{TTGVrec} where for both methods, parameters were
tuned to yield optimal PSNR.  Figure~\ref{imreconfoot} shows exemplary
results; see the supplementary material (Section C) for more detailed
results. One can see in these results that both the TGV+Shearlet and
the proposed model are able to preserve the textures well.  Further,
one can observe that the proposed model yields more naturally
appearing bones and junctions that are also closer to the ground
truth. The muscular tissues in the foot and knee image as well as the
brain tissue are also reconstructed well. For the latter, the
$\ICTGV^{\osci}$-model also has a desired smoothing effect that
reduces the ringing artifacts originating from the MR
measurements. Furthermore, note that there are some ghost stripes in
the results of TGV+Shearlet model which do not appear in the
$\ICTGV^{\osci}$ approach.
The quantitative comparison in terms of PSNR and SSIM, see
Tables~\ref{table:reconstructionPSNR}
and~\ref{table:reconstructionSSIM}, support these findings.  Except
for the PSNR for the foot image at low sampling rates, the proposed
method consistently outperforms the TGV+Shearlet approach with respect
to these quantitative measures. The only weakness one might observe in
both models may the found in the foot image: There, the ground truth
admits several fine directional structures in the bone (see the
closeup) for which the reconstructions do not correctly recover the
orientations and artificially introduce vertical structures.  This
effect can, however, be mitigated by switching to the extended model
($m=17$) that also includes oscillations of higher frequencies, see
the supplementary material.
\begin{figure}
  \center
  \subfigure[ground truth]{%
    \begin{minipage}{0.21\linewidth}
      \includegraphics[width=\textwidth]{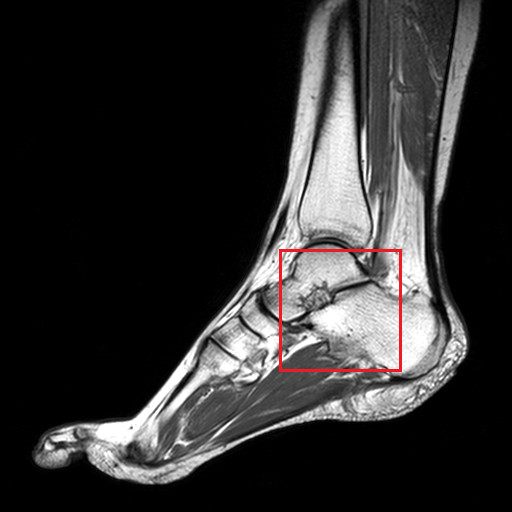}
      \\[\smallskipamount]
      \includegraphics[width=\textwidth]{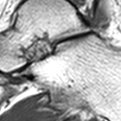}
      \\[-1em]
    \end{minipage}}\quad
  \subfigure[TGV+Shearlet]{%
    \begin{minipage}{0.21\linewidth}
      \includegraphics[width=\textwidth]{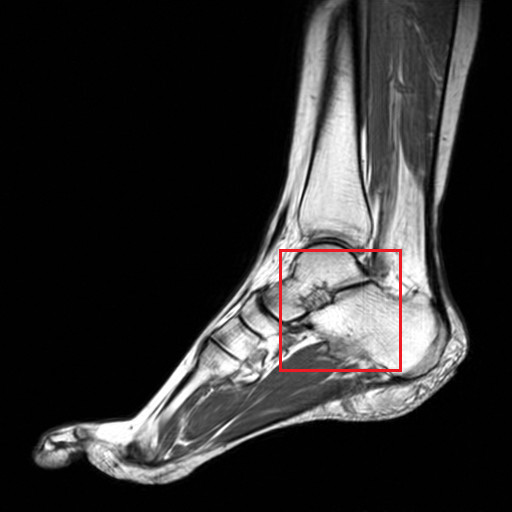}
      \\[\smallskipamount]
      \includegraphics[width=\textwidth]{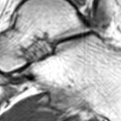}
      \\[-1em]
    \end{minipage}}\ \ 
  \subfigure[proposed model]{%
    \begin{minipage}{0.21\linewidth}
      \includegraphics[width=\textwidth]{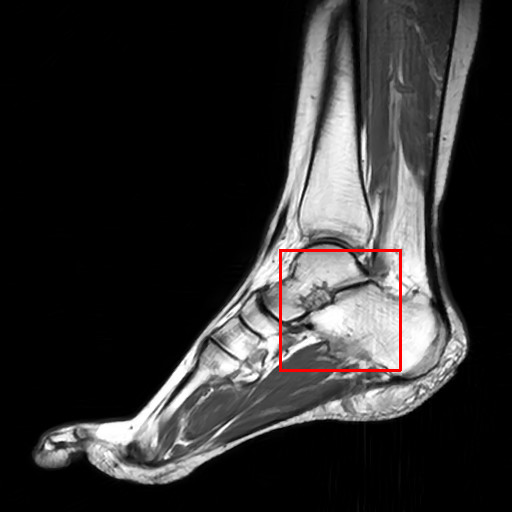}
      \\[\smallskipamount]
      \includegraphics[width=\textwidth]{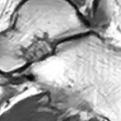}
      \\[-1em]
    \end{minipage}}

  \vspace*{-0.5em}
  \subfigure[ground truth]{%
    \begin{minipage}{0.21\linewidth}
      \includegraphics[width=\textwidth]{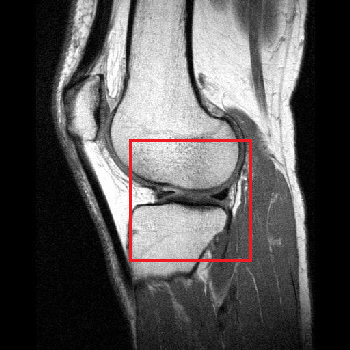}
      \\[\smallskipamount]
      \includegraphics[width=\textwidth]{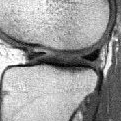}
      \\[-1em]
    \end{minipage}}\quad
  \subfigure[TGV+Shearlet]{%
    \begin{minipage}{0.21\linewidth}
      \includegraphics[width=\textwidth]{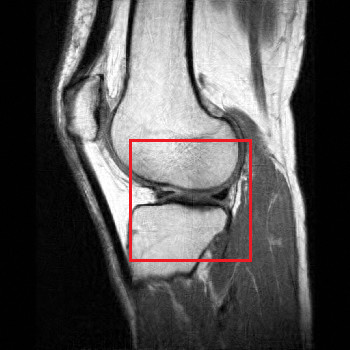}
      \\[\smallskipamount]
      \includegraphics[width=\textwidth]{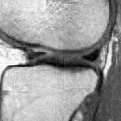}
      \\[-1em]
    \end{minipage}}\ \ 
  \subfigure[proposed model]{%
    \begin{minipage}{0.21\linewidth}
      \includegraphics[width=\textwidth]{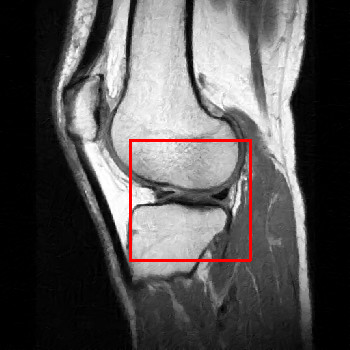}
      \\[\smallskipamount]
      \includegraphics[width=\textwidth]{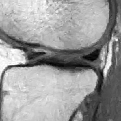}
      \\[-1em]
    \end{minipage}}

  \vspace*{-0.5em}
  \subfigure[ground truth]{%
    \begin{minipage}{0.21\linewidth}
      \includegraphics[width=\textwidth]{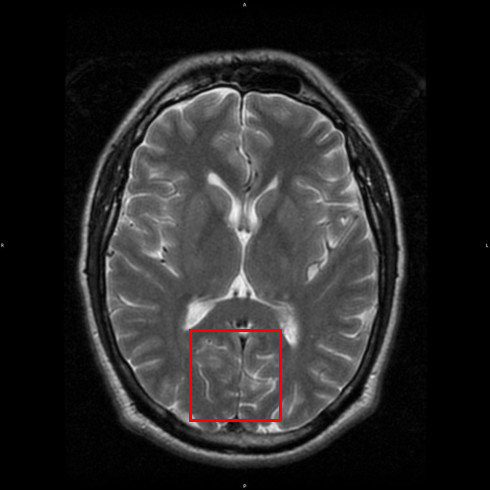}
      \\[\smallskipamount]
      \includegraphics[width=\textwidth]{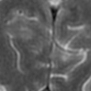}
      \\[-1em]
    \end{minipage}}\quad
  \subfigure[TGV+Shearlet]{%
    \begin{minipage}{0.21\linewidth}
      \includegraphics[width=\textwidth]{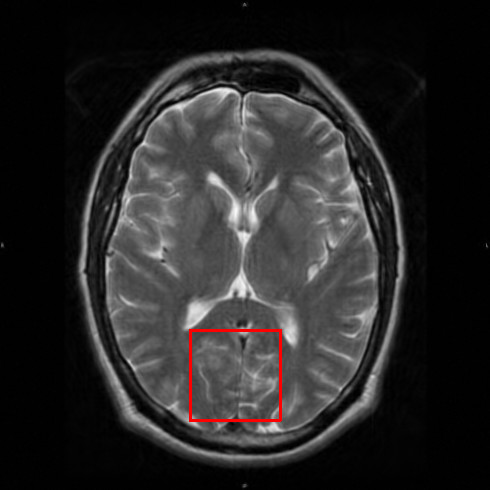}
      \\[\smallskipamount]
      \includegraphics[width=\textwidth]{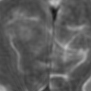}
      \\[-1em]
    \end{minipage}}\ \ 
  \subfigure[proposed model]{%
    \begin{minipage}{0.21\linewidth}
      \includegraphics[width=\textwidth]{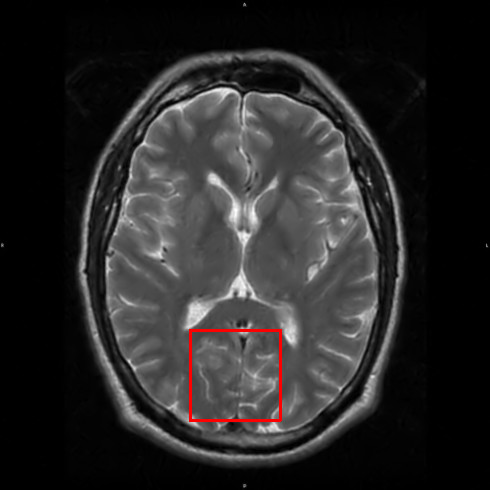}
      \\[\smallskipamount]
      \includegraphics[width=\textwidth]{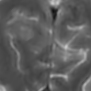}
      \\[-1em]
    \end{minipage}}
  \caption{Comparison of MRI reconstruction results.  (a) Foot image;
    (b)--(c) Recovery from $14.74\%$ radially sampled coefficients;
    (d) Knee image; (e)--(f) Recovery from $21.12\%$ radially sampled
    coefficients; (g) Brain image; (h)--(i) Recovery from $13.47\%$
    radially sampled coefficients. See the supplementary material for
    more detailed results.}
  \label{imreconfoot}
\end{figure}

\section{Summary and conclusions}
\label{sec:conclusions}

In this paper, we proposed the oscillation TGV regularizer which
arises from modifying the second-order TGV functional such that
structured oscillatory functions are in its kernel.
The infimal convolution of these novel regularization functionals can
be used to represent structured texture consisting of oscillations of
different directions and frequencies.  We have shown that the natural
space for oscillation TGV is the space of functions of bounded
variation in the sense that norm equivalence, lower semi-continuity
and coercivity can be established in that space.
Furthermore, we considered the $m$-fold infimal convolution of
oscillation TGV which can be utilized for solving many variational
image processing problems; in particular those for which
$\TV$-regularization is applicable.  The analysis of the
infimal-convolution model based on general properties of infimal
convolution of lower semi-continuous semi-norms. For such functionals,
we obtained general criteria for exactness, lower semi-continuity and
coercivity of the infimal convolution. This allowed us to obtain
existence results for variational problems regularized with this type
of functional. %
We discussed a possible discretization of the oscillation TGV
functional. It turned out that one has to choose the $\mathbf{c}_i$
specifically in order to preserve the kernel in the discretized
version.
Numerical experiments for denoising, inpainting and MRI reconstruction
show the major properties of the proposed model in terms of texture
preservation and reconstruction. %
The general framework obtained in this article allows in particular
the extension to other imaging problems.  A possible line of future
research may focus on reducing the complexity of the proposed model by
studying variants of oscillation TGV that allow for reconstruction of
oscillatory textures for multiple directions and frequencies with a
simple, differentiation-based functional.

\section*{Acknowledgments}

The first author is supported by China Scholarship Council (CSC) Scholarship Program and National Natural Science Foundation of China (Nos.~11531005 and 91330101). The second author acknowledges support by the Austrian Science Fund
(FWF): project number P-29192 ``Regularization Graphs for Variational
Imaging''.

\clearpage
\appendix

}

\section*{Supplementary material}

\tableofcontents

\section{Image denoising}

\subsection{Comparison of PSNR-optimal results}

\subsubsection{Noise level $\sigma=0.05$}

See Figures~\ref{imdenoisebar},~\ref{imdenoisezebra} and~\ref{imdenoiseparrot},
and Table~\ref{table:denoising005}.

\begin{table}[H]
  \center{}
  \begin{tabular}{|c|c|c|c|c|c|c|c|}
    \hline
    \hline
    &image&TGV &NLTV&ICTGV&Fra+LDCT& BM3D& proposed model \\
    \hline
    \multirow{3}{*}{PSNR}
               & barbara & 27.41 & 32.05 &31.15& 31.70 & \textbf{34.43} & 32.21\\
    \cline{2-8}& zebra & 28.74 & 31.16 &31.08& 30.81 & \textbf{32.25} & 31.33\\
    \cline{2-8}& parrots & 34.77 & 35.10 &36.49& 36.10 & \textbf{37.40} & 36.61\\
    \hline
    \hline
    \multirow{3}{*}{SSIM}
           & barbara & 0.8043 & 0.9005 &0.8859& 0.8857 &\textbf{0.9365}& 0.9004\\
    \cline{2-8}&zebra & 0.8386 & 0.8647 &0.8761& 0.8653 & \textbf{0.8998} & 0.8859 \\
    \cline{2-8}& parrots & 0.9157 & 0.8956 &0.9324& 0.9259 & \textbf{0.9416} & 0.9358 \\
    \hline

  \end{tabular}
  \caption{\label{table:denoising005} Comparison of denoising performance of PSNR-optimized results with noise level $\sigma=0.05$ in terms of PSNR and SSIM.}
\end{table}

\begin{figure}
  \center{} 
\subfigure[ground truth]{%
	\begin{minipage}{0.22\linewidth}
		\includegraphics[width=\textwidth]{images_denoise_barbara_barbara2_red.jpg}
		\\[\smallskipamount]
		\includegraphics[width=\textwidth]{images_denoise_barbara_barbara2_closeup.jpg}
		\\[-1em]
\end{minipage}}
  \subfigure[noisy image ($\sigma=0.05$)]{%
    \begin{minipage}{0.22\linewidth}
      \includegraphics[width=\textwidth]{images_denoise_barbara_noisy_005_red.jpg} 
      \\[\smallskipamount]
      \includegraphics[width=\textwidth]{images_denoise_barbara_noisy_005_closeup.jpg}
      \\[-1em]
    \end{minipage}}
  \subfigure[TGV]{%
    \begin{minipage}{0.22\linewidth}
      \includegraphics[width=\textwidth]{images_denoise_barbara_tgv_noise_005_red.jpg}
      \\[\smallskipamount]
      \includegraphics[width=\textwidth]{images_denoise_barbara_tgv_noise_005_closeup.jpg}
      \\[-1em]
    \end{minipage}}
  \subfigure[NLTV]{%
    \begin{minipage}{0.22\linewidth}
      \includegraphics[width=\textwidth]{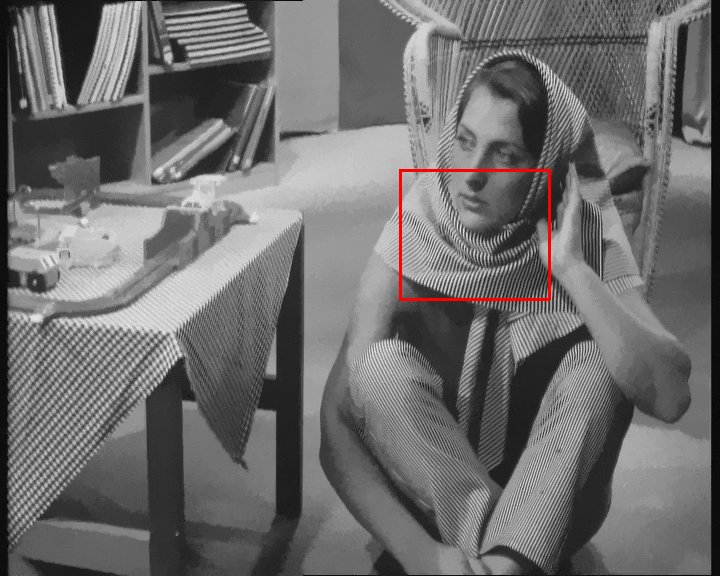}
      \\[\smallskipamount]
      \includegraphics[width=\textwidth]{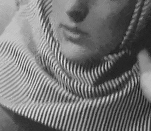}
      \\[-1em]
    \end{minipage}}

  \subfigure[$\ICTGV$]{%
	\begin{minipage}{0.22\linewidth}
		\includegraphics[width=\textwidth]{images_denoise_barbara_barbara_ictgv_005_vision_red.jpg}
		\\[\smallskipamount]
		\includegraphics[width=\textwidth]{images_denoise_barbara_barbara_ictgv_005_vision_closeup.jpg}
		\\[-1em]
\end{minipage}}
  \subfigure[Fra+LDCT]{%
    \begin{minipage}{0.22\linewidth}
      \includegraphics[width=\textwidth]{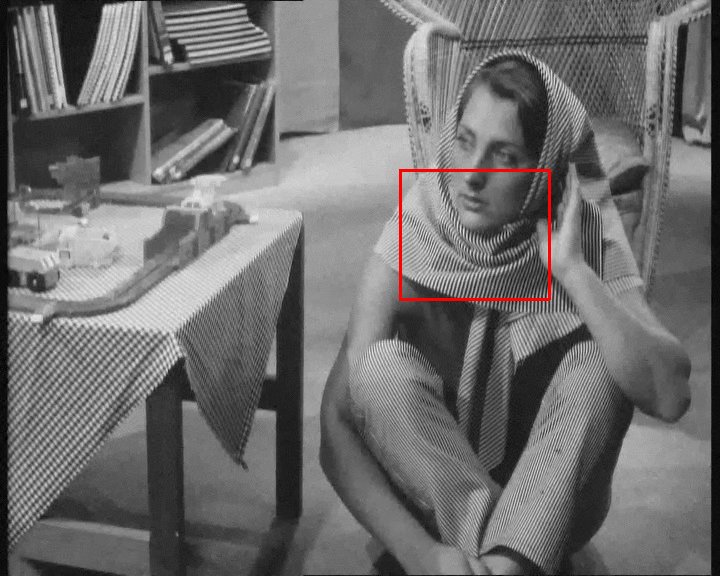}
      \\[\smallskipamount]
      \includegraphics[width=\textwidth]{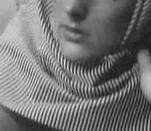}
      \\[-1em]
    \end{minipage}}
  \subfigure[BM3D]{%
    \begin{minipage}{0.22\linewidth}
      \includegraphics[width=\textwidth]{images_denoise_barbara_barbara_bm3d_005_red.jpg}
      \\[\smallskipamount]
      \includegraphics[width=\textwidth]{images_denoise_barbara_barbara_bm3d_005_closeup.jpg}
      \\[-1em]
    \end{minipage}}
  \subfigure[proposed model]{%
    \begin{minipage}{0.22\linewidth}
      \includegraphics[width=\textwidth]{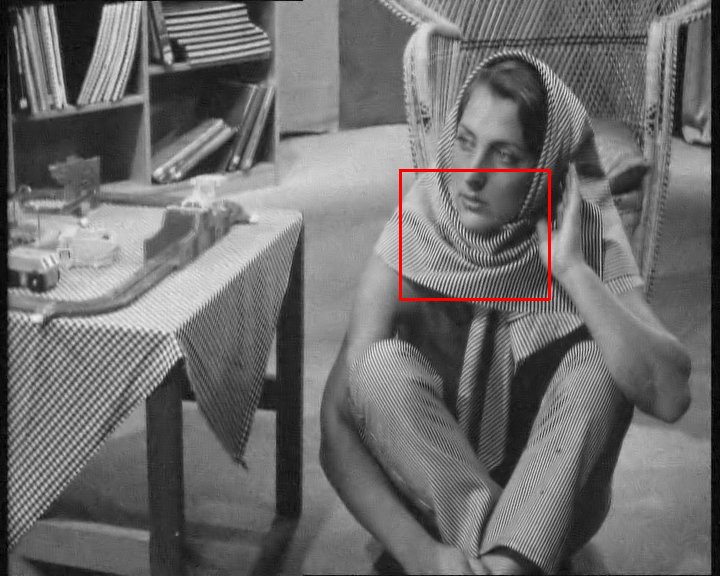}
      \\[\smallskipamount]
      \includegraphics[width=\textwidth]{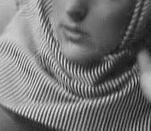}
      \\[-1em]
    \end{minipage}}
  \caption{Denoising results for ``barbara'' image ($720 \times 576$
    pixels) with noise level $\sigma=0.05$. Parameter choice for
    $\ICTGV^{\osci}$:
    $\alpha_1=0.045, \beta_1=0.7\alpha_1, \gamma_1=0,
    \alpha_i=0.9\alpha_1, \beta_i=0.6\alpha_i, \gamma_i=0.1\alpha_i,
    i=2,\ldots,17.$}
  \label{imdenoisebar}
\end{figure}

\begin{figure}
  \center{} 
    \subfigure[ground truth]{%
      \begin{minipage}{0.22\linewidth}
        \includegraphics[width=\textwidth]{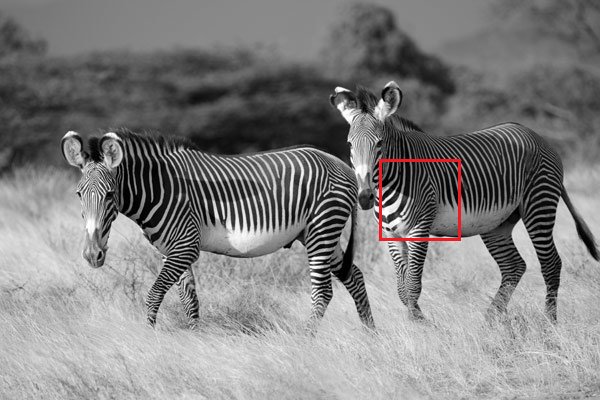}
        \\[\smallskipamount]
        \includegraphics[width=\textwidth]{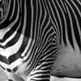}
        \\[-1em]
      \end{minipage}}
  \subfigure[noisy image ($\sigma=0.05$)]{%
    \begin{minipage}{0.22\linewidth}
      \includegraphics[width=\textwidth]{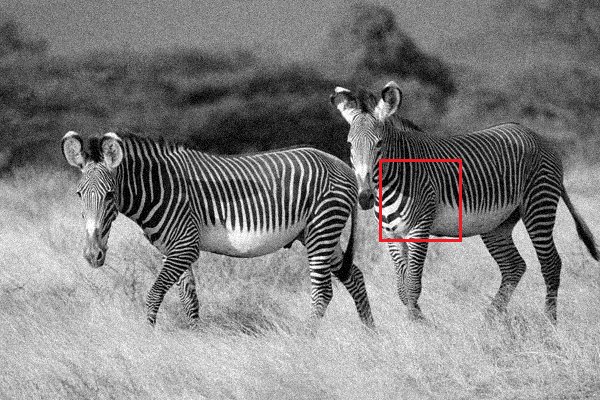}
      \\[\smallskipamount]
      \includegraphics[width=\textwidth]{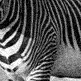}
      \\[-1em]
    \end{minipage}}
  \subfigure[TGV]{%
    \begin{minipage}{0.22\linewidth}
      \includegraphics[width=\textwidth]{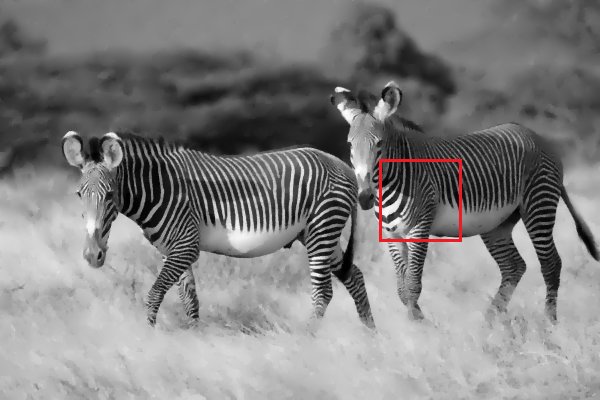}
      \\[\smallskipamount]
      \includegraphics[width=\textwidth]{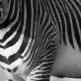}
      \\[-1em]
    \end{minipage}}
  \subfigure[NLTV]{%
    \begin{minipage}{0.22\linewidth}
      \includegraphics[width=\textwidth]{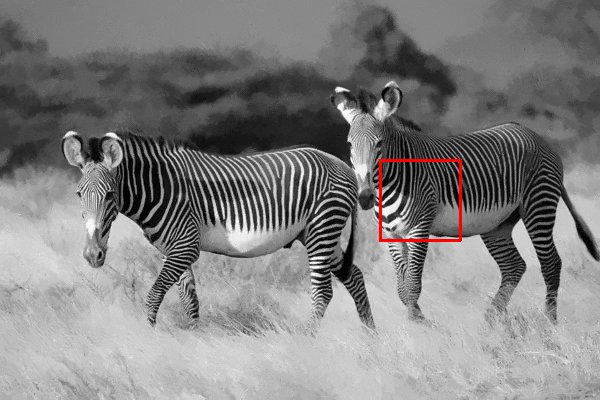}
      \\[\smallskipamount]
      \includegraphics[width=\textwidth]{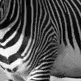}
      \\[-1em]
    \end{minipage}}

  \subfigure[ICTGV]{%
	\begin{minipage}{0.22\linewidth}
		\includegraphics[width=\textwidth]{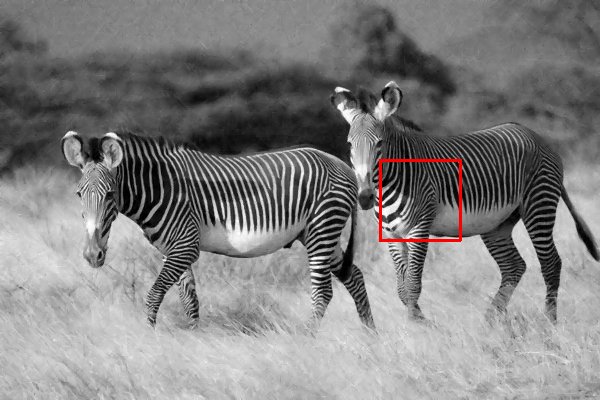}
		\\[\smallskipamount]
		\includegraphics[width=\textwidth]{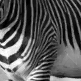}
		\\[-1em]
\end{minipage}}
  \subfigure[Fra+LDCT]{%
    \begin{minipage}{0.22\linewidth}
      \includegraphics[width=\textwidth]{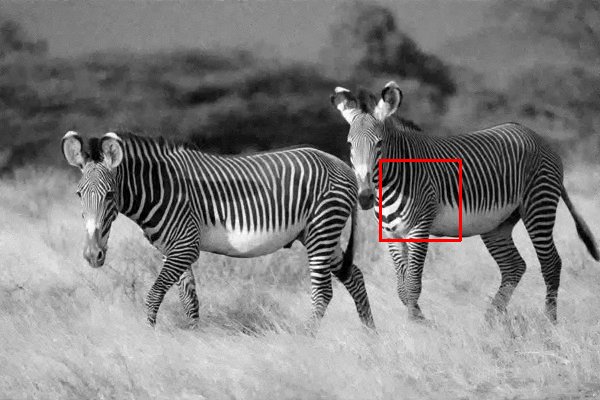}
      \\[\smallskipamount]
      \includegraphics[width=\textwidth]{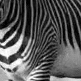}
      \\[-1em]
    \end{minipage}}
  \subfigure[BM3D]{%
    \begin{minipage}{0.22\linewidth}
      \includegraphics[width=\textwidth]{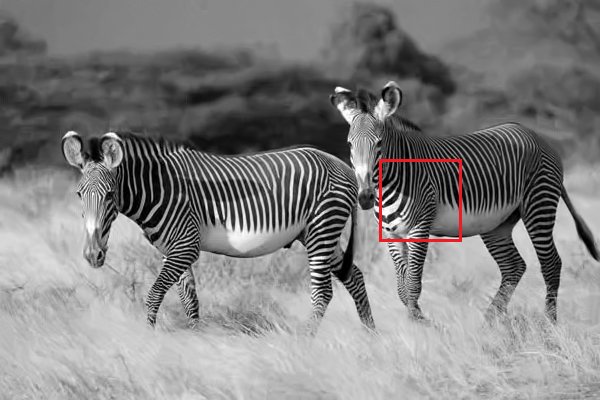}
      \\[\smallskipamount]
      \includegraphics[width=\textwidth]{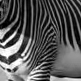}
      \\[-1em]
    \end{minipage}}
  \subfigure[proposed model]{%
    \begin{minipage}{0.22\linewidth}
      \includegraphics[width=\textwidth]{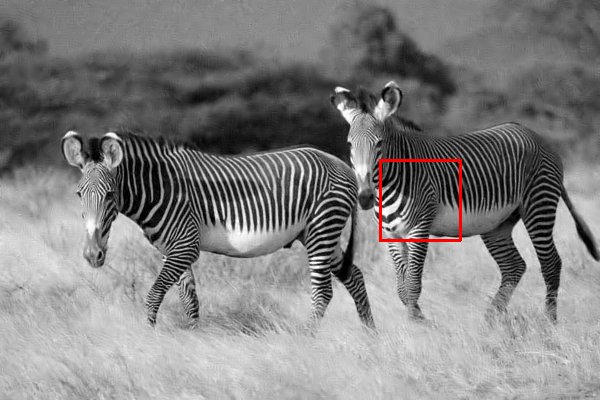}
      \\[\smallskipamount]
      \includegraphics[width=\textwidth]{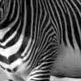}
      \\[-1em]
    \end{minipage}}
  \caption{Denoising results for ``zebra'' image ($640 \times 400$ pixels)
    with noise level $\sigma=0.05$. Parameter choice for
    $\ICTGV^{\osci}$:
    $\alpha_1=0.045, \beta_1=0.6\alpha_1, \gamma_1=0,
    \alpha_i=\alpha_1, \beta_i=0.5\alpha_i, \gamma_i=0.1\alpha_i,
    i=2,\ldots,17$.}
  \label{imdenoisezebra}
\end{figure}

\begin{figure}
  \center{} 
    \subfigure[ground truth]{%
      \begin{minipage}{0.22\linewidth}
        \includegraphics[width=\textwidth]{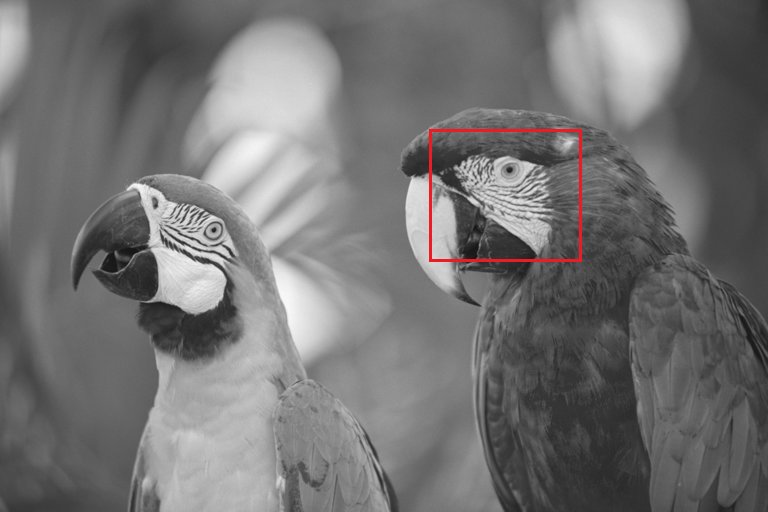}
        \\[\smallskipamount]
        \includegraphics[width=\textwidth]{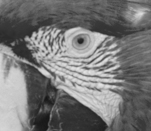}
        \\[-1em]
      \end{minipage}}
  \subfigure[noisy image ($\sigma=0.05$)]{%
    \begin{minipage}{0.22\linewidth}
      \includegraphics[width=\textwidth]{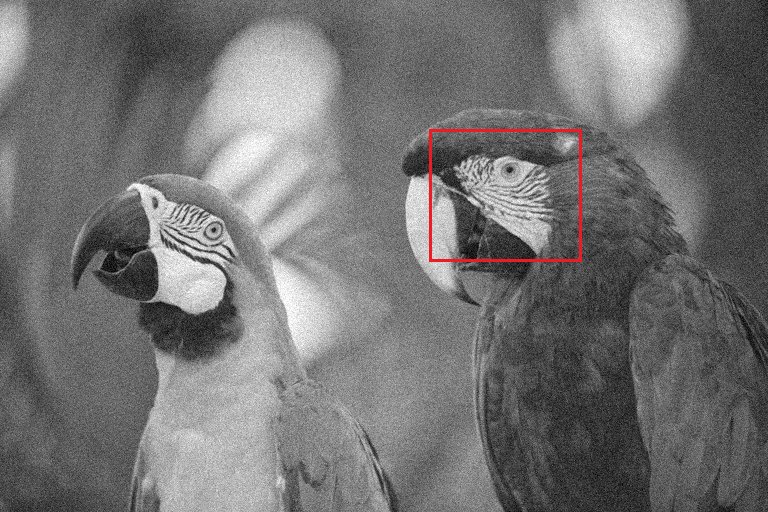}
      \\[\smallskipamount]
      \includegraphics[width=\textwidth]{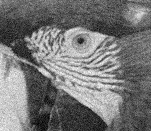}
      \\[-1em]
    \end{minipage}}
  \subfigure[TGV]{%
    \begin{minipage}{0.22\linewidth}
      \includegraphics[width=\textwidth]{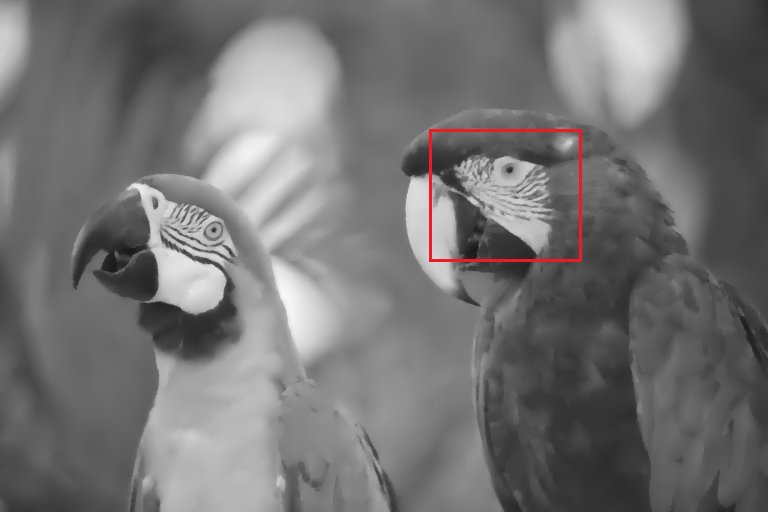}
      \\[\smallskipamount]
      \includegraphics[width=\textwidth]{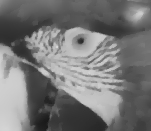}
      \\[-1em]
    \end{minipage}}
  \subfigure[NLTV]{%
    \begin{minipage}{0.22\linewidth}
      \includegraphics[width=\textwidth]{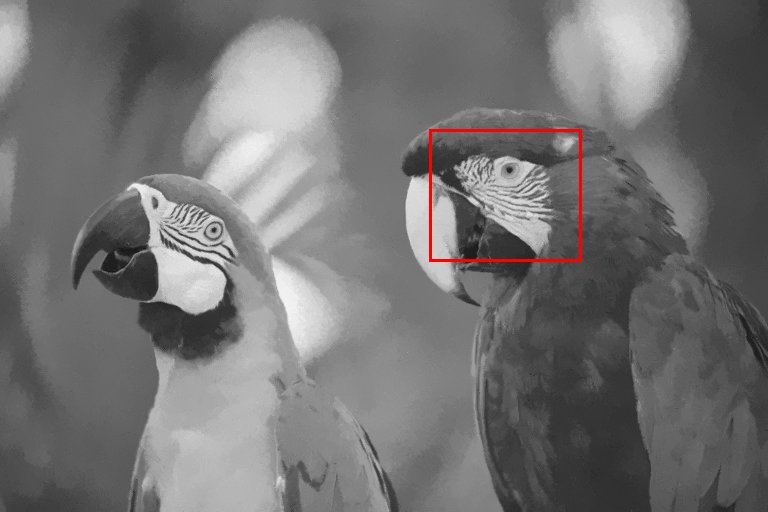}
      \\[\smallskipamount]
      \includegraphics[width=\textwidth]{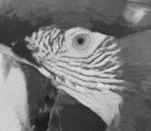}
      \\[-1em]
    \end{minipage}}

  \subfigure[ICTGV]{%
	\begin{minipage}{0.22\linewidth}
		\includegraphics[width=\textwidth]{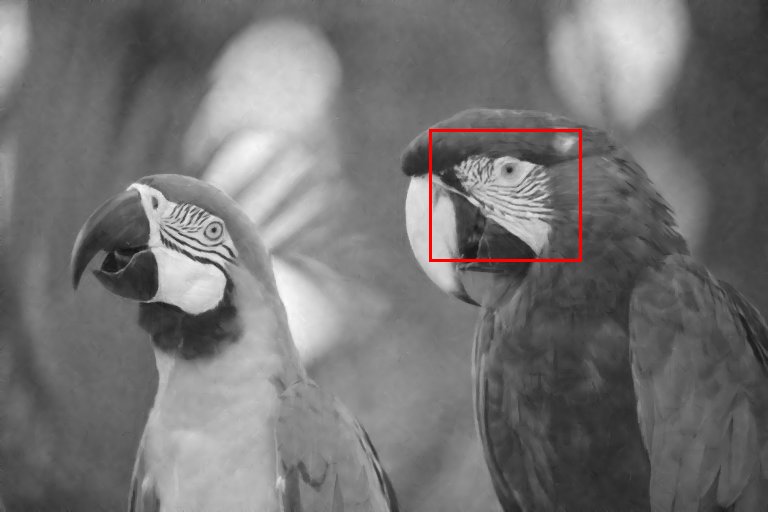}
		\\[\smallskipamount]
		\includegraphics[width=\textwidth]{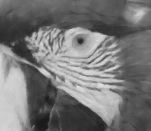}
		\\[-1em]
\end{minipage}}
  \subfigure[Fra+LDCT]{%
    \begin{minipage}{0.22\linewidth}
      \includegraphics[width=\textwidth]{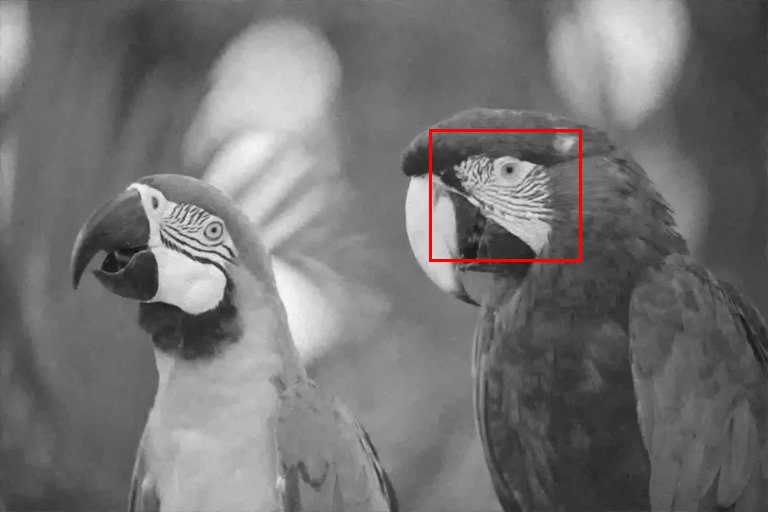}
      \\[\smallskipamount]
      \includegraphics[width=\textwidth]{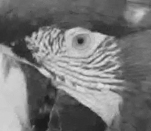}
      \\[-1em]
    \end{minipage}}
  \subfigure[BM3D]{%
    \begin{minipage}{0.22\linewidth}
      \includegraphics[width=\textwidth]{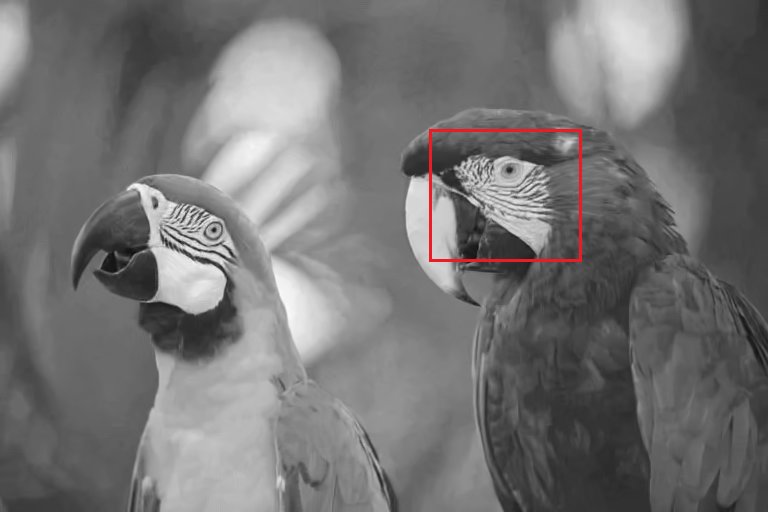}
      \\[\smallskipamount]
      \includegraphics[width=\textwidth]{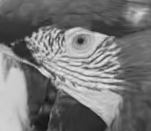}
      \\[-1em]
    \end{minipage}}
  \subfigure[proposed model]{%
    \begin{minipage}{0.22\linewidth}
      \includegraphics[width=\textwidth]{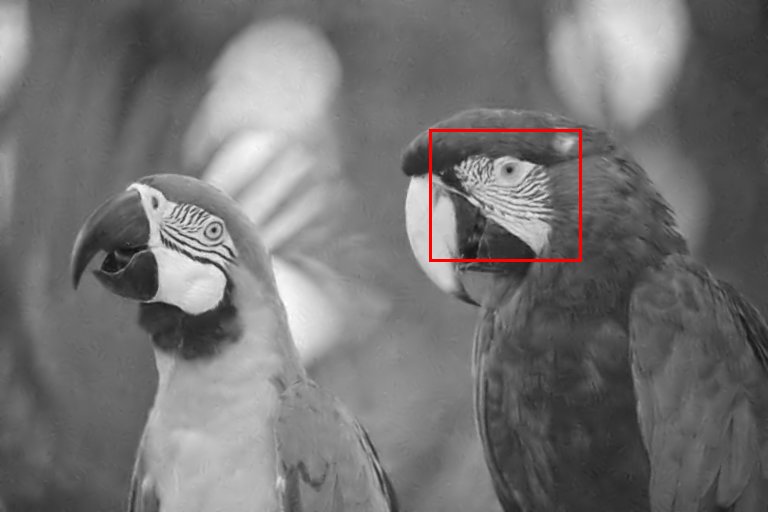}
      \\[\smallskipamount]
      \includegraphics[width=\textwidth]{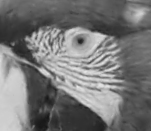}
      \\[-1em]
    \end{minipage}}
  \caption{ Denoising results for ``parrots'' image ($768 \times 512$
    pixels) with noise level $\sigma=0.05$. Parameter choice for
    $\ICTGV^{\osci}$:
    $\alpha_1=0.05, \beta_1=0.9\alpha_1, \gamma_1=0,
    \alpha_i=\alpha_1, \beta_i=0.5\alpha_i, \gamma_i=0.14\alpha_i,
    i=2,\ldots,9$.}
  \label{imdenoiseparrot}
\end{figure}

\clearpage
\subsubsection{Noise level $\sigma=0.1$}

See Figures~\ref{imdenoisebar2},~\ref{imdenoisezebra2}
and~\ref{imdenoiseparrot2}, and Table~\ref{table:denoising01}.

\begin{table}[H]
  \center{}
  \begin{tabular}{|c|c|c|c|c|c|c|c|}
    \hline
    \hline
    &image&TGV &NLTV&ICTGV&Fra+LDCT& BM3D& proposed model \\
    \hline
    \multirow{3}{*}{PSNR}
           & barbara & 25.34 & 27.58 &27.35& 28.05 & \textbf{31.06} & 28.45\\
    \cline{2-8}& zebra & 25.24 & 27.35 &27.35& 27.06 & \textbf{28.80} & 27.65 \\
    \cline{2-8}& parrots & 32.51 & 31.39 &33.28& 32.91 & \textbf{34.08} & 33.32 \\
    \hline
    \hline
    \multirow{3}{*}{SSIM}
           & barbara & 0.7297 & 0.7553 &0.7835& 0.8214 &\textbf{0.8920} & 0.8235\\
    \cline{2-8}&zebra & 0.7472 & 0.7271 &0.7796& 0.7868 & \textbf{0.8237} & 0.7965 \\
    \cline{2-8}& parrots & 0.8887 & 0.8312 &0.8975& 0.8903 & \textbf{0.9025} & 0.8979 \\
    \hline
  \end{tabular}
  \caption{\label{table:denoising01} Comparison of denoising performance of PSNR-optimized results with noise level $\sigma=0.1$ in terms of PSNR and SSIM.}
\end{table}

\begin{figure}
  \center{} 
  \subfigure[noisy image ($\sigma=0.1$)]{%
    \begin{minipage}{0.32\linewidth}
      \includegraphics[width=\textwidth]{images_denoise_barbara_noisy_01.jpg}
      \\[-1em]
    \end{minipage}}

  \subfigure[TGV]{%
    \begin{minipage}{0.32\linewidth}
      \includegraphics[width=\textwidth]{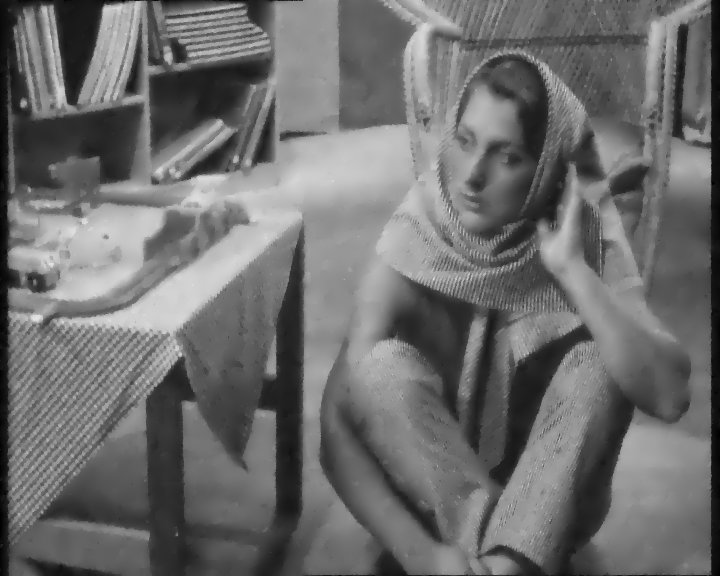}
      \\[-1em]
    \end{minipage}}
  \subfigure[NLTV]{%
    \begin{minipage}{0.32\linewidth}
      \includegraphics[width=\textwidth]{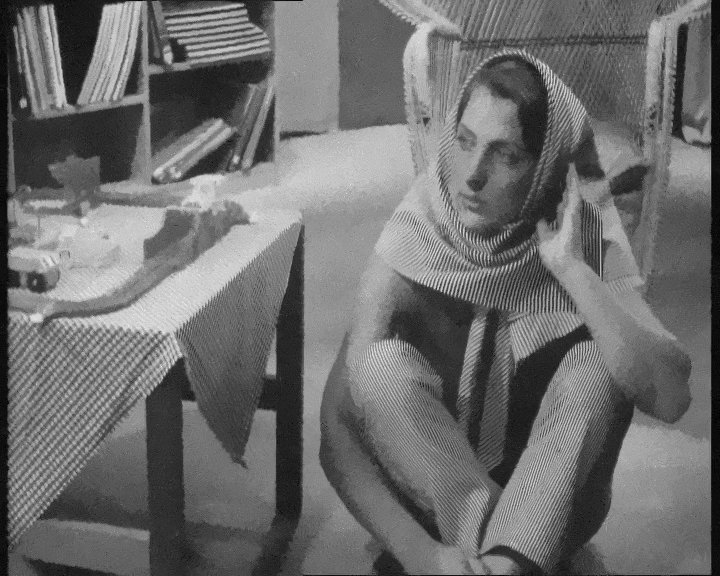}
      \\[-1em]
    \end{minipage}}
    \subfigure[ICTGV]{%
  	\begin{minipage}{0.32\linewidth}
  		\includegraphics[width=\textwidth]{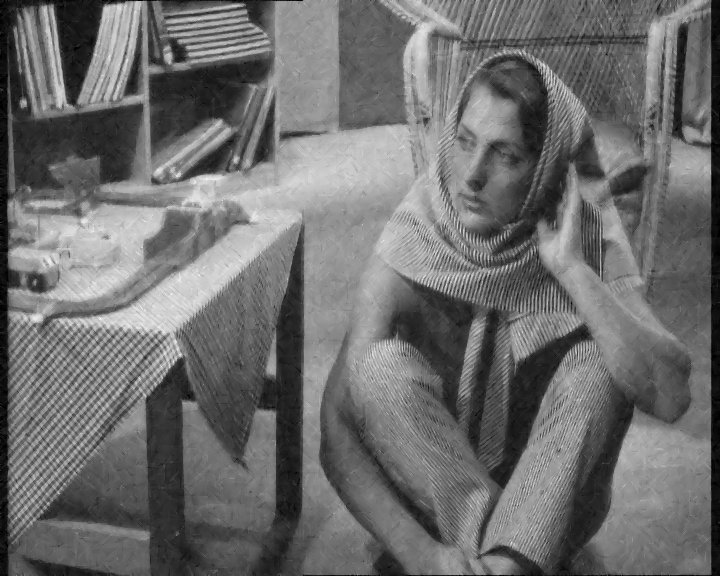}
  		\\[-1em]
  \end{minipage}}

  \subfigure[Fra+LDCT]{%
    \begin{minipage}{0.32\linewidth}
      \includegraphics[width=\textwidth]{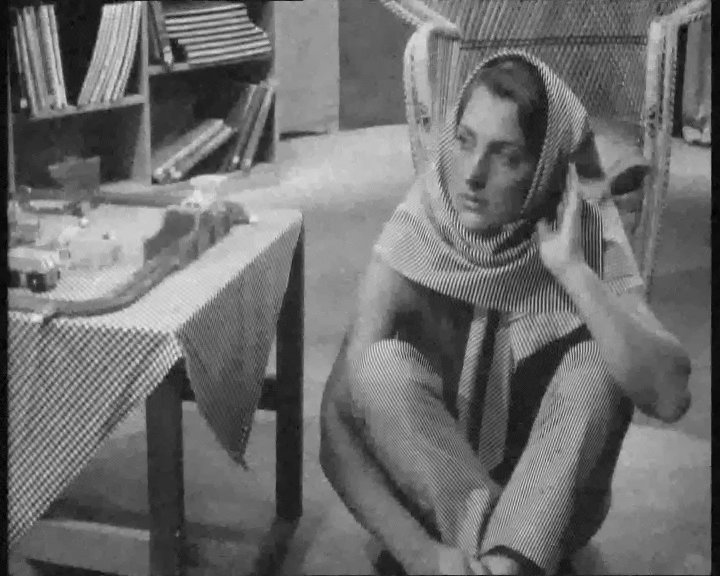}
      \\[-1em]
    \end{minipage}}
  \subfigure[BM3D]{%
    \begin{minipage}{0.32\linewidth}
      \includegraphics[width=\textwidth]{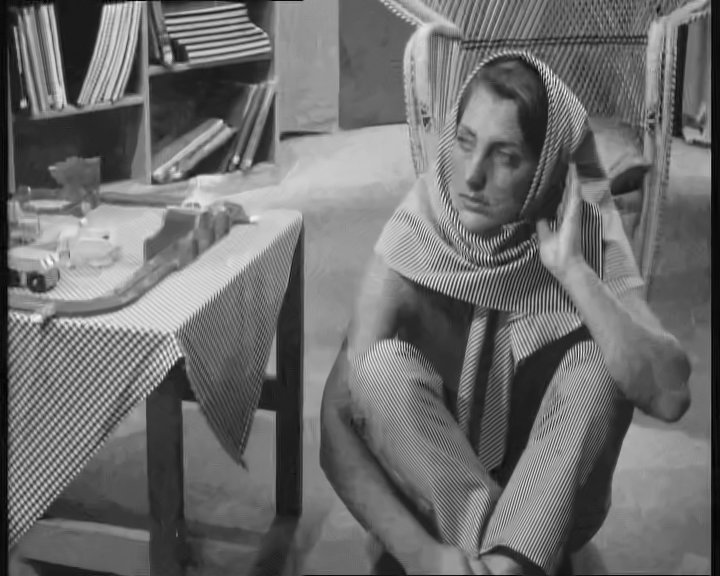}
      \\[-1em]
    \end{minipage}}
  \subfigure[proposed model]{%
    \begin{minipage}{0.32\linewidth}
      \includegraphics[width=\textwidth]{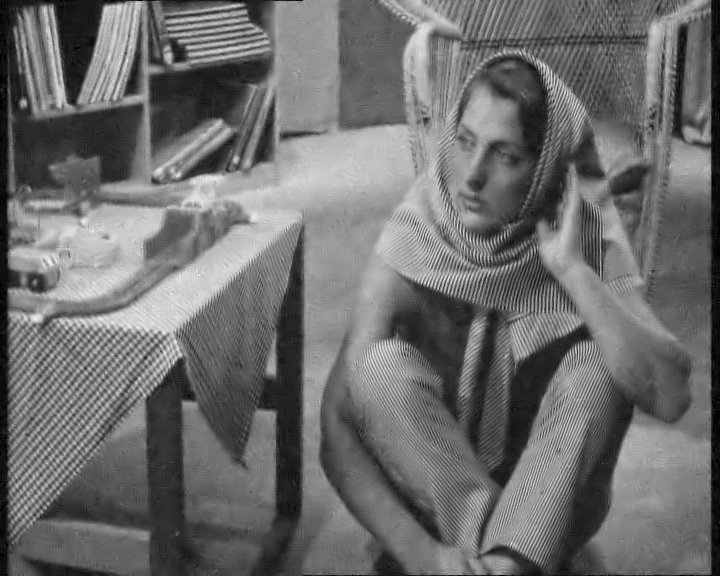}
      \\[-1em]
    \end{minipage}}
  \caption{Denoising results for ``barbara'' image ($720 \times 576$
    pixels) with noise level $\sigma=0.1$. Parameter choice for
    $\ICTGV^{\osci}$:
    $\alpha_1=0.1, \beta_1=0.8\alpha_1, \gamma_1=0,
    \alpha_i=0.8\alpha_1, \beta_i=0.7\alpha_i, \gamma_i=0.15\alpha_i,
    i=2,\ldots,17$.}
  \label{imdenoisebar2}
\end{figure}

\begin{figure}
  \center{} 
  \subfigure[noisy image ($\sigma=0.1$)]{%
    \begin{minipage}{0.32\linewidth}
      \includegraphics[width=\textwidth]{images_denoise_zebra_zebra_noisy_01.jpg}
      \\[-1em]
    \end{minipage}}

  \subfigure[TGV]{%
    \begin{minipage}{0.32\linewidth}
      \includegraphics[width=\textwidth]{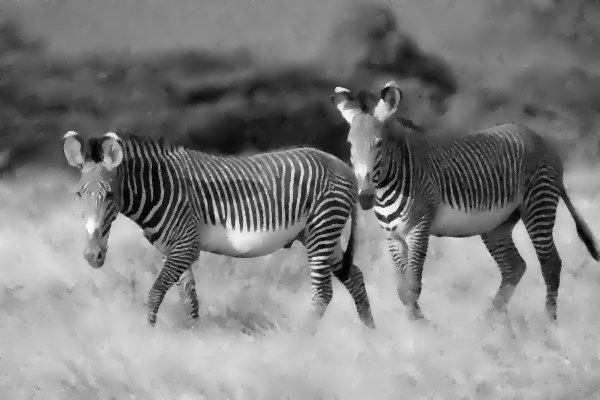}
      \\[-1em]
    \end{minipage}}
  \subfigure[NLTV]{%
    \begin{minipage}{0.32\linewidth}
      \includegraphics[width=\textwidth]{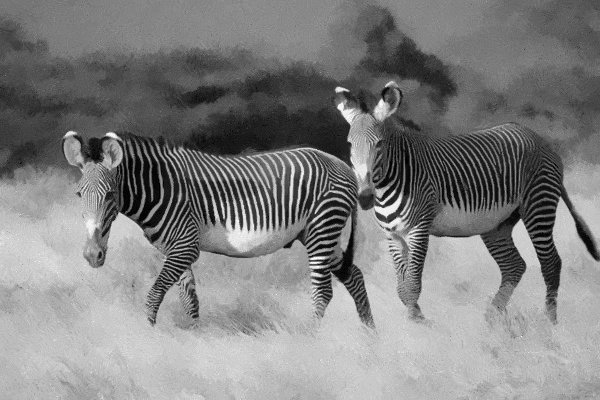}
      \\[-1em]
    \end{minipage}}
  \subfigure[ICTGV]{%
	\begin{minipage}{0.32\linewidth}
		\includegraphics[width=\textwidth]{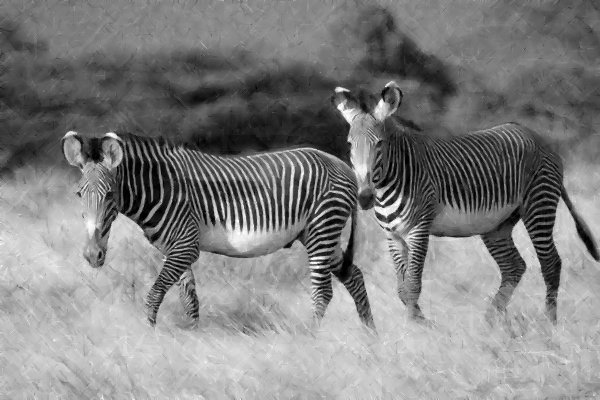}
		\\[-1em]
\end{minipage}}
  
  \subfigure[Fra+LDCT]{%
    \begin{minipage}{0.32\linewidth}
      \includegraphics[width=\textwidth]{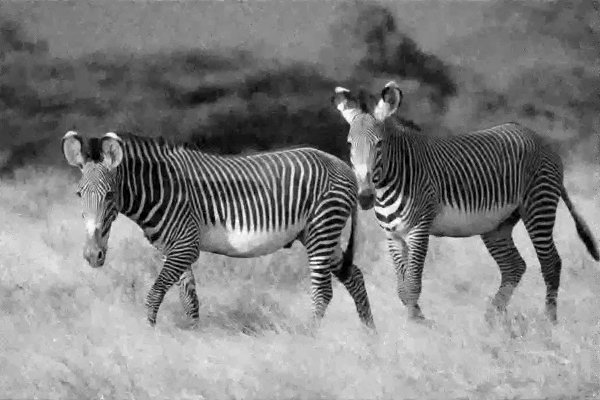}
      \\[-1em]
    \end{minipage}}
  \subfigure[BM3D]{%
    \begin{minipage}{0.32\linewidth}
      \includegraphics[width=\textwidth]{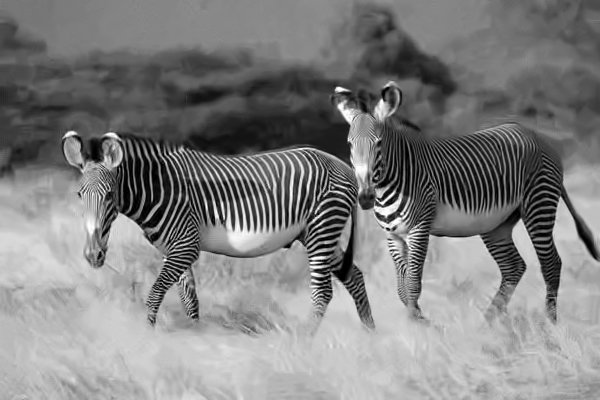}
      \\[-1em]
    \end{minipage}}
  \subfigure[proposed model]{%
    \begin{minipage}{0.32\linewidth}
      \includegraphics[width=\textwidth]{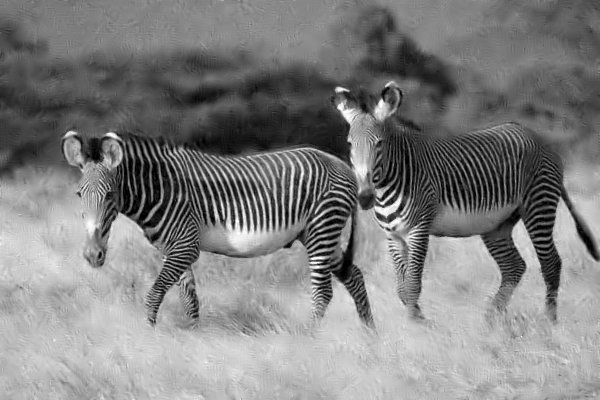}
      \\[-1em]
    \end{minipage}}
  \caption{Denoising results for ``zebra'' image ($640 \times 400$ pixels)
    with noise level $\sigma=0.1$. Parameter choice for
    $\ICTGV^{\osci}$:
    $\alpha_1=0.09, \beta_1=\alpha_1, \gamma_1=0,
    \alpha_i=1.1\alpha_1, \beta_i=0.5\alpha_i, \gamma_i=0.12\alpha_i,
    i=2,\ldots,17$.}
  \label{imdenoisezebra2}
\end{figure}

\begin{figure}
  \center{} 
  \subfigure[noisy image ($\sigma=0.1$)]{%
    \begin{minipage}{0.32\linewidth}
      \includegraphics[width=\textwidth]{images_denoise_bird_bird_noisy_01.jpg}
      \\[-1em]
    \end{minipage}}

  \subfigure[TGV]{%
    \begin{minipage}{0.32\linewidth}
      \includegraphics[width=\textwidth]{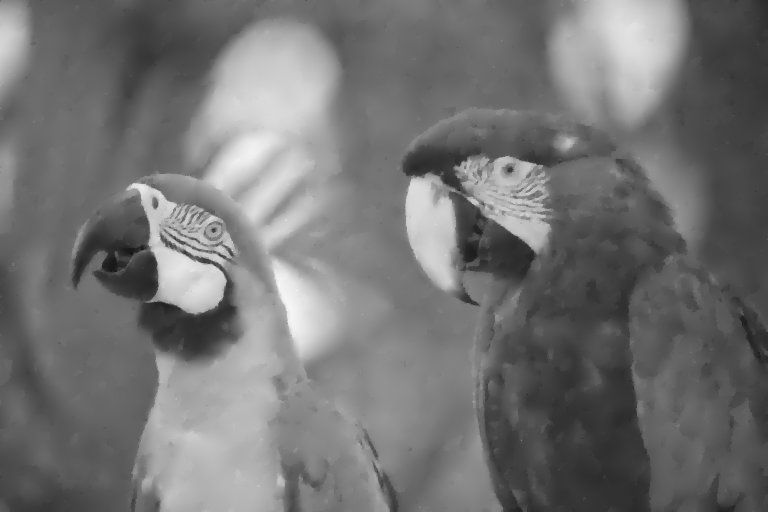}
      \\[-1em]
    \end{minipage}}
  \subfigure[NLTV]{%
    \begin{minipage}{0.32\linewidth}
      \includegraphics[width=\textwidth]{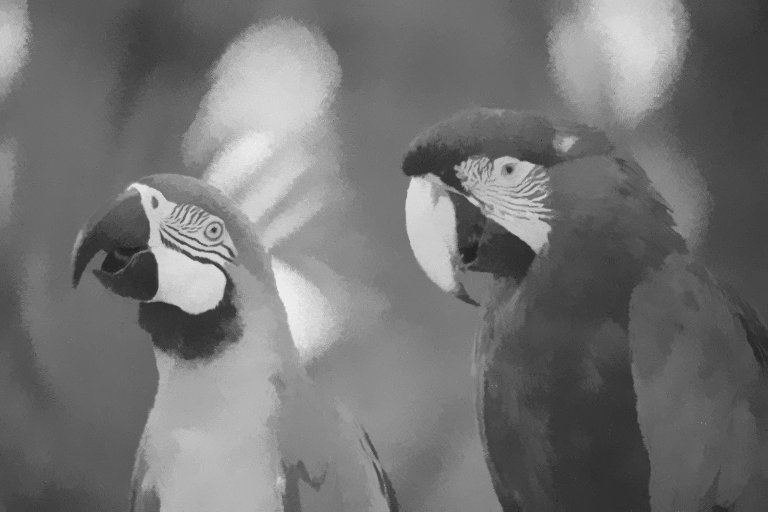}
      \\[-1em]
    \end{minipage}}
  \subfigure[ICTGV]{%
	\begin{minipage}{0.32\linewidth}
		\includegraphics[width=\textwidth]{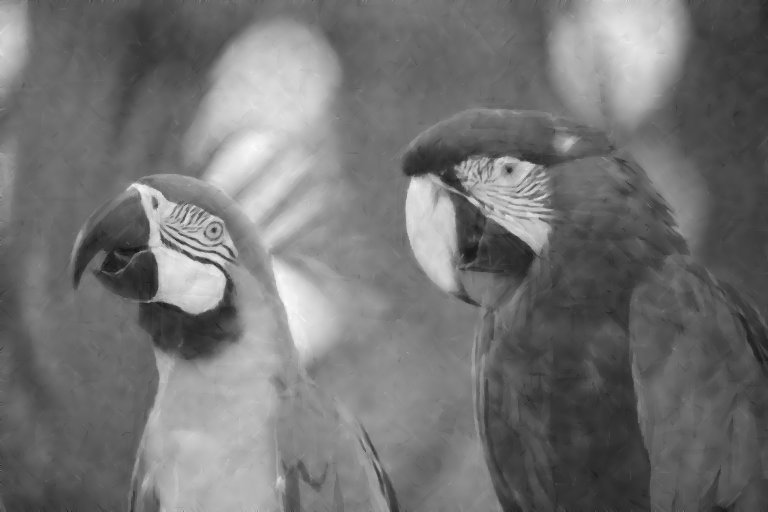}
		\\[-1em]
\end{minipage}}
  
  \subfigure[Fra+LDCT]{%
    \begin{minipage}{0.32\linewidth}
      \includegraphics[width=\textwidth]{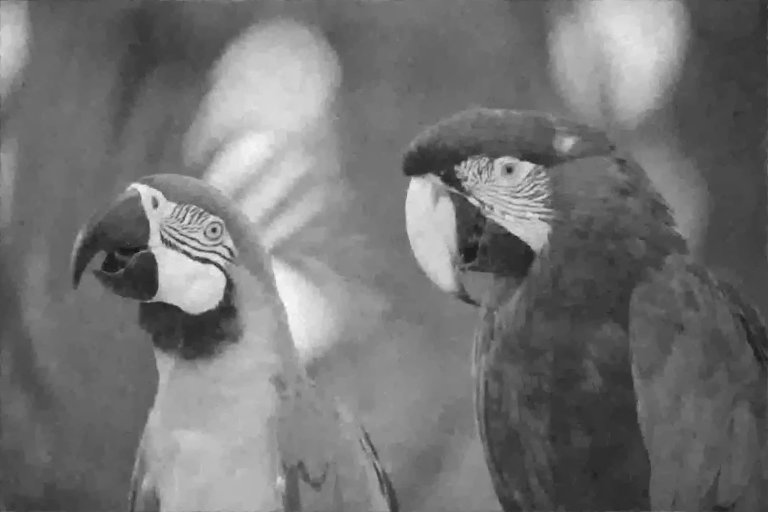}
      \\[-1em]
    \end{minipage}}
  \subfigure[BM3D]{%
    \begin{minipage}{0.32\linewidth}
      \includegraphics[width=\textwidth]{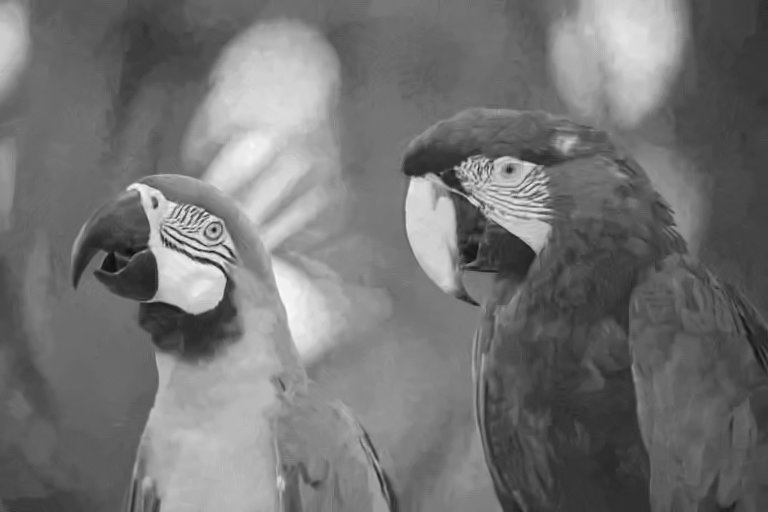}
      \\[-1em]
    \end{minipage}}
  \subfigure[proposed model]{%
    \begin{minipage}{0.32\linewidth}
      \includegraphics[width=\textwidth]{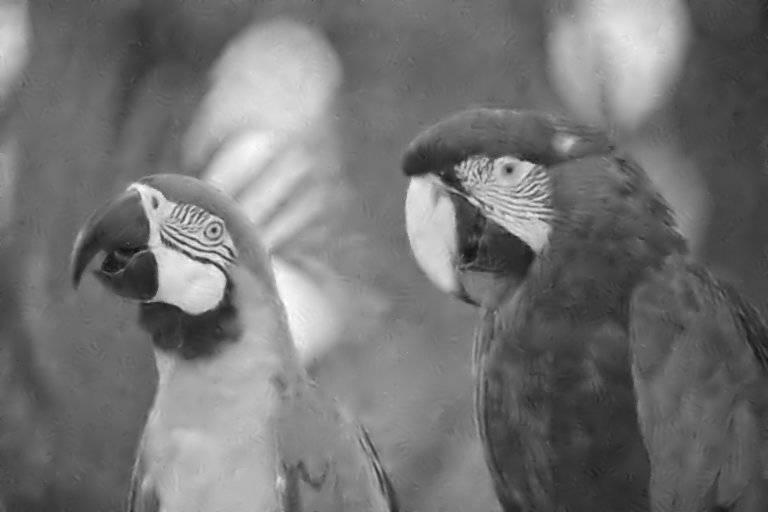}
      \\[-1em]
    \end{minipage}}
  \caption{Denoising results for parrots image ($768 \times 512$
    pixels) with noise level $\sigma=0.1$. Parameter choice for
    $\ICTGV^{\osci}$:
    $\alpha_1=0.1, \beta_1=1.2\alpha_1, \gamma_1=0, \alpha_i=1.2\alpha_1,
    \beta_i=0.5\alpha_i, \gamma_i=0.1\alpha_i, i=2,\ldots,9$.}
  \label{imdenoiseparrot2}
\end{figure}

\clearpage
\subsection{Comparison of visually-optimized results}

\subsubsection{Noise level $\sigma=0.05$}

See Figures~\ref{imdenoisebar_vision},~\ref{imdenoisezebra_vision}
and~\ref{imdenoiseparrot_vision}, and
Table~\ref{table:denoising005_vis}.

\begin{table}[H]
  \center{}
  \begin{tabular}{|c|c|c|c|c|c|c|c|}
    \hline
    \hline
    &image&TGV &NLTV&ICTGV&Fra+LDCT& BM3D& proposed model \\
    \hline
    \multirow{3}{*}{PSNR}
           & barbara & 27.41 & 31.28&30.33 & 31.19 & \textbf{34.43} & 31.44\\
    \cline{2-8}& zebra & 28.74 & 30.07 &30.43& 30.06 & \textbf{32.25} & 30.46 \\
    \cline{2-8}& parrots & 34.77 & 34.08 &35.63& 35.49 & \textbf{37.40} & 36.01 \\
    \hline
    \hline
    \multirow{3}{*}{SSIM}
          	& barbara & 0.8043 & 0.8917&0.8833 & 0.8988 &\textbf{0.9365}& 0.8988\\
          \cline{2-8}&zebra & 0.8386 & 0.8419 &0.8735 & 0.8671 & \textbf{0.8998} & 0.8777 \\
          \cline{2-8}& parrots & 0.9157 & 0.9062 &0.9272 & 0.9257 & \textbf{0.9416} &0.9303 \\
    \hline
  \end{tabular}
  \caption{\label{table:denoising005_vis} Comparison of denoising performance of visually-optimized results with noise level $\sigma=0.05$ in terms of PSNR and SSIM.}
\end{table}

\begin{figure}
  \center{} 
    \subfigure[ground truth]{%
      \begin{minipage}{0.22\linewidth}
        \includegraphics[width=\textwidth]{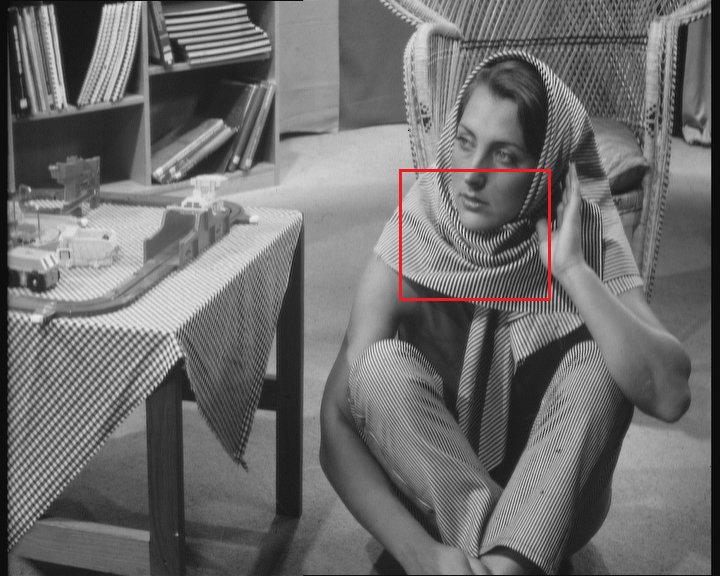}
        \\[\smallskipamount]
         \includegraphics[width=\textwidth]{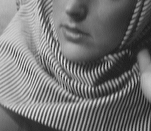}
        \\[-1em]
      \end{minipage}}
  \subfigure[noisy image ($\sigma=0.05$)]{%
    \begin{minipage}{0.22\linewidth}
      \includegraphics[width=\textwidth]{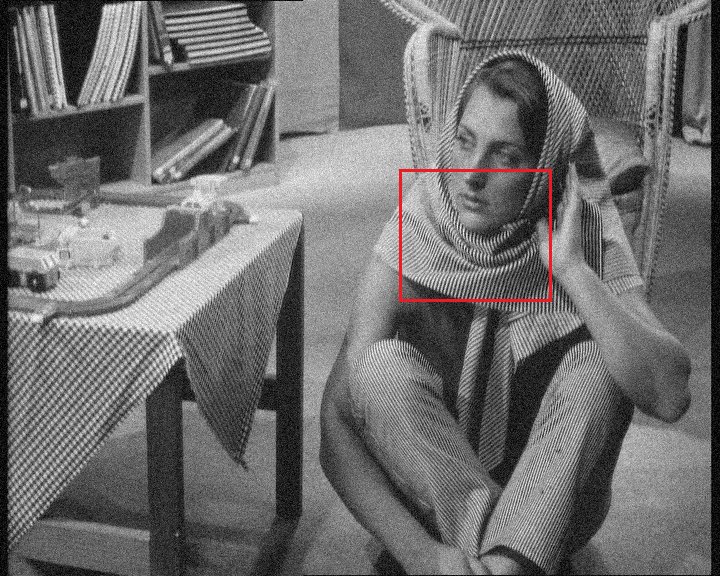} 
      \\[\smallskipamount]
      \includegraphics[width=\textwidth]{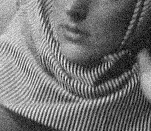}
      \\[-1em]
    \end{minipage}}
  \subfigure[TGV]{%
    \begin{minipage}{0.22\linewidth}
      \includegraphics[width=\textwidth]{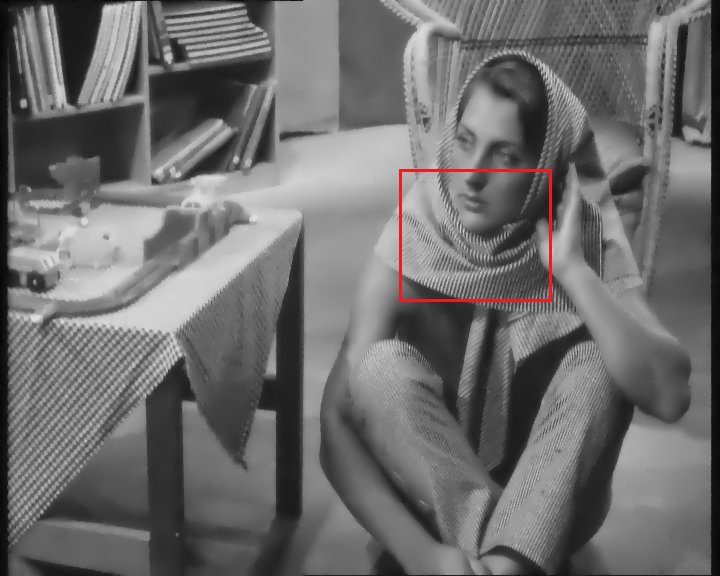}
      \\[\smallskipamount]
      \includegraphics[width=\textwidth]{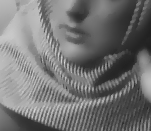}
      \\[-1em]
    \end{minipage}}
  \subfigure[NLTV]{%
    \begin{minipage}{0.22\linewidth}
      \includegraphics[width=\textwidth]{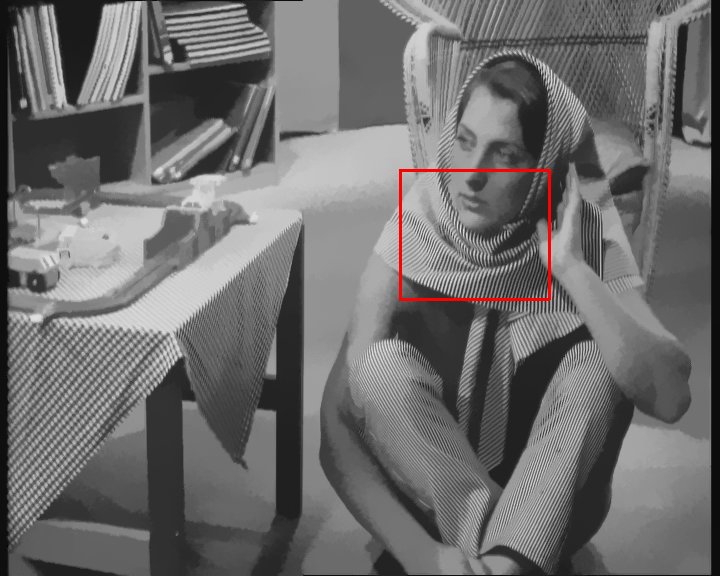}
      \\[\smallskipamount]
      \includegraphics[width=\textwidth]{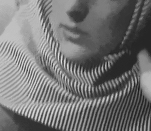}
      \\[-1em]
    \end{minipage}}

  \subfigure[ICTGV]{%
	\begin{minipage}{0.22\linewidth}
		\includegraphics[width=\textwidth]{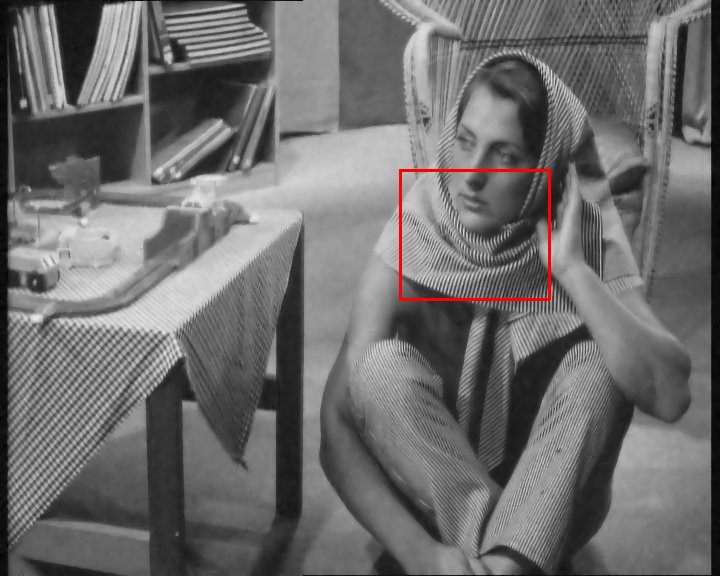}
		\\[\smallskipamount]
		\includegraphics[width=\textwidth]{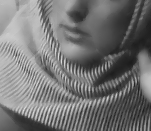}
		\\[-1em]
\end{minipage}}
  \subfigure[Fra+LDCT]{%
    \begin{minipage}{0.22\linewidth}
      \includegraphics[width=\textwidth]{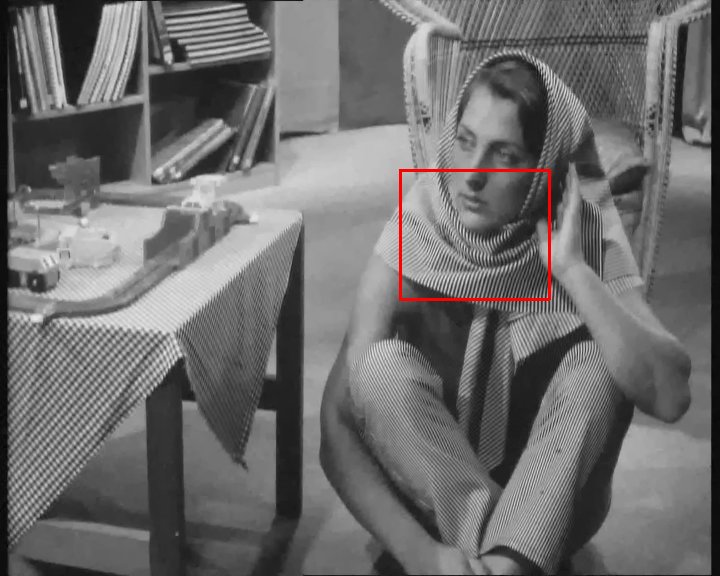}
      \\[\smallskipamount]
      \includegraphics[width=\textwidth]{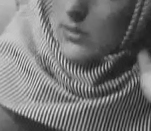}
      \\[-1em]
    \end{minipage}}
  \subfigure[BM3D]{%
    \begin{minipage}{0.22\linewidth}
      \includegraphics[width=\textwidth]{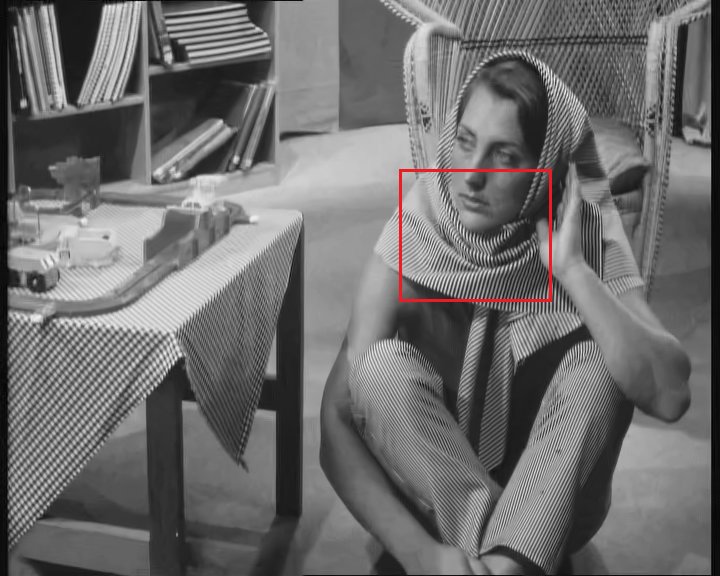}
      \\[\smallskipamount]
      \includegraphics[width=\textwidth]{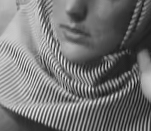}
      \\[-1em]
    \end{minipage}}
  \subfigure[proposed model]{%
    \begin{minipage}{0.22\linewidth}
      \includegraphics[width=\textwidth]{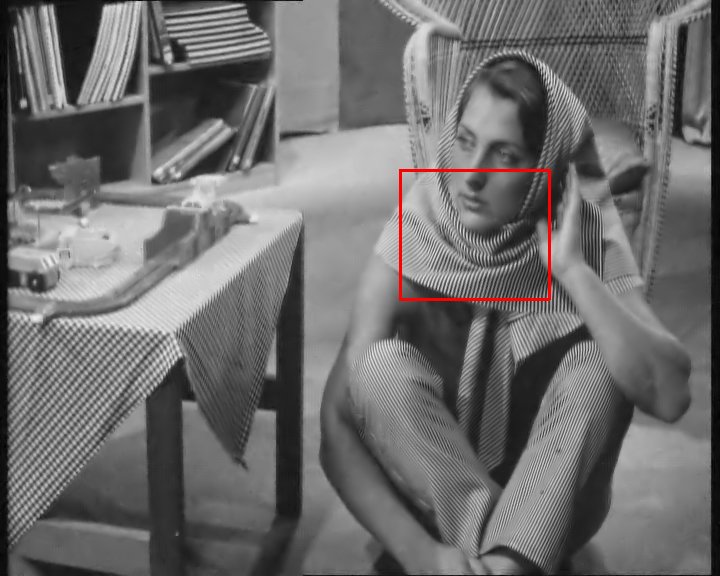}
      \\[\smallskipamount]
      \includegraphics[width=\textwidth]{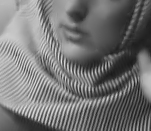}
      \\[-1em]
    \end{minipage}}
  \caption{ Denoising results for ``barbara'' image ($720 \times 576$
    pixels) with noise level $\sigma=0.05$. Parameter choice for
    $\ICTGV^{\osci}$:
    $\alpha_1=0.05, \beta_1=\alpha_1, \gamma_1=0,
    \alpha_i=0.9\alpha_1, \beta_i=\alpha_i, \gamma_i=0.1\alpha_i,
    i=2,\ldots,17.$}
  \label{imdenoisebar_vision}
\end{figure}

\begin{figure}
  \center{} 
    \subfigure[ground truth]{%
      \begin{minipage}{0.22\linewidth}
        \includegraphics[width=\textwidth]{images_denoise_zebra_zebra_red.jpg}
        \\[\smallskipamount]
         \includegraphics[width=\textwidth]{images_denoise_zebra_zebra_closeup.jpg}
        \\[-1em]
      \end{minipage}}
  \subfigure[noisy image ($\sigma=0.05$)]{%
    \begin{minipage}{0.22\linewidth}
      \includegraphics[width=\textwidth]{images_denoise_zebra_zebra_noisy_005_red.jpg}
      \\[\smallskipamount]
      \includegraphics[width=\textwidth]{images_denoise_zebra_zebra_noisy_005_closeup.jpg}
      \\[-1em]
    \end{minipage}}
  \subfigure[TGV]{%
    \begin{minipage}{0.22\linewidth}
      \includegraphics[width=\textwidth]{images_denoise_zebra_zebra_tgv_noise_005_red.jpg}
      \\[\smallskipamount]
      \includegraphics[width=\textwidth]{images_denoise_zebra_zebra_tgv_noise_005_closeup.jpg}
      \\[-1em]
    \end{minipage}}
  \subfigure[NLTV]{%
    \begin{minipage}{0.22\linewidth}
      \includegraphics[width=\textwidth]{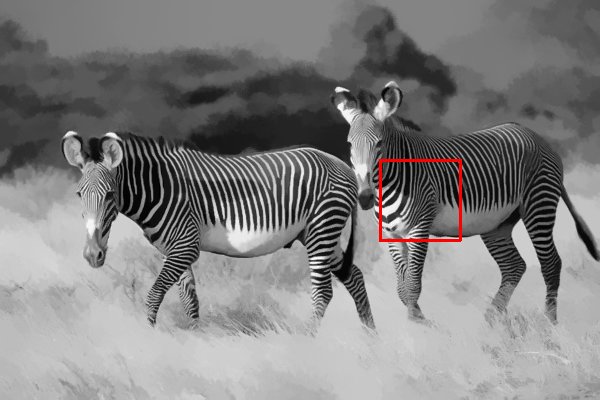}
      \\[\smallskipamount]
      \includegraphics[width=\textwidth]{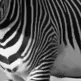}
      \\[-1em]
    \end{minipage}}

  \subfigure[ICTGV]{%
	\begin{minipage}{0.22\linewidth}
		\includegraphics[width=\textwidth]{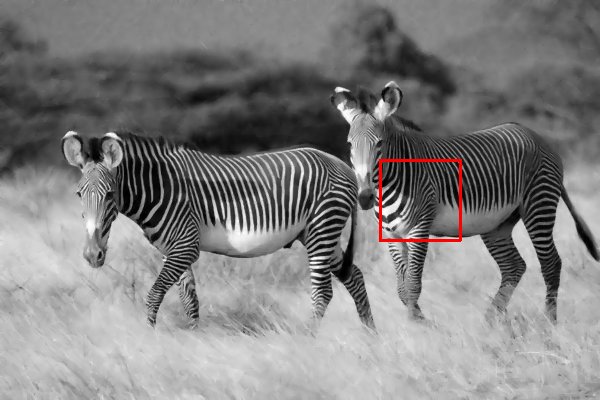}
		\\[\smallskipamount]
		\includegraphics[width=\textwidth]{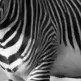}
		\\[-1em]
\end{minipage}}
  \subfigure[Fra+LDCT]{%
    \begin{minipage}{0.22\linewidth}
      \includegraphics[width=\textwidth]{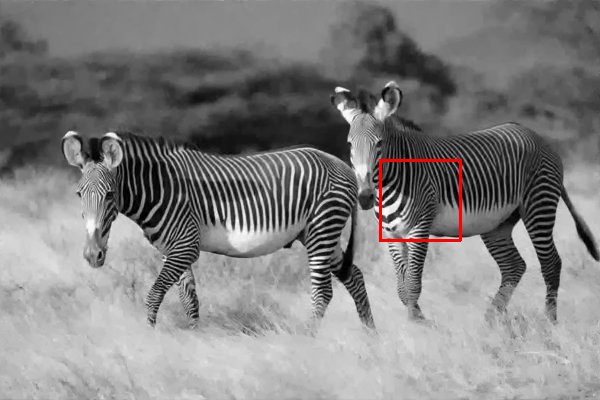}
      \\[\smallskipamount]
      \includegraphics[width=\textwidth]{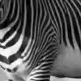}
      \\[-1em]
    \end{minipage}}
  \subfigure[BM3D]{%
    \begin{minipage}{0.22\linewidth}
      \includegraphics[width=\textwidth]{images_denoise_zebra_zebra_bm3d_005_red.jpg}
      \\[\smallskipamount]
      \includegraphics[width=\textwidth]{images_denoise_zebra_zebra_bm3d_005_closeup.jpg}
      \\[-1em]
    \end{minipage}}
  \subfigure[proposed model]{%
    \begin{minipage}{0.22\linewidth}
      \includegraphics[width=\textwidth]{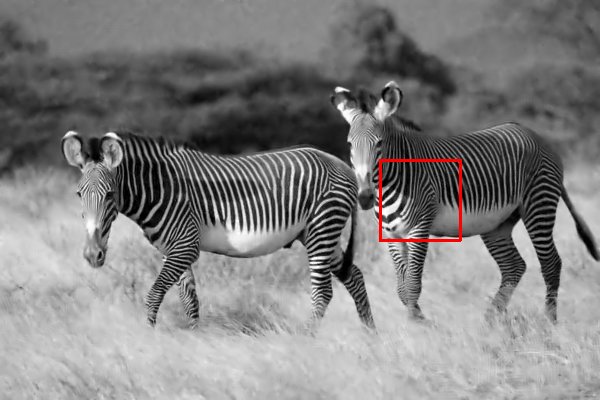}
      \\[\smallskipamount]
      \includegraphics[width=\textwidth]{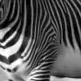}
      \\[-1em]
    \end{minipage}}
  \caption{ Denoising results for ``zebra'' image ($640 \times 400$
    pixels) with noise level $\sigma=0.05$. Parameter choice for
    $\ICTGV^{\osci}$:
    $\alpha_1=0.05, \beta_1=0.6\alpha_1, \gamma_1=0,
    \alpha_i=0.9\alpha_1, \beta_i=\alpha_i, \gamma_i=0.1\alpha_i,
    i=2,\ldots,17.$}
  \label{imdenoisezebra_vision}
\end{figure}

\begin{figure}
  \center{} 
    \subfigure[ground truth]{%
      \begin{minipage}{0.22\linewidth}
        \includegraphics[width=\textwidth]{images_denoise_bird_bird_red.jpg}
         \\[\smallskipamount]
         \includegraphics[width=\textwidth]{images_denoise_bird_bird_closeup.jpg}
        \\[-1em]
      \end{minipage}}
  \subfigure[noisy image ($\sigma=0.05$)]{%
    \begin{minipage}{0.22\linewidth}
      \includegraphics[width=\textwidth]{images_denoise_bird_bird_noisy_005_red.jpg}
      \\[\smallskipamount]
      \includegraphics[width=\textwidth]{images_denoise_bird_bird_noisy_005_closeup.jpg}
      \\[-1em]
    \end{minipage}}
  \subfigure[TGV]{%
    \begin{minipage}{0.22\linewidth}
      \includegraphics[width=\textwidth]{images_denoise_bird_bird_tgv_noise_005_red.jpg}
      \\[\smallskipamount]
      \includegraphics[width=\textwidth]{images_denoise_bird_bird_tgv_noise_005_closeup.jpg}
      \\[-1em]
    \end{minipage}}
  \subfigure[NLTV]{%
    \begin{minipage}{0.22\linewidth}
      \includegraphics[width=\textwidth]{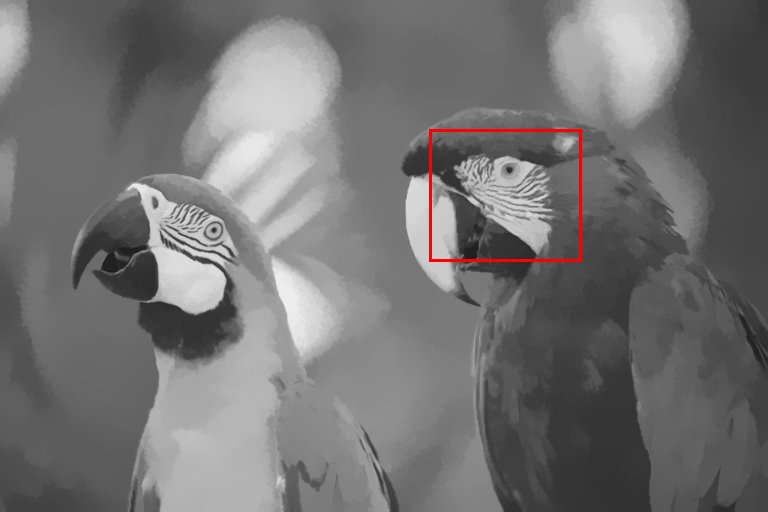}
      \\[\smallskipamount]
      \includegraphics[width=\textwidth]{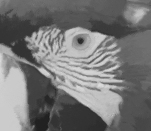}
      \\[-1em]
    \end{minipage}}

  \subfigure[ICTGV]{%
	\begin{minipage}{0.22\linewidth}
		\includegraphics[width=\textwidth]{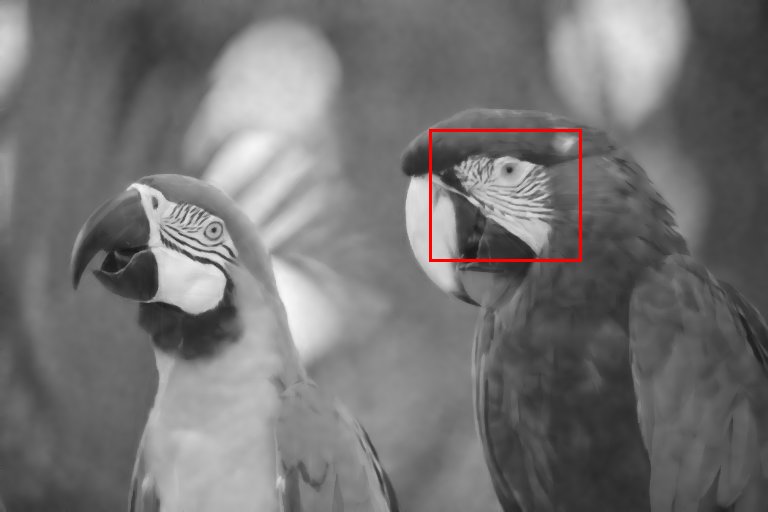}
		\\[\smallskipamount]
		\includegraphics[width=\textwidth]{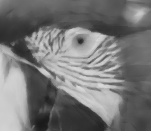}
		\\[-1em]
\end{minipage}}
  \subfigure[Fra+LDCT]{%
    \begin{minipage}{0.22\linewidth}
      \includegraphics[width=\textwidth]{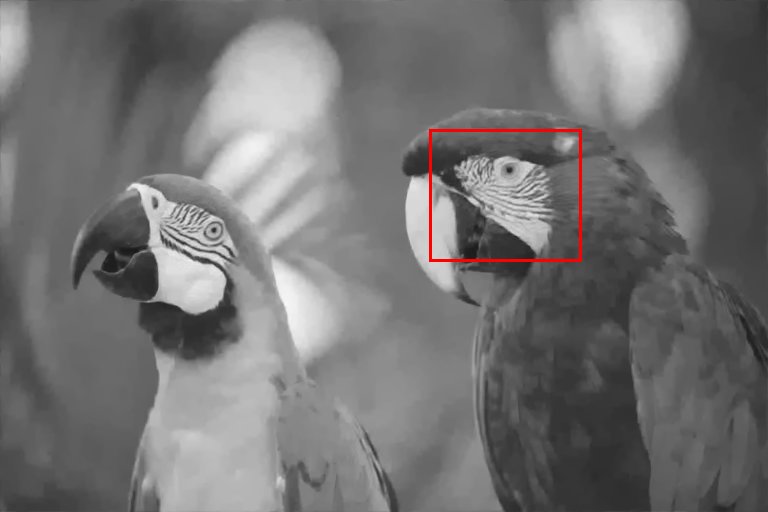}
      \\[\smallskipamount]
      \includegraphics[width=\textwidth]{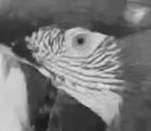}
      \\[-1em]
    \end{minipage}}
  \subfigure[BM3D]{%
    \begin{minipage}{0.22\linewidth}
      \includegraphics[width=\textwidth]{images_denoise_bird_bird_bm3d_005_red.jpg}
      \\[\smallskipamount]
      \includegraphics[width=\textwidth]{images_denoise_bird_bird_bm3d_005_closeup.jpg}
      \\[-1em]
    \end{minipage}}
  \subfigure[proposed model]{%
    \begin{minipage}{0.22\linewidth}
      \includegraphics[width=\textwidth]{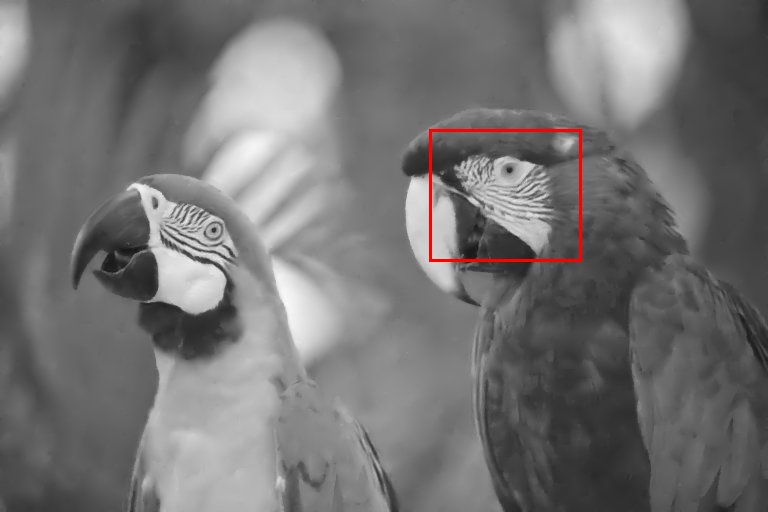}
      \\[\smallskipamount]
      \includegraphics[width=\textwidth]{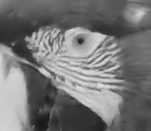}
      \\[-1em]
    \end{minipage}}
  \caption{Denoising results for ``parrots'' image ($768 \times 512$
    pixels) with noise level $\sigma=0.05$. Parameter choice for
    $\ICTGV^{\osci}$:
    $\alpha_1=0.05, \beta_1=1.5\alpha_1, \gamma_1=0,
    \alpha_i=0.75\alpha_1, \beta_i=1.5\alpha_i, \gamma_i=0.2\alpha_i,
    i=2,\ldots,9.$}
  \label{imdenoiseparrot_vision}
\end{figure}

\clearpage
\subsubsection{Noise level $\sigma=0.1$}

See Figures~\ref{imdenoisebar2_vision},~\ref{imdenoisezebra2_vision}
and~\ref{imdenoiseparrot2_vision}, and
Table~\ref{table:denoising01_vis}.

\begin{table}[H]
  \center{}
  \begin{tabular}{|c|c|c|c|c|c|c|c|}
    \hline
    \hline
    &image&TGV &NLTV&ICTGV&Fra+LDCT& BM3D& proposed model \\
    \hline
    \multirow{3}{*}{PSNR}
           & barbara & 25.34 & 27.15 &26.17 &27.66 & \textbf{31.06} & 27.89\\
    \cline{2-8}& zebra & 25.24 & 26.93 & 26.83& 26.44 & \textbf{28.80} & 27.00 \\
    \cline{2-8}& parrots & 32.51 & 30.13 & 32.35& 32.65 & \textbf{34.08} & 32.70 \\
    \hline
    \hline
    \multirow{3}{*}{SSIM}
           & barbara & 0.7297 & 0.7814 &0.7840 & 0.8201 &\textbf{0.8920} & 0.8120\\
    \cline{2-8}&zebra & 0.7472 & 0.7540 &0.7897 & 0.7795 & \textbf{0.8237} & 0.7948 \\
    \cline{2-8}& parrots & 0.8887 & 0.8530 &0.8889& 0.8928 & \textbf{0.9025} & 0.8936 \\
    \hline
  \end{tabular}
  \caption{\label{table:denoising01_vis} Comparison of denoising performance of visually-optimized results with noise level $\sigma=0.1$ in terms of PSNR and SSIM.}
\end{table}

\begin{figure}
  \center{} 
  \subfigure[noisy image ($\sigma=0.1$)]{%
    \begin{minipage}{0.32\linewidth}
      \includegraphics[width=\textwidth]{images_denoise_barbara_noisy_01.jpg}
      \\[-1em]
    \end{minipage}}

  \subfigure[TGV]{%
    \begin{minipage}{0.32\linewidth}
      \includegraphics[width=\textwidth]{images_denoise_barbara_tgv_noise_01.jpg}
      \\[-1em]
    \end{minipage}}
  \subfigure[NLTV]{%
    \begin{minipage}{0.32\linewidth}
      \includegraphics[width=\textwidth]{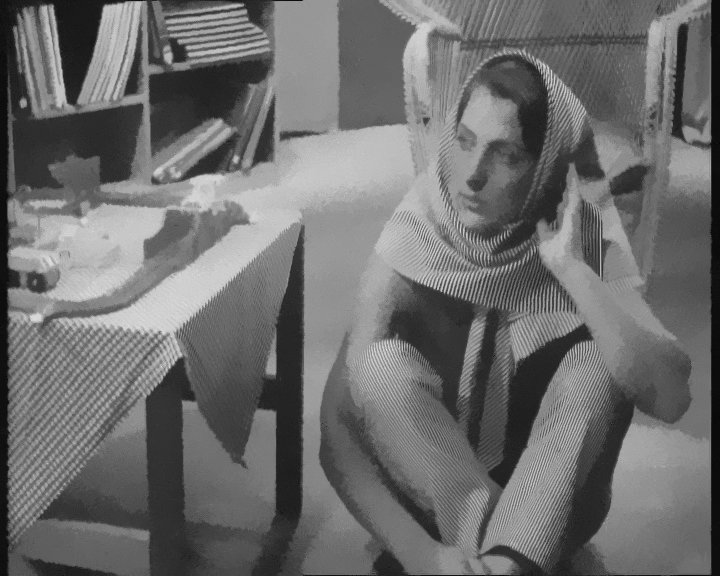}
      \\[-1em]
    \end{minipage}}
  \subfigure[ICTGV]{%
	\begin{minipage}{0.32\linewidth}
		\includegraphics[width=\textwidth]{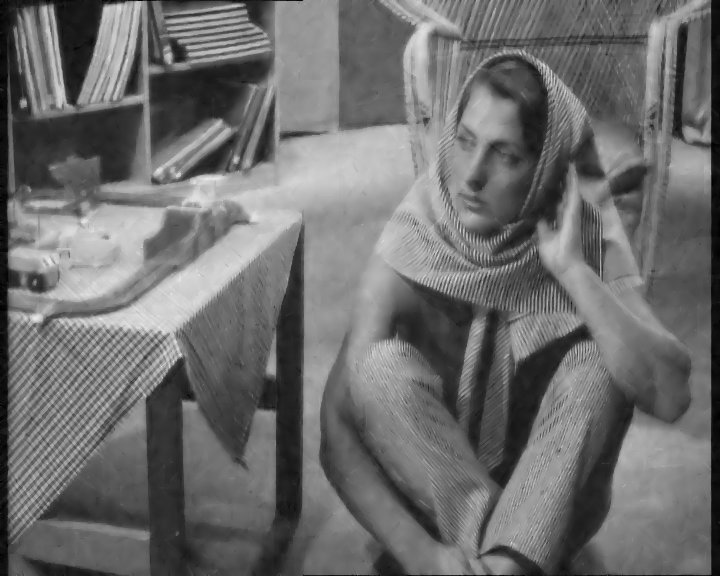}
		\\[-1em]
\end{minipage}}
  
  \subfigure[Fra+LDCT]{%
    \begin{minipage}{0.32\linewidth}
      \includegraphics[width=\textwidth]{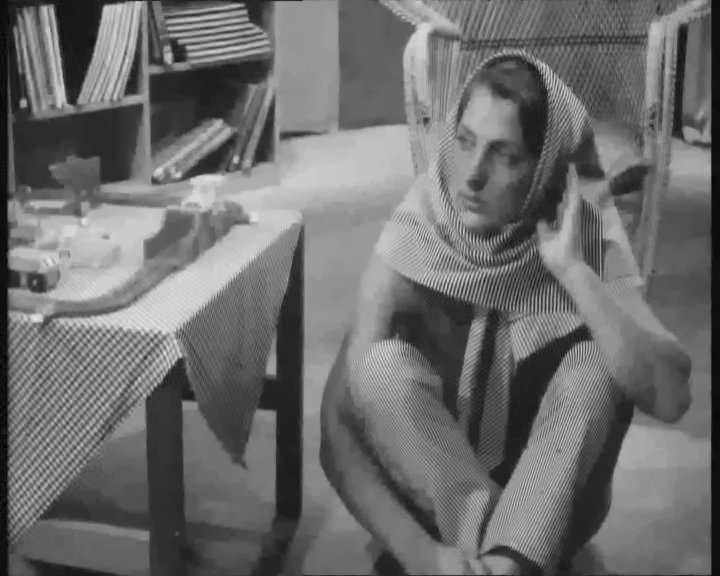}
      \\[-1em]
    \end{minipage}}
  \subfigure[BM3D]{%
    \begin{minipage}{0.32\linewidth}
      \includegraphics[width=\textwidth]{images_denoise_barbara_barbara_bm3d_01.jpg}
      \\[-1em]
    \end{minipage}}
  \subfigure[proposed model]{%
    \begin{minipage}{0.32\linewidth}
      \includegraphics[width=\textwidth]{images_denoise_barbara_barbara_osci_tgv_01_vision.jpg}
      \\[-1em]
    \end{minipage}}
  \caption{Denoising results for ``barbara'' image ($720 \times 576$
    pixels) with noise level $\sigma=0.1$. Parameter choice for
    $\ICTGV^{\osci}$:
    $\alpha_1=0.1, \beta_1=\alpha_1, \gamma_1=0, \alpha_i=0.8\alpha_1,
    \beta_i=\alpha_i, \gamma_i=0.2\alpha_i, i=2,\ldots,17.$}
  \label{imdenoisebar2_vision}
\end{figure}

\begin{figure}
  \center{} 
  \subfigure[noisy image ($\sigma=0.1$)]{%
    \begin{minipage}{0.32\linewidth}
      \includegraphics[width=\textwidth]{images_denoise_zebra_zebra_noisy_01.jpg}
      \\[-1em]
    \end{minipage}}

  \subfigure[TGV]{%
    \begin{minipage}{0.32\linewidth}
      \includegraphics[width=\textwidth]{images_denoise_zebra_zebra_tgv_noise_01.jpg}
      \\[-1em]
    \end{minipage}}
  \subfigure[NLTV]{%
    \begin{minipage}{0.32\linewidth}
      \includegraphics[width=\textwidth]{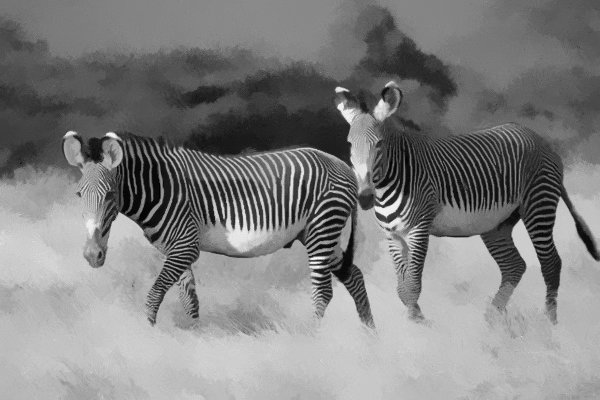}
      \\[-1em]
    \end{minipage}}
  \subfigure[ICTGV]{%
	\begin{minipage}{0.32\linewidth}
		\includegraphics[width=\textwidth]{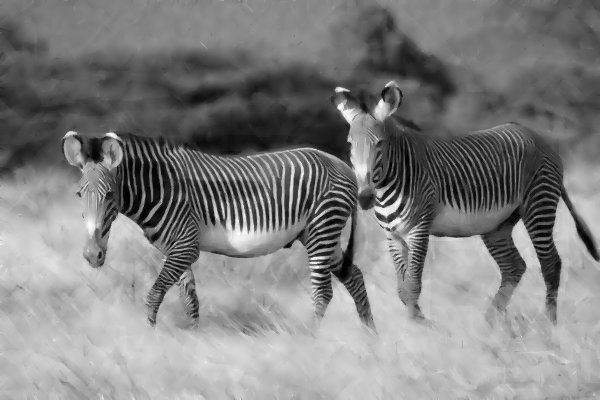}
		\\[-1em]
\end{minipage}}
  
  \subfigure[Fra+LDCT]{%
    \begin{minipage}{0.32\linewidth}
      \includegraphics[width=\textwidth]{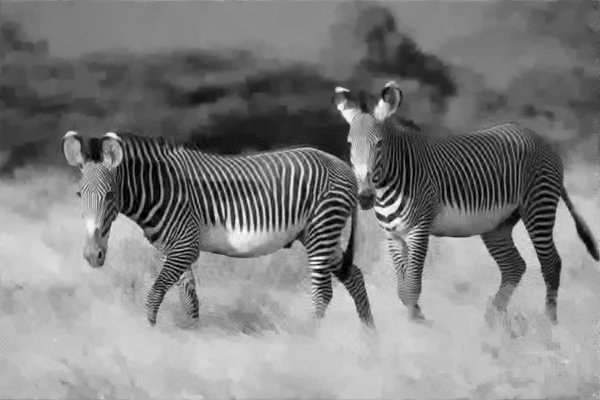}
      \\[-1em]
    \end{minipage}}
  \subfigure[BM3D]{%
    \begin{minipage}{0.32\linewidth}
      \includegraphics[width=\textwidth]{images_denoise_zebra_zebra_bm3d_01.jpg}
      \\[-1em]
    \end{minipage}}
  \subfigure[proposed model]{%
    \begin{minipage}{0.32\linewidth}
      \includegraphics[width=\textwidth]{images_denoise_zebra_zebra_osci_tgv_01_vision.jpg}
      \\[-1em]
    \end{minipage}}
  \caption{Denoising results for ``zebra'' image ($640 \times 400$ pixels)
    with noise level $\sigma=0.1$. Parameter choice for
    $\ICTGV^{\osci}$:
    $\alpha_1=0.1, \beta_1=\alpha_1, \gamma_1=0, \alpha_i=0.8\alpha_1,
    \beta_i=\alpha_i, \gamma_i=0.2\alpha_i, i=2,\ldots,17.$}
  \label{imdenoisezebra2_vision}
\end{figure}

\begin{figure}
  \center{} 
  \subfigure[noisy image ($\sigma=0.1$)]{%
    \begin{minipage}{0.32\linewidth}
      \includegraphics[width=\textwidth]{images_denoise_bird_bird_noisy_01.jpg}
      \\[-1em]
    \end{minipage}}

  \subfigure[TGV]{%
    \begin{minipage}{0.32\linewidth}
      \includegraphics[width=\textwidth]{images_denoise_bird_bird_tgv_noise_01.jpg}
      \\[-1em]
    \end{minipage}}
  \subfigure[NLTV]{%
    \begin{minipage}{0.32\linewidth}
      \includegraphics[width=\textwidth]{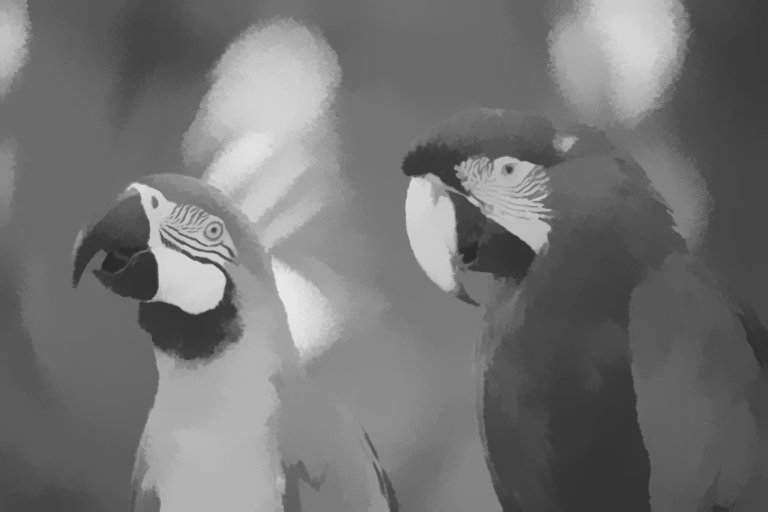}
      \\[-1em]
    \end{minipage}}
  \subfigure[ICTGV]{%
	\begin{minipage}{0.32\linewidth}
		\includegraphics[width=\textwidth]{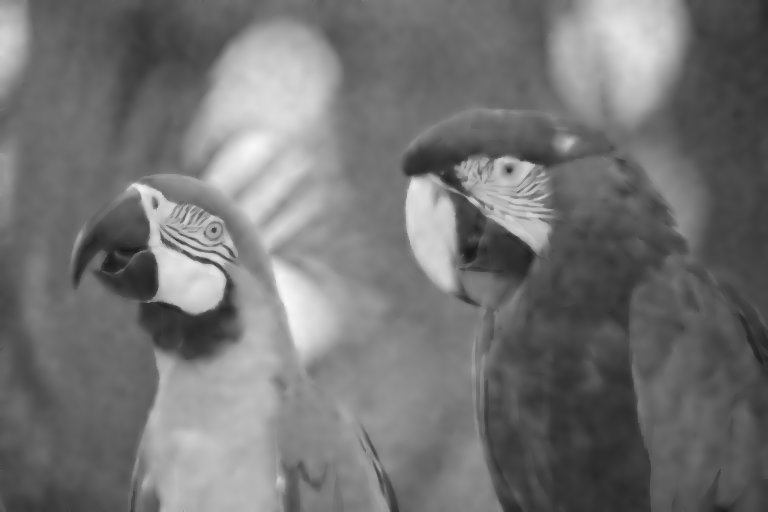}
		\\[-1em]
\end{minipage}}
  
  \subfigure[Fra+LDCT]{%
    \begin{minipage}{0.32\linewidth}
      \includegraphics[width=\textwidth]{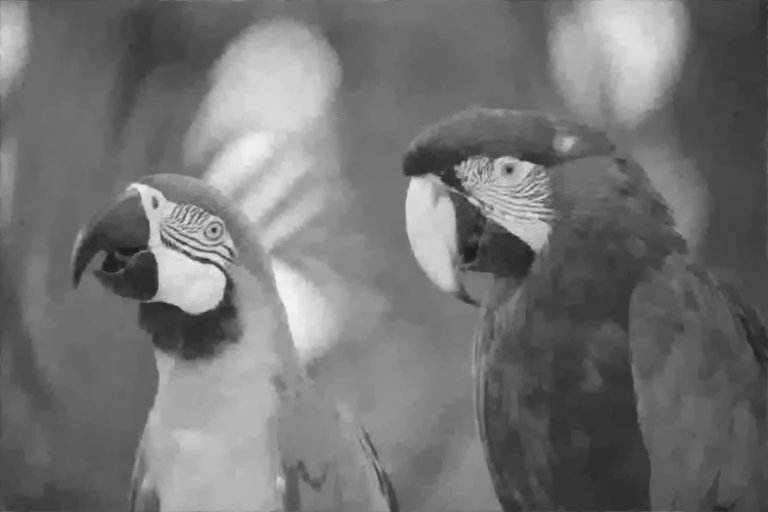}
      \\[-1em]
    \end{minipage}}
  \subfigure[BM3D]{%
    \begin{minipage}{0.32\linewidth}
      \includegraphics[width=\textwidth]{images_denoise_bird_bird_bm3d_01.jpg}
      \\[-1em]
    \end{minipage}}
  \subfigure[proposed model]{%
    \begin{minipage}{0.32\linewidth}
      \includegraphics[width=\textwidth]{images_denoise_bird_bird_osci_tgv_01_vision.jpg}
      \\[-1em]
    \end{minipage}}
  \caption{Denoising results for ``parrots'' image ($768 \times 512$
    pixels) with noise level $\sigma=0.1$. Parameter choice for
    $\ICTGV^{\osci}$:
    $\alpha_1=0.11, \beta_1=\alpha_1, \gamma_1=0,
    \alpha_i=0.75\alpha_1, \beta_i=1.5\alpha_i, \gamma_i=0.3\alpha_i,
    i=2,\ldots,9.$}
  \label{imdenoiseparrot2_vision}
\end{figure}

\clearpage
\subsection{Comparison for varying number of directions}
\begin{figure}
  \center{} 
  \subfigure[ground truth]{%
    \begin{minipage}{0.3\linewidth}
      \includegraphics[width=\textwidth]{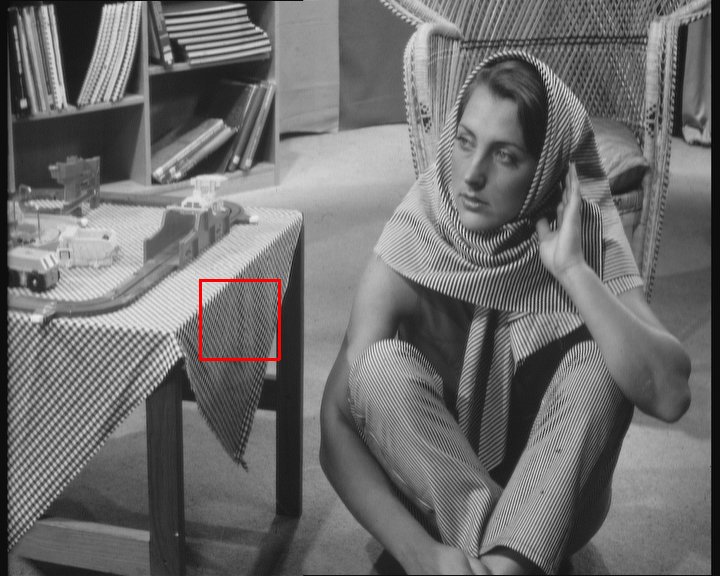}
      \\[\smallskipamount]
      \includegraphics[width=\textwidth]{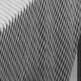}
      \\[-1em]
    \end{minipage}}
  \subfigure[4 directions]{%
    \begin{minipage}{0.3\linewidth}
      \includegraphics[width=\textwidth]{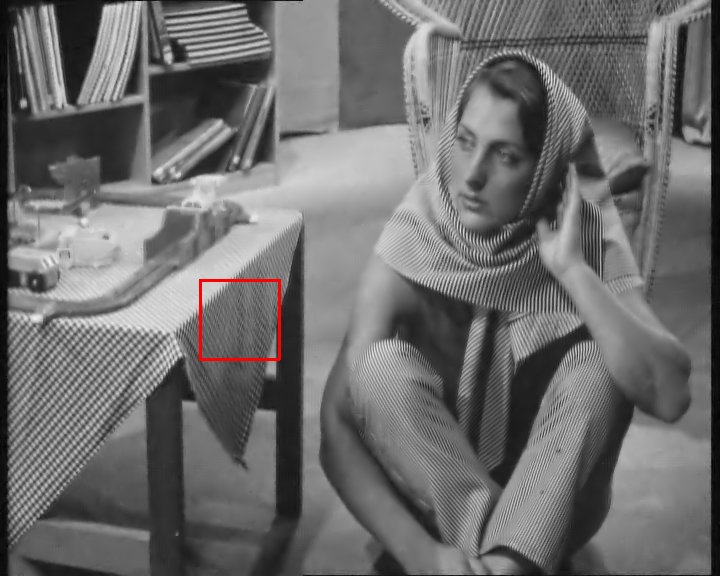}
      \\[\smallskipamount]
      \includegraphics[width=\textwidth]{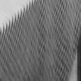}
      \\[-1em]
    \end{minipage}}
  \subfigure[6 directions]{%
    \begin{minipage}{0.3\linewidth}
      \includegraphics[width=\textwidth]{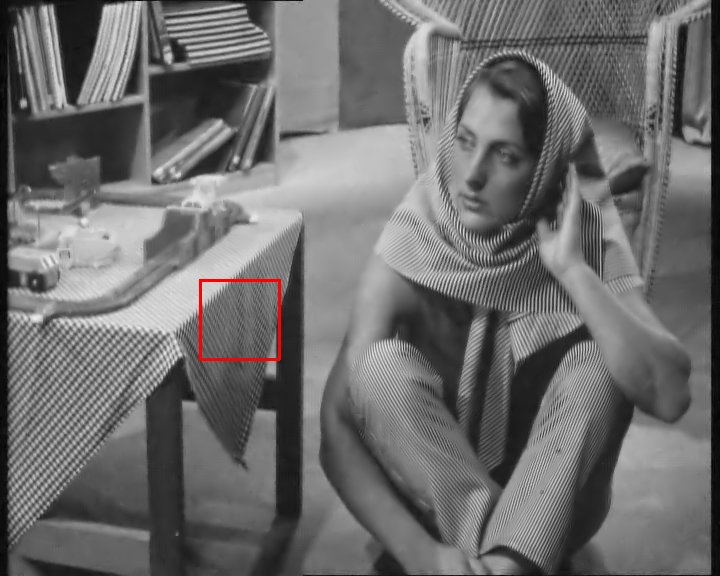}
      \\[\smallskipamount]
      \includegraphics[width=\textwidth]{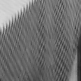}
      \\[-1em]
    \end{minipage}}
  \subfigure[8 directions]{%
    \begin{minipage}{0.3\linewidth}
      \includegraphics[width=\textwidth]{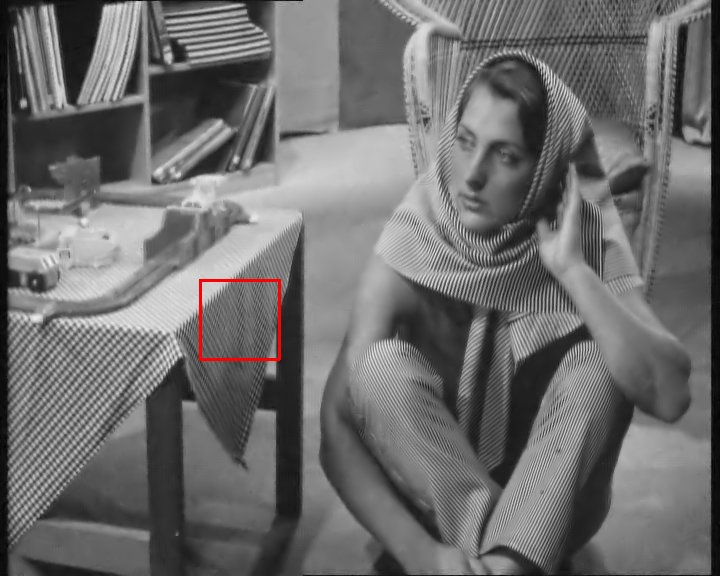}
      \\[\smallskipamount]
      \includegraphics[width=\textwidth]{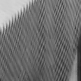}
      \\[-1em]
    \end{minipage}}
  \subfigure[10 directions]{%
    \begin{minipage}{0.3\linewidth}
      \includegraphics[width=\textwidth]{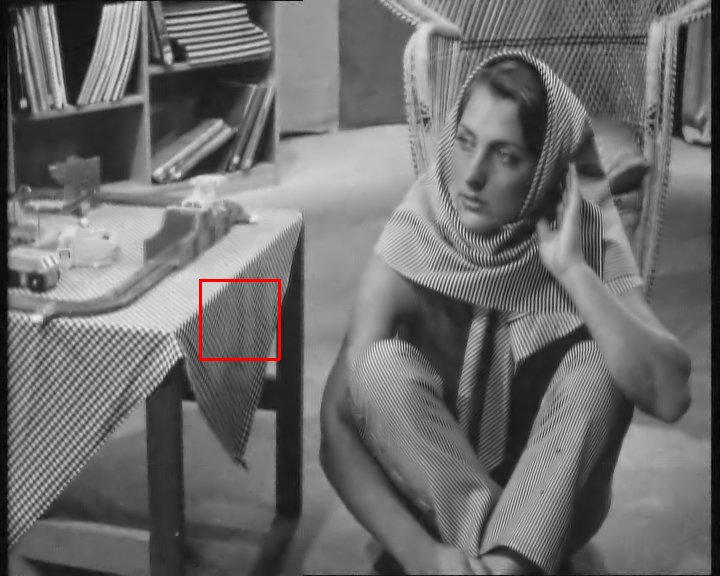}
      \\[\smallskipamount]
      \includegraphics[width=\textwidth]{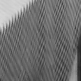}
      \\[-1em]
    \end{minipage}}
  \subfigure[12 directions]{%
    \begin{minipage}{0.3\linewidth}
      \includegraphics[width=\textwidth]{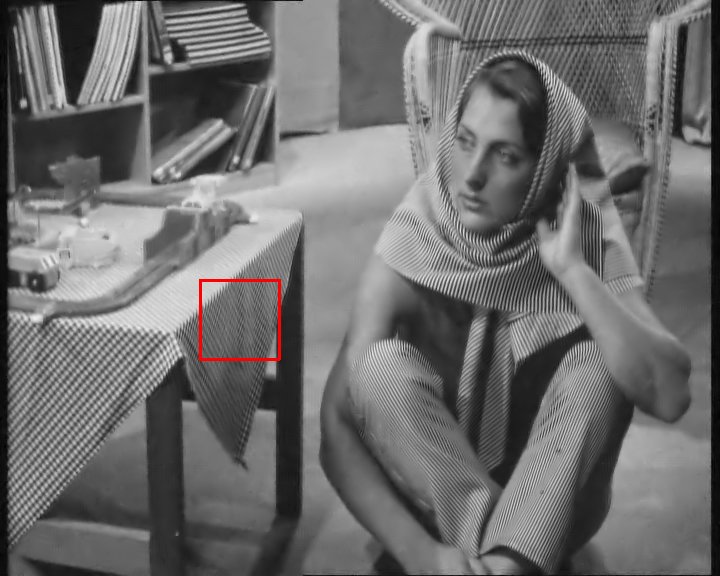}
      \\[\smallskipamount]
      \includegraphics[width=\textwidth]{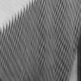}
      \\[-1em]
    \end{minipage}}
  \caption{ Denoising results for ``barbara'' image ($720 \times 576$
    pixels) with noise level $\sigma=0.05$ for varying number of
    directions. Parameter choice for $\ICTGV^{\osci}$:
    $\alpha_1=0.05, \beta_1=\alpha_1, \gamma_1=0,
    \alpha_i=0.9\alpha_1, \beta_i=\alpha_i, \gamma_i=0.1\alpha_i,
    i=2,\ldots,m$.}
  \label{imdenoisebar_vision_comp}
\end{figure}

The choice of directions for the extended $k$-direction model
($m= 2k+1$) is according to
\begin{equation*}
\left\{
\begin{aligned}
\omega_{i,1}=\sin(\tfrac{(i-2)\pi}{k}), \\
\omega_{i,2}=\cos(\tfrac{(i-2)\pi}{k}), \\
\end{aligned}\right.
\quad \text{for} \quad i=2, \ldots, k+1,
\qquad \omega_{i} = 2 \omega_{i-k} \quad \text{for} \quad i=k+2, \ldots, 2k+1,
\end{equation*}
see Figure~\ref{imdenoisebar_vision_comp} and
Table~\ref{table:denoising_directions}.

\begin{table}[H]
  \center{}
  \begin{tabular}{|c|c|c|c|c|c|}
    \hline
    \hline
    number of directions & 4 & 6 & 8 & 10 & 12 \\
    \hline
    PSNR & 
    30.64 & 31.17 & 31.44 & 31.49 &  31.53 \\
    \hline
    SSIM & 
    0.8879 & 0.8950 & 0.8988 & 0.9005 & 0.9005 \\
    \hline
  \end{tabular}
  \caption{\label{table:denoising_directions} Comparison of denoising performance for varying number of oscillation directions (``barbara image'', noise level $\sigma=0.05$) in terms of PSNR and SSIM.}
\end{table}

\section{Image inpainting}

\subsection{Comparison for natural images}

See Figures~\ref{iminpaintbarbara},~\ref{iminpaintzebra}
and~\ref{iminpaintfish}, and Table~\ref{table:inpaintingPSNR}.

\begin{table}[H]\small
  \center{}
  \begin{tabular}{|c|c|c|c|c|c|c|c|c|c|c|}
    \hline
    \hline
    & \multirow{2}{*}{image}
    & \multicolumn{3}{|c|}{TGV} &  \multicolumn{3}{|c|}{Fra+LDCT}&  \multicolumn{3}{|c|}{proposed model} \\
    \cline{3-11}&&$50\%$ & $60\%$ &$70\%$ & $50\%$  &$60\%$ & $70\%$&$50\%$ & $60\%$&$70\%$\\
    \hline 
    \multirow{3}{*}{PSNR}
    & barbara & 27.49 & 26.26 & 25.09 & 32.75 & 30.76 & 28.53 & \textbf{34.03} & \textbf{31.86} & \textbf{29.49}\\
    \cline{2-11} & zebra &26.00 & 24.23 & 22.65 & 28.50 & 26.67 & 25.08 & \textbf{30.09} & \textbf{28.11} & \textbf{26.42}\\
    \cline{2-11} & fish & 22.70 & 20.81 & 19.01 & 24.54 & 23.01 & 21.22 & \textbf{25.10} & \textbf{23.50} & \textbf{21.54} \\
    \hline
    \hline 
    \multirow{3}{*}{SSIM}    
    & barbara & 0.8942 & 0.8562 & 0.8081 & 0.9549 & 0.9320 & 0.8953 & \textbf{0.9591} & \textbf{0.9390} & \textbf{0.9078}\\
    \cline{2-11} & zebra & 0.9016 & 0.8597 & 0.8051 & 0.9213 & 0.8863 & 0.8438 & \textbf{0.9463} & \textbf{0.9194} & \textbf{0.8839}\\
    \cline{2-11} & fish & 0.8676 & 0.8059 & 0.7255 & 0.8923 & 0.8462 & 0.7902 & \textbf{0.9059} & \textbf{0.8596} & \textbf{0.7977} \\
    \hline
  \end{tabular}
  \caption{\label{table:inpaintingPSNR} Comparison of inpainting results with different rates of missing pixels in terms of PSNR and SSIM.}
\end{table}

\begin{figure}
  \center{} 
  \begin{minipage}{1\textwidth}
    \centering
    \subfigure[ground truth]{%
      \begin{minipage}{0.3\linewidth}
        \includegraphics[width=\textwidth]{images_denoise_barbara_barbara2_red.jpg}
        \\[-1em]
      \end{minipage}}
    \subfigure[closeup]{%
      \begin{minipage}{0.276\linewidth}
        \includegraphics[width=\textwidth]{images_denoise_barbara_barbara2_closeup.jpg}
        \\[-1em]
      \end{minipage}}
  \end{minipage}
  
  \subfigure[$50\%$ missing]{%
    \begin{minipage}{0.23\linewidth}
      \includegraphics[width=\textwidth]{images_inp_barbara_miss50_red.jpg}
      \\[\smallskipamount]
      \includegraphics[width=\textwidth]{images_inp_barbara_miss50_closeup.jpg}
      \\[-1em]
    \end{minipage}}
  \subfigure[TGV]{%
    \begin{minipage}{0.23\linewidth}
      \includegraphics[width=\textwidth]{images_inp_barbara_tgv_inpaint_5_red.jpg}
      \\[\smallskipamount]
      \includegraphics[width=\textwidth]{images_inp_barbara_tgv_inpaint_5_closeup.jpg}
      \\[-1em]
    \end{minipage}}
  \subfigure[Fra+LDCT]{%
    \begin{minipage}{0.23\linewidth}
      \includegraphics[width=\textwidth]{images_inp_barbara_framelet_5_red.jpg}
      \\[\smallskipamount]
      \includegraphics[width=\textwidth]{images_inp_barbara_framelet_5_closeup.jpg}
      \\[-1em]
    \end{minipage}}
  \subfigure[proposed model]{%
    \begin{minipage}{0.23\linewidth}
      \includegraphics[width=\textwidth]{images_inp_barbara_barbara_osci_tgv_5_red.jpg}
      \\[\smallskipamount]
      \includegraphics[width=\textwidth]{images_inp_barbara_barbara_osci_tgv_5_closeup.jpg}
      \\[-1em]
    \end{minipage}}

  \vspace*{-0.5em}
  \subfigure[$60\%$ missing]{%
    \begin{minipage}{0.23\linewidth}
      \includegraphics[width=\textwidth]{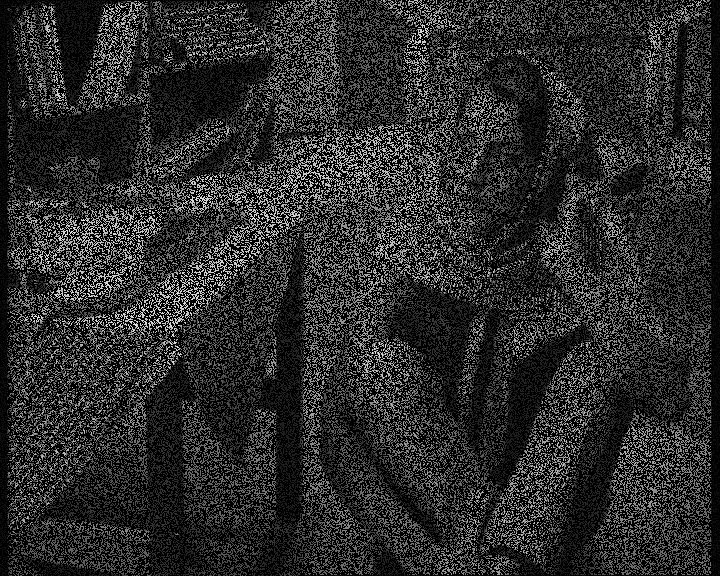}
      \\[-1em]
    \end{minipage}}
  \subfigure[TGV]{%
    \begin{minipage}{0.23\linewidth}
      \includegraphics[width=\textwidth]{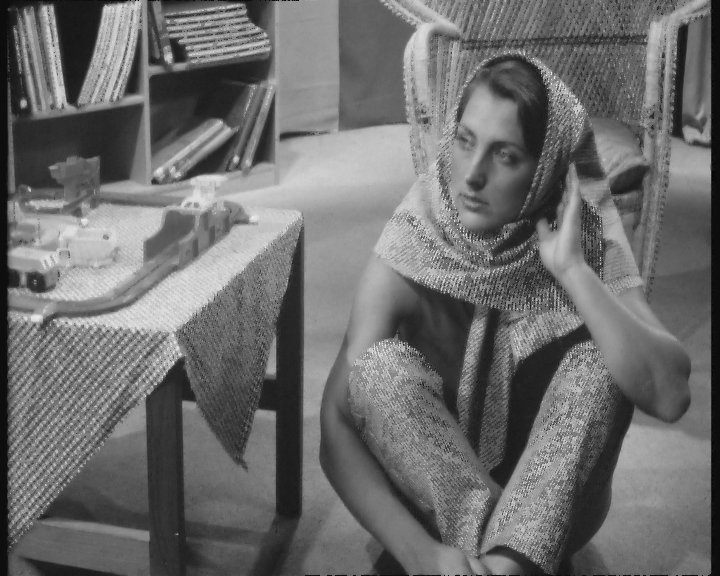}
      \\[-1em]
    \end{minipage}}
  \subfigure[Fra+LDCT]{%
    \begin{minipage}{0.23\linewidth}
      \includegraphics[width=\textwidth]{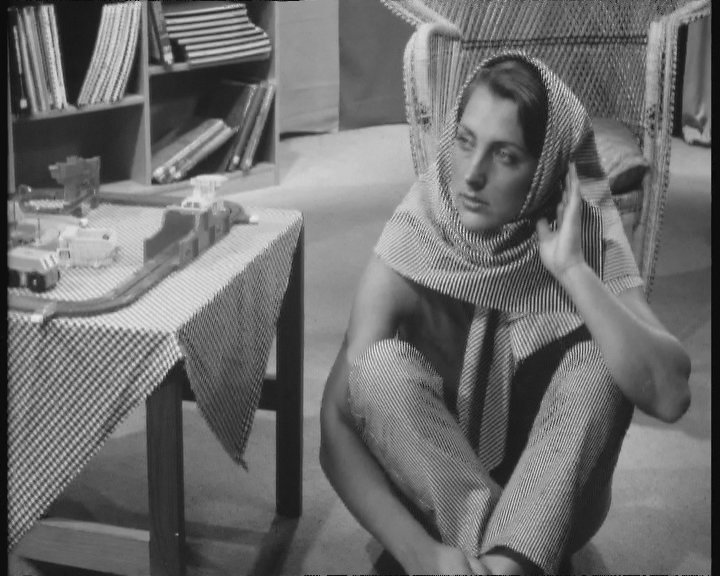}
      \\[-1em]
    \end{minipage}}
  \subfigure[proposed model]{%
    \begin{minipage}{0.23\linewidth}
      \includegraphics[width=\textwidth]{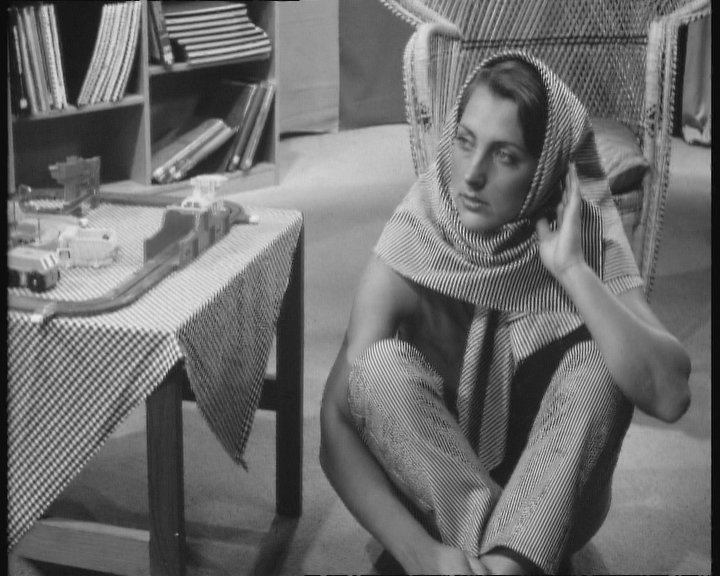}
      \\[-1em]
    \end{minipage}}

  \vspace*{-0.5em}
  \subfigure[$70\%$ missing]{%
    \begin{minipage}{0.23\linewidth}
      \includegraphics[width=\textwidth]{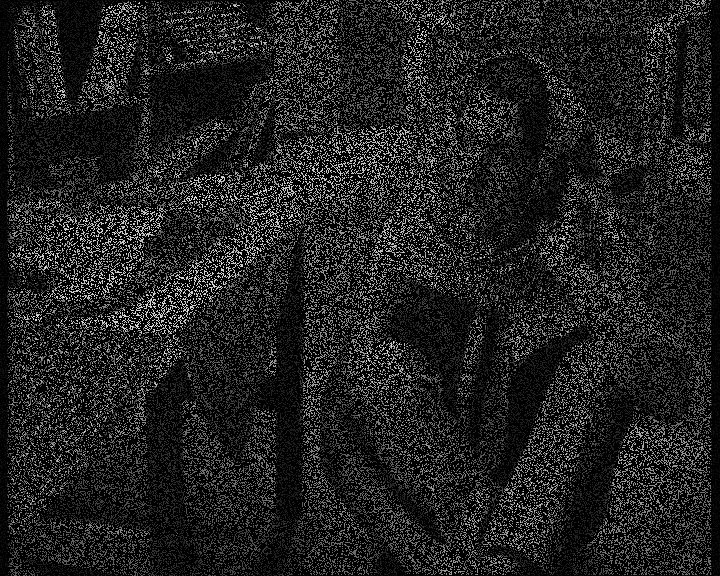}
      \\[-1em]
    \end{minipage}}
  \subfigure[TGV]{%
    \begin{minipage}{0.23\linewidth}
      \includegraphics[width=\textwidth]{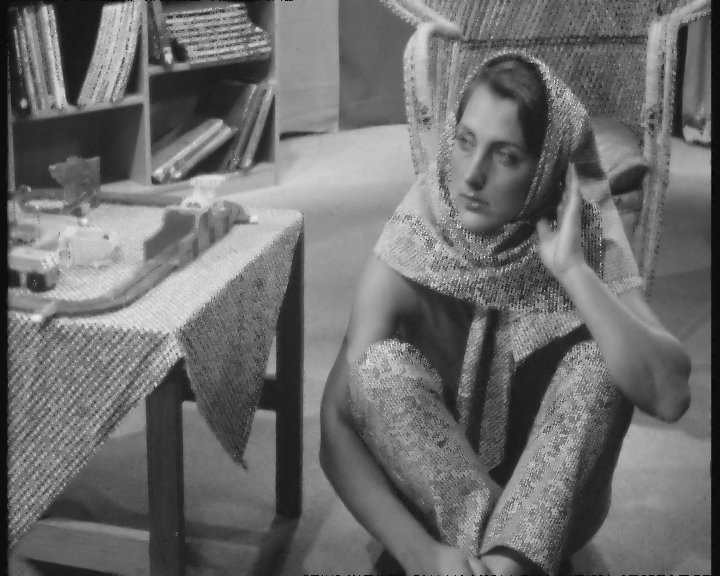}
      \\[-1em]
    \end{minipage}}
  \subfigure[Fra+LDCT]{%
    \begin{minipage}{0.23\linewidth}
      \includegraphics[width=\textwidth]{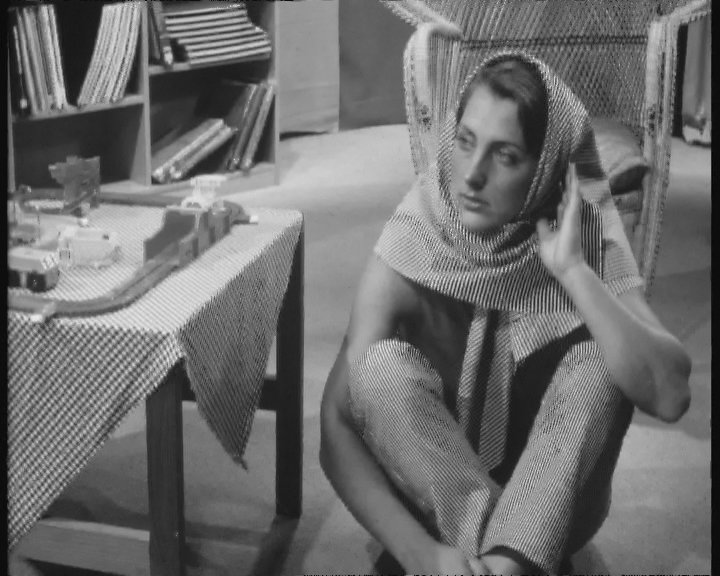}
      \\[-1em]
    \end{minipage}}
  \subfigure[proposed model]{%
    \begin{minipage}{0.23\linewidth}
      \includegraphics[width=\textwidth]{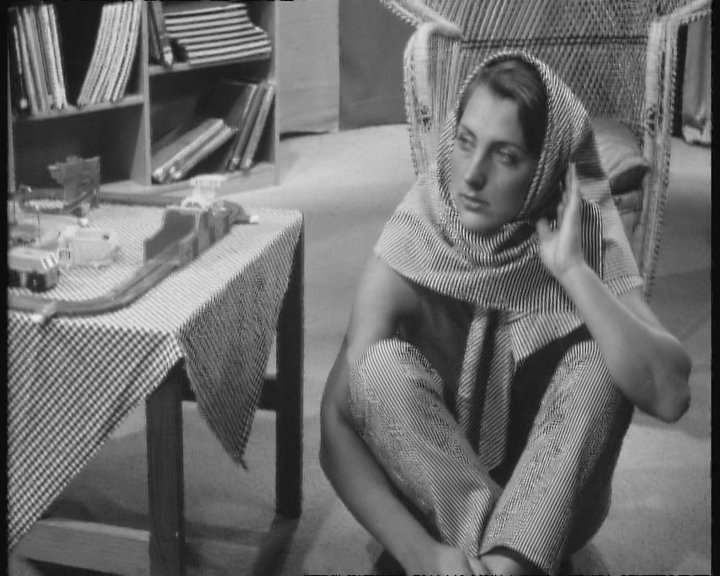}
      \\[-1em]
    \end{minipage}}
  \caption{Inpainting results for ``barbara'' image ($720 \times 576$
    pixels). Parameter choice for $\ICTGV^{\osci}$:
    $\alpha_1=0.03, \beta_1=0.5\alpha_1, \gamma_1=0,
    \alpha_i=0.9\alpha_1, \beta_i=0.4\alpha_i, \gamma_i=0.02\alpha_i,
    i=2,\ldots,17$.}
  \label{iminpaintbarbara}
\end{figure}

\begin{figure}
  \center{} 
  \begin{minipage}{1\textwidth}
    \centering
    \subfigure[ground truth]{%
      \begin{minipage}{0.3\linewidth}
        \includegraphics[width=\textwidth]{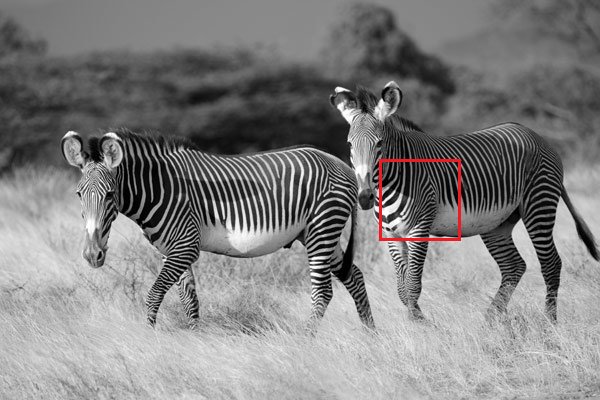}
        \\[-1em]
      \end{minipage}}
    \subfigure[closeup]{%
      \begin{minipage}{0.2\linewidth}
        \includegraphics[width=\textwidth]{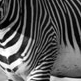}
        \\[-1em]
      \end{minipage}}
  \end{minipage}
  
  \subfigure[$50\%$ missing]{%
    \begin{minipage}{0.23\linewidth}
      \includegraphics[width=\textwidth]{images_inp_zebra_zebra_miss50_red.jpg}
      \\[\smallskipamount]
      \includegraphics[width=\textwidth]{images_inp_zebra_zebra_miss50_closeup.jpg}
      \\[-1em]
    \end{minipage}}
  \subfigure[TGV]{%
    \begin{minipage}{0.23\linewidth}
      \includegraphics[width=\textwidth]{images_inp_zebra_zebra_5_tgv_red.jpg}
      \\[\smallskipamount]
      \includegraphics[width=\textwidth]{images_inp_zebra_zebra_5_tgv_closeup.jpg}
      \\[-1em]
    \end{minipage}}
  \subfigure[Fra+LDCT]{%
    \begin{minipage}{0.23\linewidth}
      \includegraphics[width=\textwidth]{images_inp_zebra_framelet_zebra_5_red.jpg}
      \\[\smallskipamount]
      \includegraphics[width=\textwidth]{images_inp_zebra_framelet_zebra_5_closeup.jpg}
      \\[-1em]
    \end{minipage}}
  \subfigure[proposed model]{%
    \begin{minipage}{0.23\linewidth}
      \includegraphics[width=\textwidth]{images_inp_zebra_zebra_osci_tgv_5_red.jpg}
      \\[\smallskipamount]
      \includegraphics[width=\textwidth]{images_inp_zebra_zebra_osci_tgv_5_closeup.jpg}
      \\[-1em]
    \end{minipage}}

  \vspace*{-0.5em}
  \subfigure[$60\%$ missing]{%
    \begin{minipage}{0.23\linewidth}
      \includegraphics[width=\textwidth]{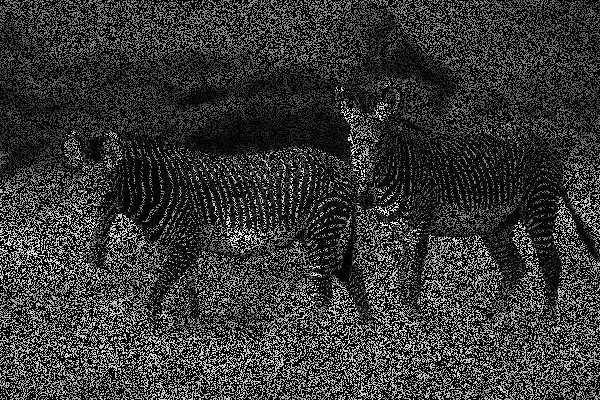}
      \\[-1em]
    \end{minipage}}
  \subfigure[TGV]{%
    \begin{minipage}{0.23\linewidth}
      \includegraphics[width=\textwidth]{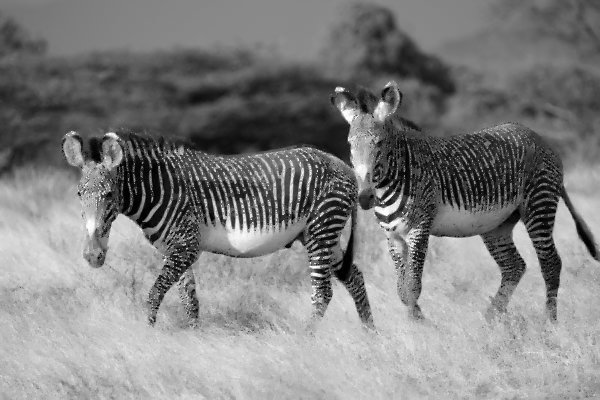}
      \\[-1em]
    \end{minipage}}
  \subfigure[Fra+LDCT]{%
    \begin{minipage}{0.23\linewidth}
      \includegraphics[width=\textwidth]{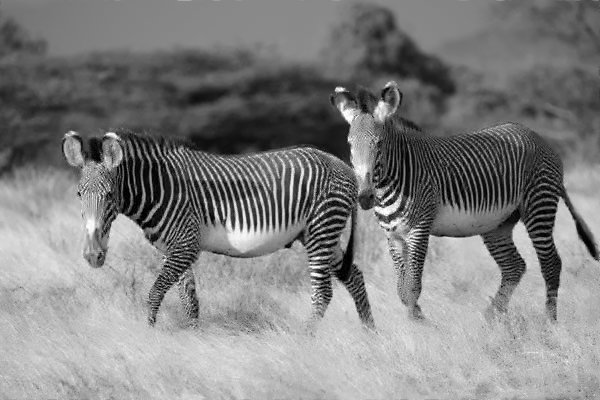}
      \\[-1em]
    \end{minipage}}
  \subfigure[proposed model]{%
    \begin{minipage}{0.23\linewidth}
      \includegraphics[width=\textwidth]{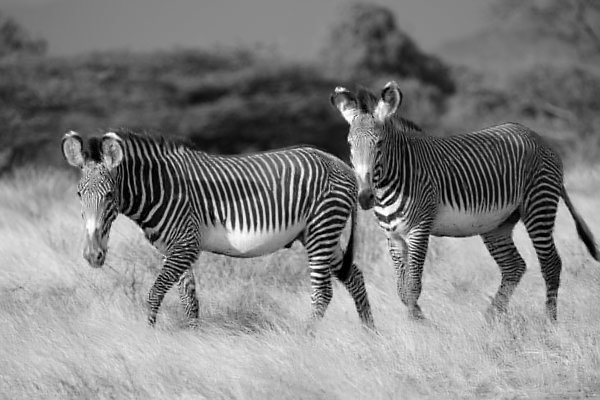}
      \\[-1em]
    \end{minipage}}

  \vspace*{-0.5em}
  \subfigure[$70\%$ missing]{%
    \begin{minipage}{0.23\linewidth}
      \includegraphics[width=\textwidth]{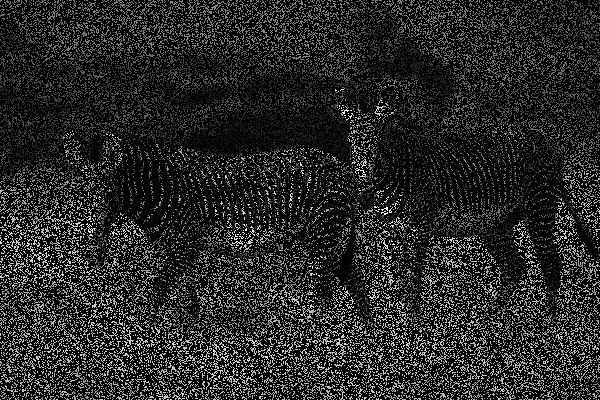}
      \\[-1em]
    \end{minipage}}
  \subfigure[TGV]{%
    \begin{minipage}{0.23\linewidth}
      \includegraphics[width=\textwidth]{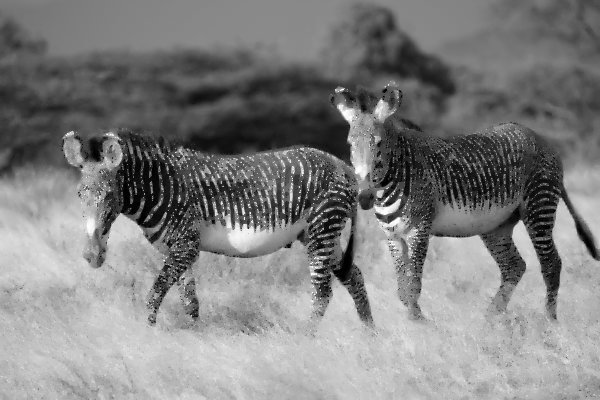}
      \\[-1em]
    \end{minipage}}
  \subfigure[Fra+LDCT]{%
    \begin{minipage}{0.23\linewidth}
      \includegraphics[width=\textwidth]{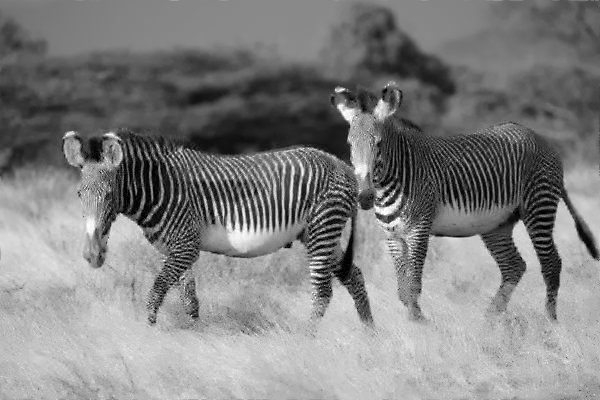}
      \\[-1em]
    \end{minipage}}
  \subfigure[proposed model]{%
    \begin{minipage}{0.23\linewidth}
      \includegraphics[width=\textwidth]{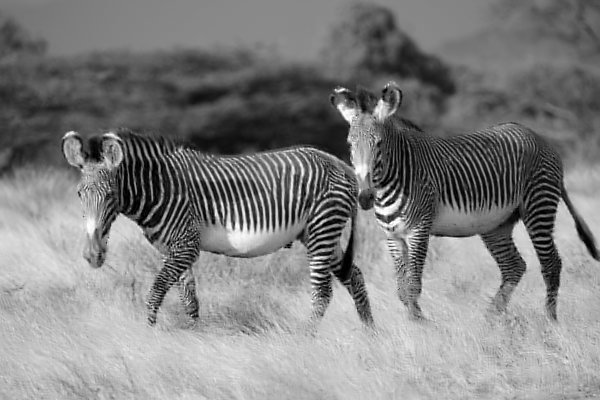}
      \\[-1em]
    \end{minipage}}
  \caption{Inpainting results for ``zebra'' image ($640 \times 400$
    pixels). Parameter choice for $\ICTGV^{\osci}$:
    $\alpha_1=0.05, \beta_1=0.6\alpha_1, \gamma_1=0,
    \alpha_i=\alpha_1, \beta_i=0.3\alpha_i, \gamma_i=0.04\alpha_i,
    i=2,\ldots,17$.}
  \label{iminpaintzebra}
\end{figure}

\begin{figure}
  \center{} 
  \begin{minipage}{1\textwidth}
    \centering
    \subfigure[ground truth]{%
      \begin{minipage}{0.3\linewidth}
        \includegraphics[width=\textwidth]{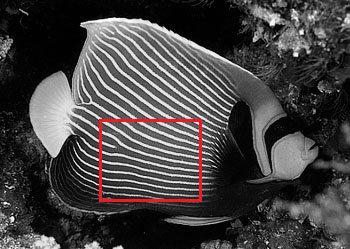}
        \\[-1em]
      \end{minipage}}
    \subfigure[closeup]{%
      \begin{minipage}{0.26667\linewidth}
        \includegraphics[width=\textwidth]{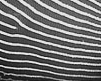}
        \\[-1em]
      \end{minipage}}
  \end{minipage}
  
  \subfigure[$50\%$ missing]{%
    \begin{minipage}{0.23\linewidth}
      \includegraphics[width=\textwidth]{images_inp_fish_fish_miss50_red.jpg}
      \\[\smallskipamount]
      \includegraphics[width=\textwidth]{images_inp_fish_fish_miss50_closeup.jpg}
      \\[-1em]
    \end{minipage}}
  \subfigure[TGV]{%
    \begin{minipage}{0.23\linewidth}
      \includegraphics[width=\textwidth]{images_inp_fish_fish_5_tgv_red.jpg}
      \\[\smallskipamount]
      \includegraphics[width=\textwidth]{images_inp_fish_fish_5_tgv_closeup.jpg}
      \\[-1em]
    \end{minipage}}
  \subfigure[Fra+LDCT]{%
    \begin{minipage}{0.23\linewidth}
      \includegraphics[width=\textwidth]{images_inp_fish_framelet_fish_5_red.jpg}
      \\[\smallskipamount]
      \includegraphics[width=\textwidth]{images_inp_fish_framelet_fish_5_closeup.jpg}
      \\[-1em]
    \end{minipage}}
  \subfigure[proposed model]{%
    \begin{minipage}{0.23\linewidth}
      \includegraphics[width=\textwidth]{images_inp_fish_fish_osci_tgv_5_red.jpg}
      \\[\smallskipamount]
      \includegraphics[width=\textwidth]{images_inp_fish_fish_osci_tgv_5_closeup.jpg}
      \\[-1em]
    \end{minipage}}

  \vspace*{-0.5em}
  \subfigure[$60\%$ missing]{%
    \begin{minipage}{0.23\linewidth}
      \includegraphics[width=\textwidth]{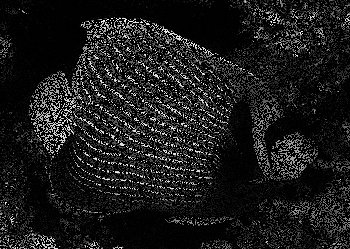}
      \\[-1em]
    \end{minipage}}
  \subfigure[TGV]{%
    \begin{minipage}{0.23\linewidth}
      \includegraphics[width=\textwidth]{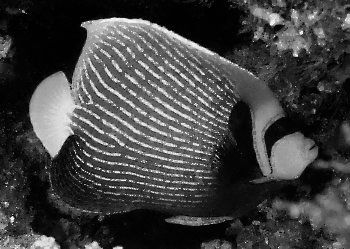}
      \\[-1em]
    \end{minipage}}
  \subfigure[Fra+LDCT]{%
    \begin{minipage}{0.23\linewidth}
      \includegraphics[width=\textwidth]{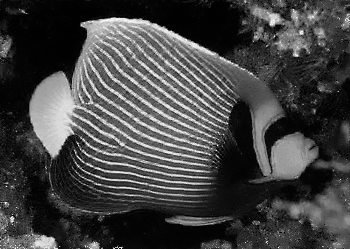}
      \\[-1em]
    \end{minipage}}
  \subfigure[proposed model]{%
    \begin{minipage}{0.23\linewidth}
      \includegraphics[width=\textwidth]{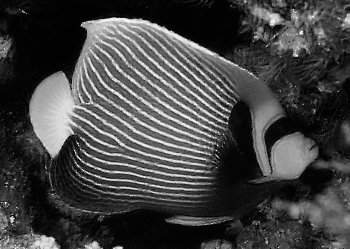}
      \\[-1em]
    \end{minipage}}

  \vspace*{-0.5em}
  \subfigure[$60\%$ missing]{%
    \begin{minipage}{0.23\linewidth}
      \includegraphics[width=\textwidth]{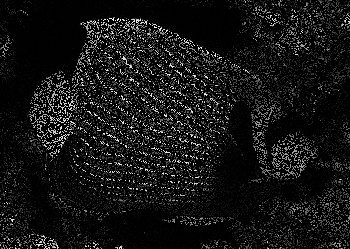}
      \\[-1em]
    \end{minipage}}
  \subfigure[TGV]{%
    \begin{minipage}{0.23\linewidth}
      \includegraphics[width=\textwidth]{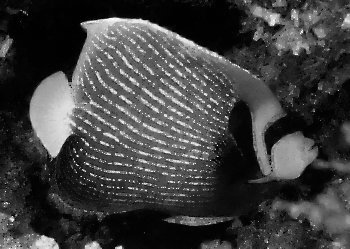}
      \\[-1em]
    \end{minipage}}
  \subfigure[Fra+LDCT]{%
    \begin{minipage}{0.23\linewidth}
      \includegraphics[width=\textwidth]{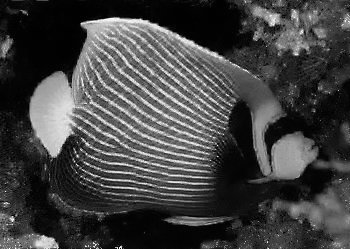}
      \\[-1em]
    \end{minipage}}
  \subfigure[proposed model]{%
    \begin{minipage}{0.23\linewidth}
      \includegraphics[width=\textwidth]{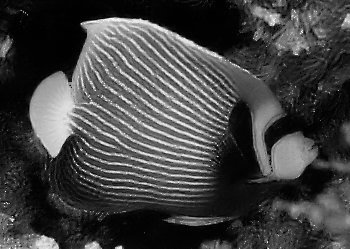}
      \\[-1em]
    \end{minipage}}
  \caption{Inpainting results for ``fish'' image ($350 \times 249$
    pixels). Parameter choice for $\ICTGV^{\osci}$:
    $\alpha_1=0.05, \beta_1=1.2\alpha_1, \gamma_1=0,
    \alpha_i=0.6\alpha_1, \beta_i=2.5\alpha_i, \gamma_i=0.03\alpha_i,
    i=2,\ldots,9$.}
  \label{iminpaintfish}
\end{figure}

\clearpage
\section{Undersampled magnetic resonance imaging}

\subsection{Comparison for radial sampling patterns}

See Figures~\ref{imreconfoot},~\ref{imreconknee}
and~\ref{imreconbrain}, and Tables~\ref{table:reconstructionPSNR}
and~\ref{table:reconstructionSSIM}.

\begin{table}
	\center{}
	\begin{tabular}{|c|c|c|c|c|c|c|c|c|}
		\hline
		\hline
		\multicolumn{2}{|c|}{sampling lines}&40 & 50 &60& 70 &80 & 90&100\\
		\hline
		\multirow{3}{*}{TGV-shearlet}
		&foot & \textbf{30.81} &\textbf{32.59}& 33.75 &34.99 & 36.51&37.34&38.59\\
		\cline{2-9}&knee & 28.01 &28.77& 29.59 &30.10 & 30.84 &31.50 &32.07\\
		\cline{2-9}&brain & 34.05 & 35.71 & 37.18 &38.44 & 39.66 &40.51 &41.43\\
		\hline
		\multirow{3}{*}{proposed model}
		&foot & 29.76 & 32.41 & \textbf{33.89} & \textbf{35.27} & \textbf{37.07} & \textbf{38.25} & \textbf{39.66}\\
		\cline{2-9}&knee & \textbf{28.40} & \textbf{29.16} & \textbf{29.98} & \textbf{30.56} & \textbf{31.23} & \textbf{31.82} & \textbf{32.36}\\
		\cline{2-9}&brain & \textbf{34.64} & \textbf{36.97} & \textbf{38.99} & \textbf{40.33} & \textbf{41.65} & \textbf{42.71} & \textbf{43.67}\\
		\hline
	\end{tabular}
	\caption{\label{table:reconstructionPSNR} Comparison of MRI reconstruction performance for different sampling rates in terms of PSNR.}
\end{table}

\begin{table}
	\center{}
	\begin{tabular}{|c|c|c|c|c|c|c|c|c|}
		\hline
		\hline
		\multicolumn{2}{|c|}{sampling lines}&40 & 50 &60& 70 &80 & 90&100\\
		\hline
		\multirow{3}{*}{TGV-shearlet}
		&foot & 0.8451 & 0.8669 & 0.8823 & 0.8956 & 0.9098 & 0.9175 & 0.9262\\
		\cline{2-9}&knee & 0.7356 & 0.7656 & 0.7950 & 0.8133 & 0.8348 & 0.8508 & 0.8661\\
		\cline{2-9}&brain & 0.8867 & 0.9059 & 0.9209 & 0.9318 & 0.9408 & 0.9466 & 0.9523\\
		\hline
		\multirow{3}{*}{proposed model}
		&foot & \textbf{0.8653} & \textbf{0.9280} & \textbf{0.9465} & \textbf{0.9567} & \textbf{0.9702} & \textbf{0.9774} & \textbf{0.9822}\\
		\cline{2-9}&knee & \textbf{0.7581} & \textbf{0.7861} & \textbf{0.8148} & \textbf{0.8344} & \textbf{0.8570} & \textbf{0.8720} & \textbf{0.8878}\\
		\cline{2-9}&brain & \textbf{0.9055} & \textbf{0.9421} & \textbf{0.9600} & \textbf{0.9681} & \textbf{0.9743} & \textbf{0.9784} & \textbf{0.9818}\\
		\hline
	\end{tabular}
	\caption{\label{table:reconstructionSSIM} Comparison of MRI reconstruction performance for different sampling rates in terms of SSIM.}
\end{table}

\begin{figure}
  \center{} 
  \begin{minipage}{1\textwidth}
    \centering
    \subfigure[ground truth]{%
      \begin{minipage}{0.49\linewidth}
        \includegraphics[width=0.49\textwidth]{images_rec_foot_foot_orig_red.jpg}%
        \hfill
        \includegraphics[width=0.49\textwidth]{images_rec_foot_foot_orig_closeup.jpg}
        \\[-1em]
      \end{minipage}}
  \end{minipage}
  
  \subfigure[TGV+Shearlet]{%
    \begin{minipage}{0.49\linewidth}
      \includegraphics[width=0.49\textwidth]{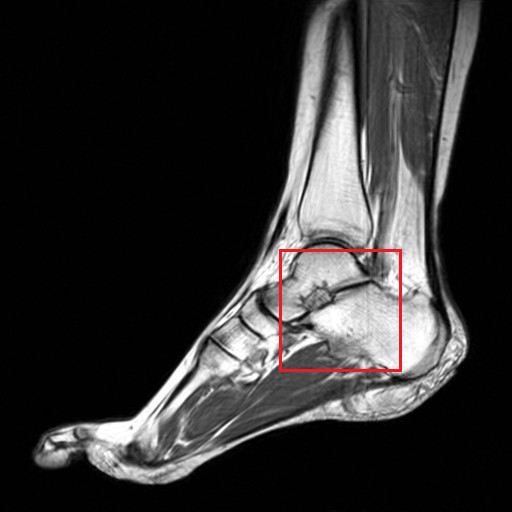}%
      \hfill
      \includegraphics[width=0.49\textwidth]{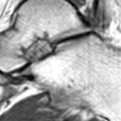}
      \\[-1em]
    \end{minipage}}\ \ 
  \subfigure[proposed model]{%
    \begin{minipage}{0.49\linewidth}
      \includegraphics[width=0.49\textwidth]{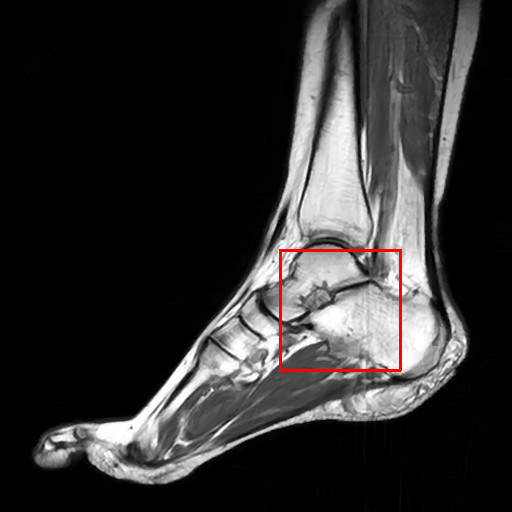}%
      \hfill
      \includegraphics[width=0.49\textwidth]{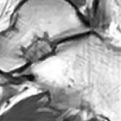}
      \\[-1em]
    \end{minipage}}
  
  \vspace*{-0.5em}
  \subfigure[TGV+Shearlet]{%
    \begin{minipage}{0.49\linewidth}
      \includegraphics[width=0.49\textwidth]{images_rec_foot_foot_tgvshearlet_7_red.jpg}%
      \hfill
      \includegraphics[width=0.49\textwidth]{images_rec_foot_foot_tgvshearlet_7_closeup.jpg}
      \\[-1em]
    \end{minipage}}\ \ 
  \subfigure[proposed model]{%
    \begin{minipage}{0.49\linewidth}
      \includegraphics[width=0.49\textwidth]{images_rec_foot_foot_osci_tgv_7_red.jpg}%
      \hfill
      \includegraphics[width=0.49\textwidth]{images_rec_foot_foot_osci_tgv_7_closeup.jpg}
      \\[-1em]
    \end{minipage}}
  \caption{MRI reconstruction results for different sampling rates for
    ``foot'' image ($512 \times 512$ pixels). The sampling rate in the
    second row and third row is $12.92\%$ and $14.74\%$,
    respectively. Parameter choice for $\ICTGV^{\osci}$:
    $\alpha_1=0.003, \beta_1=3\alpha_1, \gamma_1=0,
    \alpha_i=0.3\alpha_1, \beta_i=5\alpha_i, \gamma_i=0.35\alpha_i,
    i=2,\ldots,9$.}
  \label{imreconfoot}
\end{figure}

\begin{figure}
  \center{} 
  \begin{minipage}{1\textwidth}
    \centering
    \subfigure[ground truth]{%
      \begin{minipage}{0.49\linewidth}
        \includegraphics[width=0.49\textwidth]{images_rec_knee_knee_red.jpg}%
        \hfill
        \includegraphics[width=0.49\textwidth]{images_rec_knee_knee_closeup.jpg}
        \\[-1em]
      \end{minipage}}
  \end{minipage}
  
  \subfigure[TGV+Shearlet]{%
    \begin{minipage}{0.49\linewidth}
      \includegraphics[width=0.49\textwidth]{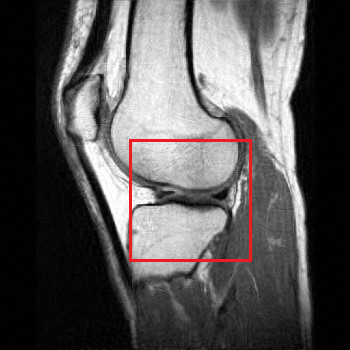}%
      \hfill
      \includegraphics[width=0.49\textwidth]{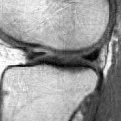}
      \\[-1em]
    \end{minipage}}\ \ 
  \subfigure[proposed model]{%
    \begin{minipage}{0.49\linewidth}
      \includegraphics[width=0.49\textwidth]{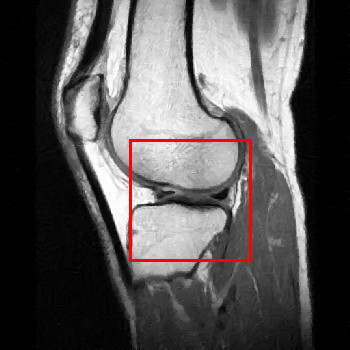}%
      \hfill
      \includegraphics[width=0.49\textwidth]{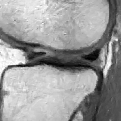}
      \\[-1em]
    \end{minipage}}
  
  \vspace*{-0.5em}
  \subfigure[TGV+Shearlet]{%
    \begin{minipage}{0.49\linewidth}
      \includegraphics[width=0.49\textwidth]{images_rec_knee_knee_tgvshearlet_7_red.jpg}%
      \hfill
      \includegraphics[width=0.49\textwidth]{images_rec_knee_knee_tgvshearlet_7_closeup.jpg}
      \\[-1em]
    \end{minipage}}\ \ 
  \subfigure[proposed model]{%
    \begin{minipage}{0.49\linewidth}
      \includegraphics[width=0.49\textwidth]{images_rec_knee_knee_osci_tgv_7_red.jpg}%
      \hfill
      \includegraphics[width=0.49\textwidth]{images_rec_knee_knee_osci_tgv_7_closeup.jpg}
      \\[-1em]
    \end{minipage}}
  \caption{MRI reconstruction results for different sampling rates for
    ``knee'' image ($350 \times 350$ pixels). The sampling rate in the
    second row and third row is $18.56\%$ and $21.12\%$,
    respectively. Parameter choice for $\ICTGV^{\osci}$:
    $\alpha_1=0.003, \beta_1=4.5\alpha_1, \gamma_1=0,
    \alpha_i=0.3\alpha_1, \beta_i=5\alpha_i, \gamma_i=0.45\alpha_i,
    i=2,\ldots,9$.}
  \label{imreconknee}
\end{figure}

\begin{figure}
  \center{} 
  \begin{minipage}{1\textwidth}
    \centering
    \subfigure[ground truth]{%
      \begin{minipage}{0.49\linewidth}
        \includegraphics[width=0.49\textwidth]{images_rec_brain2_humanbrain_red.jpg}%
        \hfill
        \includegraphics[width=0.49\textwidth]{images_rec_brain2_humanbrain_closeup.jpg}
        \\[-1em]
      \end{minipage}}
  \end{minipage}
  
  \subfigure[TGV+Shearlet]{%
    \begin{minipage}{0.49\linewidth}
      \includegraphics[width=0.49\textwidth]{images_rec_brain2_humanbrain_tgvshearlet_6_red.jpg}%
      \hfill
      \includegraphics[width=0.49\textwidth]{images_rec_brain2_humanbrain_tgvshearlet_6_closeup.jpg}
      \\[-1em]
    \end{minipage}}\ \ 
  \subfigure[proposed model]{%
    \begin{minipage}{0.49\linewidth}
      \includegraphics[width=0.49\textwidth]{images_rec_brain2_humanbrain_osci_tgv_6_red.jpg}%
      \hfill
      \includegraphics[width=0.49\textwidth]{images_rec_brain2_humanbrain_osci_tgv_6_closeup.jpg}
      \\[-1em]
    \end{minipage}}
  
  \vspace*{-0.5em}
  \subfigure[TGV+Shearlet]{%
    \begin{minipage}{0.49\linewidth}
      \includegraphics[width=0.49\textwidth]{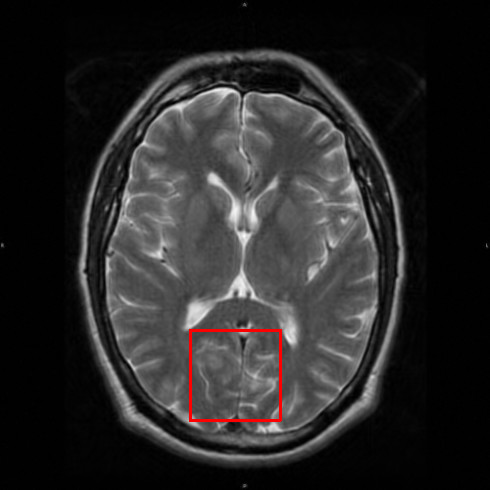}%
      \hfill
      \includegraphics[width=0.49\textwidth]{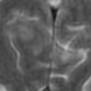}
      \\[-1em]
    \end{minipage}}\ \ 
  \subfigure[proposed model]{%
    \begin{minipage}{0.49\linewidth}
      \includegraphics[width=0.49\textwidth]{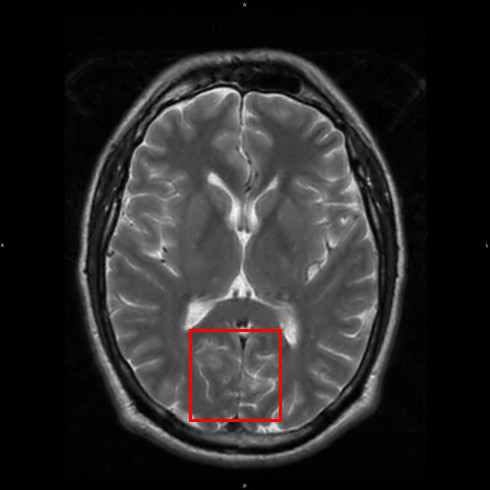}%
      \hfill
      \includegraphics[width=0.49\textwidth]{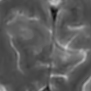}
      \\[-1em]
    \end{minipage}}
  \caption{MRI reconstruction results for different sampling rates for
    ``brain'' image ($490 \times 490$ pixels). The sampling rate in the
    second row and third row is $13.47\%$ and $15.35\%$,
    respectively. Parameter choice for $\ICTGV^{\osci}$:
    $\alpha_1=0.002, \beta_1=2.5\alpha_1, \gamma_1=0,
    \alpha_i=0.6\alpha_1, \beta_i=\alpha_i, \gamma_i=0.15\alpha_i,
    i=2,\ldots,9$.}
  \label{imreconbrain}
\end{figure}

\clearpage
\subsection{Comparison basic versus extended model}

See Figure~\ref{imreconfoot2}.

\begin{figure}
  \center
  \subfigure[ground truth]{%
    \begin{minipage}{0.3\linewidth}
      \includegraphics[width=\textwidth]{images_rec_foot_foot_orig_red.jpg}
      \\[\smallskipamount]
      \includegraphics[width=\textwidth]{images_rec_foot_foot_orig_closeup.jpg}
      \\[-1em]
    \end{minipage}}\quad
  \subfigure[basic model ($m=9$)]{%
    \begin{minipage}{0.3\linewidth}
      \includegraphics[width=\textwidth]{images_rec_foot_foot_osci_tgv_7_red.jpg}
      \\[\smallskipamount]
      \includegraphics[width=\textwidth]{images_rec_foot_foot_osci_tgv_7_closeup.jpg}
      \\[-1em]
    \end{minipage}}\ \ \subfigure[extended model ($m=17$)]{%
    \begin{minipage}{0.3\linewidth}
      \includegraphics[width=\textwidth]{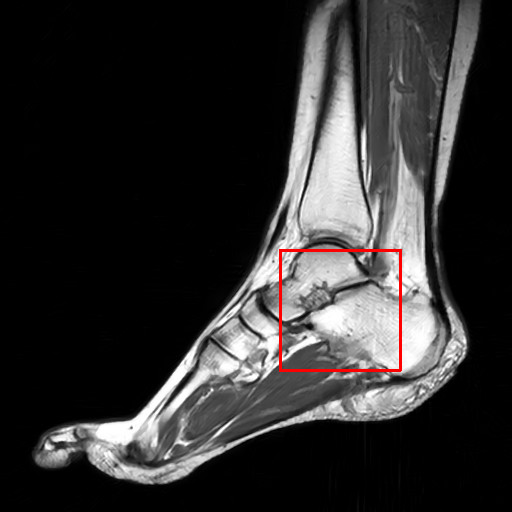}
      \\[\smallskipamount]
      \includegraphics[width=\textwidth]{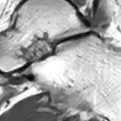}
      \\[-1em]
    \end{minipage}}
  \caption{Comparison of $\ICTGV^{\osci}$-regularized MRI
    reconstruction for ``foot'' image ($350 \times 350$ pixels) with
    different parameters. Parameter choice for (b): $m=9$,
    $\alpha_1=0.003, \beta_1=3\alpha_1, \gamma_1=0,
    \alpha_i=0.3\alpha_1, \beta_i=5\alpha_i, \gamma_i=0.35\alpha_i,
    i=2,\ldots,9$; Parameter choice for (c): $m=17$,
    $\alpha_1=0.003, \beta_1=3\alpha_1, \gamma_1=0,
    \alpha_i=0.3\alpha_1, \beta_i=5\alpha_i, \gamma_i=0.35\alpha_i$ for
    $i=2,\ldots,9$ and
    $\alpha_i=0.15\alpha_1, \beta_i=2.5\alpha_i, \gamma_i=0.35\alpha_i$
    for $i=10,\ldots,17$.}
	\label{imreconfoot2}
\end{figure}

\clearpage
\section{%
  Information on parameter choice and implementations}

\subsection{The ICTGV$^{\text{osci}}$ methods}

For all experiments, the parameter choice for $\ICTGV^{\osci}$ was
done by manually defining reasonable ranges of parameters along with
step-lengths and performing an exhaustive search over the combinations
of each parameter in order to determine the parameters that yield the
best result in terms of PSNR. Table~\ref{table:para_denoise} shows the
parameter ranges and step lengths that were used for this process
depending on the application. Here, we tuned $(\alpha_1,\beta_1)$, the
parameters for the TGV-part in $\ICTGV^{\osci}$ as well as
$(\alpha_i,\beta_i,\gamma_i)$ for all $i=2,\ldots,m$, the parameters
for the texture parts.
The parameters for the visually-optimal results were obtained in the
same manner, i.e., by visual inspection of the outcome of the
algorithms for a manually-determined range of parameters.

\begin{table}
  \center{}
  \scalebox{0.91}{%
  \begin{tabular}{|c|c|c|c|c|c|c|}
    \hline
    \hline
    \multicolumn{2}{|c|}{application} 
    &$\alpha_1$ &$\beta_1$&$\alpha_i$&$\beta_i$& $\gamma_i$\\
    \hline
    \multirow{2}{*}{denoising}
    & $\sigma=0.05$ & $[0.04,0.06]$ & $[0.5\alpha_1,1.2\alpha_1]$&$[0.5\alpha_1,1.2\alpha_1]$ &$[0.4\alpha_i,1.0\alpha_i]$ & $[0.08\alpha_i,0.15\alpha_i]$\\
    \cline{2-7}& $\sigma=0.1$ & $[0.08,0.12]$ & $[0.5\alpha_1,1.3\alpha_1]$&$[0.5\alpha_1,1.3\alpha_1]$ &$[0.4\alpha_i,1.0\alpha_i]$ & $[0.09\alpha_i,0.2\alpha_i]$\\
    \hline
    \hline
    \multicolumn{2}{|c|}{inpainting}
    & $[0.02,0.06]$ & $[0.3\alpha_1,1.5\alpha_1]$&$[0.5\alpha_1,1.2\alpha_1]$ &\tabincell{c}{$[0.2\alpha_i,1.0\alpha_i]$\\
    ($[1.0\alpha_i,4.0\alpha_i]$\\ for ``fish'')} & $[0.01\alpha_i,0.05\alpha_i]$\\
    \hline
    \hline
    \multicolumn{2}{|c|}{MRI reconstruction}
    & $[0.001,0.004]$ & $[1.5\alpha_1,5.0\alpha_1]$&$[0.1\alpha_1,1.0\alpha_1]$ &$[0.5\alpha_i,6.0\alpha_i]$ & $[0.1\alpha_i,0.5\alpha_i]$\\
    \hline
  \end{tabular}}
\caption{\label{table:para_denoise} Parameters ranges used for manual tuning 
  of $\ICTGV^{\osci}$ for the PSNR-optimal case.
  The ranges were discretized with a step-length of $0.01$ for 
  $\alpha_1$ (denoising of $\sigma=0.05$: $0.005$ and MRI reconstruction: $0.001$), $0.1$ for the factors 
  in $\beta_1, \alpha_i, \beta_i$ (inpainting of ``fish'': $0.5$ for $\beta_i$ and MRI reconstruction: 
  $0.5$ for $\beta_1$ and $\beta_i$) and $0.01$ for the factors in $\gamma_i$ (MRI reconstruction: 
  $0.05$).}
\end{table}

\subsection{Other methods}

\begin{table}
  \center{}
  \begin{tabular}{|c|c|c|}
    \hline
    \hline
    & method & code source \\
    \hline
    \cline{1-3}1& $\TV$-$G$-norm & \tabincell{c}{\url{https://www.mathworks.com/matlabcentral/fileexchange/} \\ \url{35462-bregman-cookbook}}\\
    \cline{1-3}2& $\TV$-$H^{-1}$ & \tabincell{c}{\url{http://www.numerical-tours.com/matlab/
                                   } \\ \url{inverse_6_image_separation/}} \\
    \cline{1-3}3& NLTV &
                         \tabincell{c}{\url{https://htmlpreview.github.io/?https://github.com/} \\ \url{xbresson/old_codes/blob/master/codes.html}}\\
    \cline{1-3}4& BM3D & \url{https://www.cs.tut.fi/~foi/GCF-BM3D/}\\
    \cline{1-3}5& Fra+LDCT & \url{https://www.math.ust.hk/~jfcai/}\\
    \cline{1-3}6& TGV+Shearlet & \url{http://www.montana.edu/jqin/research.html}\\
    \cline{1-3}7&ICTGV & code provided by the author Martin Holler\\
    \cline{1-3}8&TGV & own implementation\\
    \hline
  \end{tabular}
  \caption{\label{table:code} %
    Sources of the MATLAB implementations for the methods used for comparison.
    (Online resources accessed on 02/07/2018.)
  }
\end{table}

The comparisons with other methods were performed with MATLAB
implementations provided by the authors or from publicly available
toolboxes. We thank the authors for sharing their code. The sources of
the implementations are summarized in Table~\ref{table:code}. The
parameter tuning for these methods was performed in analogy to
$\ICTGV^{\osci}$, i.e., exhaustive search over a manually-determined
set of parameters and visual inspection. In particular, we chose the
parameter ranges such that the results for each method were of
comparable quality as reported in the provided codes. In the
following, we give details on the exact parameters that were chosen in
the computations.

\textbf{Remarks on the implementation and parameters for each method:}
\par
1. One can obtain the ``Bregman Cookbook'' toolbox provided by Jerome
Gilles from the MATLAB Central website, where the $\TV$-$G$-norm
method is located in the folder ``Examples''. With this
implementation, however, we did not obtain the expected
cartoon-texture decomposition result
even though we have tuned many parameters. For this reason, we decided
to write the decomposition code by ourselves using a first-order
primal-dual method.  This implementation is available to interested
readers on demand. We set $\lambda=0.6$ and $\mu=0.5$ for the model in [3].

2. The website ``Numerical tours in MATLAB'' provided by Gabriel
Peyr\'e gives a review of many cartoon+texture variational image
decomposition models and a toolbox including the $\TV$-$H^{-1}$
decomposition approach that we used for comparison. We set the
parameter $\lambda=20$ for the decomposition case.

3. The non-local TV code for denoising was obtained from Xavier
Bresson's GitHub page (Section ``An Algorithm for Nonlocal TV
Minimization''). For the case of noise level
$\sigma=0.05$, we chose $\mu=30$ for visually-optimal and $\mu=50$
(``zebra'': $\mu=55$) for PSNR-optimal results, respectively. For
$\sigma=0.1$, we chose $\mu=15$ (``parrots'': $\mu=10$) for
visually-optimal and $\mu=20$ (``parrots'': $\mu=15$) for PSNR-optimal
results, respectively.

4. This website provides the BM3D code in terms of precompiled MEX
files and we use the code directly without parameter tuning.

5. One can find the source code for framelet/LDCT on Jianfeng Cai's
homepage (Journal Paper No.~34 ``Split Bregman Methods and Frame Based
Image Restoration''). For the decomposition experiments for the
``synthetic'' image (see Figure~4 in the main publication), we set
$\mu_1=0.1$, $\mu_2=0.04$. The optimal parameters for denoising were
chosen according to Table~\ref{table:fra_ldct}.

\begin{table}
	\center{}
	\begin{tabular}{|c|c|c|c|}
		\hline
		\hline
		noise level & image & visually-optimal & PSNR-optimal\\
		\hline
		\multirow{3}{*}{$\sigma=0.05$}
		& barbara & $\mu_1=0.1$, $\mu_2=0.06$ & $\mu_1=0.08$, $\mu_2=0.04$\\
		\cline{2-4}& zebra & $\mu_1=0.1$, $\mu_2=0.06$ & $\mu_1=0.07$, $\mu_2=0.04$ \\
		\cline{2-4}& parrots & $\mu_1=0.13$, $\mu_2=0.06$ & $\mu_1=0.09$, $\mu_2=0.05$ \\
		\hline
		\multirow{3}{*}{$\sigma=0.1$}
		& barbara  & $\mu_1=0.25$, $\mu_2=0.11$ & $\mu_1=0.18$, $\mu_2=0.1$\\
		\cline{2-4}& zebra & $\mu_1=0.25$, $\mu_2=0.11$ & $\mu_1=0.16$, $\mu_2=0.1$\\
		\cline{2-4}& parrots & $\mu_1=0.25$, $\mu_2=0.11$ & $\mu_1=0.2$, $\mu_2=0.11$\\
		\hline
	\end{tabular}
	\caption{\label{table:fra_ldct} The optimal choice of parameters for the Fra+LDCT method for denoising.}
\end{table}

The parameters for inpainting case are given as: $\mu_1=0.1$, $\mu_2=0.05$~(``barbara'': $\mu_2=0.04$).

6. One can obtain the TGV+Shearlet code for MRI reconstruction on Jing
Qin's homepage. We also tuned the parameters, however, the performance
in terms of PSNR was very close to that with the original parameters
provided by the authors, so we used the code directly.

7. Martin Holler, who is author of the paper ``On Infimal Convolution
of TV Type Functionals and Applications to Video and Image
Reconstruction'', gave us the code and we extended it to the 8
directions used in the $\ICTGV^{\osci}$ models. We also tuned the
parameters manually to obtain an optimal set of parameters (summarized
in Table~\ref{table:ictgv}).
\begin{table}
  \center{}
  \begin{tabular}{|c|c|c|c|}
    \hline
    \hline
    noise level & image & visually-optimal & PSNR-optimal\\
    \hline
    \multirow{3}{*}{$\sigma=0.05$}
		& barbara & $\lambda=0.09$, $\alpha=1.6$ & $\lambda=0.12$, $\alpha=1.5$\\
    \cline{2-4}& zebra & $\lambda=0.09$, $\alpha=1.6$ & $\lambda=0.12$, $\alpha=1.5$ \\
    \cline{2-4}& parrots & $\lambda=0.07$, $\alpha=1.6$ & $\lambda=0.09$, $\alpha=1.5$ \\
    \hline
    \multirow{3}{*}{$\sigma=0.1$}
		& barbara  & $\lambda=0.04$, $\alpha=1.5$ & $\lambda=0.05$, $\alpha=1.4$\\
    \cline{2-4}& zebra & $\lambda=0.04$, $\alpha=1.5$ & $\lambda=0.05$, $\alpha=1.4$\\
    \cline{2-4}& parrots & $\lambda=0.03$, $\alpha=1.6$ & $\lambda=0.04$, $\alpha=1.5$\\
    \hline
  \end{tabular}
  \caption{\label{table:ictgv} The optimal choice of parameters for ICTGV for denoising.}
\end{table}

\enlargethispage{1em}
8. We set $\alpha=0.06$ and $\beta=2\alpha$ for denoising and $\alpha=0.01$ and $\beta=2\alpha$ for inpainting of the TGV model.
 
\end{document}